\documentclass[12pt]{article}
\usepackage{amssymb,amsmath,amsthm, tikz,multirow}
\usepackage{enumerate} 
\usepackage{hyperref} 
\usepackage [autostyle, english = british]{csquotes}
\usepackage [english]{babel}
\usepackage{mathtools}
\usepackage{bm}
\usepackage{hhline}
\usepackage{algorithm}
\usepackage[noend]{algpseudocode}
\usepackage{caption}
\usepackage{subcaption}

\usepackage[export]{adjustbox}

\usetikzlibrary{calc,arrows, arrows.meta, math} 
\usepackage{array} 

\algnewcommand{\algorithmicand}{\textbf{and }}
\algnewcommand{\algorithmicor}{\textbf{or }}
\algnewcommand{\OR}{\algorithmicor}
\algnewcommand{\AND}{\algorithmicand}
\algnewcommand{\var}{\texttt}
\algdef{SE}[SUBALG]{Indent}{EndIndent}{}{\algorithmicend\ }%
\algtext*{Indent}
\algtext*{EndIndent}

\DeclareCaptionFormat{myformat}{#3}
\newcolumntype{M}[1]{>{\centering\arraybackslash}m{#1}}

\newtheorem{theorem}{Theorem}

\newtheorem*{theorem*}{Theorem}
\newtheorem*{proposition*}{Proposition}

\theoremstyle{definition}
\newtheorem*{definition*}{Definition}
\newtheorem*{case*}{Case}
\newtheorem*{subcase*}{Subcase}
\newtheorem*{subsubcase*}{Subsubcase}

\theoremstyle{plain}
\newtheorem{thm}{Theorem}[section]
\newtheorem{lem}[thm]{Lemma}
\newtheorem{prop}[thm]{Proposition}

\theoremstyle{definition}

\theoremstyle{remark}

\numberwithin{equation}{section}

\newcommand{\bvert}{\vrule width 2pt}

\newcommand{\AVC}{\text{AVC}}
\DeclarePairedDelimiter{\floor}{\lfloor}{\rfloor}

\newcommand{\newpart}{\subsubsection*}

\newcommand{\quotes}[1]{``#1''}

\newcommand{\arcThroughThreePoints}[4][]{
\coordinate (middle1) at ($(#2)!.5!(#3)$);
\coordinate (middle2) at ($(#3)!.5!(#4)$);
\coordinate (aux1) at ($(middle1)!1!90:(#3)$);
\coordinate (aux2) at ($(middle2)!1!90:(#4)$);
\coordinate (center) at ($(intersection of middle1--aux1 and middle2--aux2)$);
\draw[#1] 
 let \p1=($(#2)-(center)$),
      \p2=($(#4)-(center)$),
      \n0={veclen(\p1)},       
      \n1={atan2(\y1,\x1)}, 
      \n2={atan2(\y2,\x2)},
      \n3={\n2>\n1?\n2:\n2+360}
    in (#2) arc(\n1:\n3:\n0);
}

\providecommand{\keywords}[1]{\noindent \textit{Keywords:} #1}

\title{Rational Angles and Tilings of the Sphere by Congruent Quadrilaterals}
\author{Ho Man CHEUNG \\ email: hmcheungae@connect.ust.hk \\[2ex] Hoi Ping LUK \\ email: hoi@connect.ust.hk}

\begin{document}
\maketitle

\begin{abstract} We apply Diophantine analysis to classify edge-to-edge tilings of the sphere by congruent almost equilateral quadrilaterals (i.e., edge combination $a^3b$). Parallel to a complete classification by Cheung, Luk and Yan, the method implemented here is more systematic and applicable to other related tiling problems. We also provide detailed geometric data for the tilings. \\

\end{abstract}


\section{Introduction}

We study edge-to-edge tilings of the sphere by congruent polygons such that each vertex has degree $\ge3$. It is well known that the polygons in these tilings are triangle, quadrilateral or pentagon. The classification of tilings of the sphere by congruent triangles was started by Sommerville \cite{so} in 1924 and completed by Ueno and Agaoka \cite{ua} in 2002. The classification for tilings of the sphere by congruent pentagons has seen considerable progress \cite{awy, ay, gsy, ly1, ly2, wy, wy2, wy3, yan, yan2} and is expected to be completed in the near future.

Akama, Sakano, van Cleemput \cite{ak, as, avc} started a preliminary classification for tilings of the sphere by congruent quadrilaterals which are equilateral or can be subdivided into two congruent triangles. It remains to classify the tilings by congruent quadrilaterals with exactly three equal edges ($a^3b$, first picture of Figure \ref{StdQuad}) and by congruent quadrilaterals with exactly two equal edges ($a^2bc$, second picture). Ueno and Agaoka \cite{ua2} studied some special cases of the tilings by congruent $a^3b$ quadrilaterals. Their work is indicative of many challenges in the classification. In 2022, Cheung, Luk and Yan \cite{cly} gave a complete classification for tilings of the sphere by congruent quadrilaterals as well as a modernised classification for the tilings by congruent triangles. 

We call a quadrilateral with edge combination $a^3b$ {\em almost equilateral}, where $a$-edge and $b$-edge are assumed to have different lengths. The angles are indicated in the first picture of Figure \ref{StdQuad}, likewise for the $a^2bc$ quadrilateral in the second picture. These standard configurations are implicitly assumed in this paper. We call an angle {\em rational} if its value is a rational multiple of $\pi$. Otherwise we call the angle {\em non-rational}.

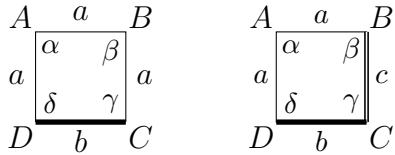
\begin{figure}[htp] 
\centering
\begin{tikzpicture}

\begin{scope}[] 

\draw
	(0,0) -- (0,1.2) -- (1.2,1.2) -- (1.2, 0.0);

\draw[ line width=2]
	(0,0) -- (1.2,0);

\node at (-0.2, -0.2) {$D$};
\node at (-0.2, 1.4) {$A$};
\node at (1.4, 1.4) {$B$};
\node at (1.4, -0.2) {$C$};

\node at (0.2,1) {\small $\alpha$};
\node at (1,0.925) {\small $\beta$};
\node at (1,0.25) {\small $\gamma$};
\node at (0.2,0.25) {\small $\delta$};

\node at (0.6, -0.25) {$b$};
\node at (-0.25, 0.6) {$a$};
\node at (1.45, 0.6) {$a$};
\node at (0.6, 1.45) {$a$};

\end{scope}

\begin{scope}[xshift = 3.2cm]

\draw
	(0,0) -- (1.2,0) -- (1.2,1.2) -- (0,1.2) -- cycle;

\draw[double, line width=0.6]
	(1.2,0) -- (1.2,1.2);

\draw[line width=2]
	(0,0) -- (1.2,0);

\node at (-0.2, -0.2) {$D$};
\node at (-0.2, 1.4) {$A$};
\node at (1.4, 1.4) {$B$};
\node at (1.4, -0.2) {$C$};

\node at (0.2,1) {\small $\alpha$};
\node at (1,0.95) {\small $\beta$};
\node at (1,0.25) {\small $\gamma$};
\node at (0.2,0.25) {\small $\delta$};

\node at (0.6,1.4) {\small $a$};
\node at (0.6,-0.25) {\small $b$};
\node at (-0.2,0.6) {\small $a$};
\node at (1.4,0.6) {\small $c$};

\end{scope}

\end{tikzpicture}
\caption{Quadrilaterals with edge combinations $a^3b, a^2bc$}
\label{StdQuad}
\end{figure}

The main purpose of this paper is to give an alternative classification for tilings of the sphere by congruent almost equilateral quadrilaterals. The key is Diophantine analysis in the following situations,
\begin{enumerate}
\item If all angles are rational, then we determine the angle values by finding all rational solutions to a trigonometric Diophantine equation which all angles must satisfy. 
\item If some angles are non-rational, then we determine all angle combinations at vertices by solving a related system of linear Diophantine equations and inequalities.
\end{enumerate}

Despite the complete classification in \cite{cly}, techniques in this paper have their own independent significance. Coolsaet \cite{co} discovered the trigonometric Diophantine equation relating the angles of convex almost equilateral quadrilateral. Myerson \cite{my} found the rational solutions to the equation. Based on these two works, we made two major advancements. The first is extending the trigonometric Diophantine equation to general (not necessarily convex) almost equilateral quadrilaterals. The second is establishing a technique to determine all angle combinations at vertices using the constraint of non-rational angles. This technique is based on the study in \cite{ly}.

Historically, trigonometric Diophantine equations have been closely connected to many geometric situations. Conway and Jones \cite{cj} have opened doors to the exploration of many interesting geometry problems. Notable work can be seen in \cite{kkpr, my, lac}.

In contrast to \cite{cly}, there are two significant advantages in our approach. Firstly, arguments in this paper are more systematic whereas those in \cite{cly} are often sophisticated and improvised. Secondly, most techniques here can be computerised. In that regard, our approach is apparently more advantageous in exhaustive search and more likely to be applied to other similar problems, such as the study of non-edge-to-edge tilings of the sphere. Promising signs of such proposal can be seen in the families of non-edge-to-edge tilings by congruent triangles obtained in this paper as degenerated cases of the tilings by quadrilaterals, which supplement the discoveries by Dawson \cite{da}.

Another feature of this paper is the extrinsic geometric data of tilings, namely the formulae for the angles and edge lengths. We include the data for tilings by congruent almost equilateral quadrilaterals and the tilings by congruent $a^2bc$ quadrilaterals. Thereby we demonstrate the relation between these tilings via edge reduction. The data are intended for wider audience such as engineers, designers and architects.

The main result of this paper is stated as follows.
\begin{theorem*} Tilings of the sphere by congruent almost equilateral quadrilaterals are earth map tilings $E$ and their flip modifications, $E^{\prime}, E^{\prime\prime}$, and rearrangement $E^{\prime\prime\prime}$, and isolated earth map tilings, $S1, S2, S3, S^{\prime}3, S5$, and special tilings $QP_6, S4,S6$. 
\end{theorem*}

The tilings in the main theorem are presented in Figure \ref{SphericalTilings}. For the earth map tiling $E$ and its flip modifications $E^{\prime}, E^{\prime\prime}$ and rearrangement $E^{\prime\prime\prime}$, we illustrate the ones with $f=10$, i.e., $E_{10}, E^{\prime}_{10}, E^{\prime\prime}_{10}, E^{\prime\prime\prime}_{10}$.  

\begin{figure}
\begin{subfigure}{0.25\textwidth}
\centering
		\adjustbox{trim=\dimexpr.5\Width-15mm\relax{} \dimexpr.5\Height-15mm\relax{}  \dimexpr.5\Width-15mm\relax{} \dimexpr.5\Height-15mm\relax{} ,clip}{\includegraphics[height=6cm]{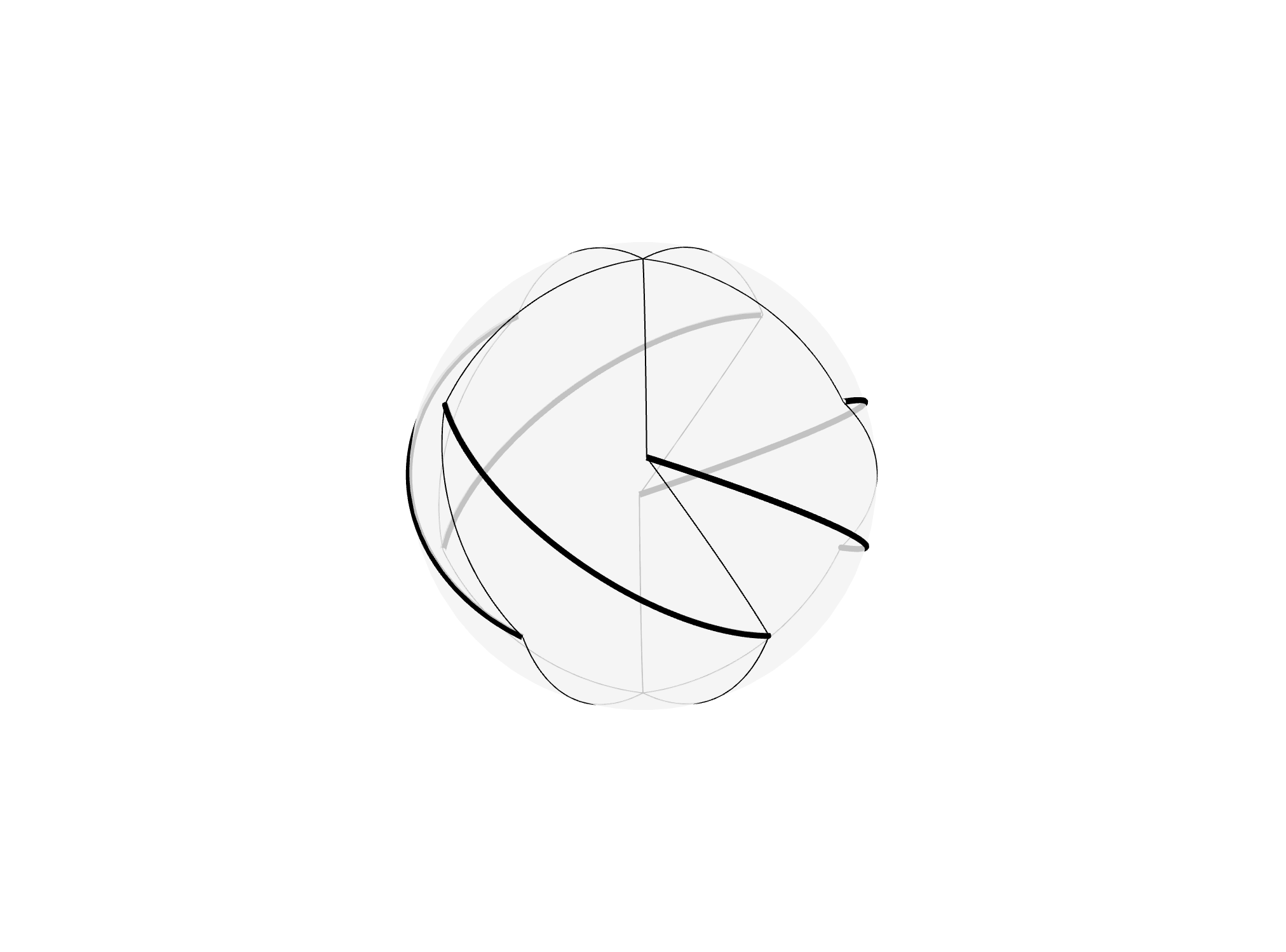}} 
\caption*{$E=E_{10}$} \hfill
 \end{subfigure} 
\begin{subfigure}{0.24\textwidth}
\centering
	\adjustbox{trim=\dimexpr.5\Width-15mm\relax{} \dimexpr.5\Height-15mm\relax{}  \dimexpr.5\Width-15mm\relax{} \dimexpr.5\Height-15mm\relax{},clip}{\includegraphics[height=6cm]{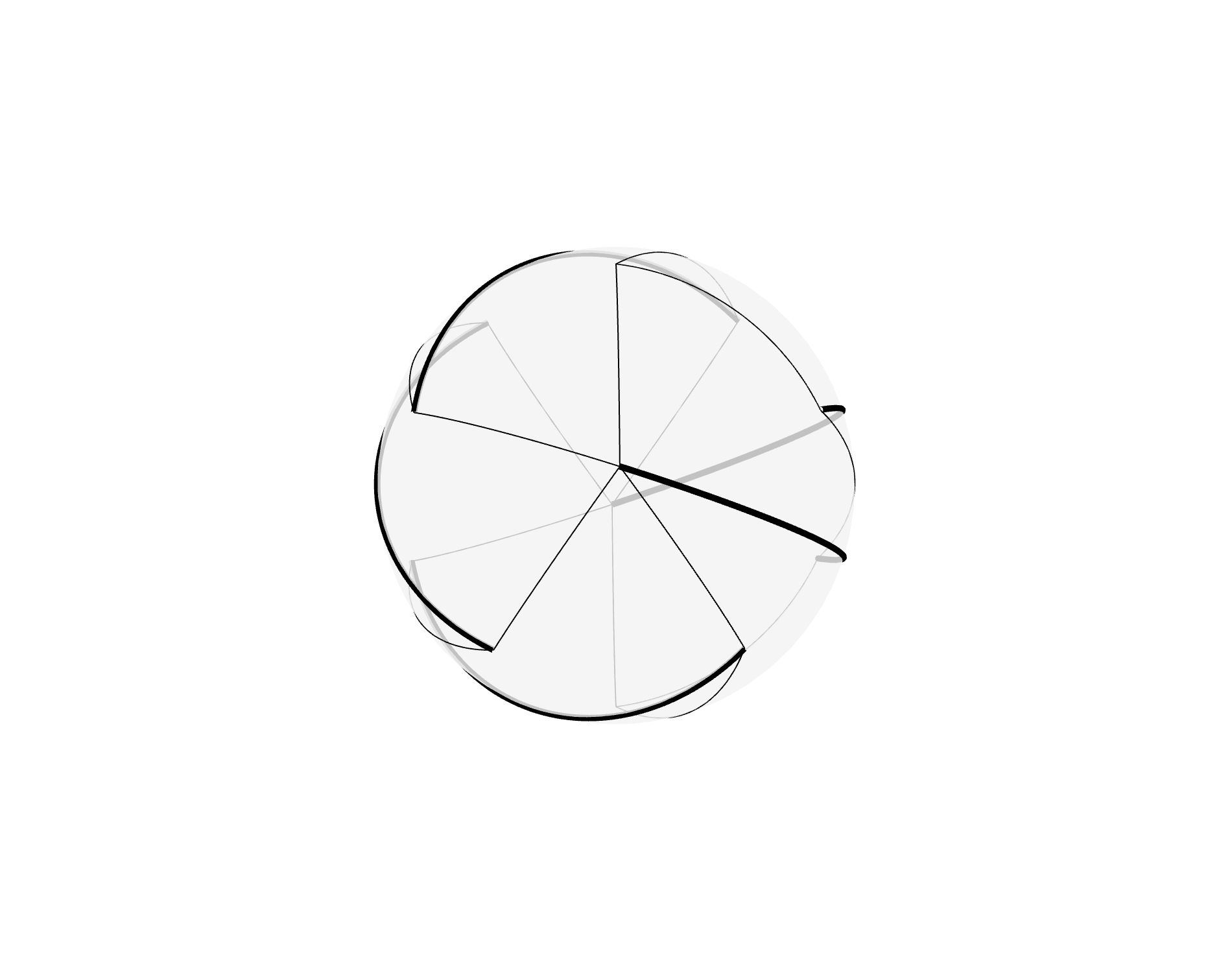}} 
\caption*{$E^{\prime}=E^{\prime}_{10}$} \hfill
 \end{subfigure} 
\begin{subfigure}{0.24\textwidth}
\centering
		\adjustbox{trim=\dimexpr.5\Width-15mm\relax{} \dimexpr.5\Height-15mm\relax{}  \dimexpr.5\Width-15mm\relax{} \dimexpr.5\Height-15mm\relax{} ,clip}{\includegraphics[height=6cm]{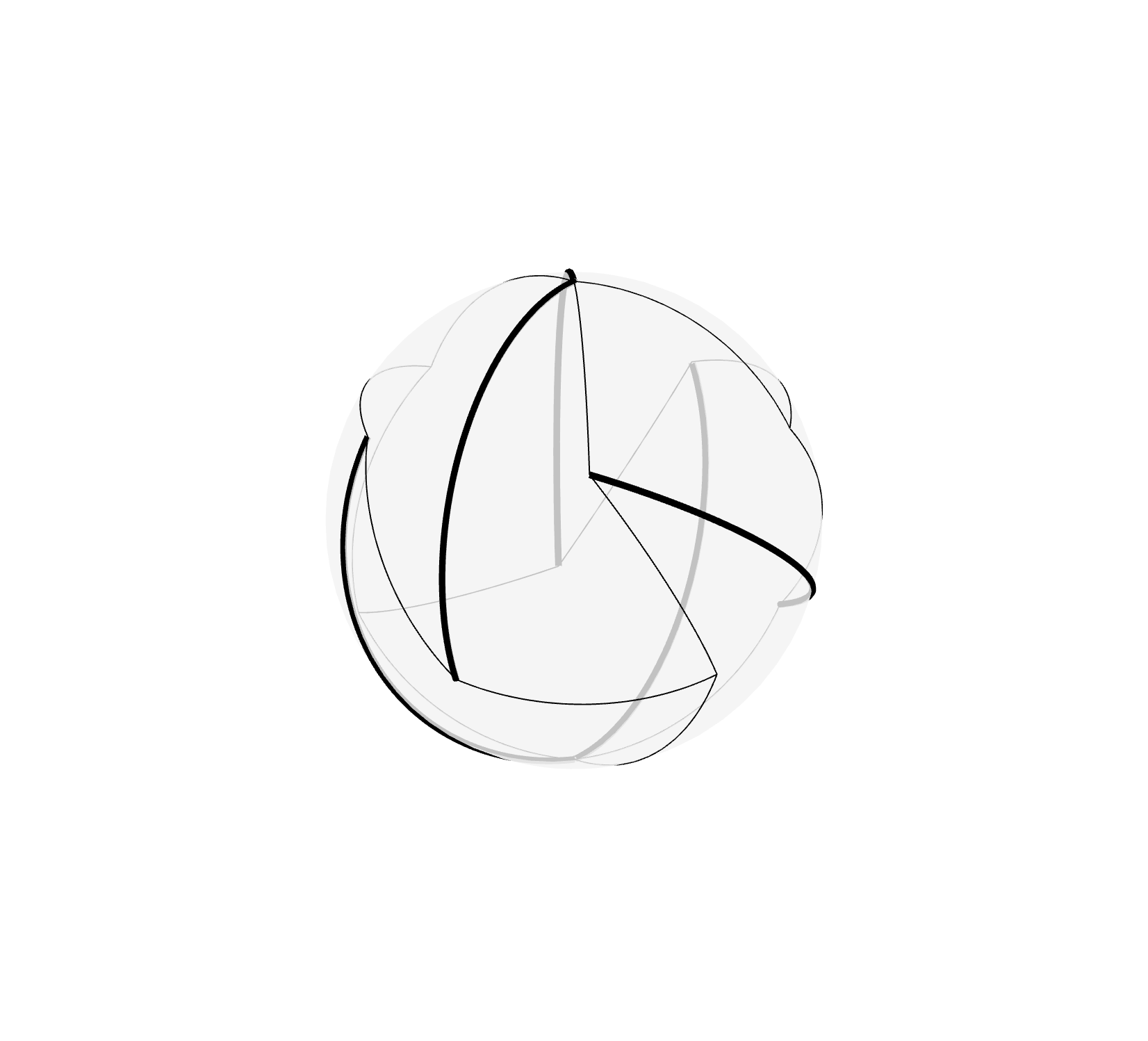}}  
\caption*{$E^{\prime\prime}=E^{\prime\prime}_{10}$} \hfill
 \end{subfigure} 
\begin{subfigure}{0.24\textwidth}
\centering
		\adjustbox{trim=\dimexpr.5\Width-15mm\relax{} \dimexpr.5\Height-15mm\relax{}  \dimexpr.5\Width-15mm\relax{} \dimexpr.5\Height-15mm\relax{} ,clip}{\includegraphics[height=6cm]{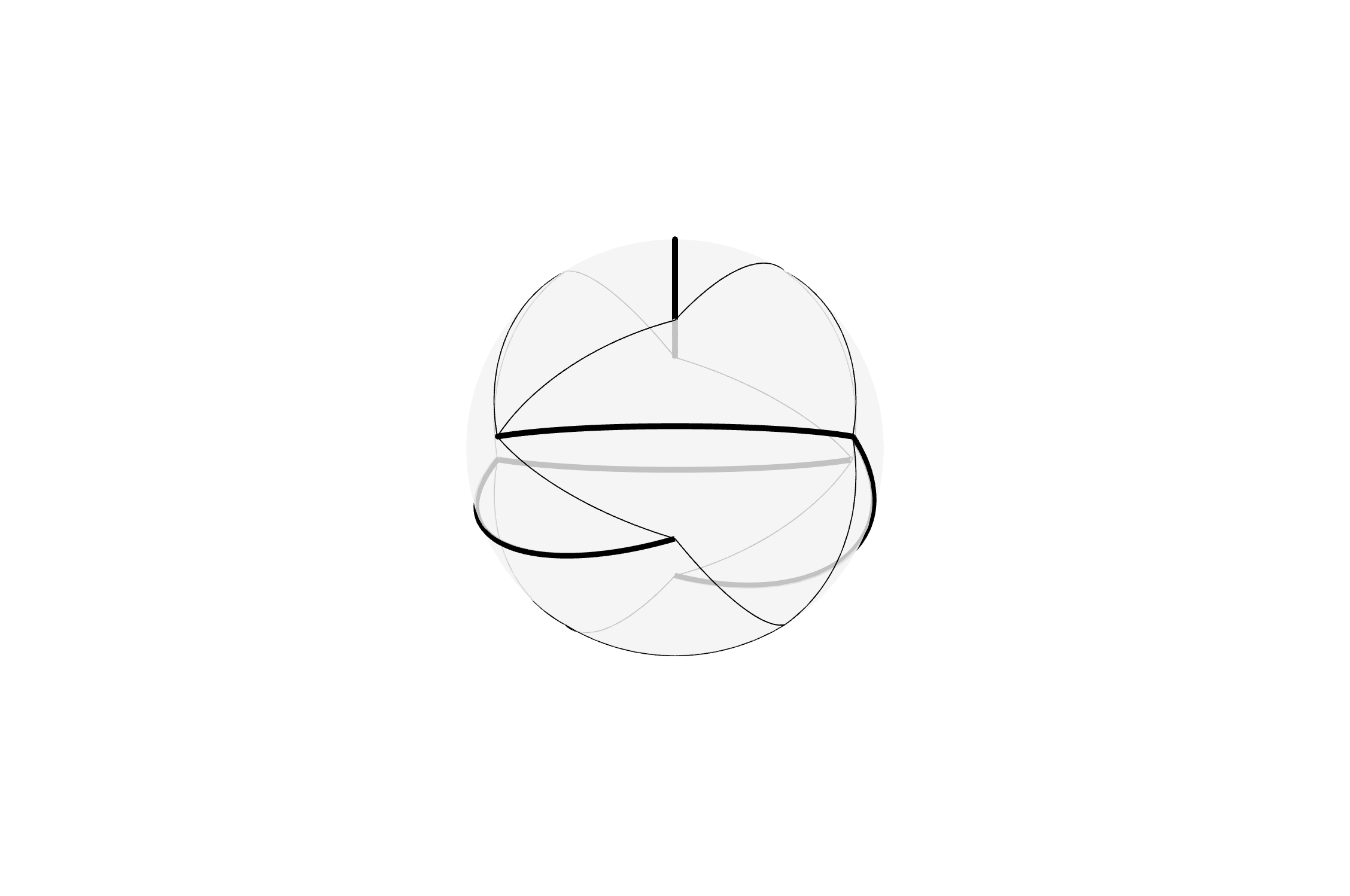}}  
\caption*{$E^{\prime\prime\prime}=E^{\prime\prime\prime}_{10}$} \hfill
\end{subfigure} 

\hfill

\begin{subfigure}[]{0.34\textwidth}
	\centering
		\adjustbox{trim=\dimexpr.5\Width-15mm\relax{} \dimexpr.5\Height-15mm\relax{}  \dimexpr.5\Width-15mm\relax{} \dimexpr.5\Height-15mm\relax{} ,clip}{\includegraphics[height=6cm]{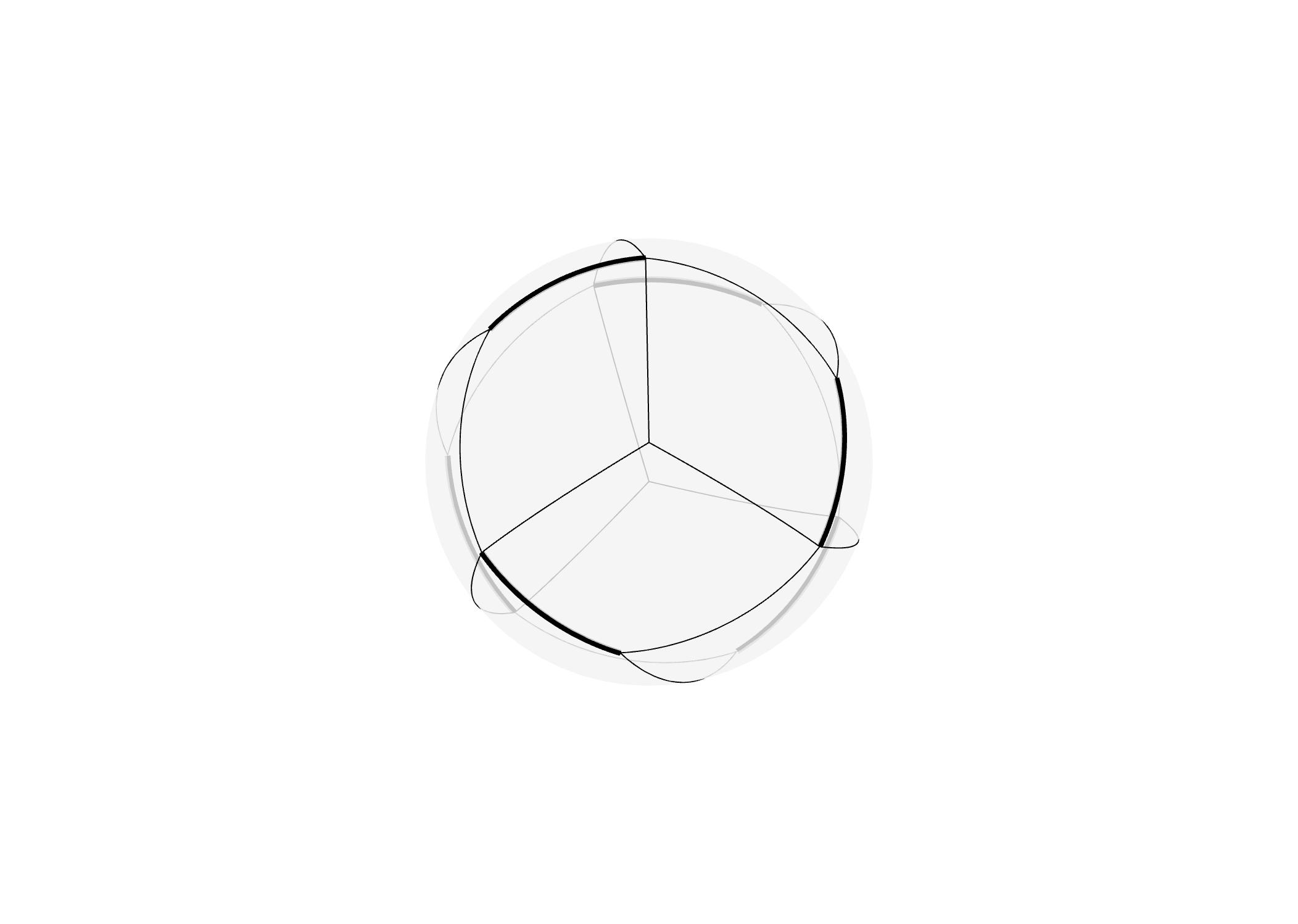}}  
\caption*{$S1=S_{12}1$} \hfill
 \end{subfigure} 
\begin{subfigure}[]{0.26\textwidth}
	\centering
		\adjustbox{trim=\dimexpr.5\Width-15mm\relax{} \dimexpr.5\Height-15mm\relax{}  \dimexpr.5\Width-15mm\relax{} \dimexpr.5\Height-15mm\relax{} ,clip}{\includegraphics[height=6cm]{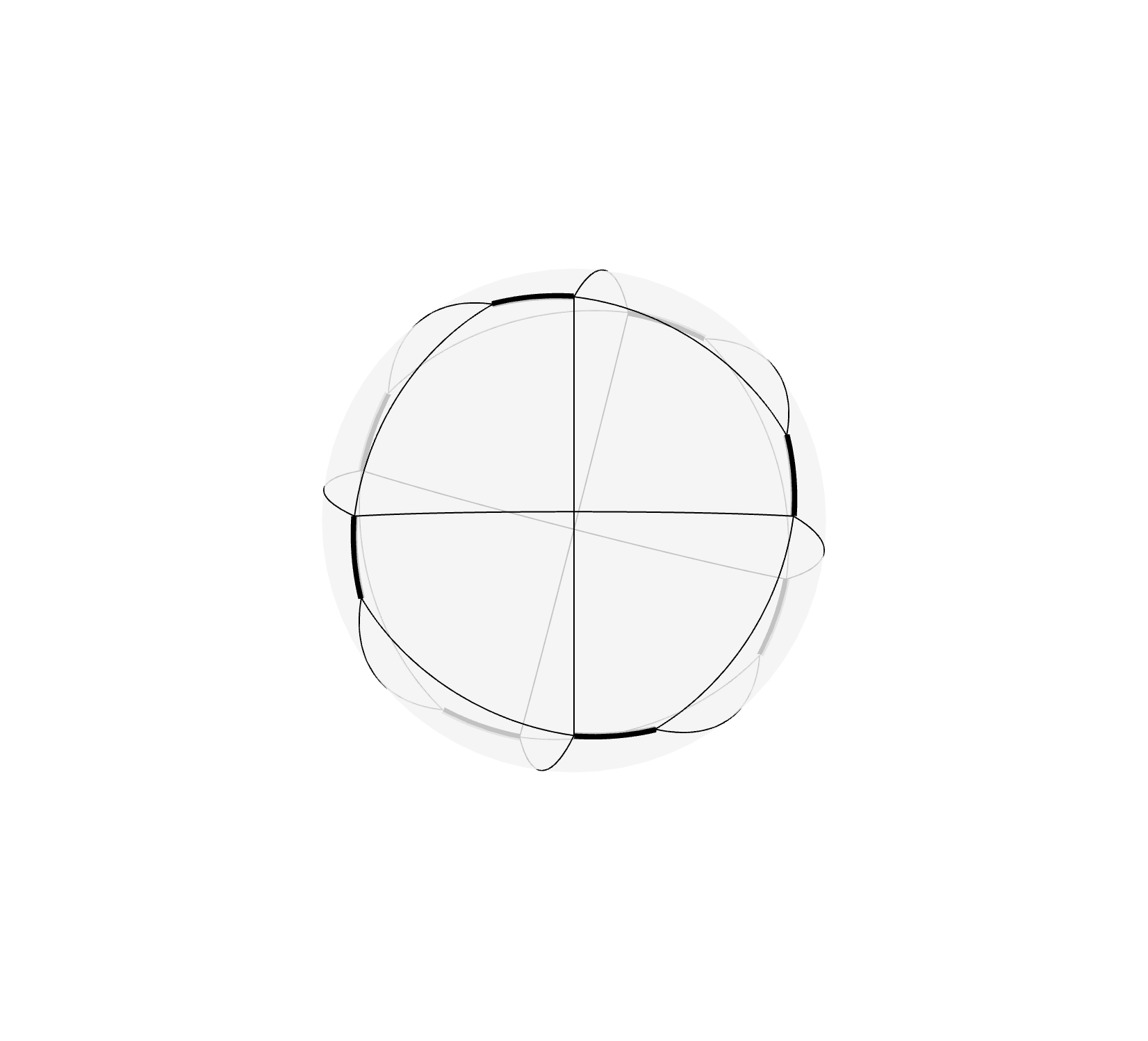}} 
\caption*{$S1=S_{16}1$} \hfill
 \end{subfigure} 
\begin{subfigure}[]{0.34\textwidth}
	\centering
		\adjustbox{trim=\dimexpr.5\Width-15mm\relax{} \dimexpr.5\Height-15mm\relax{}  \dimexpr.5\Width-15mm\relax{} \dimexpr.5\Height-15mm\relax{} ,clip}{\includegraphics[height=6cm]{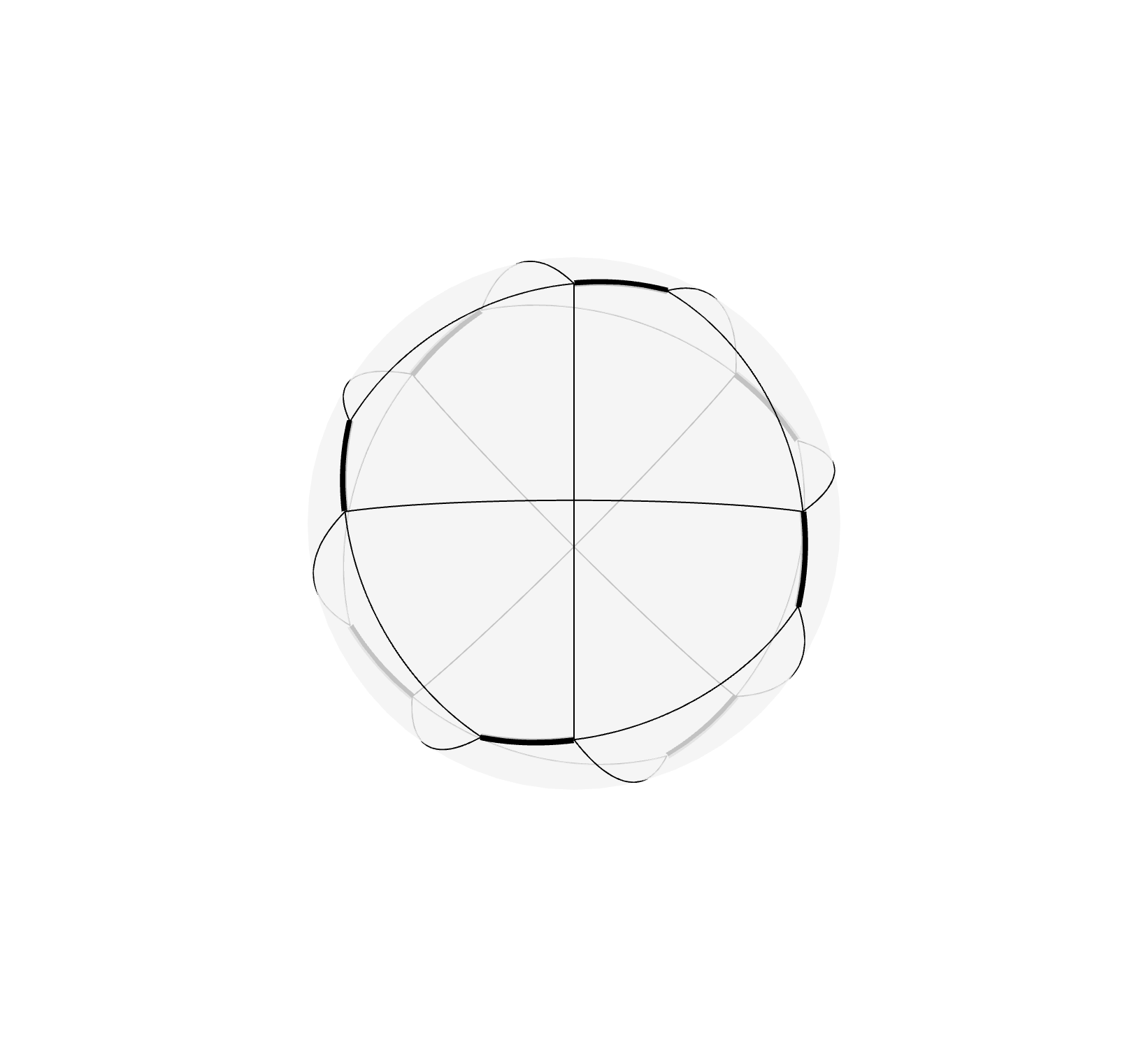}} 
\caption*{$S2$} \hfill
 \end{subfigure} 

\hfill

\begin{subfigure}{0.34\textwidth}
	\centering
		\adjustbox{trim=\dimexpr.5\Width-15mm\relax{} \dimexpr.5\Height-15mm\relax{}  \dimexpr.5\Width-15mm\relax{} \dimexpr.5\Height-15mm\relax{} ,clip}{\includegraphics[height=6cm]{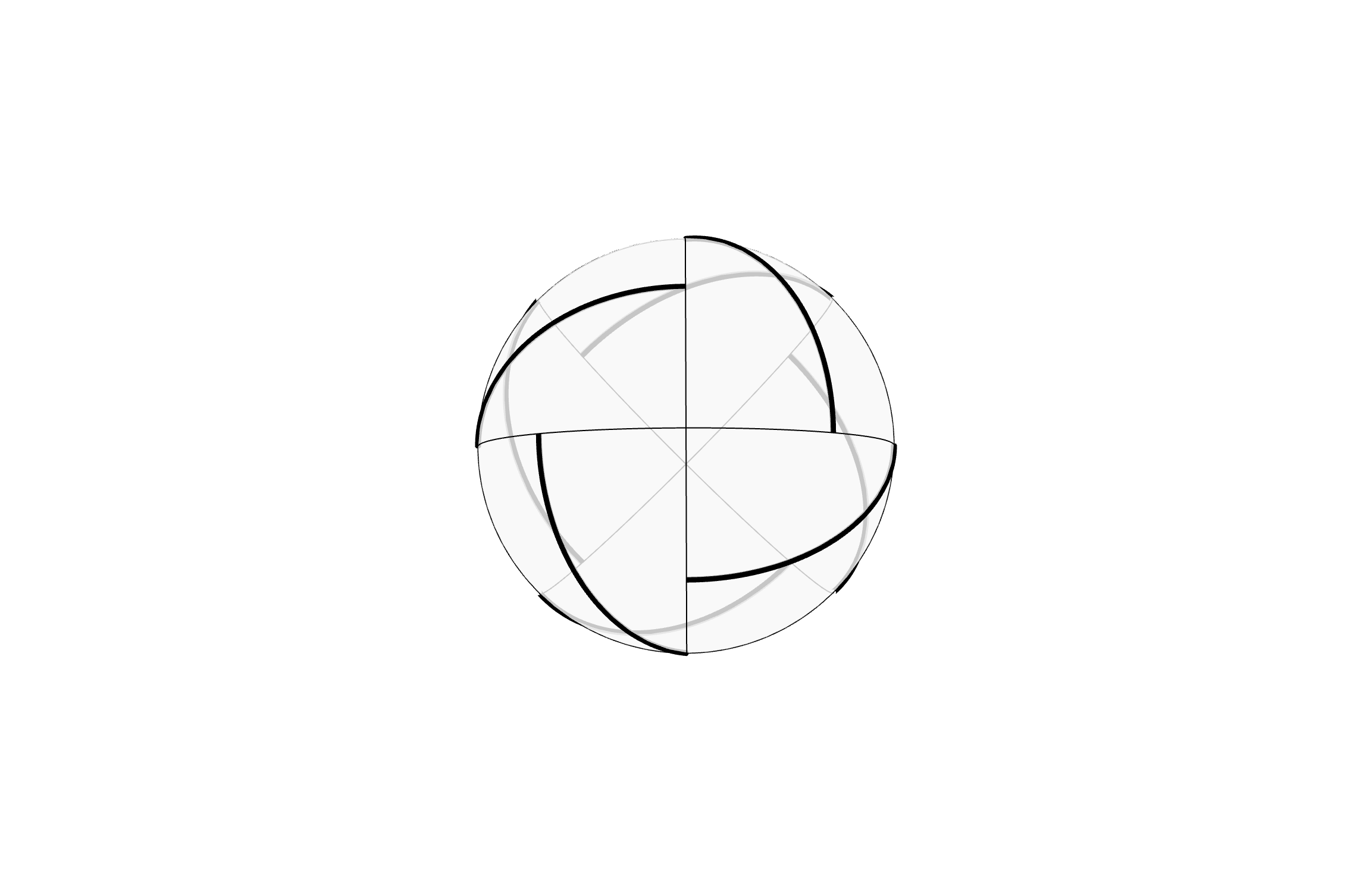}} 
\caption*{$S3$} \hfill
 \end{subfigure} 
\begin{subfigure}{0.26\textwidth}
	\centering
		\adjustbox{trim=\dimexpr.5\Width-15mm\relax{} \dimexpr.5\Height-15mm\relax{}  \dimexpr.5\Width-15mm\relax{} \dimexpr.5\Height-15mm\relax{} ,clip}{\includegraphics[height=6cm]{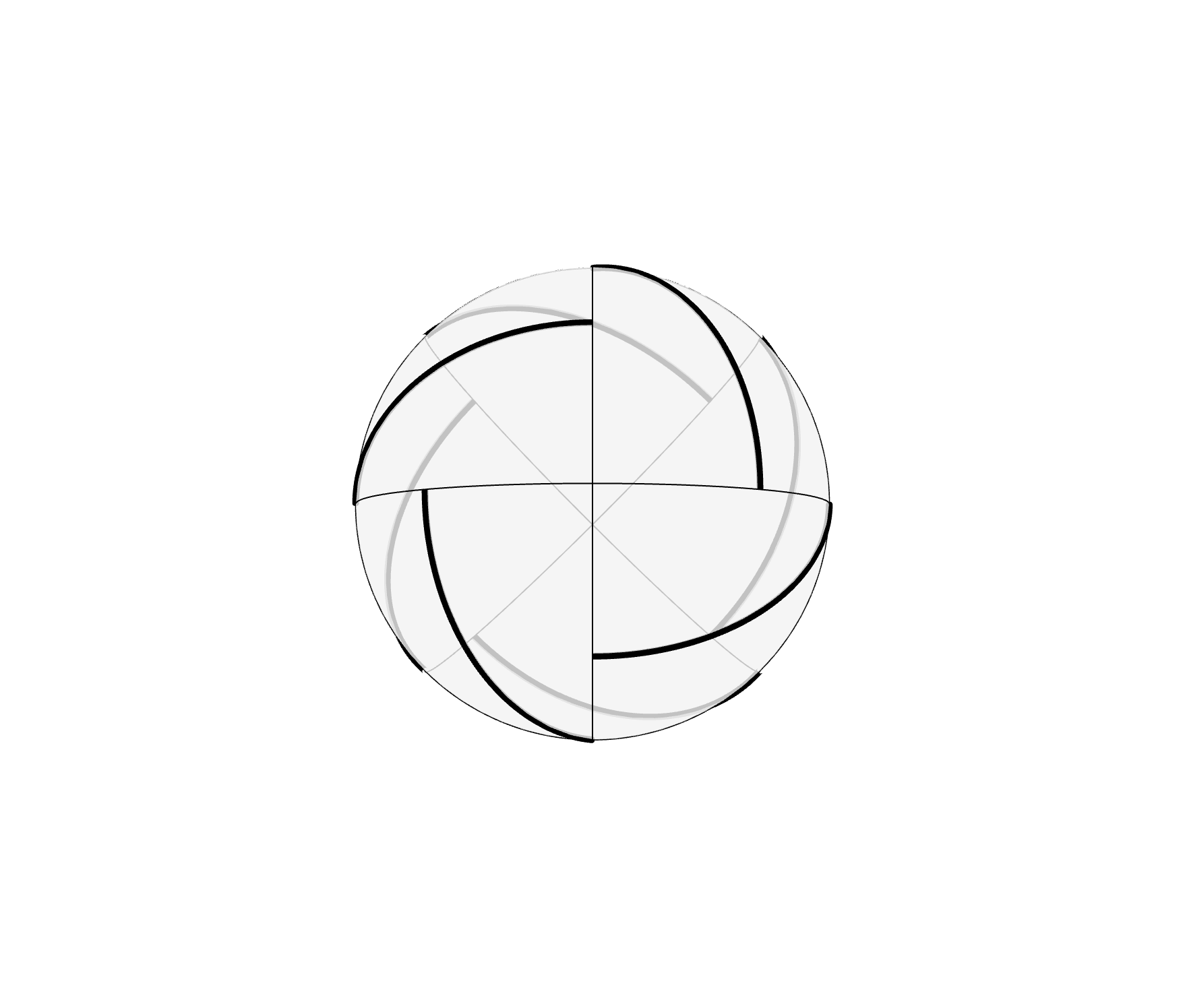}} 
\caption*{$S^{\prime}3$} \hfill
 \end{subfigure} 
\begin{subfigure}{0.34\textwidth}
	\centering
		\adjustbox{trim=\dimexpr.5\Width-15mm\relax{} \dimexpr.5\Height-15mm\relax{}  \dimexpr.5\Width-15mm\relax{} \dimexpr.5\Height-15mm\relax{} ,clip}{\includegraphics[height=6cm]{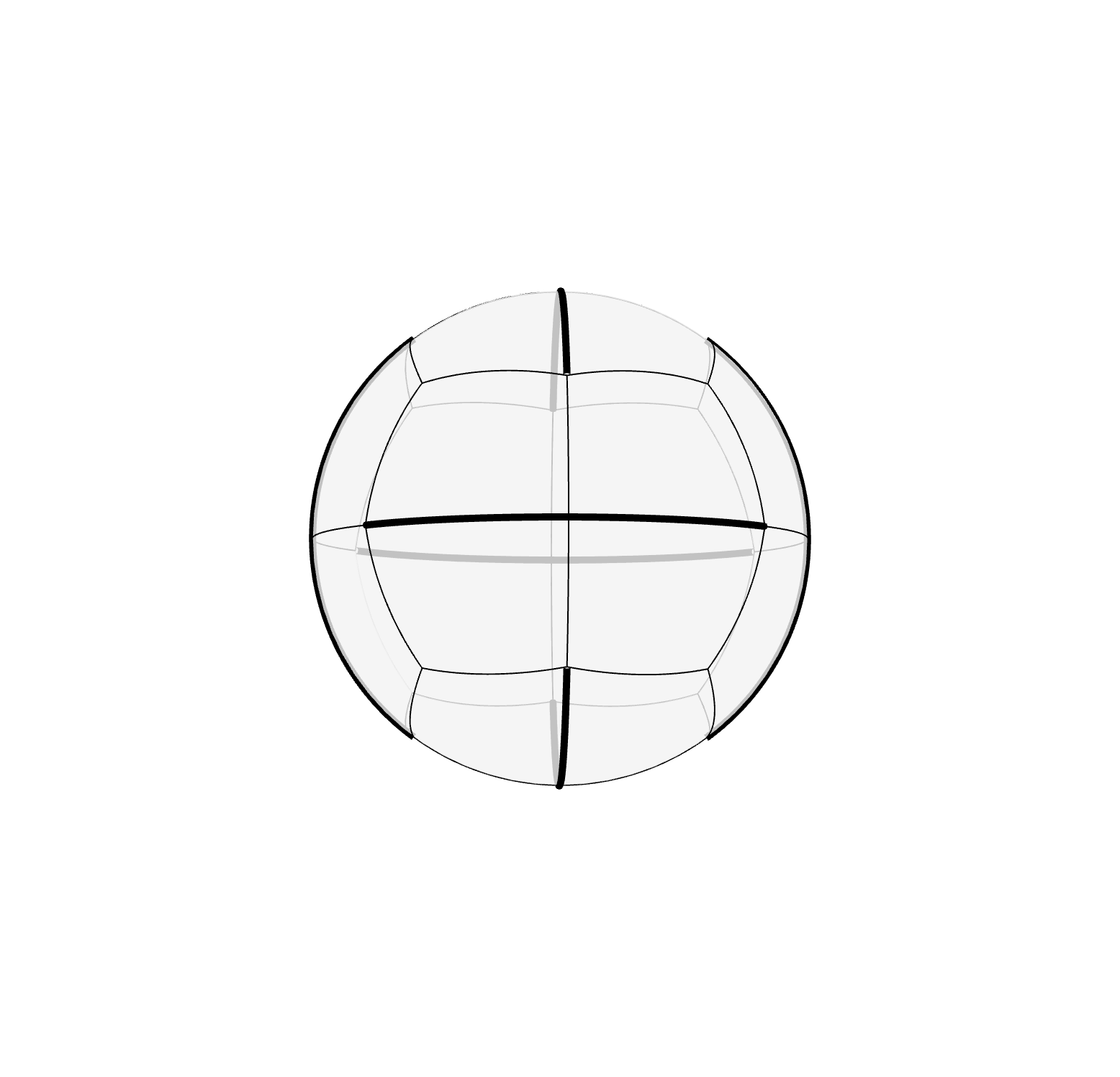}} 
\caption*{$QP_6$} \hfill
 \end{subfigure} 

\hfill

\begin{subfigure}{0.34\textwidth}
	\centering
		\adjustbox{trim=\dimexpr.5\Width-15mm\relax{} \dimexpr.5\Height-15mm\relax{}  \dimexpr.5\Width-15mm\relax{} \dimexpr.5\Height-15mm\relax{} ,clip}{\includegraphics[height=6cm]{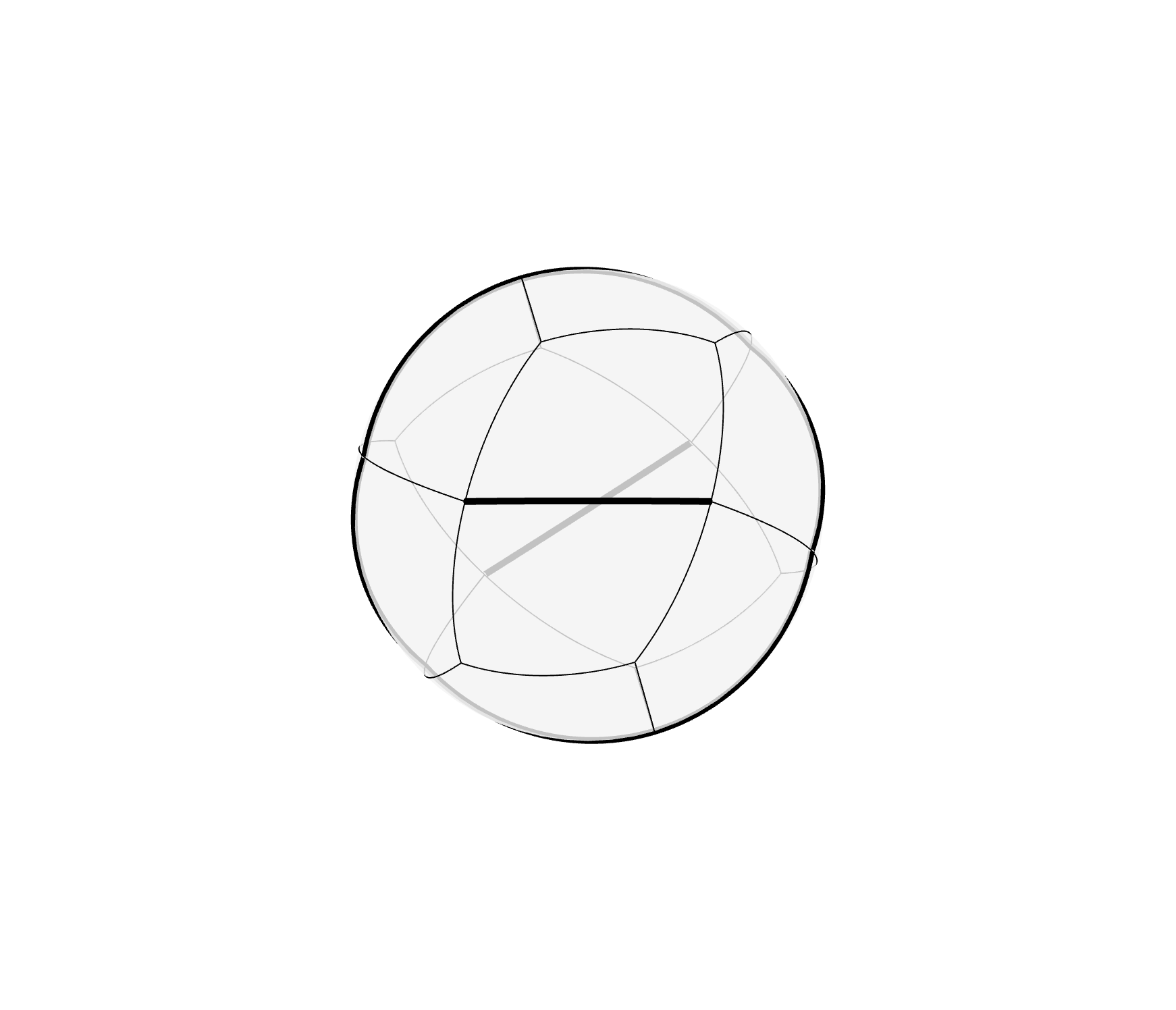}}
\caption*{$S4$} \hfill
 \end{subfigure} 
\begin{subfigure}{0.26\textwidth}
	\centering
		\adjustbox{trim=\dimexpr.5\Width-15mm\relax{} \dimexpr.5\Height-15mm\relax{}  \dimexpr.5\Width-15mm\relax{} \dimexpr.5\Height-15mm\relax{} ,clip}{\includegraphics[height=6cm]{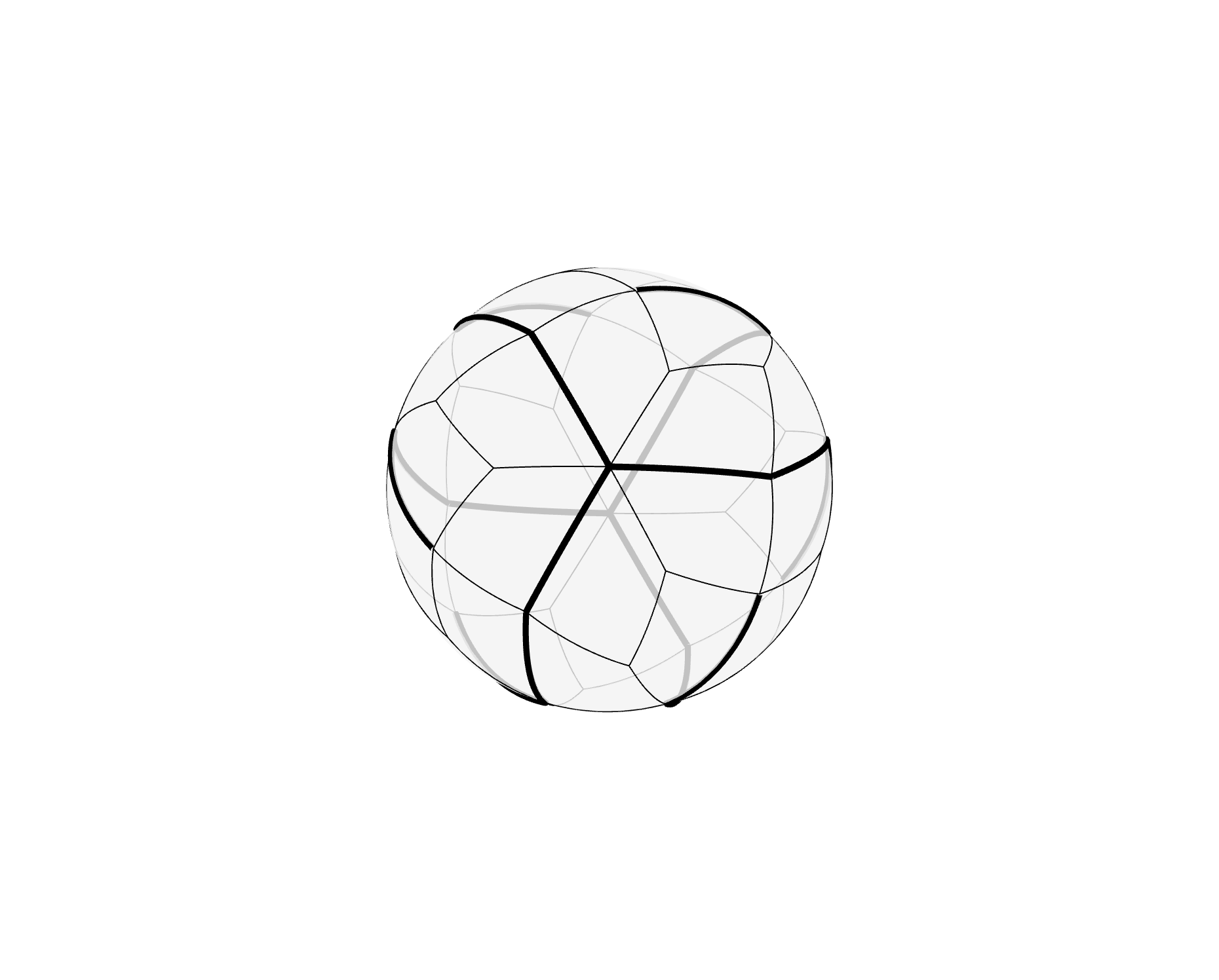}} 
\caption*{$S5$} \hfill
 \end{subfigure} 
\begin{subfigure}{0.34\textwidth}
	\centering
		\adjustbox{trim=\dimexpr.5\Width-15mm\relax{} \dimexpr.5\Height-15mm\relax{}  \dimexpr.5\Width-15mm\relax{} \dimexpr.5\Height-15mm\relax{} ,clip}{\includegraphics[height=6cm]{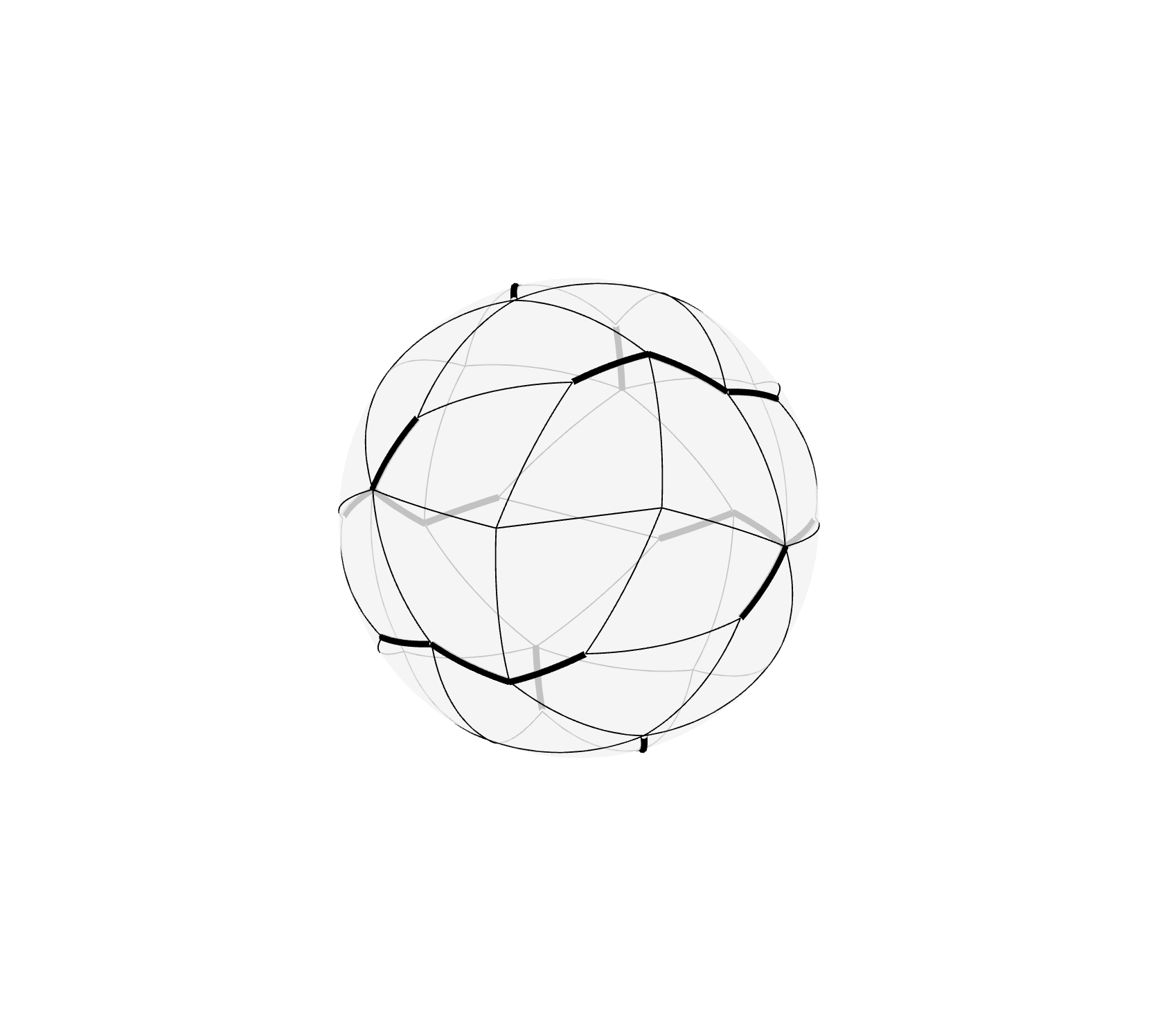}} 
\caption*{$S6$} \hfill
 \end{subfigure} 
	\caption{Tilings of the sphere by almost equilateral quadrilaterals: $E_{10}$, $E^{\prime}_{10}$, $E^{\prime\prime}_{10}$, $E^{\prime\prime\prime}_{10}$, $S_{12}1$, $S_{16}1$, $S2$, $S3$, $S^{\prime}3$, $QP_6$, $S4$, $S5$, $S6$}
	\label{SphericalTilings}
\end{figure}

We explain the structures of these tilings explicitly by their planar representations in first picture of Figure \ref{EMT}, and Figures \ref{EMT-Flips}, \ref{Polar-Tilings}, \ref{IsoEMT-S1S2S3S5}. The angles are implicitly represented according to Figure \ref{QuadOrien}. Tiles with angles arranged in the orientation in the first picture, i.e., $\alpha \rightarrow \beta \rightarrow \gamma \rightarrow \delta$ clockwise, are marked by \quotes{ - }. The other tiles, unmarked, have angles arranged counter-clockwise as in the second picture. 

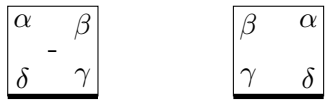
\begin{figure}[htp] 
\centering
\begin{tikzpicture}

\begin{scope}[xshift=0cm] 

\draw
	(0,0) -- (0,1.2) -- (1.2,1.2) -- (1.2, 0.0);

\draw[ line width=2]
	(0,0) -- (1.2,0);

\node at (0.2,1) {\small $\alpha$};
\node at (1,0.925) {\small $\beta$};
\node at (1,0.25) {\small $\gamma$};
\node at (0.2,0.25) {\small $\delta$};

\node at (0.6,0.6) {-};

\end{scope}

\begin{scope}[xshift=3cm] 

\draw
	(0,0) -- (0,1.2) -- (1.2,1.2) -- (1.2, 0.0);

\draw[ line width=2]
	(0,0) -- (1.2,0);

\node at (0.2,0.925) {\small $\beta$};
\node at (1,1) {\small $\alpha$};
\node at (1,0.25) {\small $\delta$};
\node at (0.2,0.25) {\small $\gamma$};

\end{scope}

\end{tikzpicture}
\caption{Orientations of almost equilateral quadrilateral tiles}
\label{QuadOrien}
\end{figure}

Let $f$ denote the {\em number of tiles} in a tiling. The notations for tilings in the theorem are introduced in \cite{cly}. We use $Si$ (or $S_fi$) to denote isolated earth map tilings and special tilings. We use $QP_6$ to denote the quadrilateral subdivision of the cube $P_6$. The notations $E, E^{\prime}, E^{\prime\prime}, E^{\prime\prime\prime}$ correspond to $E^{\prime}_{\square}2, E^{\prime (s,t)}_{\square}2, E^{\prime\prime (s',t)}_{\square}2, E^{\prime\prime\prime}_{\square}2$ in \cite{cly} respectively.

\begin{figure}[htp]
\centering
\begin{tikzpicture}

\tikzmath{ 
\tz=3;
\tzz=\tz-1;
}

\begin{scope}[] 

\begin{scope}[] 

\fill[gray!30]
	(0.8,0.6) -- (0.8,0.15) -- (1.2,-0.15) -- (1.6,0.15) -- (1.6,0.6);

\fill[gray!30, xshift=1.2cm, rotate=-180]
	(0,0.6) -- (0,0.15) -- (0.4,-0.15) -- (0.8,0.15) -- (0.8,0.6);	

\foreach \a in {0,...,\tz}
{
\begin{scope}[xshift=0.8*\a cm]

\draw
	(0.8,0.6) -- (0.8,0.15)
	(0.4,-0.6) -- (0.4,-0.15)
;

\draw[]
	(0.8,0.15) -- (0.4,-0.15);	
	
\end{scope}
}

\foreach \a in {0,...,\tzz}{

\begin{scope}[xshift=0.8*\a cm]

\draw[line width=1.4]	
	(0.8,0.15) -- (1.2,-0.15);

\end{scope}

}

\draw[very thick, dotted]
	(0.2,0) -- ++(-0.3,0)
	(3.4,0) -- ++(0.3,0);

\end{scope}

\end{scope} 

\begin{scope}[xshift=5cm] 

\begin{scope}[] 

\fill[gray!30]
	(0.8,0.6) -- (0.8,0.15) -- (1.2,-0.15) -- (1.6,0.15) -- (1.6,0.6);

\fill[gray!30, xshift=1.2cm, rotate=-180]
	(0,0.6) -- (0,0.15) -- (0.4,-0.15) -- (0.8,0.15) -- (0.8,0.6);	

\foreach \a in {0,...,\tz}
{
\begin{scope}[xshift=0.8*\a cm]

\draw
	(0.8,0.6) -- (0.8,0.15)
	(0.4,-0.6) -- (0.4,-0.15)
;

\draw[double, line width=0.6]
	(0.8,0.15) -- (0.4,-0.15);

\end{scope}
}

\foreach \a in {0,...,\tzz}{

\begin{scope}[xshift=0.8*\a cm]

\draw[line width=1.4]	
	(0.8,0.15) -- (1.2,-0.15);

\end{scope}

}
	
\draw[very thick, dotted]
	(0.2,0) -- ++(-0.3,0)
	(3.4,0) -- ++(0.3,0);

\end{scope}

\end{scope} 

\end{tikzpicture}
\caption{Earth map tilings $E$ by $a^3b$ tiles and by $a^2bc$ tiles}
\label{EMT}
\end{figure}
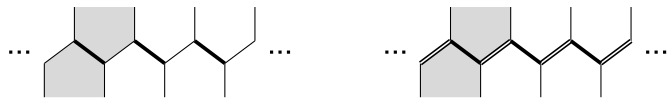

The earth map tilings $E$ is the first picture of Figure \ref{EMT}. The vertical edges in the top row of $E$ converge to a vertex (north pole) and those in the bottom row converge to another (south pole). The shaded tiles form a {\em timezone}. A tiling is a repetition of timezones. The second picture is the earth map tiling by congruent $a^2bc$ quadrilaterals. We may obtain $E$ from this earth map tiling by edge reduction $c=a$ or $b=a$. The earth map tilings with exactly three timezones is the deformed cube. 

For any positive integer $s < \frac{f}{2}$, let $\mathcal{T}_s$ be $s$ consecutive timezones. The first picture of Figure \ref{EMT-Flips} shows the boundary of $\mathcal{T}_s$. If $\alpha=s\beta$, we may flip the $\mathcal{T}_s$ part of $E$ with respect to $L^{\prime}$ to get a new tiling $E^{\prime}$. We call it a {\em flip modification}. In fact, we may simultaneously flip several disjoint copies of $\mathcal{T}_s$. Similarly, if $\gamma+\delta=s\beta$, we may simultaneously flip several disjoint copies $\mathcal{T}_s$ with respect to $L^{\prime\prime}$ to get $E^{\prime\prime}$.

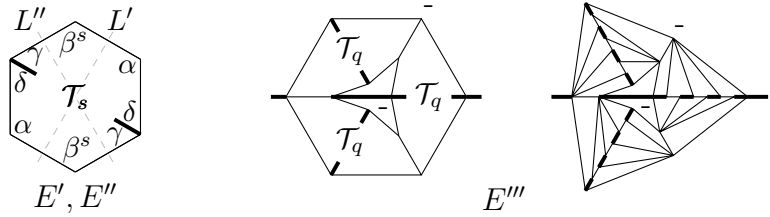
\begin{figure}[htp]
\centering
\begin{tikzpicture}


\begin{scope}[yshift=-3cm] 


\begin{scope}[xshift=0cm]

\draw[gray!50, dashed]
	(-60:1) -- (120:1)
	(60:1) -- (240:1);

\node at (60:1.2) {\small $L'$};
\node at (120:1.2) {\small $L''$};

\foreach \b in {1,-1}
{
\begin{scope}[scale=\b]

\foreach \a in {0,...,5}
\draw[rotate=60*\a]
	(30:1) -- (90:1);

\draw[line width=1.5]
	(150:1) -- (150:0.6);
	
\node at (92:0.7) {\small $\beta^s$};
\node at (30:0.8) {\small $\alpha$};
\node at (-15:0.75) {\small $\delta$};
\node at (-45:0.75) {\small $\gamma$};

\node at (0,0) {\small ${\mathcal{T}}_s$};		

\end{scope}
}

\node at (0,-1.35) {$E^{\prime}, E^{\prime\prime}$};

\end{scope}


\begin{scope}[xshift=4cm]

	
\foreach \a in {0,1,2}
{
\begin{scope}[rotate=120*\a]

\draw
	(60:0.6) -- (0:0.2) -- (-60:0.6)
	(60:0.6) -- (60:1.2) -- (0:1.2) -- (-60:1.2);

\draw[line width=1.5]
	(0.2,0) -- (0.4,0)
	(1.2,0) -- (1,0);

\node at (0.7,0) {\small ${\mathcal{T}}_q$};	
		
\end{scope}
}

\draw[line width=1.5]
	(0.2,0) -- (-0.6,0)
	(1.4,0) -- (1.2,0)
	(-1.4,0) -- (-1.2,0);
	
\node at (0.1,-0.15) {-};
\node at (60:1.4) {-};

\node at (1.7,-1.35) {$E^{\prime\prime\prime}$};

\end{scope}


\begin{scope}[xshift=7.5cm, scale=0.9, xscale=-1]

\draw[line width=1, gray!50]
	(1,0) -- (1.3,0);
	
\foreach \a in {0,1,2}
{
\begin{scope}[rotate=120*\a]

\draw
	(120:0.6) -- (60:0.2) -- (0:0.6) -- (0:1) -- (60:1.6) -- (120:1)
	(-0.2,0) -- (-0.4,0) -- (-0.65,0) -- (-0.9,0) -- (-1.15,0) -- (-1.45,0)
	(120:1) -- (-0.4,0)
	(240:0.6) -- (-0.6,0)
	(120:1) -- (-0.8,0)
	(240:0.6) -- (-1,0)
	(120:1) -- (-1.2,0)
	(240:0.6) -- (-1.4,0);
	
\draw[line width=1.5]
	(-0.2,0) -- (-0.4,0)
	(-0.6,0) -- (-0.8,0)
	(-1,0) -- (-1.2,0)
	(-1.4,0) -- (-1.6,0);
	
\end{scope}
}

\draw[line width=1.5]
	(-0.2,0) -- (0.6,0)
	(-1.6,0) -- (-1.9,0)
	(1,0) -- (1.3,0);
	
\node at (-0.1,-0.15) {-};
\node at (120:1.2) {-};

\end{scope}

\end{scope}

\end{tikzpicture}
\caption{Flip modifications $E^{\prime}, E^{\prime\prime}$ and rearrangement $E^{\prime\prime\prime}$ of $E$}
\label{EMT-Flips}
\end{figure}

For $q=\frac{f-4}{6}$, we may combine three copies of $\mathcal{T}_q$ and four more tiles as in the second picture of Figure \ref{EMT-Flips} to get a {\em rearrangement} $E^{\prime\prime\prime}$ of $E$. The third picture depicts $E^{\prime\prime\prime}$ when $q=4$. Rearrangements are only for specific combination of angle values. 

Further explanations on $E^{\prime},E^{\prime\prime},E^{\prime\prime\prime}$ can be seen in \cite[Section 2]{cly}.

The isolated earth map tilings and the special tilings are presented in Figure \ref{Polar-Tilings}.

\begin{figure}[htp]
\centering
\begin{tikzpicture}

\begin{scope} 


\begin{scope} 

\begin{scope}[scale=0.2, yshift=0.5cm]

\foreach \a in {0,...,2}
\draw[rotate=120*\a]
	(0,0)-- (-30:2.4) -- (30:4.8) -- (90:2.4)
	(-30:1.2) -- (30:2.4) -- (90:1.2)
	(-30:2.4) -- (-30:4.32)
	(30:2.4) -- (30:4.8);
			
\foreach \b in {0,...,2}
	\draw[rotate=120*\b, line width=1.25]
	(-30:1.2) -- (30:2.4)
	(90:2.4) -- (30:4.8);

\foreach \s in {0,...,2} {

\begin{scope}[rotate=120*\s]

\node at (60:2) {-};
\node at (0:2) {-};

\end{scope}
}

\end{scope}

\node at (0,-1.4) {$S1(S_{12}1)$};

\end{scope} 


\begin{scope}[xshift=2.2cm] 

\begin{scope}[scale=0.425] 

\foreach \a in {0,1,2,3}
{
\begin{scope}[rotate=90*\a]

\draw
	(0,0) -- (1.5,0) -- (1.5,1.5)
	(0,0.8) -- (0.8,0.8) -- (1.5,1.5)
	(0,1.5) -- (0,2)
;

\draw[line width=1.25]
	(0,0.8) -- (-0.8,0.8)
	(-1.5,0) -- (-1.5,1.5);

\node at (1.1,0.45) {-};
\node at (0.45,1.1) {-};

\end{scope}
}

\end{scope}

\node at (0,-1.4) {$S1(S_{16}1)$};

\end{scope} 


\begin{scope}[xshift=4.5 cm] 

\begin{scope}[scale=0.425]

\foreach \a in {0,1,2,3}
{
\begin{scope}[rotate=90*\a]

\draw
	(0,0) -- (1.5,0) -- (1.5,-1.5)
	(0.8,0) -- (0.8,0.8) -- (2,2);

\draw[line width=1.25]
	(0,0.8) -- (0.8,0.8)
	(1.5,0) -- (1.5,1.5);

\node at (0.5,1.2) {-};

\node at (2,0) {-};

\end{scope}
}

\end{scope}

\node at (0,-1.4) {$S2$};

\end{scope}


\begin{scope}[xshift=7cm]

\tikzmath { \r = 2; \r1= sqrt( \r^2 + \r^2 - 2*\r*\r*cos(45) ); }

\pgfmathsetmacro{\R}{2};
\pgfmathsetmacro{\rr}{ sqrt( \R*\R + \R*\R - 2*\R*\R*cos(45) ) };

\begin{scope}[xshift=0 cm, scale=0.25]

\foreach \a in {0,...,3}{
\draw[rotate=90*\a, line width = 1.1] 
	([shift=(-22.5:\r1)]90:2) arc (-22.5:90:\r1);

	\begin{scope}[xscale=-1]

	\draw[rotate=90*\a] 
	([shift=(-22.5:\r1)]90:2) arc (-22.5:90:\r1);

	\end{scope}

}

\foreach \a in {0,...,3}{

\begin{scope}[rotate=90*\a]

\coordinate (P1) at (0:1);
\coordinate (P2) at (45:1.2);
\coordinate (P3) at (90:\R);

\arcThroughThreePoints[line width=1.1]{P1}{P2}{P3};

\draw
	(0:0) -- (0:\R)
	(0:\R) arc (0:90:\R)
	(90:2+\rr) -- (90:3+\rr)
;

\node at (45:1.6) {-};
\node at (0:2.6) {-};

\end{scope}

}

\end{scope}

\node at (2.5,-1.4) {$S^{\prime}3$};

\begin{scope}[xshift=2.5 cm, scale=0.25] 

\foreach \a in {0,...,3}{
\draw[rotate=90*\a] 
	([shift=(-22.5:\r1)]90:2) arc (-22.5:90:\r1);

	\begin{scope}[xscale=-1]

	\draw[rotate=90*\a, line width = 1.1] 
	([shift=(-22.5:\r1)]90:2) arc (-22.5:90:\r1);

	\end{scope}

}

\foreach \a in {0,...,3}{

\begin{scope}[rotate=90*\a]

\coordinate (P1) at (0:1);
\coordinate (P2) at (45:1.2);
\coordinate (P3) at (90:\R);

\arcThroughThreePoints[line width=1.1]{P1}{P2}{P3};

\draw
	(0:0) -- (0:\R)
	(0:\R) arc (0:90:\R)
	(90:2+\rr) -- (90:3+\rr)
;

\node at (45:1.6) {-};
\node at (45:3.2) {-};

\end{scope}

}

\end{scope}

\node at (0,-1.4) {$S3$};

\end{scope}

\end{scope} 

\begin{scope}[yshift=-3.0cm] 


\begin{scope}

\begin{scope}[scale=1.25, rotate=90]

\foreach \a in {1,-1}
{
\begin{scope}[scale=\a]

\draw
	(-0.65,0.65) -- (0.65,0.65) -- (0.65,-0.65)
	(-0.25,0.25) -- (0.25,0.25) -- (0.25,-0.25)
	(0.25,0.25) -- (0.65,0.65)
	(-0.25,0.25) -- (-0.65,0.65);

\draw[line width=1.25]
	(0.45,0.45) -- (-0.45,0.45)
	(0,0.25) -- (0,0)
	(0.25,0) -- (0.65,0)
	(0,0.65) -- (0,0.9);
	
\draw[
]
	(0.45,0.45) -- (0.45,-0.45)
	(0.25,0) -- (0,0)
	(0,0.25) -- (0,0.65)
	(0.65,0) -- (0.9,0);

\end{scope}
}

\end{scope}

\foreach \b in {0,2} {

\begin{scope}[rotate=90*\b]

\node at (135:0.2) {-};
\node at (120:0.5) {-};
\node at (160:0.45) {-};

\node at (20:0.75) {-};
\node at (70:0.75) {-};
\node at (45:1.4) {-};

\end{scope}

}

\node at (0,-1.5) {$QP_6$};

\end{scope}


\begin{scope}[xshift=2.75cm]

\foreach \a in {0,...,5}
{
\begin{scope}[rotate=60*\a]

\draw
	(0:0.3) -- (60:0.3) -- (60:0.9) -- (0:0.9);
	
\draw[line width=1.25]
	(0:0.6) -- (60:0.6);

\end{scope}
}

\draw[line width=1.25]
	(-0.3,0) -- (0.3,0)
	(60:0.9) -- (60:1.2)
	(240:0.9) -- (240:1.2);
	
\node at (30:0.38) {-};
\node at (30:0.65) {-};
\node at (30:-0.38) {-};
\node at (30:-0.65) {-};

\node at (0,-1.5) {$S4$};

\end{scope}


\begin{scope}[xshift=5.4cm]

\foreach \a in {0,...,11}
\draw[rotate=30*\a]
	(0:0.4) -- (30:0.4)
	(0:0.7) -- (30:0.7)
	(0:0.95) -- (30:0.95);

\foreach \a in {0,1,2}
{
\begin{scope}[rotate=120*\a]

\draw
	(0,0) -- (60:0.4)
	(30:0.4) -- (60:0.7)
	(-30:0.4) -- (0:0.7)
	(-30:0.7) -- (0:0.95)
	(30:0.7) -- (60:0.95)
	(30:0.95) -- (30:1.2);
	
\draw[line width=1.25]
	(0,0) -- (0:0.4) -- (30:0.7) -- (0:0.95)
	(90:0.4) -- (60:0.7)
	(60:0.7) -- (90:0.95) -- (90:1.2);

\node at (90:0.25) {-};
\node at (90:0.55) {-};
\node at (0:0.8) {-};
\node at (0:1.1) {-};

\end{scope}
}

\node at (0,-1.5) {$S5$};

\end{scope}


\begin{scope}[xshift=8.75cm]

\foreach \a in {1,-1}
{
\begin{scope}[scale=0.333*\a]

\draw
	(0,0) -- (0.3,0.4) -- (2.1,0.4) 
	(-0.3,1.2) -- (0.9,-0.4) -- (1.5,-0.4) -- (2.1,-1.2) -- (3.3,0.4)
	(2.1,-1.2) -- (0.3,-1.2)
	(2.7,-2) -- (3.3,-1.2) -- (1.5,1.2)
	(2.1,0.4) -- (2.7,1.2) -- (2.1,2)
	(-1.5,-1.2) -- (0.3,-2) -- (2.1,-1.2)
	(2.7,-2) -- (0.9,-2.8) -- (0.3,-2) -- (-0.3,-2.8) -- (-2.1,-2)
	(0.3,-3.6) -- (-0.3,-2.8)
	(3.3,-1.2) -- (3.9,-0.4)
	(4.5,-1.2) -- (5.1,-1.2)
	(3.3,0.4) to[out=75,in=25, distance=1.75cm] (0.3,2.8)
	(3.9,-0.4) to[out=75,in=5, distance=2.75cm] (-0.3,3.6)
	(3.3,-1.2) to[out=-90,in=10, distance=1.75cm] (0.3,-3.6)
	(4.5,-1.2) to[out=-100,in=-20, distance=2cm] (0.3,-3.6)
	;

\draw[line width=1.25]
	(2.1,2) -- (1.5,1.2) -- (-0.3,1.2) -- (-0.9,0.4)
	(1.5,-0.4) -- (2.1,0.4) 
	(2.7,-2) -- (2.1,-1.2)
	(0.3,-3.6) -- (0.9,-2.8)
	(2.7,1.2) -- (4.5,-1.2);

\node at (-1.2,0) {-};
\node at (-1.2,0.8) {-};
\node at (-0.3,1.5) {-};
\node at (-1.5,2) {-};
\node at (0.9,0.8) {-};
\node at (1.2,0) {-};
\node at (2.1,1.2) {-};
\node at (2.7,1.8) {-};
\node at (-1.8,2.75) {-};
\node at (-3.9,1.2) {-};

\end{scope}
}

\node at (0,-1.5) {$S6$};

\end{scope}

\end{scope}

\end{tikzpicture}
\caption{Polar view of $S1, S2, S3, S'3, QP_6, S4, S5, S6$}
\label{Polar-Tilings}
\end{figure}
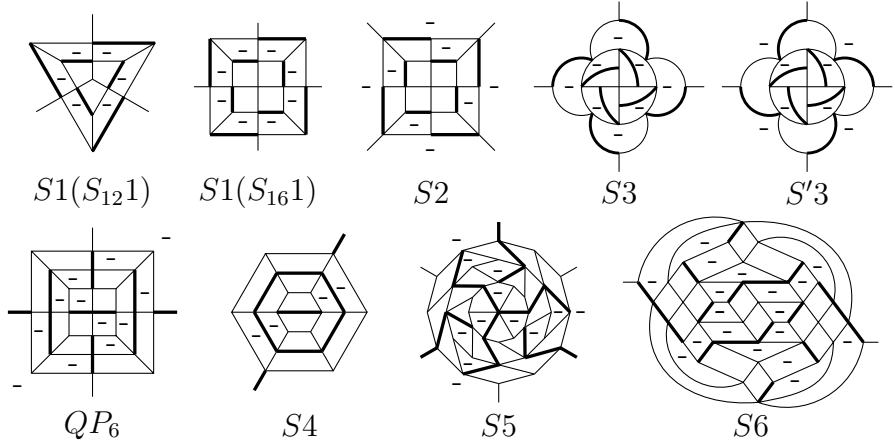

Figure \ref{IsoEMT-S1S2S3S5} presents a different view of $S1, S2, S3, S^{\prime}3, S5$. Comparing with Figure \ref{EMT}, combinatorially, each of them belongs to a family of earth map tilings (with shaded timezones different from that in $E$). However, they can be realised as geometric tilings only for specific numbers of timezones. A combinatorial study on pentagonal earth map tilings was given by Yan \cite{yan}.

\begin{figure}[htp]
\centering

\begin{tikzpicture}

\begin{scope}[]


\begin{scope}[]

\begin{scope}[yscale=1.4]

\fill[gray!30]
	(0,0.6) -- (0.8,0.6) -- (0.8,-0.6) -- (0,-0.6) -- cycle;

\foreach \a in {0,1,2}
{
\begin{scope}[xshift=0.8*\a cm]

\draw
	(0,0.2) -- (0.8,0.2)
	(0,-0.2) -- (0.8,-0.2)
;

\end{scope}
}

\foreach \a in {0,1,2}
{
\begin{scope}[xshift=0.8*\a cm]

\draw
	(0.4,0.2) -- (0.4,-0.2);

\draw[line width=1.5]
	(0,0.2) -- (0.4,0.2)
	(0.4,-0.2) -- (0.8, -0.2)
;
		
\end{scope}
}

\foreach \a in {0,1,2,3}
{
\begin{scope}[xshift=0.8*\a cm]

\draw
	(0,0.2) -- (0,0.6)
	(0,0.2) -- (0,-0.2)
	(0,-0.2) -- (0,-0.6);

\end{scope}
}

\foreach \a in {0,1,2}
	{
	\node at (0.8*\a+0.2,0) {-};	
	\node at (0.8*\a+0.6,-0) {-};
	}

\end{scope}

\node at (1.2,-1.20) {$S1(S_{12}1)$};

\end{scope} 


\begin{scope}[xshift=3.25cm]

\begin{scope}[yscale=1.4]

\fill[gray!30]
	(0,0.6) -- (0.8,0.6) -- (0.8,-0.6) -- (0,-0.6) -- cycle;

\foreach \a in {0,1,2,3}
{
\begin{scope}[xshift=0.8*\a cm]

\draw
	(0,0.2) -- (0.8,0.2)
	(0,-0.2) -- (0.8,-0.2)
;

\end{scope}
}

\foreach \a in {0,1,2,3}
{
\begin{scope}[xshift=0.8*\a cm]

\draw
	(0.4,0.2) -- (0.4,-0.2);

\draw[line width=1.5]
	(0,0.2) -- (0.4,0.2)
	(0.4,-0.2) -- (0.8, -0.2)
;
		
\end{scope}
}

\foreach \a in {0,1,2,3,4}
{
\begin{scope}[xshift=0.8*\a cm]

\draw
	(0,0.2) -- (0,0.6)
	(0,0.2) -- (0,-0.2)
	(0,-0.2) -- (0,-0.6);

\end{scope}
}

\foreach \a in {0,1,2,3}
	{
	\node at (0.8*\a+0.2,0) {-};	
	\node at (0.8*\a+0.6,-0) {-};
	}

\end{scope}

\node at (1.6,-1.2) {$S1(S_{16}1)$};

\end{scope} 


\begin{scope}[xshift = 8.55 cm] 

\begin{scope}[yscale=1.4]

\fill[gray!30]
	(-0.4,0.6) -- (-0.4,0.2) -- (0.0,0.2) -- (0,-0.6) -- (0.8,-0.6) -- (0.8, 0.2) -- (0.4,0.2) -- ( 0.4, 0.6 ) -- cycle;

\foreach \a in {0,1,2,3}
{
\begin{scope}[xshift=0.8*\a cm]

\draw
	(0,0.2) -- (0.8,0.2)
	(0,-0.2) -- (0.8,-0.2)
	(0,-0.2) -- (0,-0.6)
;
	
\draw[
]
	(0,-0.2) -- (0,0.2);
	
\end{scope}
}

\foreach \a in {0,...,3}
{
\begin{scope}[xshift=0.8*\a cm]

\draw
	(0.4,0.2) -- (0.4,-0.2);

\draw[line width=1.5]
	(0,0.2) -- (0.4,0.2)
	(0.4,-0.2) -- (0.8, -0.2)
;
		
\end{scope}
}

\foreach \a in {0,...,4}
{
\begin{scope}[xshift=0.8*\a cm]

\draw
	(-0.4,0.2) -- (-0.4,0.6)
	(-0.4,0.2) -- (0,0.2)
	(0,0.2) -- (0,-0.6)
	;

\end{scope}
}

\foreach \a in {0,1,2,3}
	{
		\node at (0.8*\a+0.2,0) {-};
		\node at (0.8*\a+0.4,-0.4) {-};	
	}

\end{scope}

\node at (1.6,-1.2) {$S2$};

\end{scope} 

\end{scope}

\begin{scope}[yshift=-2.6cm] 


\begin{scope}[]

\begin{scope}[yscale=1.4]


\fill[gray!30]
(0,0.6) -- (0,0.3)  -- (0.4,0.1) -- (0,-0.1) -- (0.4,-0.3) -- (0.4,-0.6) -- (1.2,-0.6) -- (1.2,-0.3) -- (0.8,-0.1) -- (1.2,0.1) -- (0.8, 0.3) -- (0.8,0.6) -- cycle;

\foreach \a in {0,...,3}
{
\begin{scope}[xshift=0.8*\a cm]

\draw
	(0,0.6) -- (0,0.3) -- (0.4,0.1) -- (0.8,0.3) -- (0.8,0.6)
	(1.2,-0.6) -- (1.2,-0.3) -- (0.8,-0.1) -- (1.2,0.1) -- (0.8, 0.3)
	(0,-0.1) -- (0.4,-0.3) -- (0.8,-0.1) -- (0.4,0.1) -- (0,-0.1)
	(0.4,-0.3) -- (0.4,-0.6);


\end{scope}
}

\foreach \b in {0,...,4} {
\begin{scope}[xshift=0.8*\b cm]
\draw[line width=1.5]
	(0.4,0.1) -- (0.0,0.3)
	(0.0,-0.1) -- (0.4,-0.3);
\end{scope}
}

\foreach \a in {1,2,3,4}
	{
	\node at (0.8*\a,0.1) {-};		
	\node at (0.8*\a-0.4,-0.1) {-};
}

\end{scope}

\node at (2.0,-1.2) {$S3$};
		
\end{scope}


\begin{scope}[xshift=4.1cm]

\begin{scope}[yscale=1.4]

\fill[gray!30]
(0,0.6) -- (0,0.3)  -- (0.4,0.1) -- (0,-0.1) -- (0.4,-0.3) -- (0.4,-0.6) -- (1.2,-0.6) -- (1.2,-0.3) -- (0.8,-0.1) -- (1.2,0.1) -- (0.8, 0.3) -- (0.8,0.6) -- cycle;


\foreach \a in {0,...,3}
{
\begin{scope}[xshift=0.8*\a cm]

\draw
	(0,0.6) -- (0,0.3) -- (0.4,0.1) -- (0.8,0.3) -- (0.8,0.6)
	(1.2,-0.6) -- (1.2,-0.3) -- (0.8,-0.1) -- (1.2,0.1) -- (0.8, 0.3)
	(0,-0.1) -- (0.4,-0.3) -- (0.8,-0.1) -- (0.4,0.1) -- (0,-0.1)
	(0.4,-0.3) -- (0.4,-0.6);
	
\draw[line width=1.5]
	(0.4,-0.3) -- (0.8,-0.1);

\end{scope}
}

\foreach \b in {0,...,4}{
\begin{scope}[xshift=0.8*\b cm]
\draw[line width=1.5]
	(0.4,0.1) -- (0.0,0.3);
\end{scope}
}

\foreach \a in {1,2,3,4}
	{
		\node at (0.8*\a,0.1) {-};
		\node at (0.8*\a,-0.4) {-};	
	}

\end{scope}

\node at (2.0,-1.2) {$S^{\prime}3$};

\end{scope}


\begin{scope}[xshift=9cm] 

\begin{scope}[yscale=1.05]

\fill[gray!25]
	(-0.8,0.8) -- (-0.8,0) -- (-0.4,0) -- (-0.4,-0.8) -- (1.2,-0.8) -- (1.2,0) -- (0.8, 0) -- (0.8,0.8) -- cycle;

\foreach \a in {0,1,2}
{
\begin{scope}[xshift=1.6*\a cm]

\draw
	(0,0.4) -- (0,0.8)
	(-0.8,0) -- (-0.8, 0.8)
	(-0.8,0.4) -- (0.8,0.4)
	(-0.8,0) -- (0,0) -- (0.8,0)
	(-0.4,-0.8) -- (-0.4,0) -- (-0.4, 0.4)
	(0.4,0) -- (0.4,0.4) 
	(0, 0) -- (0,-0.4)
	(0.8, 0) -- (0.8,-0.4)
	(-0.4,-0.4) -- (0.0,-0.4) -- (0.4,-0.4) -- (0.8,-0.4) -- (1.2,-0.4)
	(0.4,-0.4) -- (0.4,-0.8)
;

\draw[line width=1.5]
	(-0.4,0) -- (0.4,0.4)
	(0,-0.4) -- (0.8,0);

\end{scope}
}

\foreach \b in {0,1,2,3}
{
\begin{scope}[xshift=1.6*\b cm]

\draw
	(-0.8,0) -- (-0.4,0);

\draw[line width=1.5]
	(-0.8,0) -- (-0.8, 0.8)
	(-0.4,-0.8) -- (-0.4,0)
;

\end{scope}
}

\foreach \a in {0,1,2}
	{
		
		\node at (0.4+1.6*\a,0.6) {-};
		\node at (1.6*\a,-0.6) {-};
		\node at (0.2+1.6*\a,0.12) {-};
		\node at (0.2+1.6*\a,-0.12) {-};
}

\end{scope}

\node at (2.0,-1.2) {$S5$};

\end{scope}

\end{scope}

\end{tikzpicture}

\caption{Isolated earth map tilings $S1, S2, S3, S^{\prime}3, S5$}
\label{IsoEMT-S1S2S3S5}
\end{figure}
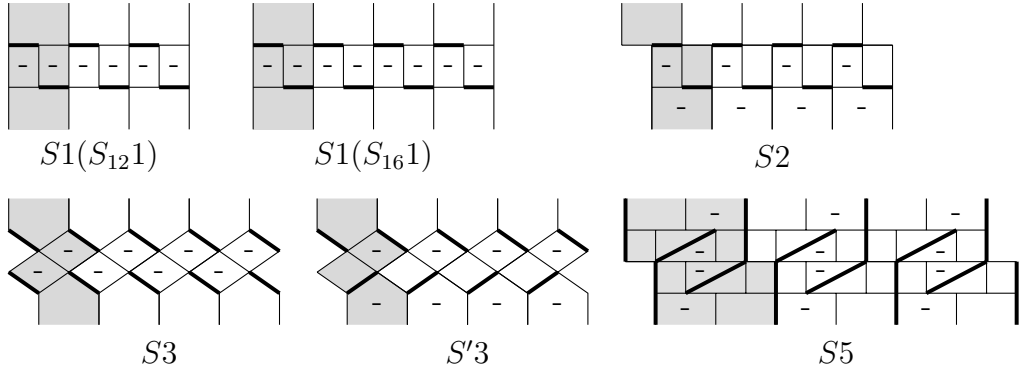

Tables \ref{SpecialTilingData1}, \ref{SpecialTilingData2}, \ref{EMTData1}, \ref{EMTData2} give the geometric and combinatoric data of the tilings.

The paper is organised as follows. Section \ref{SecBasic} explains the basic tools and the strategy. Section \ref{SecRat} studies the tilings where all angles are rational, Section \ref{SecNonRat} studies the tilings where some angles are non-rational. Section \ref{SecGeom} verifies the existence of tilings and explains the data.

\begin{table}[htp]
\begin{center}
\bgroup
\def\arraystretch{1.5}
    \begin{tabular}[]{ | c | c | c | M{1.3cm} | }
	\hline
	Tilings & $f$ & Edges and Angles & Vertices \\ \hhline{|====|}
	$P_6$ & $6$ & \parbox[c][2.15cm][c]{9cm}{$a = \cos^{-1}  \frac{\cos \alpha}{\cos \alpha - 1} $, \\ $b=\cos^{-1} \frac{ (2\cos \alpha - 1)\cos (\alpha + \frac{2}{3}\pi) - \cos^2 \alpha  }{ ( 1 - \cos \alpha )^2 }$, \\ $\alpha+ \gamma + \delta =2\pi$, $\beta =\frac{2}{3}\pi$ } & \parbox[c][2.15cm][c]{1cm}{$ 6\alpha\gamma\delta, \\ 2\beta^3 $}\\
	\hline
	$S1$ & $12$ & \parbox[c][3.5cm][c]{9cm}{$a=\cos^{-1} ( \frac{2}{3}\sqrt{5} - 1 ) \approx 0.34\pi$, \\ $b=\cos^{-1} ( 3 \sqrt{5} - 6 ) \approx 0.25\pi$, \\ $\alpha =  2\cos^{-1}  \frac{1}{4}\sqrt{10} \approx 0.42\pi$, \\ $\beta =\frac{2}{3}\pi$, \\ $\gamma = \frac{2}{3}\pi - \cos^{-1} \frac{1}{4}\sqrt{10} \approx 0.46\pi$, \\ $\delta = \pi - \cos^{-1} \frac{1}{4}\sqrt{10} \approx 0.80\pi$} & \parbox[c][2.15cm][c]{1cm}{$ 6\alpha\delta^2, \\ 6\alpha\beta\gamma^2, \\ 2\beta^3 $} \\
	\hline
	$S1$ & $16$ &  \parbox[c][4.0cm][c]{9cm}{$a=\cos^{-1}  \frac{ 1 }{2} ( -3 - \sqrt{2} + \sqrt{5} + \sqrt{10} ) \approx 0.34\pi$, \\ $b=\cos^{-1} (  -9 - 6\sqrt{2} + 4\sqrt{5} + 3\sqrt{10} ) \approx 0.11\pi$, \\ $\alpha= 2\cos^{-1}  \frac{1}{\sqrt{12}}  \sqrt{ 7 + \sqrt{2} + \sqrt{5} - \sqrt{10} } \approx 0.42\pi$, \\[0.2ex] $\beta =\frac{1}{2}\pi$, \\ $\gamma = \frac{3}{4}\pi - \cos^{-1} \frac{1}{\sqrt{12}}  \sqrt{  7 + \sqrt{2} + \sqrt{5} - \sqrt{10} } \approx 0.54\pi$, \\ $\delta = \pi - \cos^{-1}  \frac{1}{\sqrt{12}} \sqrt{ 7 + \sqrt{2} + \sqrt{5} - \sqrt{10} } \approx 0.79\pi$ } & \parbox[c][2.15cm][c]{1cm}{$ 8\alpha\delta^2, \\ 8\alpha\beta\gamma^2, \\ 2\beta^4$} \\
	\hline
	$S2$ & $16$ &  \parbox[c][4.0cm][c]{9cm}{$a=\cos^{-1}   \frac{ 1 }{ \sqrt{7} } \sqrt{   2\sqrt{2} - 1  } \approx 0.33\pi$, \\ $b=\cos^{-1} \frac{ 1}{ \sqrt{7} }  \sqrt{ 22\sqrt{2} - 25 } \approx 0.12\pi$, \\ $\alpha = \frac{1}{2}\pi$, \\ $\beta = \cos^{-1}   \frac{ 1 }{ 2 } ( \sqrt{2} - 1 ) \approx 0.43\pi$, \\[0.5ex] $\gamma =  \frac{3}{4}\pi$, \\ $\delta = \cos^{-1}  \frac{ 1 }{ 2 } ( 1 - \sqrt{2} ) \approx 0.57\pi $} & \parbox[c][2.15cm][c]{1cm}{$ 8\alpha\gamma^2, \\[0.5ex] 8\beta^2\delta^2, \\ 2\alpha^4 $} \\
	\hline
	$S3$ & \multirow{2}{*}{$16$} & \multirow{2}{*}{  \parbox[m][1.5cm][c]{9cm}{$a=\frac{1}{4}\pi$, $b=\frac{1}{2}\pi$, \\[0.75ex] $\alpha=\pi$, $\beta =\frac{1}{2}\pi$, $\gamma = \frac{1}{2}\pi$, $\delta = \frac{1}{4}\pi$ } } & \multirow{2}{*}{\parbox[c][1.55cm][c]{1cm}{$ 8\alpha\gamma^2, \\ 8\alpha\beta\delta^2, \\ 2\beta^4 $}} \\
	\cline{1-1}
	 $S^{\prime}3$ &  & &  \\
	\hline
	\end{tabular}
\egroup
\end{center}
\caption{Data of isolated tilings 1}
\label{SpecialTilingData1}
\end{table}

\begin{table}[htp]
\begin{center}
\bgroup
\def\arraystretch{1.5}
    \begin{tabular}[]{ | c | c | c | M{1.4cm} | }
	\hline
	Tilings & $f$ & Edges and Angles & Vertices \\ \hhline{|====|}
	$S4$ & $16$ & \parbox[c][3.5cm][c]{7.0cm}{$a=\frac{1}{4}\pi$, \\ $b=\cos^{-1} \frac{1}{4} (2\sqrt{2} - 1) \approx 0.35\pi$, \\[0.5ex] $\alpha=\frac{1}{2}\pi$, $\beta = \frac{3}{4}\pi$, \\ $\gamma =  \cos^{-1} \frac{1}{ \sqrt{17} }  \sqrt{  7 - 4 \sqrt{2} } \approx 0.41\pi$, \\ $\delta = \pi - \cos^{-1}  \frac{1}{ \sqrt{17} }   \sqrt{ 7 - 4 \sqrt{2} } \approx 0.59\pi$} & \parbox[c][1.5cm][c]{1.15cm}{$ 8\alpha\beta^2, \\ 4\alpha^2\gamma\delta, \\ 6\gamma^2\delta^2  $} \\
	\hline
	$QP_6$ & $24$ &  \parbox[c][4.05cm][c]{7.0cm}{$a = \cos^{-1} \frac{1}{\sqrt{13}} \sqrt{ 5 + 2\sqrt{3} }  \approx 0.20\pi$, \\ $b=\cos^{-1} \frac{1}{\sqrt{13}} \sqrt{  2(4 - \sqrt{3} ) } \approx 0.30\pi$, \\ $\alpha=\frac{2}{3}\pi$, \\ $\beta = \pi - \sin^{-1} \frac{1}{\sqrt{6}}  \sqrt{4+\sqrt{3}} \approx 0.57\pi$, \\ $\gamma = \frac{1}{2}\pi$, \\ $\delta = \sin^{-1}  \frac{1}{\sqrt{6}}  \sqrt{4+\sqrt{3}}  \approx 0.43\pi $ } & \parbox[c][1.5cm][c]{1.15cm}{$ 8\alpha^3, \\ 12\beta^2\delta^2,\\ 6\gamma^4 $} \\
	\hline	
	$S5$ & $36$ & \parbox[c][3.25cm][c]{7.0cm}{$a=\cos^{-1}  \frac{ \sin \frac{2}{9}\pi + 2 \sin \frac{4}{9}\pi } { \sqrt{3}  (1 + \cos \frac{2}{9}\pi ) } \approx 0.17\pi$, \\ $b=\cos^{-1}  \frac{1}{3} (4 \sin^2\frac{1}{9}\pi - \sqrt{3} \cot \frac{4}{9}\pi \\ + 2\sqrt{3} \cos \frac{2}{9}\pi \cot \frac{4}{9}\pi + 4 \sin \frac{4}{9}\pi \tan\frac{1}{9}\pi ) $ \\[0.75ex] $\approx 0.26\pi$, \\ $\alpha=\frac{4}{9}\pi$, $\beta = \frac{7}{9}\pi$, $\gamma=\frac{1}{3}\pi$, $\delta = \frac{5}{9}\pi$ } & \parbox[c][2.75cm][c]{1.15cm}{$ 18\alpha\beta^2, \\ 6\alpha^2\delta^2, \\ 6\gamma\delta^3, \\ 6\alpha\gamma^3\delta, \\ 2\gamma^6$} \\
	\hline
	$S6$ & $36$ & \parbox[c][2.5cm][c]{7.0cm}{$a = \cos^{-1} ( 4 \cos \frac{1}{9}\pi - 3 ) \approx 0.23\pi$, \\ $ b = \cos^{-1} (  6 \cos \frac{1}{9}\pi + 2 \sqrt{3} \sin \frac{1}{9}\pi$ \\ $- 3 \sqrt{3} \tan \frac{1}{9}\pi  - 4 ) \approx 0.12\pi$, \\[0.75ex] $\alpha = \frac{1}{3}\pi$, $\beta = \frac{5}{9}\pi$, $\gamma = \frac{7}{18}\pi$, $\delta = \frac{5}{6}\pi$.} & \parbox[c][2.5cm][c]{1.15cm}{$14 \alpha\delta^2, \\ 10\alpha\beta^3, \\ 8\gamma^3\delta, \\ 6\alpha^2\beta\gamma^2 $} \\
	\hline
	\end{tabular}
\egroup
\end{center}
\caption{Data of isolated tilings 2}
\label{SpecialTilingData2}
\end{table}

\begin{table}[htp]
\begin{center}
\bgroup
\def\arraystretch{1.5}
    \begin{tabular}[]{ | c | c | c | M{2.05cm} | }
	\hline
	Tilings & $f$ & Edges and Angles & Vertices \\ \hhline{|====|}
	$E, E^{\prime}$ & $\ge6$ & \parbox[c][1.5cm][c]{5.75 cm}{$a=\cos^{-1}  \frac{\cos \alpha}{\cos \alpha - 1} $, \\ $b=\cos^{-1} \frac{ (2\cos \alpha - 1)\cos (\alpha + \beta) - \cos^2 \alpha  }{ ( 1 - \cos \alpha )^2 } $} & \\
	 \hhline{|====|}
	$E$ & $\ge6$ &  \parbox[c][1.25cm][c]{5.75 cm}{$\alpha + \gamma + \delta = 2\pi $, $\beta = \frac{4}{f}\pi$} & \parbox[c][1.25cm][c]{4cm}{$f \alpha\gamma\delta, \\[0.5ex] 2\beta^{\frac{f}{2}} $} \\
	\hline
	\multirow{6}{*}{$E^{\prime}$} & \multirow{6}{*}{$\ge8$} & \parbox[c][1.9cm][c]{5.75 cm}{ $\alpha = \tfrac{2}{3}\pi$, $\beta =\frac{4}{f}\pi$, \\ $\gamma + \delta = \tfrac{4}{3}\pi$ } & \parbox[c][1.9cm][c]{4 cm}{$(f-6)\alpha\gamma\delta, \\[0.25ex] 2\alpha^3, \\[0.25ex] 6\beta^{\frac{f}{6}}\gamma\delta $}  \\
	\cline{3-4}
	& & \parbox[c][1.8 cm][c]{5.75 cm}{$\alpha = (1-\frac{4}{f})\pi$, $\beta =\frac{4}{f}\pi$, \\[0.5ex] $\gamma + \delta = (1 + \frac{4}{f} )\pi$ } & \parbox[c][1.8 cm][c]{4cm}{$(f-4)\alpha\gamma\delta, \\[0.25 ex] 2\alpha^2\beta^{2}, \\ 4\beta^{\frac{f}{4}-1}\gamma\delta $} \\
	\cline{3-4}
	 &  & \parbox[c][2.1cm][c]{5.75 cm}{$n \in ( \tfrac{f}{8}, \tfrac{f}{6} - \tfrac{1}{3} ]$, \\[0.5ex]  $\alpha = \frac{4n}{f}\pi$, $\beta = \frac{4}{f}\pi$, \\[0.5ex] $\gamma+\delta = (2 - \frac{4n}{f} )\pi$, $\gamma>\pi$ 
 }  & \parbox[c][2.1cm][c]{4cm}{$(f-6)\alpha\gamma\delta, \\[0.5ex] 2\alpha^3\beta^{\frac{f}{2} - 3n}, \\ 6\beta^{n}\gamma\delta$ } \\
	\cline{3-4}
	 &  & \parbox[c][2.15cm][c]{5.75 cm}{$n \in ( \tfrac{f}{8},  \tfrac{f}{4} - 1 )$, \\[0.5ex] $\alpha = \frac{4n}{f}\pi$, $\beta = \frac{4}{f}\pi$, \\[0.5ex]  $\gamma+\delta = (2 - \frac{4n}{f} )\pi$, $\gamma>\pi$ }  & \parbox[c][2.15cm][c]{4cm}{$(f-4)\alpha\gamma\delta, \\[0.5ex] 2\alpha^2\beta^{\frac{f}{2} - 2n}, \\ 4\beta^{n}\gamma\delta$ } \\
	\cline{3-4}
	& & \parbox[c][1.85cm][c]{5.75 cm}{$\alpha=\pi$, $\beta=\frac{4}{f}\pi$, \\ $\gamma+\delta=\pi$} & \parbox[c][1.85cm][c]{4cm}{$(f-2)\alpha\gamma\delta, \\[0.2ex]  2\alpha\beta^{\frac{f}{4}}, \\[0.25ex] 2\beta^{\frac{f}{4}}\gamma\delta$} \\
	\cline{3-4}
	& & \parbox[c][1.9cm][c]{5.75 cm}{$\alpha= (1 - \frac{4}{f} ) \pi$, $\beta=\frac{4}{f}\pi$, \\[0.5ex] $\gamma+\delta= (1 + \frac{4}{f} ) \pi$} & \parbox[c][1.9cm][c]{4cm}{$(f-2)\alpha\gamma\delta, \\[0.4ex]  2\alpha\beta^{\frac{f}{4}+1}, \\[0.2 ex] 2\beta^{\frac{f}{4}-1}\gamma\delta$} \\
	\cline{3-4}
	& & \parbox[c][2.75cm][c]{5.75 cm}{$n \in \begin{cases}  ( \frac{f}{4}, \frac{3f}{8} ), \ \text{if } \alpha > \pi, \\ ( \frac{f}{8}, \frac{f}{4} - 1 ), \ \text{if } \gamma > \pi; \end{cases}$ \\[0.5ex] $\alpha =  \frac{4n}{f}\pi$, $\beta = \frac{4}{f}\pi$, \\ $\gamma+\delta = ( 2 - \tfrac{4n}{f} )\pi $ } & \parbox[c][2.2cm][c]{4cm}{$ (f-2)\alpha\gamma\delta, \\[0.25ex]  2\alpha\beta^{\frac{f}{2} - n}, \\ 2\beta^{n}\gamma\delta $} \\
	\hline
	\end{tabular}
\egroup
\end{center}
\caption{Data of earth map tilings 1}
\label{EMTData1}
\end{table}

\begin{table}[htp]
\begin{center}
\bgroup
\def\arraystretch{1.5}
    \begin{tabular}[]{ | c | c | c | @{}c@{} | }
	\hline
	Tilings & $f$ & Edges and Angles & Vertices \\ \hhline{|====|}
	$E^{\prime\prime}, E^{\prime\prime\prime}$ & $\ge8$ & \parbox[c][1.5cm][c]{5.75 cm}{$a=\cos^{-1}  \frac{\cos \alpha}{\cos \alpha - 1} $, \\ $b=\cos^{-1} \frac{ (2\cos \alpha - 1)\cos (\alpha + \beta) - \cos^2 \alpha  }{ ( 1 - \cos \alpha )^2 } $} & \parbox[c][1cm][c]{2.4 cm}{\quad} \\
	 \hhline{|====|}
	\multirow{4}{*}{$E^{\prime\prime}$} & \multirow{4}{*}{$\ge8$} & \parbox[c][2cm][c]{5.75 cm}{$\alpha=\pi$, $\beta = \frac{4}{f}\pi$, \\ $\gamma+\delta = \pi$ } &  \parbox[c][2cm][c]{2 cm}{$(f-4)\alpha\gamma\delta, \\[0.25ex] 4\alpha\beta^{\frac{f}{4}}, \\[0.4ex] 2\gamma^2\delta^2$} \\
	\cline{3-4}
	 &  &   \parbox[c][2.1cm][c]{5.75 cm}{$n \in (  \tfrac{f}{8},  \tfrac{f}{4} )$, \\[0.5ex] $\alpha = (2 - \frac{4n}{f} )\pi$, $\beta = \frac{4}{f}\pi$, \\ $\gamma+\delta = \frac{4n}{f}\pi$} &  \parbox[c][2cm][c]{2 cm}{$(f-4)\alpha\gamma\delta, \\[0.5ex]  4\alpha\beta^{n}, \\[0.5ex] 2\beta^{\frac{f}{2}-2n}\gamma^2\delta^2$}  \\
	\cline{3-4}
	 &  &   \parbox[c][2.1cm][c]{5.75 cm}{$n \in (  \tfrac{f}{8},  \tfrac{f}{6} )$, \\[0.5ex] $\alpha = (2 - \frac{4n}{f} )\pi$, $\beta = \frac{4}{f}\pi$, \\ $\gamma+\delta = \frac{4n}{f}\pi$} & \parbox[c][2cm][c]{2 cm}{ $(f-6)\alpha\gamma\delta, \\[0.5ex] 6\alpha\beta^{n}, \\[0.5ex] 2\beta^{\frac{f}{2}-3n}\gamma^3\delta^3$ } \\
	\cline{3-4}
	 &  &   \parbox[c][2 cm][c]{5.75 cm}{$\alpha =\tfrac{4}{3}\pi$, $\beta = \frac{4}{f}\pi$, \\ $\gamma+\delta = \tfrac{2}{3}\pi$} & \parbox[c][2cm][c]{2 cm}{ $(f-6)\alpha\gamma\delta, \\[0.5ex] 6\alpha\beta^{\frac{f}{6}}, \\[0.5ex] 2\gamma^3\delta^3$ } \\
	\cline{4-4}
	\hline
$E^{\prime\prime\prime}$ &  $\ge8$ &  \parbox[c][2.6cm][c]{5.75 cm}{ $\alpha = ( \frac{4}{3} - \frac{4}{3f} )\pi$, $\beta = \frac{4}{f}\pi$, \\[0.75ex] $\gamma = ( \frac{2}{3} - \frac{2}{3f} )\pi$, $\delta = \frac{2}{f}\pi$ } & \parbox[c][2cm][c]{2 cm}{$(f-6)\alpha\gamma\delta,  \\[0.5ex]  2\gamma^3\delta, \\[0.5ex] 4\alpha\beta^{\frac{f+2}{6}}, \\ 2\alpha\beta^{\frac{f-4}{6}}\delta^2$} \\
	\hline
	\end{tabular}
\egroup
\end{center}
\caption{Data of earth map tilings 2}
\label{EMTData2}
\end{table}

\section{Basic Tools} \label{SecBasic}

\subsection{Concepts and Notations} \label{SubsecConcept}

\subsubsection*{Quadrilateral}

A polygon is {\em simple} if the boundary is simple. A polygon is {\em convex} if it is simple and every angle $\le \pi$. By \cite[Lemma 1]{gsy}, at least one tile in a tiling of the sphere is simple. If all the tiles are congruent, then all tiles are simple. 

The polygons in Sections \ref{SecBasic}, \ref{SecRat}, \ref{SecNonRat} are assumed to be almost equilateral quadrilaterals. In Section \ref{SecGeom}, we also calculate geometric data of $a^2bc$ quadrilaterals in tilings. 

For quadrilaterals in tilings, we assume the angles and the edges have values in $(0,2\pi)$. The simple tile implies $a < \pi$ for both quadrilaterals in Figure \ref{StdQuad}. 

Since the area of the quadrilateral is the surface area of the unit sphere divided by $f$, the {\em quadrilateral angle sum} is 
\begin{align}\label{QuadSum}
 \alpha +  \beta +  \gamma + \delta = (2 + \tfrac{4}{f} )\pi.
\end{align}

\subsubsection*{Vertex}

The {\em vertex angle sum} of a vertex $\alpha^m\beta^n\gamma^k\delta^l$, consisting of $m$ copies of $\alpha$ and $n$ copies of $\beta$ and $k$ copies of $\gamma$ and $l$ copies of $\delta$, is
\begin{align}\label{VertexAngSum}
m \alpha + n \beta + k \gamma + l \delta = 2\pi.
\end{align} 
By \eqref{QuadSum}, at least one of the non-negative integers $m,n,k,l$ is zero. At a vertex, these are generic notations reserved for the numbers of $\alpha, \beta, \gamma, \delta$ respectively. For example, $\alpha\beta^2$ is a vertex with $m=1, n=2$ and $k=l=0$. 

In our practice, $m,n,k,l$ in a vertex notation are assumed to be $>0$ unless otherwise specified. That is, we only express the angles appearing at a vertex whenever possible. For example, $\alpha^m\beta^n$ does not include $\alpha^m, \beta^n$. Such practice is one subtle difference from \cite{cly}. To streamline the discussion, we give a shorthand argument: when $\alpha\beta^2$ is a vertex, we simply say \quotes{by $\alpha\beta^2$} to mean by \quotes{$\alpha\beta^2$ being a vertex} or \quotes{by the angle sum $\alpha + 2\beta = 2\pi$ of $\alpha\beta^2$}. We use $\alpha=\beta$ to mean $\alpha,\beta$ having the same value. We use $\alpha\neq\beta$ to mean $\alpha, \beta$ having distinct values.

The notation $\alpha\beta^2\cdots$ means a vertex with at least one $\alpha$ and two $\beta$'s, i.e., $m\ge1$ and $n\ge2$. The angle combination in $\cdots$ is called the {\em remainder} of the vertex. A {\em $b$-vertex} is a vertex with a $b$-edge (i.e., with $\gamma,\delta$) and a {\em $\hat{b}$-vertex} is a vertex without (i.e., without $\gamma,\delta$).

The critical step in classifying tilings is to find all the possible vertices. There are various constraints on the possible angle combinations at vertices in a tiling. We call the combinations satisfying the constraints {\em admissible}. An {\em anglewise vertex combination} ($\AVC$) is a collection of all admissible vertices in a tiling. Examples of such constraints are the vertex angle sum and the quadrilateral angle sum. The following is an example of a set of admissible vertices, $\AVC$ \eqref{GenAVC-algade-ga3de} from Proposition \ref{RatAlGaDeProp},
\begin{align*} 
\AVC = \{ \alpha\gamma\delta,  \gamma^3\delta,  \beta^n,   \alpha\beta^n,   \alpha\beta^n\delta^2,  \beta^n\gamma\delta,  \beta^n\gamma^2\delta^2  \}.
\end{align*}
The generic $n$ may take different values at different vertex. We remark that some vertices in an $\AVC$ may not appear in a tiling. For example, the $\AVC$ of the earth map tiling $E$ below has only two vertices,
\begin{align*}
\AVC \equiv \{  \alpha\gamma\delta, \beta^{\frac{f}{2}} \}.
\end{align*}
Here we use \quotes{$\equiv$} instead of \quotes{$=$} to denote the set of all vertices which actually appear in a tiling.

An {\em angle sum system} is a linear system consisting of the quadrilateral angle sum and vertex angle sums. For example, for vertices $\alpha^{m_1}\beta^{n_1}\gamma^{k_1}\delta^{l_1}$, $\alpha^{m_2}\beta^{n_2}\gamma^{k_2}\delta^{l_2}, \alpha^m\beta^n\gamma^k\delta^l$ in a tiling, where $m_i, n_i, k_i, l_i, m, n, k, l \ge 0$ and $1 \le i \le 2$, the angles satisfy the angle sum system below,
\begin{align*}
\begin{cases}
\alpha + \beta + \gamma  + \delta = (2+\tfrac{4}{f} )\pi, \\
m_1 \alpha + n_1 \beta + k_1 \gamma + l_1 \delta = 2\pi, \\
m_2 \alpha + n_2 \beta + k_2 \gamma + l_2 \delta = 2\pi, \\
m \alpha + n \beta + k \gamma + l \delta = 2\pi.
\end{cases}
\end{align*}
If the four equations are linearly independent, then the unique solution implies that all four angles are rational. If some angle is non-rational, then this angle sum system has rank $\le 3$, which we call the {\em non-rationality condition}. 

There is an exception if $\alpha\gamma\delta$ is a vertex. The system is a bit different and the non-rationality condition is rank $=2$. The earth map tiling is an example of tiling having only two vertices. Therefore the angle sum system consists of three equations. It is indeed not always possible to obtain enough linearly independent equations to effectively determine the angles. 

In fact, as seen in \cite{cly}, technical and mostly ad hoc combinatorial arguments are required to derive three vertices in the majority of cases. By dividing into rational angle and non-rational angle analysis albeit artificial, we can systematically determine all the vertices. Our strategy is outlined in Section \ref{SubsecStrategy}, and implementation is in Sections \ref{SecRat}, \ref{SecNonRat}.

The notations $\# \alpha, \# \beta$,... denote the total number of $\alpha$, the total number of $\beta$, etc., in a tiling. If each angle appears exactly once at the quadrilateral, then 
\begin{align*}
f=\# \alpha = \# \beta = \# \gamma = \# \delta.
\end{align*}
We also, for example, denote by $\# \alpha\delta^2$ the total number of vertex $\alpha\delta^2$ in a tiling.
In $\AVC$ \eqref{ad2ab3AVC}, we have $\AVC=\{ \alpha\delta^2, \alpha\beta^3, \gamma^3\delta, \alpha^2\beta\gamma^2 \}$. Then we have
\begin{align*}
f &= \# \alpha = \# \alpha\delta^2 + \# \alpha\beta^3 + 2\# \alpha^2\beta\gamma^2,  \\
f &= \# \beta = 3\# \alpha\beta^3 + \# \alpha^2\beta\gamma^2, \\
f &= \# \gamma = 3\# \gamma^3\delta + 2\# \alpha^2\beta\gamma^2, \\
f &= \# \delta =2 \# \alpha\delta^2 + \# \gamma^3\delta.
\end{align*}

\subsubsection*{Adjacent Angle Deduction}

Angles at a vertex can be arranged in various ways. An {\em adjacent angle deduction} (AAD) is a convenient notation representing the angle arrangement and the tile arrangement at a vertex. Symbolically, \quotes{ $\vert$ } denotes an $a$-edge and \quotes{ $\bvert$ } denotes a $b$-edge. For example, all three pictures in Figure \ref{AADEg} are AADs of $\beta^2\gamma^2$ for the almost equilateral quadrilateral. The AADs of $\bvert \, \gamma \vert \beta \vert \beta \vert \gamma \, \bvert$ in the pictures can be further represented by $\bvert^{\,\delta}  \gamma^{\beta} \vert^{\gamma} \beta^{\alpha} \vert^{\alpha} \beta^{\gamma} \vert^{\beta} \gamma^{\delta} \, \bvert$, $\bvert^{\,\delta} \gamma^{\beta} \vert^{\gamma} \beta^{\alpha} \vert^{\gamma} \beta^{\alpha} \vert^{\beta} \gamma^{\delta} \, \bvert$ and $\bvert^{\,\delta} \gamma^{\beta} \vert^{\alpha} \beta^{\gamma} \vert^{\gamma} \beta^{\alpha} \vert^{\beta} \gamma^{\delta} \, \bvert$ respectively.

\begin{figure}[htp]
\centering
\begin{tikzpicture}[>=latex,scale=1]

\begin{scope}
\draw
	(1.2,0) -- (1.2, 1.2)-- (0, 1.2) -- (-1.2, 1.2) -- (-1.2, 0)
	(-1.2,0) -- (0,0) -- (0, 1.2)
	(0,0) -- (1.2,0)
	(0, -1.2) -- (1.2,-1.2)
	(1.2,0) -- (1.2, -1.2)
	(-1.2,0) -- (-1.2, -1.2);

\draw[line width=2]
	(0,0) -- (0,-1.2)
	(-1.2,0) -- (-1.2, 1.2)
	(1.2,0) -- (1.2,1.2);

\draw[
]
	(0,1.2) -- (1.2, 1.2)
	(0,1.2) -- (-1.2,1.2)
	(-1.2,-1.2) -- (0, -1.2)
	(0, -1.2) -- (1.2,-1.2);

\node at (0.2,1) {\small $\alpha$}; 
\node at (0.2,0.2) {\small $\beta$}; 
\node at (1,0.2) {\small $\gamma$}; 
\node at (1,0.95) {\small $\delta$}; 

\node at (-0.2,0.2) {\small $\beta$}; 
\node at (-0.2,1) {\small $\alpha$}; 
\node at (-1,0.95) {\small $\delta$}; 
\node at (-1,0.2) {\small $\gamma$}; 

\node at (-0.2,-0.25) {\small $\gamma$}; 
\node at (-0.2,-1) {\small $\delta$}; 
\node at (-1,-1) {\small $\alpha$}; 
\node at (-1,-0.25) {\small $\beta$}; 

\node at (0.2,-0.25) {\small $\gamma$}; 
\node at (0.2,-1) {\small $\delta$}; 
\node at (1,-1) {\small $\alpha$}; 
\node at (1,-0.25) {\small $\beta$}; 
\end{scope}

\begin{scope}[xshift = 4 cm]
\draw
	(1.2,0) -- (1.2, 1.2)-- (0, 1.2) -- (-1.2, 1.2) -- (-1.2, 0)
	(-1.2,0) -- (0,0) -- (0, 1.2)
	(0,0) -- (1.2,0)
	(0, -1.2) -- (1.2,-1.2)
	(1.2,0) -- (1.2, -1.2)
	(-1.2,0) -- (-1.2, -1.2);

\draw[line width=2]
	(0,0) -- (0,-1.2)
	(-1.2,0) -- (-1.2, 1.2)
	(0,1.2) -- (1.2, 1.2);

\draw[
]
	(1.2,0) -- (1.2,1.2)
	(0,1.2) -- (-1.2,1.2)
	(-1.2,-1.2) -- (0, -1.2)
	(0, -1.2) -- (1.2,-1.2);

\node at (0.2,0.2) {\small $\beta$}; 
\node at (1,0.2) {\small $\alpha$}; 
\node at (1,0.95) {\small $\delta$}; 
\node at (0.15,0.95) {\small $\gamma$}; 

\node at (-0.2,0.2) {\small $\beta$}; 
\node at (-0.2,1) {\small $\alpha$}; 
\node at (-1,0.95) {\small $\delta$}; 
\node at (-1,0.2) {\small $\gamma$}; 

\node at (-0.2,-0.25) {\small $\gamma$}; 
\node at (-0.2,-1) {\small $\delta$}; 
\node at (-1,-1) {\small $\alpha$}; 
\node at (-1,-0.25) {\small $\beta$}; 

\node at (0.2,-0.25) {\small $\gamma$}; 
\node at (0.2,-1) {\small $\delta$}; 
\node at (1,-1) {\small $\alpha$}; 
\node at (1,-0.25) {\small $\beta$}; 
\end{scope}

\begin{scope}[xshift = 8 cm]
\draw
	(1.2,0) -- (1.2, 1.2)-- (0, 1.2) -- (-1.2, 1.2) -- (-1.2, 0)
	(-1.2,0) -- (0,0) -- (0, 1.2)
	(0,0) -- (1.2,0)
	(0, -1.2) -- (1.2,-1.2)
	(1.2,0) -- (1.2, -1.2)
	(-1.2,0) -- (-1.2, -1.2);

\draw[line width=2]
	(0,0) -- (0,-1.2)
	(0,1.2) -- (-1.2,1.2)
	(0,1.2) -- (1.2, 1.2);

\draw[
]
	(1.2,0) -- (1.2,1.2)
	(-1.2,0) -- (-1.2, 1.2)
	(-1.2,-1.2) -- (0, -1.2)
	(0, -1.2) -- (1.2,-1.2);

\node at (0.15,0.95) {\small $\gamma$}; 
\node at (0.2,0.2) {\small $\beta$}; 
\node at (1,0.2) {\small $\alpha$}; 
\node at (1,0.95) {\small $\delta$}; 

\node at (-0.2,0.2) {\small $\beta$}; 
\node at (-0.15,0.95) {\small $\gamma$}; 
\node at (-1,0.95) {\small $\delta$}; 
\node at (-1,0.2) {\small $\alpha$}; 

\node at (-0.2,-0.25) {\small $\gamma$}; 
\node at (-0.2,-1) {\small $\delta$}; 
\node at (-1,-1) {\small $\beta$}; 
\node at (-1,-0.2) {\small $\alpha$}; 

\node at (0.2,-0.25) {\small $\gamma$}; 
\node at (0.2,-1) {\small $\delta$}; 
\node at (1,-1) {\small $\beta$}; 
\node at (1,-0.2) {\small $\alpha$}; 
\end{scope}

\end{tikzpicture}
\caption{Adjacent angle deduction (AAD)}
\label{AADEg}
\end{figure}
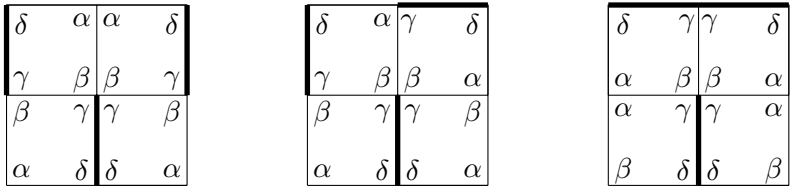

As seen above, the AAD notations can be regarded as mini pictures. Similar to their pictorial counterparts, the notations can be rotated and reversed. For example, the AAD of the  second picture can also be written as $\vert^{\gamma} \beta^{\alpha} \vert^{\gamma} \beta^{\alpha} \vert^{\beta} \gamma^{\delta} \, \bvert^{\,\delta} \gamma^{\beta} \vert$ (rotation) and $\bvert^{\,\delta} \gamma^{\beta} \vert^{\alpha} \beta^{\gamma} \vert^{\alpha} \beta^{\gamma} \vert^{\beta} \gamma^{\delta} \, \bvert$ (reversion).

The AAD notation can be flexible. For example, we write $\beta^{\alpha}\vert^{\alpha}\beta$ (the first picture of Figure \ref{AADEg}) if it is our focus on $\beta^2\gamma^2$. We use $\beta^{\alpha} \vert^{\alpha} \beta \cdots $ to denote a vertex with such angle arrangement.

The AAD has {\em reciprocity property}: an AAD $\lambda^{\theta} \vert ^{\rho} \mu$ at $\lambda\mu \cdots$ implies an AAD at $\theta^{\lambda} \vert^{\mu}\rho$ at $\theta\rho\cdots$ and vice versa.

We give an example of proof by AAD. Up to rotation and reversion, the possible AADs for $\alpha\vert\alpha$ are $\alpha^{\beta} \vert^{\beta} \alpha, \alpha^{\beta} \vert^{\delta} \alpha, \alpha^{\delta} \vert^{\delta} \alpha$. If $\beta^2\cdots, \delta^2\cdots$ are not vertices, then $\alpha \vert \alpha$ has {\em unique} AAD $\alpha ^{\beta} \vert^{\delta} \alpha$. Moreover, a vertex $\alpha^3$ has two possible AADs $\vert^{\delta}\alpha^{\beta}\vert^{\delta}\alpha^{\beta}\vert^{\delta}\alpha^{\beta}\vert, \vert^{\delta}\alpha^{\beta}\vert^{\delta}\alpha^{\beta}\vert^{\beta}\alpha^{\delta}\vert$, depicted in Figure \ref{AAD-al3}. This implies that $\beta\vert\delta\cdots$ is always a vertex. 

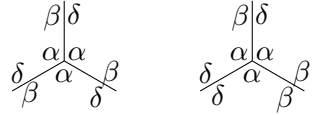
\begin{figure}[htp] 
\centering
\begin{tikzpicture}

\tikzmath{\r=0.8;}

\begin{scope}

\foreach \a in {0,...,2} {
\draw[rotate=120*\a]
	(0:0) -- (90:\r);	
}

\node at (30:0.25*\r) {\small $\alpha$};
\node at (150:0.25*\r) {\small $\alpha$};
\node at (270:0.25*\r) {\small $\alpha$};

\node at (78:0.8*\r) {\small $\delta$};
\node at (105:0.75*\r) {\small $\beta$};

\node at (195:0.8*\r) {\small $\delta$};
\node at (225:0.8*\r) {\small $\beta$};

\node at (312:0.8*\r) {\small $\delta$};
\node at (345:0.8*\r) {\small $\beta$};

\end{scope}

\begin{scope}[xshift=2.5cm]

\foreach \a in {0,...,2} {
\draw[rotate=120*\a]
	(0:0) -- (90:\r);	
}

\node at (30:0.25*\r) {\small $\alpha$};
\node at (150:0.25*\r) {\small $\alpha$};
\node at (270:0.25*\r) {\small $\alpha$};

\node at (78:0.8*\r) {\small $\delta$};
\node at (105:0.75*\r) {\small $\beta$};

\node at (195:0.8*\r) {\small $\delta$};
\node at (225:0.8*\r) {\small $\delta$};

\node at (310:0.85*\r) {\small $\beta$};
\node at (345:0.85*\r) {\small $\beta$};

\end{scope}

\end{tikzpicture}
\caption{The two possible AADs of $\alpha^3$}
\label{AAD-al3}
\end{figure}

Some typical applications of AAD are listed below:
\begin{itemize}
\item If $\beta\vert\delta\cdots$ is not a vertex, then $m$ is even in $\alpha^m$.
\item If $\delta \vert \delta \cdots$ is not a vertex, then $\alpha^{\delta}\vert^{\delta}\alpha\cdots$ is also not a vertex.
\item If $\beta \vert \beta \cdots, \delta \vert \delta \cdots$ are not vertices, then $\alpha\vert \alpha $ has the unique AAD $\alpha^{\beta} \vert ^{\delta} \alpha$.
\item If $\beta\vert \beta \cdots, \beta \vert \delta \cdots$ are not vertices, then $\alpha\alpha\alpha$ cannot be a vertex. In other words, there are no three consecutive $\alpha$'s at a vertex.
\end{itemize}
The application of AAD depends on the information available. In principle, the AAD argument can be programmed in decision algorithms.  

These notations have a significant advantage. The discussion of angle and tile arrangements can be more efficiently and concisely conducted in place of drawing pictures. Luk and Yan invented this system of notations to substitute tens of pictures in the studies of pentagonal tilings.

\subsection{Technique} 

In the almost equilateral quadrilateral in Figure \ref{StdQuad}, we can see a symmetry of $\alpha \leftrightarrow \beta$ and $\gamma \leftrightarrow \delta$. We use \quotes{up to symmetry} to refer to this.

\subsection*{Combinatorics}

Let $v_i$ be the number of vertices of degree $i \ge 3$. From \cite{cly}, the basic formulae about edge-to-edge tilings of the sphere by quadrilaterals are
\begin{align}
\label{FCount}
f &= 6 + \sum_{h\ge4} (h-3)v_h, \\
\label{VCount}
v_3 &= 8 + \sum_{h\ge4} (h - 4) v_h.
\end{align}
Equation \eqref{FCount} implies $f\ge6$, and $f=6$ if and only if all vertices have degree $3$. Equation \eqref{VCount} implies $v_3 \ge 8$, which further implies that degree $3$ vertices always exist.

In \cite{awy, gsy, ua, wy, wy2}, a crucial step in classifying tilings is to find all admissible vertices. This means, we need to find various constraints that angle combinations at vertices must satisfy. Here we list some combinatorial constraints.

\begin{lem}[Counting Lemma, {\cite[Lemma 4]{cly}}] In a tiling of the sphere by congruent polygons, suppose two different angles $\theta,\varphi$ appear the same number of times in the polygon. If the number of times $\theta$ appears at every vertex is no more than the number of times $\varphi$ appears, then at every vertex the number of times $\theta$ appears is the same as the number of times $\varphi$ appears.
\end{lem}

The assumption is that every $\theta^k\varphi^l\cdots$ has $l \ge k \ge 0$ and no $\theta, \varphi$ in the remainder. The conclusion is that every such vertex is $\theta^k\varphi^k\cdots$, with no $\theta, \varphi$ in the remainder.

\begin{lem}[Parity Lemma, {\cite[Lemma 2]{cly}}]  \label{PaLem} The total number of $\gamma$ and $\delta$ at any vertex is even.
\end{lem}

\begin{lem}[Balance Lemma, {\cite[Lemma 6]{cly}}] \label{BaLem} In a tiling of the sphere by congruent almost equilateral quadrilaterals, $\gamma^2\cdots$ is not a vertex if and only if $\delta^2\cdots$ is not a vertex. If $\gamma^2\cdots, \delta^2\cdots$ are not vertices, then every $b$-vertex has exactly one $\gamma$ and one $\delta$.
\end{lem}

\begin{lem}[{\cite[Lemma 9]{cly}}] \label{2HDAngLem} In a tiling of the sphere by congruent quadrilaterals, if two angles $\theta_1,\theta_2$ do not appear at any degree $3$ vertex, then there is a degree $4$ vertex $\theta_i^3\cdots$ ($i=1$ or $2$) or $\theta_i^2\theta_j\cdots$ ($i,j=1,2$), or a degree $5$ vertex $\theta_1^p\theta_2^q$ ($p+q=5$).
\end{lem}

\begin{lem}[{\cite[Lemma 10]{cly}}] \label{a3Lem} In a tiling of the sphere by congruent quadrilaterals, if $\theta^3$ is the unique degree $3$ vertex, then $f \ge 24$ and there is a degree $4$ vertex without $\theta$. 
\end{lem}

\begin{lem}[{\cite[Lemma 11]{cly}}] \label{ab2Lem} In a tiling of the sphere by congruent quadrilaterals, if $\theta^2\varphi$ is the unique degree $3$ vertex, then $f \ge 16$ and there is a degree $4$ vertex without $\theta$.
\end{lem}

In the last three lemmas, the technique of counting angles is involved. Whenever counting is applied, implicitly there is a criterion for distinguishing angles which is often clear in the context.

\subsection*{Geometry}

More constraints on angle combinations at vertices can be derived from the conditions on angle values due to geometric reasons. Here are some of these geometric constraints.

\begin{lem}[{\cite[Lemma 7]{cly}}] \label{AlConvexLem} In a tiling of the sphere by congruent quadrilaterals, there is at most one angle $\ge\pi$ in the quadrilateral. 
\end{lem}

\begin{lem}[{\cite[Lemma 3]{wy2}}] \label{ExchLem} In a simple almost equilateral quadrilateral, $\alpha \ge \beta$ if and only if $\gamma \ge \delta$.
\end{lem}

\begin{lem}[{\cite[Lemma 14]{cly}}, {\cite[Lemma 2.1]{avc}}]\label{LunEstLem} In a simple almost equilateral quadrilateral, 
\begin{itemize}
\item if $\alpha, \beta, \gamma < \pi$, then $\beta + \pi > \gamma + \delta$ and $\delta + \pi > \beta + \gamma$; 
\item if $\alpha, \beta, \delta < \pi$, then $\alpha + \pi > \gamma + \delta$ and $\gamma + \pi > \alpha  + \delta$.
\end{itemize}
\end{lem}

\begin{lem}[{\cite[Lemma 15]{cly}}] \label{ATriLem} In a simple almost equilateral quadrilateral, 
\begin{itemize}
\item if $\gamma, \delta < \pi$, then $\alpha > \gamma$ if and only if $\beta > \delta$; 
\item if $\gamma < \pi$, then $\beta = \delta$ if and only if $a = b$;
\item if $\delta < \pi$, then $\alpha = \gamma$ if and only if $a = b$. 
\end{itemize}
\end{lem}

In fact, the proof of \cite[Lemma 15]{cly} shows that, if $\gamma<\pi$, then $\beta>\delta$ if and only if $a<b$, and $\beta=\delta$ if and only if $a=b$.

\begin{lem}\label{Edges<pi} In a simple quadrilateral, if three angles $<\pi$ and the two edges between these angles are $<\pi$, then the other two edges are also $<\pi$.
\end{lem}

\begin{proof} We call a triangle {\em standard} when each edge and each angle $<\pi$. A standard triangle is simple and convex.

\begin{figure}[htp] 
\centering
\begin{tikzpicture}[>=latex,scale=0.75]

\begin{scope}[] 


\tikzmath{
\r=2; \R=sqrt(3);
\a=120; \aa=\a/2;
\xP=0; \yP=\R;
\aOP=36; 
\aPOne=64;
\aPThree = 160;
\aPX = 125;
\mPOne=tan(90+\aPOne);
\aPTwo=\aOP;
\mPTwo=tan(-90-\aPTwo);
\mPThree=tan(-90-\aPThree);
\mPX=tan(-90-\aPX);
\xCR=-3; \yCR=0;
\xCOne = 1; \yCOne=0;
\xCTwo = -1; \yCTwo=0;
\xCThree=\xCTwo; \yCThree = \yCTwo;
\xCX=\xCTwo; \yCX= \yCTwo;
}

\coordinate (O) at (0,0);
\coordinate (C1) at (1,0);
\coordinate (C2) at (-1,0);


\draw[dotted]
	(O) circle ({sqrt(3)});
	
\draw[gray!50]
	(C2) circle (\r)
	(C1) circle (\r);

\draw[gray!50]
	([shift={(-60:2)}]-1,0) arc (-60:60:2);

\draw[gray!50]
	([shift={(-60+180:2)}]1,0) arc (-60+180:60+180:2);

\pgfmathsetmacro{\xPOne}{ ( 2*((\mPOne)*(\yCOne)+(\xCOne)) - sqrt( ( 2*( (\mPOne)*(\yCOne)+\xCOne )  )^2 - 4*( (\mPOne)^2 + 1 )*( (\xCOne)^2 + (\yCOne)^2 - \r^2 ) ) )/( 2*( (\mPOne)^2+1 ) ) };
\pgfmathsetmacro{\yPOne}{ \mPOne*\xPOne };

\pgfmathsetmacro{\xPTwo}{ ( 2*((\mPTwo)*(\yCTwo)+(\xCTwo) ) + sqrt( ( 2*( (\mPTwo)*\yCTwo+\xCTwo )  )^2 - 4*( (\mPTwo)^2 + 1 )*( (\xCTwo)^2 + (\yCTwo)^2 - \r^2 ) ) )/( 2*( (\mPTwo)^2+1 ) ) };
\pgfmathsetmacro{\yPTwo}{ \mPTwo*\xPTwo };

\pgfmathsetmacro{\xPThree}{ ( 2*((\mPThree)*(\yCThree)+(\xCThree) ) + sqrt( ( 2*( (\mPThree)*\yCThree+\xCThree )  )^2 - 4*( (\mPThree)^2 + 1 )*( (\xCThree)^2 + (\yCThree)^2 - \r^2 ) ) )/( 2*( (\mPThree)^2+1 ) ) };
\pgfmathsetmacro{\yPThree}{ \mPThree*\xPThree };

\pgfmathsetmacro{\xPX}{ ( 2*((\mPX)*(\yCX)+(\xCX) ) + sqrt( ( 2*( (\mPX)*\yCX+\xCX )  )^2 - 4*( (\mPX)^2 + 1 )*( (\xCX)^2 + (\yCX)^2 - \r^2 ) ) )/( 2*( (\mPX)^2+1 ) ) };
\pgfmathsetmacro{\yPX}{ \mPX*\xPX };

\pgfmathsetmacro{\xPPOne}{ ( 2*((\mPOne)*(\yCOne)+(\xCOne)) + sqrt( ( 2*( (\mPOne)*(\yCOne)+\xCOne )  )^2 - 4*( (\mPOne)^2 + 1 )*( (\xCOne)^2 + (\yCOne)^2 - \r^2 ) ) )/( 2*( (\mPOne)^2+1 ) ) };
\pgfmathsetmacro{\yPPOne}{ \mPOne*\xPPOne };

\pgfmathsetmacro{\xPPTwo}{ ( 2*((\mPTwo)*(\yCTwo)+(\xCTwo) ) - sqrt( ( 2*( (\mPTwo)*\yCTwo+\xCTwo )  )^2 - 4*( (\mPTwo)^2 + 1 )*( (\xCTwo)^2 + (\yCTwo)^2 - \r^2 ) ) )/( 2*( (\mPTwo)^2+1 ) ) };
\pgfmathsetmacro{\yPPTwo}{ \mPTwo*\xPPTwo };

\pgfmathsetmacro{\xPPThree}{ ( 2*((\mPThree)*(\yCThree)+(\xCThree) ) - sqrt( ( 2*( (\mPThree)*\yCThree+\xCThree )  )^2 - 4*( (\mPThree)^2 + 1 )*( (\xCThree)^2 + (\yCThree)^2 - \r^2 ) ) )/( 2*( (\mPThree)^2+1 ) ) };
\pgfmathsetmacro{\yPPThree}{ \mPThree*\xPPThree };

\pgfmathsetmacro{\dPOneP}{ sqrt( (\xPOne - \xP)^2 + (\yPOne - \yP)^2 ) };
\pgfmathsetmacro{\aCPOne}{ acos( (2*\r^2 - \dPOneP^2 )/(2*\r^2) ) };
\pgfmathsetmacro{\l}{ \aCPOne/\a };
\pgfmathsetmacro{\rR}{ sqrt( \R^2 + (\xCR)^2  )  };
\pgfmathsetmacro{\aPCR}{ acos( (\xCR)/(\rR) ) }
\pgfmathsetmacro{\aR}{  -( 360 - 2*\aPCR )*(1-\l) };

\coordinate (CR) at (\xCR, \yCR);

\coordinate (P) at (0,{sqrt(3)});

\coordinate[rotate around={\aR:(CR)}] (R) at (P);

\coordinate (PP) at (0,{-sqrt(3)});

\coordinate (P1) at (\xPOne, \yPOne); 
\coordinate (P2) at (\xPTwo, \yPTwo);
\coordinate (P3) at (\xPThree, \yPThree);

\coordinate (PX) at (\xPX, \yPX);

\coordinate (PP1) at (\xPPOne,\yPPOne);
\coordinate (PP2) at (\xPPTwo,\yPPTwo);
\coordinate (PP3) at (\xPPThree, \yPPThree);

\draw[gray!20, dashed]
	(P1) -- (PP1)
	(P2) -- (PP2)
;


\arcThroughThreePoints[gray!50, dashed, <-]{PP2}{R}{P2};

\arcThroughThreePoints[]{P}{PP}{P1};
\arcThroughThreePoints[]{P2}{PP}{P};


\arcThroughThreePoints[gray!50, dashed, ->]{P1}{PP1}{PX};

\arcThroughThreePoints[]{P1}{PP1}{R};


\arcThroughThreePoints[]{R}{PP2}{P2};


\draw[->]
	(P2) -- ([shift={(\xPTwo, \yPTwo)}]270:0.35);

\draw[->]
	(P1) -- ([shift={(\xPOne, \yPOne)}]-35:0.35);


\node[circle,fill=black,inner sep=0pt,minimum size=2pt] at (P) {};
\node[circle,fill=black,inner sep=0pt,minimum size=2pt] at (P1) {};
\node[circle,fill=black,inner sep=0pt,minimum size=2pt] at (P2) {};
\node[circle,fill=black,inner sep=0pt,minimum size=2pt] at (R) {};

\node at (90: 2.0) {\small $P$};
\node at (270: 2.1) {\small $P^{\ast}$};
\node at (160: 1.35) {\small $Q$};
\node at (335: 3.25) {\small $Q^{\ast}$};
\node at (50: 1.5) {\small $S$};
\node at (-30: 1.4) {\small $Q'$};
\node at (285:0.55) {\small $R$};
\node at (236:2.75) {\small \textcolor{gray!50}{$S^{\ast}$}};


\end{scope}

\begin{scope}[xshift=7cm] 


\tikzmath{
\r=2; \R=sqrt(3);
\a=120; \aa=\a/2;
\xP=0; \yP=\R;
\aOP=36; 
\aPOne=64;
\aPThree = 160;
\aPX = 125;
\mPOne=tan(90+\aPOne);
\aPTwo=\aOP;
\mPTwo=tan(-90-\aPTwo);
\mPThree=tan(-90-\aPThree);
\mPX=tan(-90-\aPX);
\xCR=-3; \yCR=0;
\xCOne = 1; \yCOne=0;
\xCTwo = -1; \yCTwo=0;
\xCThree=\xCTwo; \yCThree = \yCTwo;
\xCX=\xCTwo; \yCX= \yCTwo;
}

\coordinate (O) at (0,0);
\coordinate (C1) at (1,0);
\coordinate (C2) at (-1,0);


\draw[dotted]
	(O) circle ({sqrt(3)});
	
\draw[gray!50]
	(C2) circle (\r)
	(C1) circle (\r);

\draw[gray!50]
	([shift={(-60:2)}]-1,0) arc (-60:60:2);

\draw[gray!50]
	([shift={(-60+180:2)}]1,0) arc (-60+180:60+180:2);

\pgfmathsetmacro{\xPOne}{ ( 2*((\mPOne)*(\yCOne)+(\xCOne)) - sqrt( ( 2*( (\mPOne)*(\yCOne)+\xCOne )  )^2 - 4*( (\mPOne)^2 + 1 )*( (\xCOne)^2 + (\yCOne)^2 - \r^2 ) ) )/( 2*( (\mPOne)^2+1 ) ) };
\pgfmathsetmacro{\yPOne}{ \mPOne*\xPOne };

\pgfmathsetmacro{\xPTwo}{ ( 2*((\mPTwo)*(\yCTwo)+(\xCTwo) ) + sqrt( ( 2*( (\mPTwo)*\yCTwo+\xCTwo )  )^2 - 4*( (\mPTwo)^2 + 1 )*( (\xCTwo)^2 + (\yCTwo)^2 - \r^2 ) ) )/( 2*( (\mPTwo)^2+1 ) ) };
\pgfmathsetmacro{\yPTwo}{ \mPTwo*\xPTwo };

\pgfmathsetmacro{\xPThree}{ ( 2*((\mPThree)*(\yCThree)+(\xCThree) ) + sqrt( ( 2*( (\mPThree)*\yCThree+\xCThree )  )^2 - 4*( (\mPThree)^2 + 1 )*( (\xCThree)^2 + (\yCThree)^2 - \r^2 ) ) )/( 2*( (\mPThree)^2+1 ) ) };
\pgfmathsetmacro{\yPThree}{ \mPThree*\xPThree };

\pgfmathsetmacro{\xPX}{ ( 2*((\mPX)*(\yCX)+(\xCX) ) + sqrt( ( 2*( (\mPX)*\yCX+\xCX )  )^2 - 4*( (\mPX)^2 + 1 )*( (\xCX)^2 + (\yCX)^2 - \r^2 ) ) )/( 2*( (\mPX)^2+1 ) ) };
\pgfmathsetmacro{\yPX}{ \mPX*\xPX };

\pgfmathsetmacro{\xPPOne}{ ( 2*((\mPOne)*(\yCOne)+(\xCOne)) + sqrt( ( 2*( (\mPOne)*(\yCOne)+\xCOne )  )^2 - 4*( (\mPOne)^2 + 1 )*( (\xCOne)^2 + (\yCOne)^2 - \r^2 ) ) )/( 2*( (\mPOne)^2+1 ) ) };
\pgfmathsetmacro{\yPPOne}{ \mPOne*\xPPOne };

\pgfmathsetmacro{\xPPTwo}{ ( 2*((\mPTwo)*(\yCTwo)+(\xCTwo) ) - sqrt( ( 2*( (\mPTwo)*\yCTwo+\xCTwo )  )^2 - 4*( (\mPTwo)^2 + 1 )*( (\xCTwo)^2 + (\yCTwo)^2 - \r^2 ) ) )/( 2*( (\mPTwo)^2+1 ) ) };
\pgfmathsetmacro{\yPPTwo}{ \mPTwo*\xPPTwo };

\pgfmathsetmacro{\xPPThree}{ ( 2*((\mPThree)*(\yCThree)+(\xCThree) ) - sqrt( ( 2*( (\mPThree)*\yCThree+\xCThree )  )^2 - 4*( (\mPThree)^2 + 1 )*( (\xCThree)^2 + (\yCThree)^2 - \r^2 ) ) )/( 2*( (\mPThree)^2+1 ) ) };
\pgfmathsetmacro{\yPPThree}{ \mPThree*\xPPThree };

\pgfmathsetmacro{\dPOneP}{ sqrt( (\xPOne - \xP)^2 + (\yPOne - \yP)^2 ) };
\pgfmathsetmacro{\aCPOne}{ acos( (2*\r^2 - \dPOneP^2 )/(2*\r^2) ) };
\pgfmathsetmacro{\l}{ \aCPOne/\a };
\pgfmathsetmacro{\rR}{ sqrt( \R^2 + (\xCR)^2  )  };
\pgfmathsetmacro{\aPCR}{ acos( (\xCR)/(\rR) ) }
\pgfmathsetmacro{\aR}{  -( 360 - 2*\aPCR )*(1-\l) };

\coordinate (CR) at (\xCR, \yCR);

\coordinate (P) at (0,{sqrt(3)});

\coordinate[rotate around={\aR:(CR)}] (R) at (P);

\coordinate (PP) at (0,{-sqrt(3)});

\coordinate (P1) at (\xPOne, \yPOne); 
\coordinate (P2) at (\xPTwo, \yPTwo);
\coordinate (P3) at (\xPThree, \yPThree);

\coordinate (PX) at (\xPX, \yPX);

\coordinate (PP1) at (\xPPOne,\yPPOne);
\coordinate (PP2) at (\xPPTwo,\yPPTwo);
\coordinate (PP3) at (\xPPThree, \yPPThree);

\draw[gray!20, dashed]
	(P1) -- (PP1)
	(P3) -- (PP3)
;

\arcThroughThreePoints[gray!50, dashed, ->]{P3}{R}{PP3};


\arcThroughThreePoints[]{P}{PP}{P1};
\arcThroughThreePoints[]{P2}{PP}{P};


\arcThroughThreePoints[gray!50, dashed, ->]{P1}{PP1}{PX};

\arcThroughThreePoints[]{P1}{PP1}{R};

\arcThroughThreePoints[]{P3}{PP3}{R};

\arcThroughThreePoints[]{P3}{P2}{P};



\draw[->]
	(P3) -- ([shift={(\xPThree, \yPThree)}]90:0.35);

\draw[->]
	(P1) -- ([shift={(\xPOne, \yPOne)}]-35:0.35);

\node[circle,fill=black,inner sep=0pt,minimum size=2pt] at (P) {};
\node[circle,fill=black,inner sep=0pt,minimum size=2pt] at (P1) {};
\node[circle,fill=black,inner sep=0pt,minimum size=2pt] at (P3) {};
\node[circle,fill=black,inner sep=0pt,minimum size=2pt] at (R) {};

\node at (90: 2.0) {\small $P$};
\node at (270: 2.1) {\small $P^{\ast}$};
\node at (160: 1.35) {\small $Q$};
\node at (335: 3.25) {\small $Q^{\ast}$};
\node at (-30: 1.4) {\small $Q'$};
\node at (-64: 1.65) {\small $S$};
\node at (285:0.55) {\small $R$};
\node at (110:2.4) {\small \textcolor{gray!50}{$S^{\ast}$}};


\end{scope}

\end{tikzpicture}
\caption{Quadrilateral $\square PQRS$ with $\angle P, \angle Q, \angle S, PQ, PS <\pi$} 
\label{AllEdges<pi}
\end{figure}
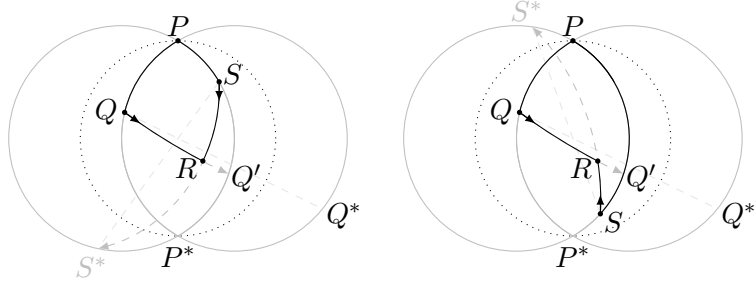

Suppose $\square PQRS$ is such quadrilateral in Figure \ref{AllEdges<pi} where $\angle P, \angle Q, \angle S$, $PQ, PS <\pi$. Then $PQ, PS$ are respectively contained in the left part and right part of the boundary of the lune (the intersection of two hemispheres) defined by antipodal points $P, P^{\ast}$, and $\angle P$. As $\angle  Q, \angle S < \pi$, the rays from $Q$ and $S$, which respectively coincide with $QR$ and $SR$, point towards the interior of the lune. Extending the ray from $Q$ until it meets at $Q'$ on the other side of the boundary, we get a standard triangle $\triangle PQQ'$ where $QQ' < \pi$.

If $S$ is contained in $PQ'$ in the first picture, then $\square PQRS$ being simple and $\angle S < \pi$ imply that the ray from $S$ will eventually intersect at $R$ where $R$ lies between $QQ'$. If $S$ is outside $PQ'$, then it is contained in $Q'P^{\ast}$ in the second picture. So $\square PQRS$ being simple and $\angle S < \pi$ also imply that the ray from $S$ will eventually intersect at $R$ where $R$ lies between $QQ'$. In either case, $QR,RS$ are contained in the lune and hence $QR,RS<\pi$.
\end{proof}

\begin{lem}\label{CTriLem} In a simple almost equilateral quadrilateral, if $\alpha, \beta, \delta < \pi$, then $\gamma > \pi$ implies $\beta > \delta$.
\end{lem}

\begin{proof} By $a, \alpha, \beta, \delta < \pi$, Lemma \ref{Edges<pi} implies $b<\pi$. Moreover, $AC, BD$ in Figure \ref{Ga>pi} is contained in the lune defined by $A, A^{\ast}, \alpha$. So $AC,BD<\pi$ and every triangle contained in $\triangle ABD$ is a standard triangle. We also know that $\triangle ABD$ contains $\square ABCD$ and $\triangle BCD$. Let $\beta', \delta'$ be the base angles of $\triangle BCD$ adjacent to $\beta, \delta$ respectively.

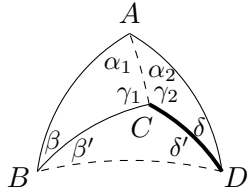
\begin{figure}[htp] 
\centering
\begin{tikzpicture}[>=latex,scale=1.25]

\begin{scope}[xshift=0cm] 

\tikzmath{
\r=2; \R=sqrt(3);
\a=120; \aa=\a/2;
\xP=0; \yP=\R;
\aOP=75; 
\aPOne=\aOP;
\mPOne=tan(90+\aPOne);
\aPTwo=\aOP;
\mPTwo=tan(-90-\aPTwo);
\xCQ=-1.0; \yCQ=0.0;
\xCOne = 1; \yCOne=0;
\xCTwo = -1; \yCTwo=0;
}

\coordinate (O) at (0,0);
\coordinate (C1) at (1,0);
\coordinate (C2) at (-1,0);

\pgfmathsetmacro{\xPOne}{ ( 2*((\mPOne)*(\yCOne)+(\xCOne)) - sqrt( ( 2*( (\mPOne)*(\yCOne)+\xCOne )  )^2 - 4*( (\mPOne)^2 + 1 )*( (\xCOne)^2 + (\yCOne)^2 - \r^2 ) ) )/( 2*( (\mPOne)^2+1 ) ) };
\pgfmathsetmacro{\yPOne}{ \mPOne*\xPOne };

\pgfmathsetmacro{\xPTwo}{ ( 2*((\mPTwo)*(\yCTwo)+(\xCTwo) ) + sqrt( ( 2*( (\mPTwo)*\yCTwo+\xCTwo )  )^2 - 4*( (\mPTwo)^2 + 1 )*( (\xCTwo)^2 + (\yCTwo)^2 - \r^2 ) ) )/( 2*( (\mPTwo)^2+1 ) ) };
\pgfmathsetmacro{\yPTwo}{ \mPTwo*\xPTwo };

\pgfmathsetmacro{\xPPOne}{ ( 2*((\mPOne)*(\yCOne)+(\xCOne)) + sqrt( ( 2*( (\mPOne)*(\yCOne)+\xCOne )  )^2 - 4*( (\mPOne)^2 + 1 )*( (\xCOne)^2 + (\yCOne)^2 - \r^2 ) ) )/( 2*( (\mPOne)^2+1 ) ) };
\pgfmathsetmacro{\yPPOne}{ \mPOne*\xPPOne };

\pgfmathsetmacro{\xPPTwo}{ ( 2*((\mPTwo)*(\yCTwo)+(\xCTwo) ) - sqrt( ( 2*( (\mPTwo)*\yCTwo+\xCTwo )  )^2 - 4*( (\mPTwo)^2 + 1 )*( (\xCTwo)^2 + (\yCTwo)^2 - \r^2 ) ) )/( 2*( (\mPTwo)^2+1 ) ) };
\pgfmathsetmacro{\yPPTwo}{ \mPTwo*\xPPTwo };

\pgfmathsetmacro{\xPQRef}{\xPTwo};
\pgfmathsetmacro{\yPQRef}{\yPTwo};

\coordinate (P) at (0,{sqrt(3)});

\coordinate(Q) at (78:1);

\coordinate (PP) at (0,{-sqrt(3)});

\coordinate (P1) at (\xPOne, \yPOne);
\coordinate (P2) at (\xPTwo, \yPTwo);

\coordinate (PP1) at (\xPPOne,\yPPOne);

\coordinate (PP2) at (\xPPTwo,\yPPTwo);

\arcThroughThreePoints[]{P}{PP}{P1};
\arcThroughThreePoints[]{P2}{PP}{P};

\arcThroughThreePoints[]{Q}{PP1}{P1};
\arcThroughThreePoints[line width=1.5]{P2}{PP2}{Q};
\arcThroughThreePoints[dashed]{P2}{PP1}{P1};

\arcThroughThreePoints[dashed]{Q}{PP}{P};

\node at (90: 1.95) {\small $A$};
\node at (170: 1.2) {\small $B$};
\node at (10: 1.15) {\small $D$};
\node at (80: 0.75) {\small $C$};

\node at (95: 1.4) {\footnotesize $\alpha_1$};
\node at (75: 1.35) {\footnotesize $\alpha_2$};
\node at (145: 1.0) {\footnotesize $\beta$};
\node at (135: 0.7) {\footnotesize $\beta'$};
\node at (90: 1.07) {\footnotesize $\gamma_1$};
\node at (70: 1.15) {\footnotesize $\gamma_2$};
\node at (45: 0.75) {\footnotesize $\delta'$};
\node at (45: 1.05) {\footnotesize $\delta$};

\end{scope}

\end{tikzpicture}
\caption{$\square ABCD$ with $\alpha,\beta, \delta < \pi$ and $\gamma>\pi$}
\label{Ga>pi}
\end{figure}

Since $AB=AD=a$, we know that $\triangle ABD, \triangle ABC$ are isosceles triangles. Then $\beta + \beta' = \delta + \delta'$ and $\alpha_1 =\gamma_1$. So $\gamma > \alpha$ implies $\gamma_2>\alpha_2$. This means $CD < AD= BC$. Then in $\triangle BCD$, we get $\beta' < \delta'$. Hence $\beta > \delta$.
\end{proof}

\begin{lem}\label{TriQuadLem} In a tiling of the sphere by congruent almost equilateral quadrilaterals, we have
\begin{itemize}
\item $\beta = \delta$ if and only if $\gamma = \pi$. 
\item $\alpha = \gamma$ if and only if $\delta = \pi$.
\end{itemize}
\end{lem}

\begin{proof} If $\beta = \delta$, then Lemma \ref{AlConvexLem} implies $\beta, \delta < \pi$. By $b\neq a$, Lemma \ref{ATriLem} implies $\gamma \ge \pi$. Then by Lemma \ref{AlConvexLem}, we get $\alpha<\pi$. If $\gamma > \pi$, then $ \gamma > \alpha $ and Lemma \ref{CTriLem} imply $\beta>\delta$, a contradiction. Hence $\gamma = \pi$.

\begin{figure}[htp] 
\centering
\begin{tikzpicture}

\begin{scope}[scale=1.1]

\draw
	(0,0) -- (1.2,1.2) -- (0,1.2) -- cycle;

\draw[line width=2]
	(0,0) -- (0.5,0.5);

\node at (-0.2, -0.2) {$D$};
\node at (-0.2, 1.4) {$A$};
\node at (1.4, 1.4) {$B$};
\node at (0.8, 0.4) {$C$};

\node at (0.2,1) {\small $\alpha$};
\node at (0.75,0.95) {\small $\beta$};
\node at (0.15,0.45) {\small $\delta$};

\end{scope}

\end{tikzpicture}
\caption{$\triangle ABD$ with $\angle C =\gamma= \pi$}
\label{TriABD}
\end{figure}
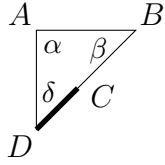
If $\gamma = \pi$, then the quadrilateral is in fact an isosceles triangle $\triangle ABD$ in Figure \ref{TriABD} with edges $AB=AD=a$ and $BD=a+b$. Hence $\beta = \delta$.
\end{proof}

The four angles of the almost equilateral quadrilateral should be related by one single equation. To explain the equation, we need to expand our definition of polygons. 

A {\em general polygon} is a closed path of piecewise geodesic arcs together with a choice of a side. A geodesic arc is a part of a great circle on the sphere. The {\em edges} of a general polygon are geodesic arcs and the {\em vertices} are where edges meet. There are two complementary angles at each vertex. A {\em side} is fixed by a choice of one angle. Figure \ref{PolygonSide} demonstrates how a side of a general quadrilateral is fixed by the choice of angle $\ast$.

\begin{figure}[htp] 
\centering
\begin{tikzpicture}

\begin{scope}[] 

\foreach \a in {0,...,3} {
\draw[rotate=90*\a]
	(45:1) -- (135:1)
	([shift=(270:0.2)]135:1) arc (270:360:0.2);
}

\node at (135:0.6) {$\ast$};

\end{scope}

\begin{scope}[xshift=3cm] 

\foreach \a in {0,2} {
\draw[rotate=90*\a]
	(45:1) -- (135:1);
}

\foreach \b in {0,...,3} {
\draw[rotate=90*\b]
	(0:0) -- (45:1);
}

\draw[]
	([shift=(180:0.2)]45:1) arc (180:225:0.2)
	([shift=(315:0.2)]135:1) arc (315:360:0.2)
	([shift=(45:0.2)]225:1) arc (45:360:0.2)
	([shift=(180:0.2)]315:1) arc (180:495:0.2)
;

\node at (120:0.6) {$\ast$};

\end{scope}

\end{tikzpicture}
\caption{General quadrilaterals as closed paths with chosen sides}
\label{PolygonSide}
\end{figure}
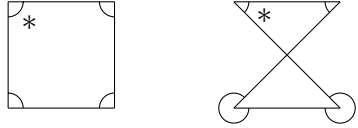

Coolsaet \cite[(2.3), Theorem 2.1]{co} proved the following identity for convex almost equilateral quadrilateral. Cheung \cite{ch,cly} proved the identity without convexity assumption.

\begin{lem}[{\cite[Theorem 21]{ch}, \cite[Lemma 18]{cly}}]\label{CoolsaetLem} The four angles of an almost equilateral quadrilateral satisfy
\begin{align}\label{Coolsaet-Id}
\sin\tfrac{1}{2}\alpha
\sin(\delta-\tfrac{1}{2}\beta)
=\sin\tfrac{1}{2}\beta
\sin(\gamma-\tfrac{1}{2}\alpha).
\end{align}
\end{lem}
We remark that \eqref{Coolsaet-Id} is also true if the quadrilateral is non-simple. It matches the trigonometric Diophantine equation in \cite[Equation (4)]{my}. In Section \ref{SecRat}, we generalise Coolsaet's method \cite[Theorem 3.2]{co} to determine rational angles. In Section \ref{SecGeom}, we shall see in Lemma \ref{AEQuadExists} that \eqref{Coolsaet-Id} serves as one of the criteria for verifying the geometric existence of tiles.

\subsection{Preliminary Cases} 

There are up to four distinct angle values among $\alpha, \beta, \gamma, \delta$. If all angles share the same value, then $a=b$. So a genuine almost equilateral quadrilateral has at least two distinct angle values. 

\begin{prop}\label{TwoAngProp} There is no tiling by congruent almost equilateral quadrilaterals with two distinct angle values.
\end{prop}

\begin{proof} Suppose there are two distinct angle values. Lemma \ref{ExchLem} implies no three angles in the tile sharing the same value. Then we have three possibilities, (1) $\alpha=\beta$ and $\gamma=\delta$, (2) $\alpha=\gamma$ and $\beta = \delta$, (3) $\alpha=\delta$ and $\beta = \gamma$.

Suppose $\alpha=\gamma$ and $\beta = \delta$. Lemma \ref{AlConvexLem} implies $\alpha, \beta, \gamma, \delta < \pi$. By $b\neq a$ and Lemma \ref{TriQuadLem}, $\alpha=\gamma$ if and only if $\delta = \pi$, contradicting $\delta < \pi$.

Suppose $\alpha=\delta$ and $\beta = \gamma$. Up to symmetry, Lemma \ref{ExchLem} implies $\alpha \ge \beta = \gamma \ge \delta = \alpha$. This implies $\alpha = \beta = \gamma = \delta$, a contradiction.

Suppose $\alpha=\beta$ and $\gamma=\delta$, we know $\alpha\neq\gamma$. Lemma \ref{AlConvexLem} implies every angles $< \pi$ so the tile is convex. The quadrilateral angle sum becomes
\begin{align*}
2\alpha + 2\gamma = (2 + \tfrac{4}{f})\pi.
\end{align*}
By \eqref{VCount}, we get $v_3 > 0$. Then Parity Lemma implies that $\alpha\gamma^2$ or $\alpha^3$ is a vertex. 

If $\alpha\gamma^2$ is a vertex, the angle sum system implies $\alpha = \frac{4}{f}\pi$ and $\gamma = (1 - \frac{2}{f})\pi$. By convexity, Lemma \ref{LunEstLem} implies $\alpha + \pi > 2\gamma$ and hence $f<8$, or $f=6$. Then $\alpha = \gamma = \frac{2}{3}\pi$, contradicting $\alpha\neq\gamma$. 

Now $\alpha^3$ must be a vertex. Then $\gamma$ appears at some degree $\ge 4$ vertex. The angle sum system gives $\alpha = \frac{2}{3}\pi$ and $\gamma = (\frac{1}{3} + \frac{2}{f})\pi$. Then $2\alpha + 2\gamma, \alpha + 4\gamma, 6\gamma >2\pi$, which implies $\alpha^2\gamma^{k\ge2}\cdots,  \alpha\gamma^{k\ge4}\cdots, \gamma^{k\ge6}\cdots$ are not vertices. So $\gamma$ only appears at $\gamma^2\cdots = \gamma^4$ and $\gamma = \frac{1}{2}\pi$. By $\gamma = (\frac{1}{3} + \frac{2}{f})\pi$, we get $f=12$. By $\alpha = \frac{2}{3}\pi$ and $\gamma = \frac{1}{2}\pi$, there are no other vertices, notably no $\alpha\gamma\cdots$.

The AAD $\bvert \, \gamma^{\alpha} \vert ^{\alpha} \gamma \, \bvert$ at $\gamma^4$ implies $\alpha^2\cdots$, which is $\alpha^3$. By $\beta=\alpha$ and $\delta=\gamma$, the two possible AADs of $\alpha^3$ in Figure \ref{AAD-al3} are $\vert^{\gamma}\alpha^{\alpha}\vert^{\gamma}\alpha^{\alpha}\vert^{\gamma}\alpha^{\alpha}\vert$ or $\vert^{\gamma}\alpha^{\alpha}\vert^{\gamma}\alpha^{\alpha}\vert^{\alpha}\alpha^{\gamma}\vert$. Both imply $\alpha\gamma\cdots$, a contradiction. So $\gamma^4$ is not a vertex and there is no tiling. 
\end{proof}

\begin{lem}\label{PairsLem} In a tiling of the sphere by congruent almost equilateral quadrilaterals with at least three distinct angle values, up to symmetry, either $\alpha\gamma\delta$ is a vertex, or one of the pairs below are vertices. 
\begin{itemize}
\item $\alpha^3$ and one of $\alpha\gamma^2, \alpha\delta^2, \beta\gamma^2, \beta\delta^2$,
\item $\alpha^2\beta$ and one of $\alpha\gamma^2, \alpha\delta^2, \beta\delta^2$,
\item $\alpha\delta^2$ and $\beta\gamma^2$,
\item $\alpha^3$ and one of $\gamma^4, \delta^4, \gamma^3\delta, \gamma\delta^3, \gamma^2\delta^2$,
\item $\alpha\beta^2$ and one of $\gamma^4, \delta^4, \gamma^3\delta, \gamma\delta^3, \gamma^2\delta^2$,
\item $\alpha\gamma^2$ and one of $\alpha^4, \beta^4, \delta^4, \alpha^3\beta, \alpha\beta^3, \alpha^2\beta^2, \alpha^2\delta^2, \beta^2\delta^2, \alpha\beta\delta^2$,
\item $\alpha\delta^2$ and one of $\alpha^4, \beta^4, \gamma^4, \alpha^3\beta, \alpha\beta^3, \alpha^2\beta^2, \alpha^2\gamma^2, \beta^2\gamma^2, \alpha\beta\gamma^2$.
\end{itemize}
In each of the last four items, the tiling has a unique degree $3$ vertex.
\end{lem}

If a tiling has a unique degree $3$ vertex, then $f\ge 16$ (or $f \ge 24$) due to Lemma  \ref{ab2Lem} (or Lemma \ref{a3Lem}).

\begin{proof} By \eqref{VCount}, $v_3 > 0$ means there exists some degree $3$ vertex. By Parity Lemma, the degree $3$ $b$-vertices are $\alpha\gamma\delta, \beta\gamma\delta, \alpha\gamma^2, \alpha\delta^2, \beta\gamma^2,\beta\delta^2$, and the degree $3$ $\hat{b}$-vertices are $\alpha^3, \beta^3, \alpha^2\beta, \alpha\beta^2$. The degree $4$ vertices are
\begin{align*}
&\alpha^4, \beta^4, \gamma^4, \delta^4, \alpha^3\beta, \alpha^2\beta^2, \alpha\beta^3,  \alpha^2\gamma^2, \alpha^2\delta^2, \alpha^2\gamma\delta, \\
& \alpha\beta\gamma^2, \alpha\beta\delta^2, \beta^2\gamma^2, \beta^2\gamma\delta, \beta^2\delta^2, \gamma^3\delta, \gamma^2\delta^2, \gamma\delta^3.
\end{align*}

If $\alpha\gamma\delta, \beta\gamma\delta$ are both vertices, Then $\alpha=\beta$ and Lemma \ref{ExchLem} implies $\gamma = \delta$, contradicting at least three distinct angle values. Hence only one of them can be a vertex. Pairs leading to these two equalities are dismissed.  

Suppose $\alpha\gamma\delta, \beta\gamma\delta$ are not vertices. 

If there are two degree $3$ vertices, we then dismiss the pairs contradicting Lemma \ref{ExchLem}. For example, $\alpha\gamma^2, \beta\delta^2$ are dismissed for this reason. Meanwhile, $\alpha^2\beta, \beta\gamma^2$ imply $\alpha = \gamma$ whereby Lemma \ref{TriQuadLem} implies $\delta = \pi$. Then $\delta^2\cdots$ is not a vertex. By Balance Lemma, $\beta\gamma^2$ cannot be a vertex. So $\alpha^2\beta, \beta\gamma^2$ are also dismissed. Up to symmetry, we obtain all degree $3$ pairs.

Suppose there is only one degree $3$ vertex. Up to symmetry, it suffices to discuss $\alpha^3,\alpha\beta^2, \alpha\gamma^2, \alpha\delta^2$. If one of $\alpha^3, \alpha\beta^2$ is the only degree $3$ vertex, then Lemma \ref{2HDAngLem} and Parity Lemma imply that one of $\gamma^4, \delta^4, \gamma^3\delta, \gamma\delta^3, \gamma^2\delta^2$ is a vertex. If $\alpha\gamma^2$ is the only degree $3$ vertex, then Lemma \ref{ab2Lem} assures a degree $4$ vertex without $\gamma$. So one of $\alpha^4, \beta^4,  \delta^4, \alpha^3\beta, \alpha\beta^3, \alpha^2\beta^2, \alpha^2\delta^2, \beta^2\delta^2, \alpha\beta\delta^2$ is a vertex. Same for $\alpha\delta^2$, one of $\alpha^4, \beta^4, \gamma^4,  \alpha^3\beta, \alpha\beta^3, \alpha^2\beta^2, \alpha^2\gamma^2, \beta^2\gamma^2$, $\alpha\beta\gamma^2$ is a vertex. These are the remaining pairs. 
\end{proof}

We remark that, in the proof above, counting is used in Lemmas \ref{2HDAngLem}, \ref{ab2Lem}. Because the four angles are distinguished by three distinct angle values and the $b$-edge, counting angles is made possible. 

It will be explained in Section \ref{SubsecStrategy} that knowing two vertices is sufficient to determine all angle combinations at vertices. By the above lemma, we only need extra discussion for the case where $\alpha\gamma\delta$ is a vertex.

\begin{lem} \label{AlGaDe-Al2} If a tiling of the sphere by congruent almost equilateral quadrilaterals has $\alpha\gamma\delta$, then $\alpha^2\cdots$ does not have $\gamma, \delta$. 
\end{lem}

The conclusion is that $\alpha^2\cdots$ is a $\hat{b}$-vertex. So $\alpha^2\cdots = \alpha^m, \alpha^{m\ge2}\beta^n$. 

\begin{proof} Assume the contrary. By $\alpha\gamma\delta$, Parity Lemma implies that one of $\alpha^2\gamma^2\cdots, \alpha^2\delta^2\cdots$ is a vertex. Up to symmetry of $\gamma \leftrightarrow \delta$, we may assume $\alpha^2\gamma^2\cdots$ is a vertex. Then $\alpha+\gamma\le\pi$ and $\alpha\gamma\delta$ implies $\delta \ge \pi$. This implies that $\delta^2\cdots$ is not a vertex. Then Balance Lemma implies that $\gamma^2\cdots$ is also not a vertex, contradicting $\alpha^2\gamma^2\cdots$.
\end{proof}

\begin{lem}\label{AlGaDeLem} If a tiling with $f\ge8$ has at least three distinct angle values and $\alpha\gamma\delta$ is a vertex, then $\alpha > \beta$ and $\gamma > \delta$. In particular, $\delta<\pi$. Moreover, if $\alpha^2\cdots$ is not a vertex, then the $\hat{b}$-vertices are $\beta^n, \alpha\beta^n$ and the vertices having strictly more $\delta$ than $\gamma$ are $\alpha\delta^2$, $\alpha\beta^n\delta^2$.
\end{lem}

\begin{proof} Assume $\delta > \gamma$. Lemma \ref{ExchLem} implies $\alpha<\beta$. Then Lemma \ref{LunEstLem} and $\alpha\gamma\delta$ imply $\beta +\pi > \gamma+\delta =2\pi - \alpha$, which gives $\alpha+\beta>\pi$. The angle sum system implies $\beta=\frac{4}{f}\pi$. So $\frac{8}{f}\pi = 2\beta > \alpha+\beta>\pi$ gives $f<8$, a contradiction. Lemma \ref{ExchLem} implies $\alpha > \beta$ and $\delta < \gamma$. Then Lemma \ref{AlConvexLem} implies $\delta < \pi$.

By $\alpha\gamma\delta$, we get $\bvert \, \gamma \vert \cdots \vert \delta \, \bvert=$ $\bvert \, \gamma \vert \delta \, \bvert$, $\bvert \, \gamma \vert  \alpha \vert \delta \, \bvert$, $\bvert \, \gamma \vert \beta\cdots\beta \vert \delta \, \bvert$.  

If $\alpha^2\cdots$ is not a vertex, then $\beta^{\alpha}\vert^{\alpha}\beta\cdots$, $\beta^{\alpha}\vert^{\alpha}\delta\cdots$, $\delta^{\alpha}\vert^{\alpha}\delta\cdots$ are not vertices. By no $\beta^{\alpha}\vert^{\alpha}\beta\cdots$, we get $\beta^{\alpha} \beta\cdots \beta = \beta^{\alpha} \beta^{\alpha}\cdots \beta^{\alpha}$ and $\delta^{\alpha} \beta \cdots \beta = \delta^{\alpha}\beta^{\alpha}\cdots \beta^{\alpha}$. Then $\bvert \, \delta \vert \beta \cdots \beta \vert \delta \, \bvert =$ $\bvert \, \delta^{\alpha} \vert \beta^{\alpha} \cdots \beta^{\alpha} \vert^{\alpha} \delta \, \bvert$ implies $\alpha^2\cdots$ being a vertex, a contradiction. So $\bvert \, \delta \vert \beta \cdots \beta \vert \delta \, \bvert$ cannot happen. Hence $\bvert \, \delta \vert \cdots \vert \delta \, \bvert =$ $\bvert \, \delta \vert \alpha \vert \delta \, \bvert$, $\bvert \, \delta \vert \alpha \vert \beta \cdots \beta \vert \delta \, \bvert$, $\bvert \, \delta \vert \beta \cdots \beta \vert \alpha \vert \beta \cdots \beta \vert \delta \, \bvert$.

A vertex with strictly more $\delta$ than $\gamma$ contains $\bvert \, \delta \vert \cdots \vert \delta \, \bvert$. Since $\bvert \, \delta \vert \cdots \vert \delta \, \bvert$ has $\alpha$, by $\alpha\gamma\delta$, we know that the vertex has no $\gamma$. Moreover, by no $\alpha^2\cdots$, the vertex has only one $\bvert \, \delta \vert \cdots \vert \delta \, \bvert$. Meanwhile, $\alpha\gamma\delta$ implies that $\bvert \, \delta \vert \alpha \vert \delta \, \bvert$ is not a vertex. Therefore the vertex is $\bvert \, \delta \vert \alpha \vert \beta \cdots \beta \vert \delta \, \bvert$ or $\bvert \, \delta \vert \beta \cdots \beta \vert \alpha \vert \beta \cdots \beta \vert \delta \, \bvert$, which is $\alpha\beta^n\delta^2$. 
\end{proof}

\begin{prop} \label{MinProp} If $f=6$, then the tiling is uniquely given by the earth map tiling $E (P_6)$ in Figure \ref{P6} with the set of admissible vertices $\AVC \equiv \{ \alpha\gamma\delta, \beta^3 \}$.

\begin{figure}
	\centering
		\adjustbox{trim=\dimexpr.5\Width-15mm\relax{} \dimexpr.5\Height-15mm\relax{}  \dimexpr.5\Width-15mm\relax{} \dimexpr.5\Height-15mm\relax{} ,clip}{\includegraphics[height=5.5cm]{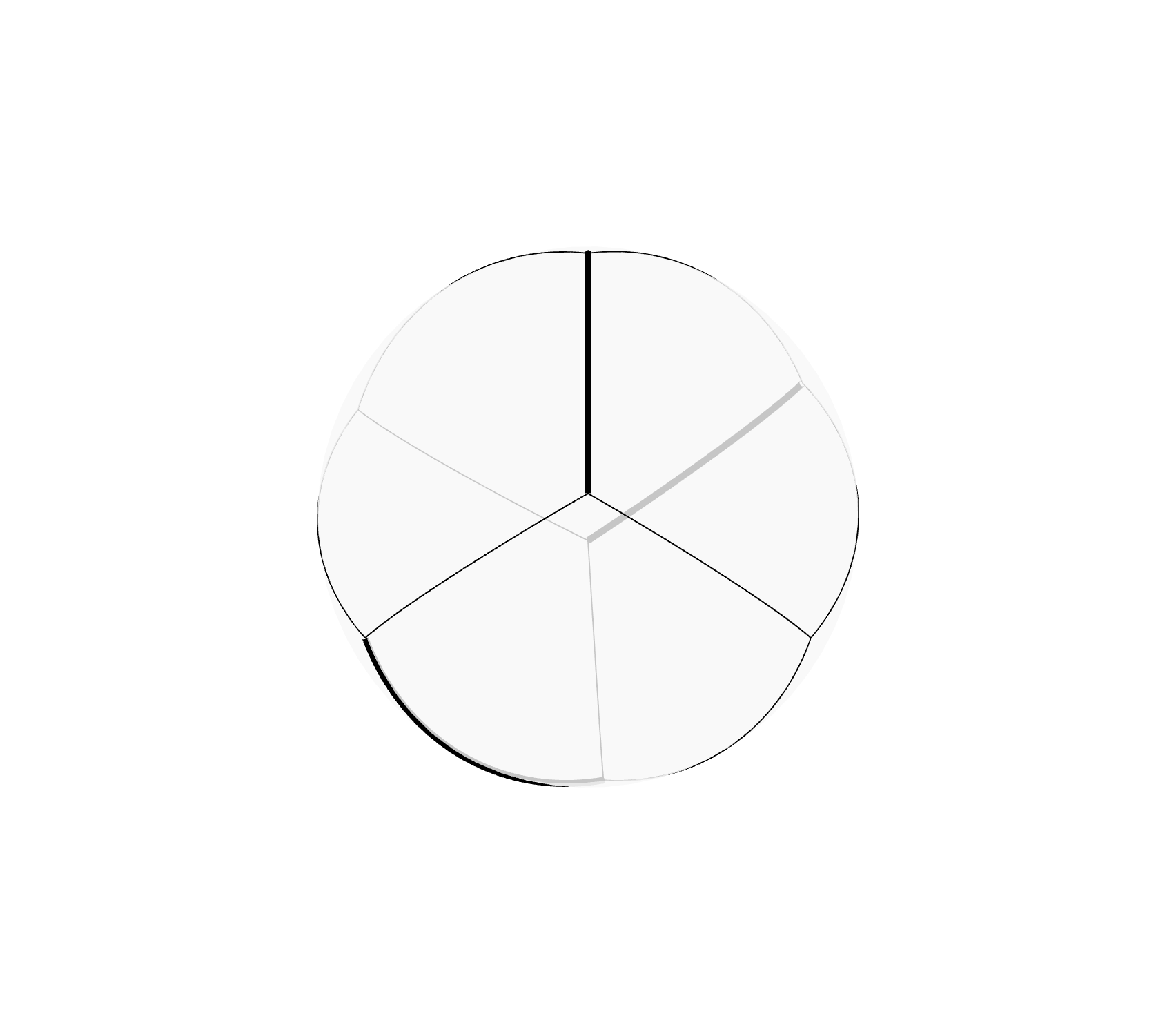}} \quad
	\caption{Tiling: $E=P_6$}
	\label{P6}
\end{figure}

\end{prop}

Although it can be easily checked by computer, we give the details here.

\begin{proof} Suppose $f=6$. The tiling is given by the cube. To determine the $b$-edge distribution, we assign $b$-edge in a way that each quadrilateral has exactly one $b$-edge. By the symmetries of the cube in Figure \ref{Edge-f6}, we may take $e_1$ as a $b$-edge. Then the other edges of the quadrilaterals sharing $e_1$ are $a$-edges. By mirror symmetry, choosing $e_2$ as a $b$-edge implies that $e_3$ must also be a $b$-edge. Therefore we obtain the $b$-edge distribution.

\begin{figure}[htp] 
\centering
\begin{tikzpicture}

\begin{scope}[scale=0.6] 

\tikzmath{\x=1;}

\draw[line width=2, gray!50]
	(0:0) -- (90:1)
	(30:2) -- (30:3);

\draw[rotate=120, line width=2, gray!50]	
	(90:1) -- (150:2);

\foreach \a in {0,1,2}
\draw[rotate=120*\a]
	(0:0) -- (90:1)
	(90:1) -- (30:2)
	(90:1) -- (150:2)
	(30:2) -- (30:3)
;

\node at (45:0.55*\x) {$e_1$};
\node at (260:1*\x) {$e_2$};
\node at (20:2.5*\x) {$e_3$};

\end{scope}

\end{tikzpicture}
\caption{Edge distribution for $f=6$}
\label{Edge-f6}
\end{figure}
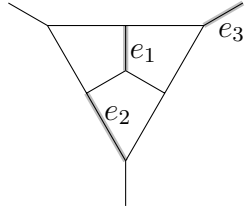

By the $b$-edge distribution in Figure \ref{Edge-f6}, we next determine the angles. By the remark after \eqref{FCount}, every vertex has degree $3$. Along the $b$-edge, the adjacent tiles either have matched angles $\gamma \, \bvert \, \gamma, \delta \, \bvert \, \delta$, or a twisted pair of $\gamma \, \bvert \, \delta$.  

Suppose a tiling has a twisted pair $\gamma \, \bvert \, \delta$. Up to symmetry, we assume $\gamma\delta\cdots = \alpha\gamma\delta$. As $\alpha\neq\beta$, we know that $\beta\gamma\delta$ is not a vertex. Vertex angle sums of $\alpha\gamma\delta, \beta\delta^2$ give $\alpha+\gamma = \beta + \delta$, contradicting Lemma \ref{ExchLem}. So $\beta\delta^2$ is not vertex. This means that $\beta\delta\cdots$ is not a vertex. If the central $b$-edge is configured in the first picture of Figure \ref{Angles-f6}, we get $\beta\delta\cdots$, a contradiction. So it must be configured the way as in the second picture. We uniquely determine angles and hence the set of admissible vertices $\AVC \equiv \{ \alpha\gamma\delta, \beta^3 \}$.

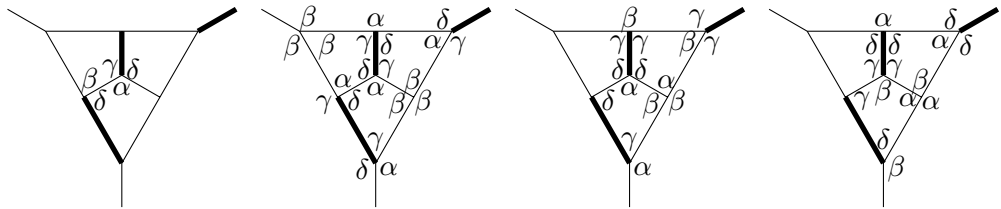
\begin{figure}[htp] 
\centering
\begin{tikzpicture}[scale=0.9]

\begin{scope}[scale=0.65] 

\foreach \a in {0,1,2}
\draw[rotate=120*\a]
	(0:0) -- (90:1)
	(90:1) -- (30:2)
	(90:1) -- (150:2)
	(30:2) -- (30:3)
;

\draw[line width=2]
	(0:0) -- (90:1)
	(30:2) -- (30:3);

\draw[rotate=120, line width=2]	
	(90:1) -- (150:2);

\node at (270:0.3) {\small $\alpha$};
\node at (30:0.3) {\small $\delta$};
\node at (150:0.3) {\small $\gamma$};

\node at (190:0.75) {\small $\beta$};
\node at (230:0.75) {\small $\delta$};

\end{scope}

\begin{scope}[xshift=3.75cm, scale=0.65] 

\foreach \a in {0,1,2}
\draw[rotate=120*\a]
	(0:0) -- (90:1)
	(90:1) -- (30:2)
	(90:1) -- (150:2)
	(30:2) -- (30:3)
;

\draw[line width=2]
	(0:0) -- (90:1)
	(30:2) -- (30:3);

\draw[rotate=120, line width=2]	
	(90:1) -- (150:2);

\node at (270:0.3) {\small $\alpha$}; 
\node at (30:0.3) {\small $\gamma$};
\node at (150:0.3) {\small $\delta$};

\node at (70:0.75) {\small $\delta$}; 
\node at (110:0.75) {\small $\gamma$};  
\node at (90:1.25) {\small $\alpha$};  

\node at (190:0.75) {\small $\alpha$}; 
\node at (230:0.75) {\small $\delta$};
\node at (210:1.35) {\small $\gamma$};

\node at (350:0.85) {\small $\beta$}; 
\node at (305:0.85) {\small $\beta$};  
\node at (330:1.25) {\small $\beta$}; 

\node at (30:1.5) {\small $\alpha$}; 
\node at (20:2) {\small $\gamma$}; 
\node at (40:2) {\small $\delta$}; 

\node at (150:1.275) {\small $\beta$}; 
\node at (140:2) {\small $\beta$}; 
\node at (162:2) {\small $\beta$}; 

\node at (270:1.5) {\small $\gamma$}; 
\node at (262:2.15) {\small $\delta$};
\node at (278:2.15) {\small $\alpha$};

\end{scope}

\begin{scope}[xshift=7.5cm, scale=0.65] 

\foreach \a in {0,1,2}
\draw[rotate=120*\a]
	(0:0) -- (90:1)
	(90:1) -- (30:2)
	(90:1) -- (150:2)
	(30:2) -- (30:3)
;

\draw[line width=2]
	(0:0) -- (90:1)
	(30:2) -- (30:3);

\draw[rotate=120, line width=2]	
	(90:1) -- (150:2);

\node at (270:0.3) {\small $\alpha$}; 
\node at (30:0.3) {\small $\delta$};
\node at (150:0.3) {\small $\delta$};

\node at (70:0.75) {\small $\gamma$}; 
\node at (110:0.75) {\small $\gamma$};  
\node at (90:1.25) {\small $\beta$};  

\node at (230:0.75) {\small $\delta$};

\node at (350:0.85) {\small $\alpha$}; 
\node at (305:0.85) {\small $\beta$};  
\node at (330:1.25) {\small $\beta$}; 

\node at (28:1.5) {\small $\beta$}; 
\node at (20:2) {\small $\gamma$}; 
\node at (40:2) {\small $\gamma$};

\node at (270:1.5) {\small $\gamma$}; 
\node at (278:2.15) {\small $\alpha$};

\end{scope}

\begin{scope}[xshift=11.25cm, scale=0.65] 

\foreach \a in {0,1,2}
\draw[rotate=120*\a]
	(0:0) -- (90:1)
	(90:1) -- (30:2)
	(90:1) -- (150:2)
	(30:2) -- (30:3)
;

\draw[line width=2]
	(0:0) -- (90:1)
	(30:2) -- (30:3);

\draw[rotate=120, line width=2]	
	(90:1) -- (150:2);

\node at (270:0.4) {\small $\beta$}; 
\node at (30:0.3) {\small $\gamma$};
\node at (150:0.3) {\small $\gamma$};

\node at (70:0.75) {\small $\delta$}; 
\node at (110:0.75) {\small $\delta$};  
\node at (90:1.25) {\small $\alpha$};  

\node at (230:0.75) {\small $\gamma$};

\node at (350:0.85) {\small $\beta$}; 
\node at (315:0.8) {\small $\alpha$};  
\node at (330:1.25) {\small $\alpha$}; 

\node at (30:1.5) {\small $\alpha$}; 
\node at (20:2) {\small $\delta$}; 
\node at (40:2) {\small $\delta$}; 


\node at (270:1.45) {\small $\delta$}; 
\node at (278:2.15) {\small $\beta$};

\end{scope}

\end{tikzpicture}
\caption{Angle distributions for $f=6$}
\label{Angles-f6}
\end{figure}

If the tiles are all matched with respect to some $b$-edge, then the angles at the $b$-edges are $\gamma \, \bvert \, \gamma, \delta \, \bvert \, \delta$. The corresponding vertices are $\alpha\gamma^2, \alpha\delta^2, \beta\gamma^2, \beta\delta^2$. By Lemma \ref{PairsLem}, the $b$-vertices are $\alpha\delta^2, \beta\gamma^2$. So $\alpha\gamma\cdots, \beta\delta\cdots$ are not vertices. If the central $b$-edge is configured the way in the third picture of Figure \ref{Angles-f6}, we get $\alpha\gamma\cdots$, a contradiction. Then the central $b$-edge must be configured the way in the fourth picture. Similar argument arrives at $\beta\delta\cdots$, a contradiction. So there is no tiling. This completes the proof.
\end{proof}

\subsection{Strategy} \label{SubsecStrategy}

With groundwork in place, we assume at least three distinct angles and $f\ge8$. Notably, in the remaining part of this paper, we assume  
\begin{align*}
\alpha\neq\beta, \quad
\gamma\neq\delta.
\end{align*}

Among $\alpha, \beta, \gamma, \delta$, there are two possibilities: every angle is rational or some is non-rational. 

If every angle is rational (Section \ref{SecRat}), then we analyse \eqref{Coolsaet-Id} using \cite[Theorem 4]{my} and determine the angles. This is Coolsaet's method \cite[Theorem 3.2]{co} and we extend it to general almost equilateral quadrilaterals via Lemmas \ref{CoolsaetLem}, \ref{PairsLem}. For the case where $\alpha\gamma\delta$ is a vertex, we use Lemmas \ref{AlGaDe-Al2}, \ref{AlGaDeLem} to find all the vertices. Once all vertices are determined, it is only a matter of constructing the tiling.

If some angle is non-rational (Section \ref{SecNonRat}), with the exception of $\alpha\gamma\delta$ being a vertex, then we apply Lemma \ref{PairsLem} and the non-rationality condition to determine the vertices. For the case where $\alpha\gamma\delta$ is a vertex, we apply Lemmas \ref{AlGaDe-Al2}, \ref{AlGaDeLem} to find all the vertices.

In Section \ref{SecGeom}, we discuss whether the tilings are geometrically possible. This means the expected angle values are realised by a simple almost equilateral quadrilateral. We say the tile {\em exists} in this case. 

\section{Rational Angles} \label{SecRat}

In this section, we assume $\alpha, \beta, \gamma, \delta \in (0, 2\pi)$ are rational, i.e., their values are rational multiples of $\pi$. Myerson \cite{my} provided rational solutions to \eqref{Coolsaet-Id}. 

\begin{theorem}[{\cite[Theorem 4]{my}}] \label{MyThm} The rational angle solutions $(x_1, x_2, x_3, x_4)$ where $x_i \in [0, \frac{1}{2}\pi]$ for $1\le i \le 4$ to 
\begin{align}\label{MyTrigEq}
\sin x_1 \sin x_2 = \sin x_3 \sin x_4
\end{align}
are given by
\begin{enumerate}[I.]
\item one of the following, where every $x_i \in [0, \frac{1}{2}\pi]$, 
\begin{itemize}
\item $x_1 = x_3 = 0$ and any rational angles $ x_2, x_4$;
\item $x_1 = x_3$ and $x_2 =  x_4$;
\end{itemize}
and their permutations.
\item $(\tfrac{1}{6}\pi, \theta, \tfrac{1}{2}\theta, \tfrac{1}{2}\pi - \tfrac{1}{2}\theta)$ for any rational angle $\theta \in [0,\frac{1}{2}\pi]$ and its permutations;
\item the fifteen rational angle solutions listed in Table \ref{MySpSol} and the permutations of each row.

\begin{table}[h]
\begin{center}
\begin{minipage}[t]{.45\linewidth}
\bgroup
\def\arraystretch{1.5}
    \begin{tabular}[t]{ | c | c | c | c | c | }
	\hline
	 &  $x_1$ & $x_2$ & $x_3$ & $x_4$ \\ \hhline{|=====|}
	1. & $\frac{1}{21}\pi$  &  $\frac{8}{21}\pi$ & $\frac{1}{14}\pi$ & $\frac{3}{14}\pi$ \\
       	\hline 
	2. & $\frac{1}{14}\pi$  &  $\frac{5}{14}\pi$ & $\frac{2}{21}\pi$ & $\frac{5}{21}\pi$ \\
       	\hline 
	3. & $\frac{4}{21}\pi$  &  $\frac{10}{21}\pi$ & $\frac{3}{14}\pi$ & $\frac{5}{14}\pi$ \\
       	\hline 
	4. & $\frac{1}{20}\pi$  &  $\frac{9}{20}\pi$ & $\frac{1}{15}\pi$ & $\frac{4}{15}\pi$ \\
       	\hline 
	5. & $\frac{2}{15}\pi$  &  $\frac{7}{15}\pi$ & $\frac{3}{20}\pi$ & $\frac{7}{20}\pi$ \\
       	\hline 
	6. & $\frac{1}{30}\pi$  &  $\frac{3}{10}\pi$ & $\frac{1}{15}\pi$ & $\frac{2}{15}\pi$ \\
       	\hline 
	7. & $\frac{1}{15}\pi$  &  $\frac{7}{15}\pi$ & $\frac{1}{10}\pi$ & $\frac{7}{30}\pi$ \\
       	\hline 
	8. & $\frac{1}{10}\pi$  &  $\frac{13}{30}\pi$ & $\frac{2}{15}\pi$ & $\frac{4}{15}\pi$  \\  
\hline
	\end{tabular}
\egroup
 \end{minipage} 
\begin{minipage}[t]{.45\linewidth}
\bgroup
\def\arraystretch{1.5}
    \begin{tabular}[t]{ | c | c | c | c | c | }
	\hline
	 &  $x_1$ & $x_2$ & $x_3$ & $x_4$ \\ \hhline{|=====|}
	9. & $\frac{4}{15}\pi$  &  $\frac{7}{15}\pi$ & $\frac{3}{10}\pi$ & $\frac{11}{30}\pi$ \\
       	\hline 
	10. & $\frac{1}{30}\pi$  &  $\frac{11}{30}\pi$ & $\frac{1}{10}\pi$ & $\frac{1}{10}\pi$ \\
       	\hline 
	11. & $\frac{7}{30}\pi$  &  $\frac{13}{30}\pi$ & $\frac{3}{10}\pi$ & $\frac{3}{10}\pi$ \\
       	\hline 
	12. & $\frac{1}{15}\pi$  &  $\frac{4}{15}\pi$ & $\frac{1}{10}\pi$ & $\frac{1}{6}\pi$ \\
       	\hline 
	13. & $\frac{2}{15}\pi$  &  $\frac{7}{15}\pi$ & $\frac{1}{6}\pi$ & $\frac{3}{10}\pi$ \\
       	\hline 
	14. & $\frac{1}{12}\pi$  &  $\frac{5}{12}\pi$ & $\frac{1}{10}\pi$ & $\frac{3}{10}\pi$ \\
       	\hline 
	15. & $\frac{1}{10}\pi$  &  $\frac{3}{10}\pi$ & $\frac{1}{6}\pi$ & $\frac{1}{6}\pi$ \\
       	\hline 
	\end{tabular}
\egroup
 \end{minipage} 
\end{center}
\caption{Rational solutions to \eqref{MyTrigEq} in $[0, \frac{1}{2}\pi]$}
\label{MySpSol}
\end{table}
\end{enumerate}

The permutations for Type I, II, III solutions are the following symmetries of \eqref{MyTrigEq},
\begin{align}\label{MyPerm}
(x_1, x_2, x_3, x_4), (x_1, x_2, x_4, x_3), (x_2, x_1, x_3, x_4), (x_2, x_1, x_4, x_3), \\ \notag
(x_3, x_4, x_1, x_2), (x_4, x_3, x_1, x_2), (x_3, x_4, x_2, x_1), (x_4, x_3, x_2, x_1).
\end{align}
\end{theorem}

We remark that Type I solutions are not included in the original theorem as they are solutions to $\{ \sin x_1 = 0,  \sin x_3 = 0 \}$ or $\{ \sin x_1 = \sin x_3, \ \sin x_2 = \sin x_4 \}$, which may have been deemed \quotes{trivial}. 

We also note that Type II solutions can be summarised by the identity,
\begin{align*}
\sin\tfrac{1}{6}\pi\sin \theta = \sin \tfrac{1}{2}\theta \sin ( \tfrac{1}{2}\pi - \tfrac{1}{2}\theta ).
\end{align*} 

For Type III solutions in each row of Table \ref{MySpSol}, applying \eqref{MyPerm} also gives rational solutions to \eqref{MyTrigEq}. We remark a misprint in the previous literatures where $x_2$ of thirteenth row should be $\frac{7}{15}$ instead of $\frac{8}{15}$. With the correct value, the conclusion in \cite[Theorem 3.2]{co} is valid. 

By \eqref{Coolsaet-Id}, we know that $x_1 = \frac{1}{2}\alpha, x_2 = \delta - \frac{1}{2}\beta, x_3 = \frac{1}{2}\beta, x_4 = \gamma - \frac{1}{2}\alpha$ satisfy \eqref{MyTrigEq}. If all $x_i \in [0, \frac{1}{2}\pi]$, then we can apply Theorem \ref{MyThm} to determine the angles. We know $\tfrac{1}{2}\alpha,\tfrac{1}{2}\beta \in (0, \pi)$ and the ranges of $\delta - \tfrac{1}{2}\beta, \gamma - \tfrac{1}{2}\alpha$ can be wider. To apply Theorem \ref{MyThm}, we therefore need to \quotes{recalibrate}: for example, if $x_i \in (\frac{1}{2}\pi, \pi)$, then it should be changed to $\pi - x_i \in (0, \frac{1}{2}\pi)$. By similar modification of switching signs and/or adding an integer multiple of $\pi$ and using $\sin (\pi - x) = \sin x$ and $\sin(-x)=-\sin x$, we may reduce all angle values to $[0, \frac{1}{2}\pi]$ and \eqref{MyTrigEq} still holds. Then we can apply Theorem \ref{MyThm}.  

To apply Theorem \ref{MyThm} with Type I solutions, we may bypass the calibration with the angle relations given by the next lemma. Modifying the discussion of \cite[Theorem 3.2]{co} and Type I solutions to \eqref{Coolsaet-Id}, we have one of the following,
\begin{align}
\label{TrigEq1} &\sin ( \gamma - \tfrac{1}{2}\alpha ) =0,& &\sin ( \delta - \tfrac{1}{2}\beta )=0;& \\
\label{TrigEq2} &\sin ( \gamma - \tfrac{1}{2}\alpha ) = \sin \tfrac{1}{2}\alpha,& &\sin ( \delta - \tfrac{1}{2}\beta )=\sin\tfrac{1}{2}\beta;& \\
\label{TrigEq3} &\sin ( \gamma - \tfrac{1}{2}\alpha ) =  \sin ( \delta - \tfrac{1}{2}\beta ),& &\sin\tfrac{1}{2}\beta  =  \sin \tfrac{1}{2}\alpha.& 
\end{align} 
They correspond to the following relations between the angles,
\begin{align}
&\begin{cases}
\label{CoAngRel1}
2\gamma = \alpha + 2N_1\pi, \\
2\delta = \beta + 2N_2\pi;
\end{cases} \\
\label{CoAngRel2}
&\begin{cases}
2\gamma =(1 + (-1)^{N_1})\alpha +  2N_1\pi, \\
2\delta =(1 + (-1)^{N_2})\beta +  2N_2\pi; \\
\end{cases}\\
\label{CoAngRel3}
&\begin{cases}
2\gamma =(-1)^{N_1}(2\delta - \beta) + \alpha +  2N_1\pi, \\
\alpha = (-1)^{N_2}\beta +  2N_2\pi.
\end{cases}
\end{align} 
After further simplification, the result is summarised below.

\begin{lem}\label{ReducedTrigAng} In an almost equilateral quadrilateral tile with at least three distinct angles, if the angles satisfy one of \eqref{TrigEq1}, \eqref{TrigEq2}, \eqref{TrigEq3}, then we have one of the following, 
\begin{enumerate}
\item if $\alpha, \beta, \gamma, \delta < \pi$, then $\alpha=2\gamma$ and $\beta = 2\delta$ hold.
\item if either one of $\alpha, \beta \ge \pi$ and all other angles $<\pi$, then one of the following is true,
\begin{enumerate}[i.]
\item $\alpha=2\gamma$ and $\beta = 2\delta$,
\item $\alpha+\beta=2\pi$ and $\alpha+2\delta=\beta + 2\gamma$;
\end{enumerate}
\item  if $\gamma \ge \pi$ and all other angles $<\pi$, then one of the following is true, 
\begin{enumerate}[i.]
\item $\gamma=\pi$ and $\beta = \delta$,
\item $\alpha+2\pi = 2\gamma$ and $\beta = 2\delta$;
\end{enumerate}
\item  if $\delta \ge \pi$ and all other angles $<\pi$, then one of the following is true, 
\begin{enumerate}[i.]
\item $\delta = \pi$ and $\alpha=\gamma$,
\item $\alpha = 2\gamma$ and $\beta +2\pi  = 2\delta$.
\end{enumerate}
\end{enumerate}
\end{lem}

For Type II, III solutions, for the purpose of tiling, Lemma \ref{AlConvexLem} implies that we only need to consider the calibrations in Table \ref{Recali}. As per Lemma \ref{AlConvexLem}, \quotes{case $\alpha\ge\pi$} in the table means $\alpha\ge\pi$ and the other three angles $<\pi$, and etc. 

\begin{table}[htp]
\begin{center}
\bgroup
\def\arraystretch{1.4}
    \begin{tabular}[]{ | c | c | c | c | c | }
	\hline
	Case & $x_1$ & $x_2$ & $x_3$ & $x_4$ \\ \hhline{|=====|}
	\multirow{5}{*}{$\alpha, \beta, \gamma, \delta<\pi$} & $\gamma - \tfrac{1}{2}\alpha$ & $\tfrac{1}{2}\beta$ & $\tfrac{1}{2}\alpha$ & $\delta - \tfrac{1}{2}\beta$  \\
	\cline{2-5}
	& $\tfrac{1}{2}\alpha - \gamma$ & $\tfrac{1}{2}\beta$ & $\tfrac{1}{2}\alpha$ & $\tfrac{1}{2}\beta - \delta$ \\
	\cline{2-5}
	& $\gamma - \tfrac{1}{2}\alpha$ & $\tfrac{1}{2}\beta$ & $\tfrac{1}{2}\alpha$ & $\pi  + \tfrac{1}{2}\beta - \delta$ \\
	\cline{2-5}
	& $\pi + \tfrac{1}{2}\alpha - \gamma$ & $\tfrac{1}{2}\beta$ & $\tfrac{1}{2}\alpha$ & $\delta - \tfrac{1}{2}\beta$ \\
	\cline{2-5}
	& $\pi + \tfrac{1}{2}\alpha - \gamma$ & $\tfrac{1}{2}\beta$ & $\tfrac{1}{2}\alpha$ & $\pi + \tfrac{1}{2}\beta - \delta$ \\
	\hline
	\multirow{4}{*}{$\alpha \ge \pi$} & $\gamma - \tfrac{1}{2}\alpha$ & $\tfrac{1}{2}\beta$ & $\pi - \tfrac{1}{2}\alpha$ & $\delta - \tfrac{1}{2}\beta$ \\
	\cline{2-5}
	& $\gamma - \tfrac{1}{2}\alpha$ & $\tfrac{1}{2}\beta$ & $\pi - \tfrac{1}{2}\alpha$ & $\pi + \tfrac{1}{2}\beta - \delta$ \\ 
	\cline{2-5}
	& $\pi - \tfrac{1}{2}\alpha + \gamma$ & $\tfrac{1}{2}\beta$ & $\pi - \tfrac{1}{2}\alpha$ & $\tfrac{1}{2}\beta - \delta$ \\
	\cline{2-5}
	& $\tfrac{1}{2}\alpha - \gamma$ & $\tfrac{1}{2}\beta$ & $\pi - \tfrac{1}{2}\alpha$ & $\tfrac{1}{2}\beta - \delta$  \\
	\hline
	\multirow{4}{*}{$\beta \ge \pi$} & $\gamma - \tfrac{1}{2}\alpha$ & $\pi - \tfrac{1}{2}\beta$ &  $\tfrac{1}{2}\alpha$ & $\delta - \tfrac{1}{2}\beta$ \\
	\cline{2-5}
	& $\pi + \tfrac{1}{2}\alpha - \gamma$ & $\pi - \tfrac{1}{2}\beta$ & $\tfrac{1}{2}\alpha$ & $\delta - \tfrac{1}{2}\beta$ \\  
	\cline{2-5}
	& $\tfrac{1}{2}\alpha - \gamma$ & $\pi - \tfrac{1}{2}\beta$ & $\tfrac{1}{2}\alpha$ & $\pi - \tfrac{1}{2}\beta + \delta$ \\
	\cline{2-5}
	& $\tfrac{1}{2}\alpha - \gamma$ & $\pi - \tfrac{1}{2}\beta$ & $\tfrac{1}{2}\alpha$ & $\tfrac{1}{2}\beta - \delta$ \\
	\hline
	\multirow{4}{*}{$\gamma \ge \pi$} & $\pi + \tfrac{1}{2}\alpha - \gamma$ & $\tfrac{1}{2}\beta$ & $\tfrac{1}{2}\alpha$ & $\delta - \tfrac{1}{2}\beta$ \\
	\cline{2-5}
	& $\pi + \tfrac{1}{2}\alpha - \gamma$ & $\tfrac{1}{2}\beta$ & $\tfrac{1}{2}\alpha$ & $\pi  + \tfrac{1}{2}\beta - \delta$  \\ 
	\cline{2-5}
	& $\gamma - \tfrac{1}{2}\alpha - \pi$ & $\tfrac{1}{2}\beta$ & $\tfrac{1}{2}\alpha$ & $\tfrac{1}{2}\beta - \delta$\\
	\cline{2-5}
	& $2\pi + \tfrac{1}{2}\alpha - \gamma$ & $\tfrac{1}{2}\beta$ & $\tfrac{1}{2}\alpha$ & $\tfrac{1}{2}\beta - \delta$ \\
	\hline
	\multirow{4}{*}{$\delta \ge \pi$} & $\gamma - \tfrac{1}{2}\alpha$ & $\tfrac{1}{2}\beta$ & $\tfrac{1}{2}\alpha$ & $\pi  + \tfrac{1}{2}\beta - \delta$ \\
	\cline{2-5}
	& $\pi + \tfrac{1}{2}\alpha - \gamma$ & $\tfrac{1}{2}\beta$ & $\tfrac{1}{2}\alpha$ & $\pi + \tfrac{1}{2}\beta -\delta$ \\ 
	\cline{2-5}
	& $\tfrac{1}{2}\alpha - \gamma$ & $\tfrac{1}{2}\beta$ & $\tfrac{1}{2}\alpha$ & $\delta - \tfrac{1}{2}\beta - \pi$\\
	\cline{2-5}
	& $\tfrac{1}{2}\alpha - \gamma$ & $\tfrac{1}{2}\beta$ & $\tfrac{1}{2}\alpha$ & $2\pi + \tfrac{1}{2}\beta - \delta$ \\
	\hline
	\end{tabular}
\egroup
\end{center}
\caption{Angle value recalibrations}
\label{Recali}
\end{table}

In general, there are more calibrations for angles with wider ranges. However, those ranges are not needed for tiling classification.

We generalise the scheme in \cite[Theorem 3.2]{co} in the following steps.
\begin{enumerate}[Step 1.]
\item Determine the angles via
\begin{itemize}
\item Type I solutions to angle relations in Lemma \ref{ReducedTrigAng},
\item Type II solutions and recalibrations in Table \ref{Recali},
\item Type III solutions and recalibrations in Table \ref{Recali}.
\end{itemize}
\item Dismiss the angle values not satisfying any of the following
\begin{itemize}
\item $0 < \alpha, \beta, \gamma, \delta < 2\pi$;
\item at least three distinct angle and at most one of them $\ge \pi$;
\item the angles are consistent with Lemmas \ref{ExchLem}, \ref{LunEstLem}, \ref{ATriLem}, \ref{CTriLem}.
\end{itemize}
\item For each set of the angle values, apply each vertex in Lemma \ref{PairsLem} (except $\alpha\gamma\delta$) to determine $f$. Select only the pairs in the lemma producing consistent even integer $f \ge 8$. If one of $\alpha\beta^2, \alpha\gamma^2, \alpha\delta^2$ is the unique degree $3$ vertex, then we further require $f\ge16$; and if $\alpha^3$ is the unique degree $3$ vertex, then we further require $f\ge24$.
\item We call the selected angle values {\em valid} and use them to determine their sets of admissible vertices ($\AVC$s).
\end{enumerate}
For $\alpha\gamma\delta$, we carry out Steps 1,2 and the first part of Step 3 to select the angle values. We then use Lemmas \ref{AlGaDe-Al2}, \ref{AlGaDeLem}, \ref{Ga=Pi} to complete Step 4 in Proposition \ref{RatAlGaDeProp}.

Lastly, we construct the tilings from the $\AVC$s.

\begin{prop}\label{RatProp} If $f\ge8$ and $\alpha\gamma\delta, \beta\gamma\delta$ are not vertices, then the tilings of the sphere by congruent almost equilateral quadrilaterals, where every angle is rational, are isolated earth map tilings $S3,S^{\prime}3$ and special tilings $S5, S6$. 
\end{prop}

\begin{proof} By Lemma \ref{AlConvexLem}, the discussion is divided into five cases: every angle $<\pi$, or exactly one angle $\ge \pi$. In each case, we compute the angle values via Type I, II, III solutions in Theorem \ref{MyThm} and select the valid ones. We give an example of this routine and leave out the details for the others. The routine can be more efficiently executed in computer.

\begin{case*}[$\alpha, \beta, \gamma, \delta<\pi$] Type I: By Lemma \ref{ReducedTrigAng}, we combine $\alpha=2\gamma$ and $\beta=2\delta$ with the vertex angle sums of the pairs in Lemma \ref{PairsLem}. There are no valid angle values. The conclusion is consistent with \cite[Theorem 3.2]{co}.  

Type II: Matching $( \gamma - \frac{1}{2}\alpha, \frac{1}{2}\beta, \frac{1}{2}\alpha, \delta - \frac{1}{2}\beta )$ and its permutations in \eqref{MyPerm} with each row of the first case of Table \ref{Recali}, we determine the angles. For example, if $(\frac{1}{6}\pi, \theta, \frac{1}{2}\pi - \frac{1}{2}\theta, \frac{1}{2}\theta )$ is matched with $( \gamma - \frac{1}{2}\alpha, \frac{1}{2}\beta, \frac{1}{2}\alpha, \delta - \frac{1}{2}\beta )$, then we have
\begin{align*}
&\tfrac{1}{6}\pi=\gamma - \tfrac{1}{2}\alpha, \quad
\theta = \tfrac{1}{2}\beta, \quad
\tfrac{1}{2}\pi - \tfrac{1}{2}\theta = \tfrac{1}{2}\alpha, \quad
\tfrac{1}{2}\theta = \delta - \tfrac{1}{2}\beta.
\end{align*} 
Combining the above with the quadrilateral angle sum, we solve for the angles and get
\begin{align*}
\alpha=(\tfrac{5}{6}-\tfrac{2}{f})\pi, \quad 
\beta=(\tfrac{1}{3}+\tfrac{4}{f})\pi,\quad 
\gamma=(\tfrac{7}{12}-\tfrac{1}{f})\pi, \quad 
\delta=(\tfrac{1}{4}+\tfrac{3}{f})\pi.
\end{align*} 

Substituting the above into $\alpha\gamma^2, \alpha\delta^2, \beta\gamma^2, \beta\delta^2$, we get $f=\infty, 6, 4, \frac{60}{7}$ respectively. These vertices fail the filter on $f$. Such vertices are dismissed. The vertices yield meaningful $f$ are $\alpha^3 (f=12), \alpha\beta^2 (f=12), \alpha^2\beta (\text{any } f), \beta^4 (f=24), \alpha\beta^3 (f=60), \gamma^4(f=12), \delta^4 (f=12), \gamma^3\delta (\text{any } f)$, $\gamma\delta^3(f=12), \gamma^2\delta^2(f=12), \beta^2\gamma^2(f=36), \alpha\beta\delta^2(f=24)$. Then we select the pairs in Lemma \ref{PairsLem} with consistent $f$. For example, $\alpha^2\beta$ is dismissed because it is paired with only $\alpha\gamma^2, \alpha\delta^2, \beta\delta^2$ in Lemma \ref{PairsLem}, none of which has meaningful $f$. The remaining pairs have unique degree $3$ vertex $\alpha^3$ or $\alpha\beta^2$. However, $\alpha^3$ or $\alpha\beta^2$ implies $f=12$, contradicting $f\ge24$ or $f\ge16$. Hence these angle values are dismissed. 

Other invalid angle values are dismissed by similar or simpler arguments. The valid ones are listed in Table \ref{ConvexType-II}. 

\begin{table}[h]
\begin{center}
\bgroup
\def\arraystretch{1.4}
    \begin{tabular}[]{ | c | c | c | c | c | c  | }
	\hline
	Pairs & $f$ & $\alpha$ & $\beta$ & $\gamma$ & $\delta$ \\ \hhline{|======|}
	$\{ \alpha\beta^2,\gamma\delta^3 \}$ & $\ge16$ & $(\frac{1}{3}+\frac{4}{f})\pi$ & $(\frac{5}{6}-\frac{2}{f})\pi$ & $(\frac{1}{4}+\frac{3}{f})\pi$ & $(\frac{7}{12}-\frac{1}{f})\pi$  \\
	\hline
	$\{ \alpha\delta^2, \alpha\beta^3 \}$ & $36$ & $\frac{1}{3}\pi$ & $\frac{5}{9}\pi$ & $\frac{7}{18}\pi$ & $\frac{5}{6}\pi$  \\
	\hline
	\end{tabular}
\egroup
\end{center}
\caption{Type II and convex: vertex pairs and angle values}
\label{ConvexType-II}
\end{table}

In $\{ \alpha\beta^2,\gamma\delta^3 \}$, the angle values and $f\ge16$ give the lower bounds, $\alpha > \tfrac{1}{3}\pi, \beta \ge \tfrac{17}{24}\pi, \gamma > \tfrac{1}{4}\pi, \delta \ge \tfrac{25}{48}\pi$. This implies $m<6$ and $n<3$ and $k<8$ and $l<4$ in \eqref{VertexAngSum}. We substitute finitely many non-negative integers $m,n,k,l$ such that \eqref{VertexAngSum} holds and calculate the corresponding $f$. We select only those with $f\ge16$. The admissible vertices are listed below with their corresponding $f$ values,
\begin{align*}
f=20, \quad &\{ \alpha\beta^2,\gamma\delta^3, \alpha^2\gamma\delta \}; \\
f=24,\quad &\{ \alpha\beta^2, \alpha^4, \gamma\delta^3, \alpha\beta\gamma^2, \alpha\gamma^4 \}; \\
f=36, \quad &\{ \alpha\beta^2, \alpha^2\delta^2, \gamma\delta^3, \alpha^3\gamma^2, \alpha\gamma^3\delta,  \gamma^6 \}; \\
f=60,\quad &\{ \alpha\beta^2,\gamma\delta^3, \alpha^3\beta, \alpha^5, \beta\gamma^4, \alpha^2\gamma^4 \}; \\
f=84,\quad &\{ \alpha\beta^2,\gamma\delta^3, \alpha^3\gamma\delta, \gamma^5\delta \}; \\
f=132,\,\,\, &\{ \alpha\beta^2,\gamma\delta^3, \alpha^4\gamma^2, \alpha\gamma^6 \}.
\end{align*}

In $\{ \alpha\delta^2, \alpha\beta^3 \}$, the angle values imply that the other vertices are $\gamma^3\delta, \alpha^2\beta\gamma^2$, $\alpha^6$. Therefore we get the set of admissible vertices,
\begin{align*}
f=36, \quad \{ \alpha\delta^2, \alpha\beta^3, \gamma^3\delta, \alpha^2\beta\gamma^2, \alpha^6 \}.
\end{align*}

Type III: Matching the permutations (in \eqref{MyPerm}) of each row of Table \ref{MySpSol} and each row of the first case of Table \ref{Recali}, the process yields no valid angle values. 
\end{case*}

\begin{case*}[$\beta, \gamma, \delta<\pi$ and $\alpha \ge \pi$] By $\alpha\ge\pi$, we know that $\alpha^2\cdots$ is not a vertex. It suffices to study those in the list of Lemma \ref{PairsLem} without $\alpha^2\cdots$.  

Type I: By Lemma \ref{ReducedTrigAng}, we have $\alpha=2\gamma$ and $\beta = 2\delta$, or $\alpha+\beta=2\pi$ and $\alpha+2\delta=\beta + 2\gamma$. By the same argument in the previous case, namely solving for $0<\beta, \gamma, \delta<\pi$ and $\alpha \ge \pi$ with even integer $f\ge8$, we get the valid angle values in Table \ref{AlphaType-I}.

\begin{table}[h]
\begin{center}
\bgroup
\def\arraystretch{1.4}
    \begin{tabular}[]{ | c | c | c | c | c | c  | }
	\hline
	Pairs & $f$ & $\alpha$ & $\beta$ & $\gamma$ & $\delta$ \\ \hhline{|======|}
	 $\{ \alpha\beta^2,\gamma^4 \}, \{ \alpha\gamma^2,\beta^4 \}, \{ \alpha\gamma^2,\alpha\beta\delta^2 \}$ & $16$ & $\pi$ & $\frac{1}{2}\pi$ & $\frac{1}{2}\pi$ & $\frac{1}{4}\pi$  \\
	\hline
	\end{tabular}
\egroup
\end{center}
\caption{Type I and $\alpha \ge \pi$: vertex pairs and angle values}
\label{AlphaType-I}
\end{table}

We combine the results in Table \ref{AlphaType-I} and therefore,
\begin{align*} 
f=16, \, \AVC = \{ \alpha\beta^2, \alpha\gamma^2, \alpha\beta\delta^2, \beta^4, \beta^2\gamma^2, \gamma^4, \alpha\delta^4, \beta^3\delta^2,  \beta\gamma^2\delta^2, \beta^2\delta^4, \gamma^2\delta^4, \beta\delta^6, \delta^8 \}.
\end{align*}

Type II: By the same argument (using \eqref{MyPerm} and the second case of Table \ref{Recali}), the process yields no valid angle values.

Type III: By the same argument (using \eqref{MyPerm} on Table \ref{MySpSol} and the second case of Table \ref{Recali}), the process yields no valid angle values.   
\end{case*}

\begin{case*}[$\alpha, \gamma, \delta<\pi$ and $\beta \ge \pi$] By $\beta\ge\pi$, we know that $\beta^2\cdots$ is not a vertex. 

Type I: The same argument gives valid angle values in Table \ref{BetaType-I}. 

\begin{table}[h]
\begin{center}
\bgroup
\def\arraystretch{1.5}
    \begin{tabular}[]{ | c | c | c | c | c | c  | }
	\hline
	Pairs & $f$ & $\alpha$ & $\beta$ & $\gamma$ & $\delta$ \\ \hhline{|======|}
	$\{ \alpha^3, \beta\delta^2 \}, \{ \alpha^3, \delta^4 \}$ & $8$ & $\frac{2}{3}\pi$ & $\pi$ & $\frac{1}{3}\pi$ & $\frac{1}{2}\pi$  \\
	\hline
	$\{ \alpha^2\beta, \beta\delta^2 \}$ & $16$ & $\frac{1}{2}\pi$ & $\pi$ & $\frac{1}{4}\pi$ & $\frac{1}{2}\pi$ \\
	\hline
	\end{tabular}
\egroup
\end{center}
\caption{Type I and $\beta \ge \pi$: vertex pairs and angle values}
\label{BetaType-I}
\end{table}

In $\{ \alpha^3, \beta\delta^2 \}$, by \eqref{VCount}, $f=8$ implies $v_{\ge 6}=0$. So the other vertices are $\delta^4, \alpha^2\gamma^2, \alpha\gamma^4$. Therefore,
\begin{align*}
f=8, \quad \AVC = \{ \alpha^3, \beta\delta^2, \delta^4, \alpha^2\gamma^2, \alpha\gamma^4 \}. 
\end{align*}

In $\{ \alpha^2\beta, \beta\delta^2 \}$, exchanging $\alpha \leftrightarrow \beta$ and $\gamma \leftrightarrow \delta$, we recover the $\AVC$ derived by Table \ref{AlphaType-I}.  

Type II, III: By the same argument and $\beta \ge \pi$ and $\alpha, \gamma, \delta<\pi$, there are no valid angle values.
\end{case*}

For the case $\alpha, \beta, \delta<\pi$ and $\gamma \ge \pi$ and the case $\alpha, \beta, \gamma<\pi$ and $\delta \ge \pi$, the same arguments imply no valid angle values. 

All the $\AVC$s are summarised in Table \ref{RatAVC}. Next we explain why some of the $\AVC$s do not constitute tilings and how tilings can be constructed from the others.

\begin{table}[h]
\begin{center}
\bgroup
\def\arraystretch{1.5}
    \begin{tabular}[]{ | c | c | }
	\hline
	$f$ & $\AVC$ \\ \hhline{|==|}
	$8$ & $\{ \alpha^3, \beta\delta^2, \delta^4, \alpha^2\gamma^2, \alpha\gamma^4 \}$ \\ 
	\hline
	$16$ & $ \{ \alpha\beta^2, \alpha\gamma^2, \alpha\beta\delta^2, \beta^4, \beta^2\gamma^2, \gamma^4, \alpha\delta^4, \beta^3\delta^2,  \beta\gamma^2\delta^2, \beta^2\delta^4, \gamma^2\delta^4, \beta\delta^6, \delta^8 \}$ \\
	\hline
	$20$ & $\{ \alpha\beta^2,\gamma\delta^3, \alpha^2\gamma\delta \}$  \\
	\hline
	$24$ & $\{ \alpha\beta^2, \alpha^4, \gamma\delta^3, \alpha\beta\gamma^2, \alpha\gamma^4 \}$  \\
	\hline
	$36$ & $\{ \alpha\beta^2, \alpha^2\delta^2, \gamma\delta^3, \alpha^3\gamma^2, \alpha\gamma^3\delta,  \gamma^6 \}$ \\
	\hline
	$36$ & $\{ \alpha\delta^2, \alpha\beta^3, \gamma^3\delta, \alpha^2\beta\gamma^2, \alpha^6 \}$ \\
	\hline
	$60$ & $\{ \alpha\beta^2,\gamma\delta^3, \alpha^3\beta, \alpha^5, \beta\gamma^4, \alpha^2\gamma^4 \}$ \\
	\hline
	$84$ & $\{ \alpha\beta^2,\gamma\delta^3, \alpha^3\gamma\delta, \gamma^5\delta \}$ \\
	\hline
	$132$ & $\{ \alpha\beta^2,\gamma\delta^3, \alpha^4\gamma^2, \alpha\gamma^6 \}$ \\
	\hline
	\end{tabular}
\egroup
\end{center}
\caption{$\AVC$s of rational angles without $\alpha\gamma\delta$}
\label{RatAVC}
\end{table}

\newpart{$\AVC$s without tiling}

In $f=8, \AVC = \{ \alpha^3, \beta\delta^2, \delta^4, \alpha^2\gamma^2, \alpha\gamma^4 \}$, by no $\beta^2\cdots$, we know that $\alpha^{\beta}\vert^{\beta}\alpha\cdots$, $\alpha^{\beta}\vert^{\beta}\gamma \cdots, \gamma^{\beta}\vert^{\beta}\gamma \cdots$ are not vertices. Hence $\alpha^2\gamma^2, \alpha\gamma^4$ are not vertices. The $\AVC$ is reduced to $\{ \alpha^3, \beta\delta^2, \delta^4 \}$ in which $\gamma$ does not appear, a contradiction.

In $f=20, \AVC = \{ \alpha\beta^2,\gamma\delta^3, \alpha^2\gamma\delta \}$, applying Counting Lemma to $\gamma, \delta$, we know that $\gamma\delta^3$ is not a vertex, which contradicts Lemma \ref{2HDAngLem}.

The $\AVC$ for $f=24$ implies $\beta^2\cdots = \alpha\beta^2$ and $\gamma\delta\cdots = \gamma\delta^3$ and $\alpha^2\gamma\cdots$, $\alpha\delta\cdots, \beta\delta\cdots$ are not vertices. By no $\alpha\delta\cdots, \beta\delta\cdots$, the vertex $\alpha\beta\gamma^2$ has unique AAD $\bvert  \, \gamma^{\beta} \vert^{\beta} \alpha^{\delta} \vert ^{\gamma} \beta^{\alpha} \vert^{\beta} \gamma \, \bvert$. The AAD of $\gamma^{\beta}\vert^{\beta}\alpha$ in the first picture of Figure \ref{AAD-alal-bebe-gaga-bebe} determines tiles $T_1, T_2$. By $\beta^2\cdots = \alpha\beta^2$ and no $\alpha\delta\cdots$, we get $T_3$. By $\gamma\delta\cdots =\gamma\delta^3$, we determine $T_4$ and then $T_5$. This implies $\gamma^{\beta}\vert^{\beta}\alpha\cdots = \alpha^2\gamma\cdots$, a contradiction. Then $\beta\gamma\cdots=\alpha\beta\gamma^2$ is not a vertex. Then the AAD $\alpha^{\beta}\vert^{\beta}\alpha$ in the second picture determines $T_1, T_2$. As $\beta^2\cdots = \alpha\beta^2$, by mirror symmetry we also know $T_3$, which implies $\beta_3\gamma_2\cdots$, contradicting no $\beta\gamma\cdots$. So there is no $\alpha^{\beta}\vert^{\beta}\alpha$. Then no $\alpha^{\beta}\vert^{\delta}\alpha,\alpha^{\beta}\vert^{\beta}\alpha$ implies no $\alpha\alpha\alpha$. So $\alpha^4$ is not a vertex. The AAD $\gamma^{\beta} \vert^{\beta} \gamma$ in the third picture implies $\alpha\delta\cdots$, a contradiction. So $\alpha\gamma^4$ is not a vertex. The $\AVC$ is reduced to $\{ \alpha\beta^2, \gamma\delta^3 \}$. Applying Counting Lemma to $\alpha,\beta$, we get a contradiction.

\begin{figure}[htp]
\centering
\begin{tikzpicture}[>=latex,scale=1]

\begin{scope}

\foreach \a in {1,...,2}
\draw[rotate=90*\a]
	(0,0) -- (0,-1.2) -- (1.2,-1.2) -- (1.2,0);

\draw[]
	(-1.2,1.2) -- (-1.2,0) 
	(1.2,1.2) -- (1.2,0) -- (1.8,0)
	(1.2,1.2) -- (1.8,1.2);

\draw[line width=2]
	(1.2,0) -- (1.2,1.2) -- (1.2, 1.8)
	(0,0) -- (-1.2,0)
	(1.2,-0.6) -- (1.2,0);

\node at (0.2,0.2) {\small $\alpha$}; 
\node at (0.2,0.90) {\small $\beta$};
\node at (0.98,0.98) {\small $\gamma$}; 
\node at (0.98,0.2) {\small $\delta$};

\node at (-0.2,0.2) {\small $\gamma$}; 
\node at (-0.2,0.90) {\small $\beta$};
\node at (-0.98,0.98) {\small $\alpha$};
\node at (-0.98,0.2) {\small $\delta$};

\node at (0.0,1.45) {\small $\alpha$}; 
\node at (1.0, 1.45) {\small $\delta$};
\node at (-1.0, 1.45) {\small $\beta$};

\node at (1.4, 0.98) {\small $\delta$}; 
\node at (1.4,0.2) {\small $\gamma$};

\node at (0.98,-0.25) {\small $\delta$}; 
\node at (0.2,-0.2) {\small $\alpha$}; 

\node[inner sep=1,draw,shape=circle] at (0.58,0.55) {\small $1$};
\node[inner sep=1,draw,shape=circle] at (-0.60,0.6) {\small $2$};
\node[inner sep=1,draw,shape=circle] at (0.0,1.95) {\small $3$};
\node[inner sep=1,draw,shape=circle] at (1.58,0.55) {\small $4$};
\node[inner sep=1,draw,shape=circle] at (0.58,-0.55) {\small $5$};

\end{scope}

\begin{scope}[xshift=4.5cm]

\foreach \a in {1,...,2}
\draw[rotate=90*\a]
	(0,0) -- (0,-1.2) -- (1.2,-1.2) -- (1.2,0);

\draw[]
	(0,0) -- (1.2,0)
	(0,0) -- (-1.2,0);

\draw[line width=2]
	(-1.2,1.2) -- (-1.2,0) 
	(1.2,1.2) -- (1.2,0);

\node at (0.2,0.2) {\small $\alpha$}; 
\node at (0.2,0.90) {\small $\beta$};
\node at (0.98,0.98) {\small $\gamma$}; 
\node at (0.98,0.2) {\small $\delta$};

\node at (-0.2,0.2) {\small $\alpha$}; 
\node at (-0.2,0.90) {\small $\beta$};
\node at (-0.98,0.98) {\small $\gamma$};
\node at (-0.98,0.2) {\small $\delta$};

\node at (0.0,1.45) {\small $\alpha$}; 
\node at (1.0, 1.45) {\small $\delta$};
\node at (-1.0, 1.45) {\small $\beta$};

\node[inner sep=1,draw,shape=circle] at (0.58,0.55) {\small $1$};
\node[inner sep=1,draw,shape=circle] at (-0.60,0.6) {\small $2$};
\node[inner sep=1,draw,shape=circle] at (0.0,1.95) {\small $3$};

\end{scope}

\begin{scope}[xshift=9cm]

\foreach \a in {1,...,2}
\draw[rotate=90*\a]
	(0,0) -- (0,-1.2) -- (1.2,-1.2) -- (1.2,0);

\draw[]
	(-1.2,1.2) -- (-1.2,0) 
	(1.2,1.2) -- (1.2,0);

\draw[line width=2]
	(0,0) -- (1.2,0)
	(0,0) -- (-1.2,0);

\node at (0.2,0.2) {\small $\gamma$}; 
\node at (0.2,0.90) {\small $\beta$};
\node at (0.98,0.98) {\small $\alpha$}; 
\node at (0.98,0.2) {\small $\delta$};

\node at (-0.2,0.2) {\small $\gamma$}; 
\node at (-0.2,0.90) {\small $\beta$};
\node at (-0.98,0.98) {\small $\alpha$};
\node at (-0.98,0.2) {\small $\delta$};

\node at (0.0,1.45) {\small $\alpha$}; 
\node at (1.0, 1.45) {\small $\delta$};
\node at (-1.0, 1.45) {\small $\beta$};


\node[inner sep=1,draw,shape=circle] at (0.58,0.55) {\small $1$};
\node[inner sep=1,draw,shape=circle] at (-0.60,0.6) {\small $2$};
\node[inner sep=1,draw,shape=circle] at (0.0,1.95) {\small $3$};

\end{scope}

\end{tikzpicture}
\caption{The AADs of $\alpha^{\beta}\vert^{\beta}\alpha$ and $\gamma^{\beta}\vert^{\beta}\gamma$}
\label{AAD-alal-bebe-gaga-bebe}
\end{figure}
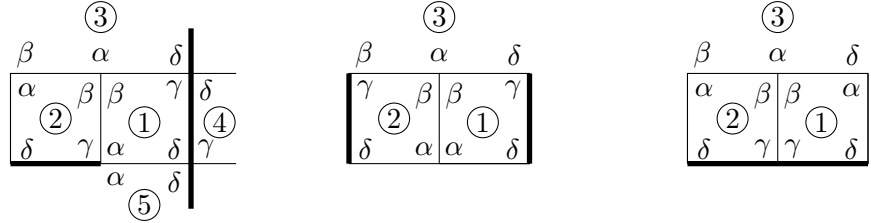

In $f=60, \AVC=\{ \alpha\beta^2,\gamma\delta^3, \alpha^3\beta, \alpha^5, \beta\gamma^4, \alpha^2\gamma^4 \}$, by no $\alpha\delta\cdots$, the AAD in the third picture of Figure \ref{AAD-alal-bebe-gaga-bebe} arrives at a contradiction. It shows that $\gamma^{\beta} \vert^{\beta} \gamma \cdots$ is not a vertex. So $\alpha^2\gamma^4$ has AAD $\bvert \, \gamma^{\beta} \vert \alpha \vert^{\beta} \gamma \, \bvert \, \gamma^{\beta} \vert \alpha \vert^{\beta} \gamma \, \bvert$ and $\beta\gamma^4$ is not a vertex. By no $\beta\gamma\cdots, \beta\delta\cdots$, there is no $\alpha\alpha\alpha, \gamma\alpha\gamma$ and hence $\alpha^3\beta, \alpha^5, \alpha^2\gamma^4$ are not vertices. Then $\alpha^2\cdots$ is not a vertex and hence the same for $\delta^{\alpha}\vert^{\alpha}\delta \cdots$. So $\gamma\delta^3$ is not a vertex, a contradiction.

Among the $f=36, \AVC=\{ \alpha\beta^2, \alpha^2\delta^2, \gamma\delta^3, \alpha^3\gamma^2, \alpha\gamma^3\delta,  \gamma^6 \}$ and $f=84, \AVC=\{ \alpha\beta^2,\gamma\delta^3, \alpha^3\gamma\delta, \gamma^5\delta \}$ and $f=132, \AVC=\{ \alpha\beta^2,\gamma\delta^3, \alpha^4\gamma^2, \alpha\gamma^6 \}$, we have $\beta^2\cdots=\alpha\beta^2$ and no $\beta\delta\cdots$. By no $\beta\delta\cdots$, the AAD $\alpha^{\beta}\vert^{\delta}\alpha$ is dismissed. If $\beta\gamma\cdots$ is not a vertex, then the AAD in the second picture of Figure \ref{AAD-alal-bebe-gaga-bebe} implies no $\alpha^{\beta}\vert^{\beta}\alpha$. Then no $\alpha^{\beta}\vert^{\delta}\alpha$ and $\alpha^{\beta}\vert^{\beta}\alpha$ imply no $\alpha\alpha\alpha$. So $ \alpha^3\gamma^2, \alpha^3\gamma\delta, \alpha^4\gamma^2$ cannot be vertices when $f=36, 84, 132$ and $\alpha\beta^2$ is a vertex. Then there is no $\alpha^2\cdots$ in the $\AVC$s of $f=84, 132$ which implies that $\delta^{\alpha}\vert^{\alpha}\delta\cdots$ is not a vertex. So $\gamma\delta^3$ is not a vertex, a contradiction. The $\AVC$s for $f=84, 132$ are dismissed and the one for $f=36$ constitutes to a tiling which remains to be discussed below.

\newpart{$\AVC$s with tilings}

In the $\AVC$ for $f=16$, by no $\alpha^2\cdots$, the AADs $\delta^{\alpha}\vert^{\alpha}\delta, \beta^{\alpha}\vert^{\alpha}\beta, \beta^{\alpha}\vert^{\alpha} \delta$ are dismissed and $\beta\cdots\beta$ has unique AAD $\beta^{\alpha} \vert \cdots \vert \beta^{\alpha}$. Then there is no $\delta\delta\delta$. This means $\alpha\delta^4,\beta^3\delta^2, \beta^2\delta^4$, $\gamma^2\delta^4, \beta\delta^6, \delta^8$ are not vertices. The $\AVC$ is reduced to
\begin{align*}
f=16, \quad &\AVC = \{ \alpha\beta^2, \alpha\gamma^2, \alpha\beta\delta^2, \beta^4, \beta^2\gamma^2, \gamma^4,  \beta\gamma^2\delta^2 \}.
\end{align*}
By the proof of \cite[Proposition 39]{cly}, $\alpha\beta^2, \beta\gamma^2\delta^2, \beta^2\gamma^2$ are not vertices. The $\AVC$ is further reduced to
\begin{align*}
f=16, \quad \AVC = \{ \alpha\gamma^2, \alpha\beta\delta^2, \beta^4,  \gamma^4 \}.
\end{align*}
By no $\alpha^2\cdots$, the vertex $\beta^4$ has unique AAD $\vert^{\gamma} \beta^{\alpha}\vert^{\gamma} \beta^{\alpha}\vert^{\gamma} \beta^{\alpha}\vert^{\gamma} \beta^{\alpha}\vert$. Then the AAD $\gamma^{\beta}\vert^{\beta}\gamma$ implies $\beta^2\cdots = \beta^4$ which contradicts its unique AAD. So $\gamma^4$ is not a vertex. The $\AVC$ is reduced to
\begin{align}
\label{ab2ag2AVC}
f=16, \quad \AVC = \{ \alpha\gamma^2, \alpha\beta\delta^2, \beta^4 \}.
\end{align}
By $\AVC$ \eqref{ab2ag2AVC}, we construct $S3, S^{\prime}3$ in Figure \ref{Tiling-S3}. By Proposition \ref{Geom-AlGa2-AlBeDe2}, the tiles are actually triangles and hence the pictures in Figure \ref{Polar-Tilings}. 

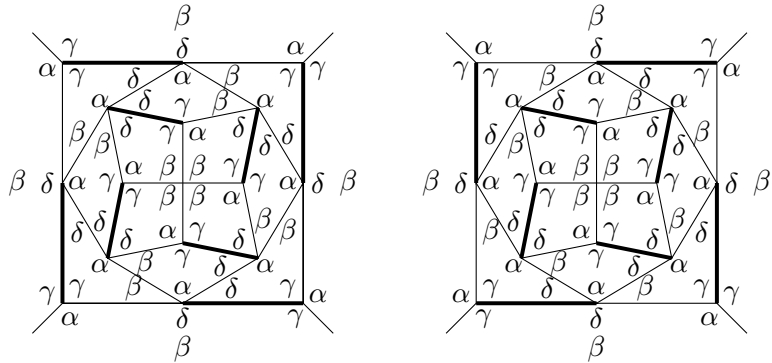
\begin{figure}[htp]
\centering
\begin{tikzpicture}[scale=1]

\foreach \a in {0,...,3}
{

\begin{scope}[rotate=90*\a]

\draw
	(0,0) -- (0.8,0) -- (1,1) -- (0,0.8)
	(1.6,0) -- (1,1) -- (0,1.6) -- (1.6,1.6) -- (-1.6,1.6)
	(1.6,1.6) -- (2,2);

\draw[line width=1.5]
	(0.8,0) -- (1,1)
	(1.6,1.6) -- (1.6,0);

\node at (0.2,0.65) {\small $\alpha$};
\node at (0.2,0.2) {\small $\beta$};
\node at (0.65,0.2) {\small $\gamma$};
\node at (0.75,0.75) {\small $\delta$};

\node at (1.4,0) {\small $\alpha$};
\node at (1.07,-0.5) {\small $\beta$};
\node at (1,0) {\small $\gamma$};
\node at (1.1,0.5) {\small $\delta$};

\node at (1.1,1.1) {\small $\alpha$};
\node at (0.65,1.4) {\small $\beta$};
\node at (1.4,0.65) {\small $\delta$};
\node at (1.4,1.4) {\small $\gamma$};

\node at (1.8,-1.5) {\small $\alpha$};
\node at (2.2,0) {\small $\beta$};
\node at (1.8,1.5) {\small $\gamma$};
\node at (1.8,0) {\small $\delta$};

\end{scope}

\begin{scope}[xshift=5.5cm, rotate=90*\a]

\draw
	(0,0) -- (0.8,0) -- (1,1) -- (0,0.8)
	(1.6,0) -- (1,1) -- (0,1.6) -- (1.6,1.6) -- (-1.6,1.6)
	(1.6,1.6) -- (2,2);

\draw[line width=1.5]
	(0.8,0) -- (1,1)
	(1.6,1.6) -- (0,1.6);

\node at (0.2,0.65) {\small $\alpha$};
\node at (0.2,0.2) {\small $\beta$};
\node at (0.65,0.2) {\small $\gamma$};
\node at (0.75,0.75) {\small $\delta$};

\node at (1.4,0) {\small $\alpha$};
\node at (1.07,-0.5) {\small $\beta$};
\node at (1,0) {\small $\gamma$};
\node at (1.1,0.5) {\small $\delta$};

\node at (1.1,1.1) {\small $\alpha$};
\node at (1.4,0.65) {\small $\beta$};
\node at (1.4,1.4) {\small $\gamma$};
\node at (0.65,1.4) {\small $\delta$};

\node at (1.8,1.5) {\small $\alpha$};
\node at (2.2,0) {\small $\beta$};
\node at (1.8,-1.5) {\small $\gamma$};
\node at (1.8,0) {\small $\delta$};

\end{scope}

}

\end{tikzpicture}
\caption{Tilings $S3,S^{\prime}3$}
\label{Tiling-S3}
\end{figure}

In $f=36, \AVC=\{ \alpha\beta^2, \alpha^2\delta^2, \gamma\delta^3, \alpha^3\gamma^2, \alpha\gamma^3\delta,  \gamma^6 \}$, the earlier discussion already shows that $\alpha^3\gamma^2$ is not a vertex. The $\AVC$ is reduced to 
\begin{align}\label{ab2a2d2AVC}
f=36, \quad \{ \alpha\beta^2, \alpha^2\delta^2, \gamma\delta^3, \alpha\gamma^3\delta,  \gamma^6 \}.
\end{align}
By $\AVC$ \eqref{ab2a2d2AVC}, we construct $S5$ in Figure \ref{Tiling-S5}.

\begin{figure}[htp]
\centering
\begin{tikzpicture}[xscale=-1]

\foreach \a in {0,1,2}
\foreach \b in {1,-1}
{
\begin{scope}[xshift=3.2*\a cm + 0.4*\b cm, scale=\b]

\draw
	(-1.6,0) -- (1.6,0)
	(-1.6,1.2) -- (1.6,1.2)
	(-0.8,0) -- (-0.8,1.2)
	(0.8,0) -- (0.8,1.2)
	(0,1.2) -- (0,2);

\draw[line width=1.5]
	(-1.6,2) -- (-1.6,0)
	(1.6,2) -- (1.6,0)
	(-0.8,1.2) -- (0.8,0);

\node at (0.8,1.35) {\small $\alpha$};
\node at (0.2,1.4) {\small $\beta$};
\node at (1.4,1.4) {\small $\delta$};
\node at (0.8,1.9) {\small $\gamma$};

\node at (-0.8,1.35) {\small $\alpha$};
\node at (-0.2,1.4) {\small $\beta$};
\node at (-1.4,1.4) {\small $\delta$};
\node at (-0.8,1.9) {\small $\gamma$};

\node at (0,1.05) {\small $\alpha$};
\node at (0.65,1) {\small $\beta$};
\node at (-0.3,1.05) {\small $\delta$};
\node at (0.65,0.4) {\small $\gamma$};

\node at (-0.65,0.2) {\small $\alpha$};
\node at (0,0.2) {\small $\beta$};
\node at (-0.65,0.8) {\small $\delta$};
\node at (0.25,0.17) {\small $\gamma$};

\node at (1,0.2) {\small $\alpha$};
\node at (1,1) {\small $\beta$};
\node at (1.4,0.2) {\small $\delta$};
\node at (1.4,1) {\small $\gamma$};

\node at (-1,1) {\small $\alpha$};
\node at (-1,0.2) {\small $\beta$};
\node at (-1.4,1) {\small $\delta$};
\node at (-1.4,0.2) {\small $\gamma$};

\end{scope}
}

\end{tikzpicture}
\caption{Tiling $S5$}
\label{Tiling-S5}
\end{figure}
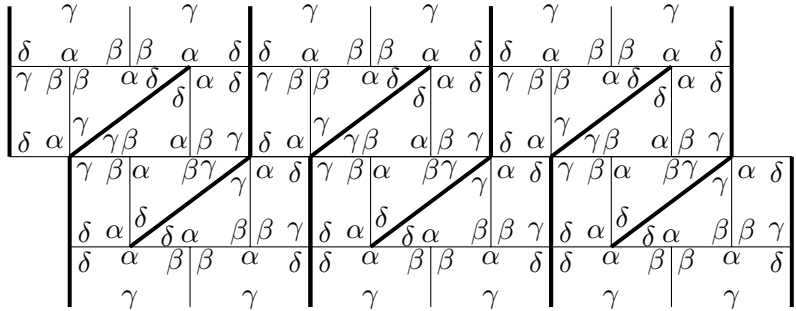

In $f=36, \AVC = \{ \alpha\delta^2, \alpha\beta^3, \gamma^3\delta, \alpha^2\beta\gamma^2, \alpha^6 \}$, by no $\beta\delta\cdots$ and $\delta \vert \delta \cdots$, the AADs $\alpha^{\beta} \vert^{\delta} \alpha, \alpha^{\delta} \vert^{\delta} \alpha$ are dismissed. So there is no $\alpha\alpha\alpha$ and $\alpha^6$ is not a vertex. The $\AVC$ is reduced to
\begin{align}\label{ad2ab3AVC}
f=36, \quad \{ \alpha\delta^2, \alpha\beta^3, \gamma^3\delta, \alpha^2\beta\gamma^2 \}.
\end{align}
By $\AVC$ \eqref{ad2ab3AVC}, we construct $S6$ in Figure \ref{Tiling-S6}.

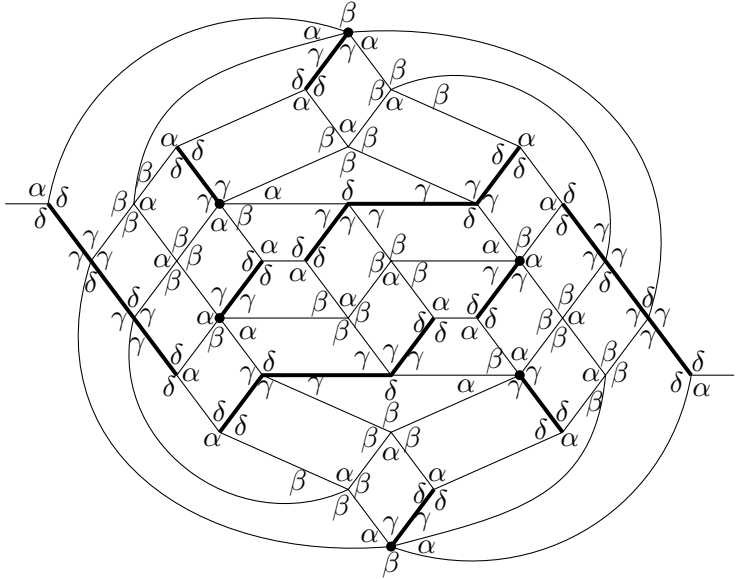
\begin{figure}[htp]
\centering
\begin{tikzpicture}

\begin{scope}[scale=0.95]

\foreach \a in {1,-1}
{
\begin{scope}[scale=\a]

\draw
	(0,0) -- (0.3,0.4) -- (2.1,0.4) 
	(-0.3,1.2) -- (0.9,-0.4) -- (1.5,-0.4) -- (2.1,-1.2) -- (3.3,0.4)
	(2.1,-1.2) -- (0.3,-1.2)
	(2.7,-2) -- (3.3,-1.2) -- (1.5,1.2)
	(2.1,0.4) -- (2.7,1.2) -- (2.1,2)
	(-1.5,-1.2) -- (0.3,-2) -- (2.1,-1.2)
	(2.7,-2) -- (0.9,-2.8) -- (0.3,-2) -- (-0.3,-2.8) -- (-2.1,-2)
	(0.3,-3.6) -- (-0.3,-2.8)
	(3.3,-1.2) -- (3.9,-0.4)
	(4.5,-1.2) -- (5.1,-1.2)
	(3.3,0.4) to[out=75,in=25, distance=1.75cm] (0.3,2.8)
	(3.9,-0.4) to[out=72,in=5, distance=2.75cm] (-0.3,3.6)
	(3.3,-1.2) to[out=-95,in=15, distance=1.75cm] (0.3,-3.6)
	(4.5,-1.2) to[out=-100,in=-20, distance=2cm] (0.3,-3.6)
	;

\draw[line width=1.5]
	(2.1,2) -- (1.5,1.2) -- (-0.3,1.2) -- (-0.9,0.4)
	(1.5,-0.4) -- (2.1,0.4) 
	(2.7,-2) -- (2.1,-1.2)
	(0.3,-3.6) -- (0.9,-2.8)
	(2.7,1.2) -- (4.5,-1.2);

\fill
	(2.1,-1.2) circle (0.07)
	(2.1,0.4) circle (0.07)
	(0.3,-3.6) circle (0.07);

\node at (-0.3,-0.15) {\small $\alpha$};
\node at (0.1,0.35) {\small $\beta$};
\node at (-0.3,0.9) {\small $\gamma$};
\node at (-0.7,0.45) {\small $\delta$};

\node at (-1.4,0.6) {\small $\alpha$};
\node at (-1.75,1) {\small $\beta$};
\node at (-0.65,1) {\small $\gamma$};
\node at (-1,0.6) {\small $\delta$};

\node at (-1,0.2) {\small $\alpha$};
\node at (-0.7,-0.2) {\small $\beta$};
\node at (-1.7,-0.2) {\small $\gamma$};
\node at (-1.45,0.2) {\small $\delta$};

\node at (-1.7,-0.6) {\small $\alpha$};
\node at (-0.4,-0.6) {\small $\beta$};
\node at (-0.1,-1) {\small $\gamma$};
\node at (-1.4,-1) {\small $\delta$};

\node at (-2.1,0.9) {\small $\alpha$};
\node at (-2.45,0.45) {\small $\beta$};
\node at (-2.1,-0.1) {\small $\gamma$};
\node at (-1.7,0.4) {\small $\delta$};

\node at (-2.3,-0.4) {\small $\alpha$};
\node at (-2.75,0.1) {\small $\beta$};
\node at (-3.1,-0.4) {\small $\gamma$};
\node at (-2.7,-0.9) {\small $\delta$};

\node at (-2.5,-1.2) {\small $\alpha$};
\node at (-2.15,-0.7) {\small $\beta$};
\node at (-1.7,-1.2) {\small $\gamma$};
\node at (-2.1,-1.7) {\small $\delta$};

\node at (-1.35,1.35) {\small $\alpha$};
\node at (-0.3,1.75) {\small $\beta$};
\node at (0.75,1.35) {\small $\gamma$};
\node at (-0.3,1.4) {\small $\delta$};

\node at (-3.1,1.2) {\small $\alpha$};
\node at (-2.65,0.7) {\small $\beta$};
\node at (-2.3,1.2) {\small $\gamma$};
\node at (-2.7,1.7) {\small $\delta$};

\node at (-0.95,2.6) {\small $\alpha$};
\node at (-0.6,2.1) {\small $\beta$};
\node at (-2.05,1.4) {\small $\gamma$};
\node at (-2.4,1.95) {\small $\delta$};

\node at (0.35,2.6) {\small $\alpha$};
\node at (-0,2.1) {\small $\beta$};
\node at (1.45,1.4) {\small $\gamma$};
\node at (1.8,1.9) {\small $\delta$};

\node at (-0.3,2.3) {\small $\alpha$};
\node at (0.1,2.75) {\small $\beta$};
\node at (-0.3,3.3) {\small $\gamma$};
\node at (-0.7,2.85) {\small $\delta$};

\node at (-2.9,0.4) {\small $\alpha$};
\node at (-3.35,0.9) {\small $\beta$};
\node at (-3.7,0.4) {\small $\gamma$};
\node at (-3.3,-0.1) {\small $\delta$};

\node at (-2.2,-2.1) {\small $\alpha$};
\node at (-1,-2.7) {\small $\beta$};
\node at (-3.25,-0.8) {\small $\gamma$};
\node at (-2.8,-1.35) {\small $\delta$};

\node at (-0,-3.45) {\small $\alpha$};
\node at (-0.4,-3.05) {\small $\beta$};
\node at (-3.5,-0.45) {\small $\gamma$};
\node at (-3.9,0.1) {\small $\delta$};

\node at (-2.8,2.1) {\small $\alpha$};
\node at (-3.15,1.6) {\small $\beta$};
\node at (-0.75,3.25) {\small $\gamma$};
\node at (-1,2.95) {\small $\delta$};

\node at (-0.8,3.6) {\small $\alpha$};
\node at (-3.5,1.2) {\small $\beta$};
\node at (-3.9,0.7) {\small $\gamma$};
\node at (-4.3,1.3) {\small $\delta$};

\node at (-4.65,1.4) {\small $\alpha$};
\node at (-0.3,3.85) {\small $\beta$};
\node at (-4.1,0.4) {\small $\gamma$};
\node at (-4.6,1) {\small $\delta$};

\end{scope}
}

\end{scope}

\end{tikzpicture}
\caption{Tiling $S6$}
\label{Tiling-S6}
\end{figure}

This completes the proof.
\end{proof}

We provide the pseudocode for Propositions \ref{RatProp}, \ref{RatAlGaDeProp}. In preprocessing, we define the functions, f\_Condition and Angle\_Condition, for executing Step 3 in our scheme. For example, the pseudocode as written, is for the convex case. The other cases can be defined similarly. 

\begin{algorithm}[H]
\begin{algorithmic}[1] \caption*{Preprocessing}
\State \text{Declare Function: f\_Condition($f$):=}
\Indent
\If {$f \neq \emptyset$ \AND consistent \AND even \AND $\ge 8$ }
\Indent 
\Return {true} 
\EndIndent
\Else \quad
\Return {false}
\EndIf 
\EndIndent
\State \text{Declare Function: Angle\_Condition($[\alpha, \beta, \gamma, \delta]$):=}
\Indent
\If {$[\alpha, \beta, \gamma, \delta] \neq \emptyset$ 
\AND $0 < \alpha, \beta, \gamma, \delta < \pi$ \AND valid} 
\Indent 
\Return {true} 
\EndIndent
\Else \quad
\Return {false}
\EndIf 
\EndIndent
\end{algorithmic}
\end{algorithm}

The pseudocode for computing angles and $f$ via Type I solutions is given in Algorithm \ref{Type-I-Algo} and the pseudocode for computing angles and $f$ via Type II, III solutions is given in Algorithm \ref{Type-II-III-Algo}. 

In Algorithm \ref{Type-I-Algo}, we define Vertex\_Eqns by the angle relation(s) in Lemma \ref{ReducedTrigAng}. For example, in the convex case, we define Case\_Eqns by $\alpha=2\gamma$ and $\beta=2\delta$. We define Vertex\_Eqns by the vertex angle sums given by the vertices in Lemma \ref{PairsLem}. Then we execute Step 2 and Step 3.

In Algorithm \ref{Type-II-III-Algo}, we define Vertex\_Eqns in the same way by Lemma \ref{PairsLem}. We define Case\_Recal by the recalibrations in Table \ref{Recali}. We define Myerson\_Sol by Type II or III solutions. The quadrilateral angle sum defines Angle\_Eqns. After solving for $\alpha, \beta, \gamma, \delta$ (and $\theta$) in terms of $f$ in the first procedure (Step 2), we dismiss angle values if they fail the criteria in Step 2 of our scheme. Then we carry out Step 3 in the second procedure. 

\begin{algorithm}[H]
\caption{Algorithm 1: Type I Rational Angle Values} \label{Type-I-Algo}
\begin{algorithmic}[1]
\Procedure{Solve and Select Angle Values}{}
\State \text{Declare Array: Vertex\_Eqns, Case\_Eqns, Angle\_Eqns;}
\State \text{Declare Array: Angle\_Values, Valid\_Angle\_Values;}
\State \text{Declare Rational Number: f\_Soln;}
\For {$i:1$ \ \textbf{through} \ length(Vertex\_Eqns) } 
\State \text{Angle\_Eqns: concatenate(Vertex\_Eqns[$i$], Case\_Eqns),}
\State \text{f\_Soln: solve(Angle\_Eqns, $f$),}
\State \text{Angle\_Soln: solve(Angle\_Eqns, $[\alpha, \beta, \gamma, \delta]$),}
\If {f\_Condition(f\_Soln) \AND Angle\_Condition(Angle\_Values)}
\State \text{append(Valid\_Angle\_Values, Angle\_Values)}
\EndIf
\EndFor 
\EndProcedure
\end{algorithmic}
\end{algorithm}

\begin{algorithm}[]
\caption{Algorithm 2: Type II, III Rational Angle Values} \label{Type-II-III-Algo}
\begin{algorithmic}[1]
\Procedure{Solve Angle Values}{}
\State \text{Declare Array: Myerson\_Sol, Case\_Recal, Angle\_Eqns;}
\State \text{Declare Array: Angle\_Soln, Angle\_Values;}
\State \text{Declare Rational Number: f\_Soln;}
\For {\quad $m:1$ \ \textbf{through} \ length(Myerson\_Sol) \ } 
\For {\quad $c:1$ \ \textbf{through} \ length(Case\_Recal) \ }
\For  {\quad $i:1$ \ \textbf{while} \ $i\le4$ \ }
\State \text{append(Angle\_Eqns, Myerson\_Sol}$[m][i]=$\text{Case\_Recal}$[c][i]$)
\EndFor 
\State \text{f\_Soln: solve(Angle\_Eqns, $f$),}
\State \text{Angle\_Soln: solve(Angle\_Eqns, $[\alpha, \beta, \gamma, \delta]$),}
\If {f\_Condition(f\_Soln) \AND Angle\_Condition(Angle\_Soln)}
\State \text{append(Angle\_Values, Angle\_Soln)}
\EndIf
\EndFor \quad 
\EndFor \quad 
\State 
\EndProcedure
\Procedure{Select Valid Angle Values}{}
\State \text{Declare Array: Vertex\_Eqns;}
\State \text{Declare Array: Sub\_Vertices, Vertex\_Angles, Valid\_Angle\_Values;}
\State \text{Declare Rational Number: f\_Value;}
\For {\quad $a:1$ \ \textbf{through} \ length(Angle\_Values) \ } 
\For {\quad $v:1$ \ \textbf{through} \ length(Vertex\_Eqns) \ }
\State \text{Sub\_Vertices: Substitute(Angle\_Values[a], Vertex\_Eqns[v])}, 
\State \text{f\_Value: solve(Sub\_Vertices, $f$)},
\If { f\_Condition(f\_Value) }
\State \text{Vertex\_Angles: solve(Sub\_Vertices, $[\alpha, \beta, \gamma, \delta, \text{f\_Value}]$),}
\If { Angle\_Condition(Vertex\_Angles) }
\State \text{append(Valid\_Angle\_Values, Vertex\_Angles)}
\EndIf
\EndIf
\EndFor
\EndFor
\EndProcedure
\end{algorithmic}
\end{algorithm}

We now turn our attention to tilings with $\alpha\gamma\delta$ as a vertex. To simplify the discussion, we first establish the following fact.

\begin{lem} \label{Ga=Pi} If $\gamma = \pi$ and $f\ge8$, then the set of admissible vertices is 
\begin{align}\label{avc-ga=pi}
\AVC = \{ \alpha\gamma\delta, \alpha^3, \alpha^2\beta^2, \alpha\beta^n, \beta^n, \beta^n\gamma\delta \}. 
\end{align}
\end{lem}

\begin{proof} Suppose $\gamma = \pi$. Lemma \ref{TriQuadLem} implies that the quadrilateral is in fact an isosceles triangle $\triangle ABD$ in Figure \ref{TriABD} with edges $AB=AD=a$, and $BD=a+b$, and $\beta=\delta$. By Lemma \ref{AlConvexLem}, $\alpha, \beta, \delta < \pi$. Then $\triangle ABD$ is a standard isosceles triangle. Then by $BD>AB=AD$, $\gamma = \pi > \alpha > \beta = \delta$. 

By $\gamma=\pi$ and the quadrilateral angle sum, $\alpha + 2\beta = (1 + \frac{4}{f})\pi$. By $\alpha > \beta$, we get $\alpha > ( \frac{1}{3} + \frac{4}{3f} )\pi > \beta = \delta$. Since $\gamma = \pi$, we know that $\gamma^2\cdots$ is not a vertex. Balance Lemma implies that $\delta^2\cdots$ is also not a vertex and every $b$-vertex has exactly one $\gamma$ and one $\delta$. 

Assume $\alpha\gamma\delta$ is not a vertex. Then the only $b$-vertex is $\gamma\cdots=\delta\cdots = \beta^n\gamma\delta$. Counting Lemma on $\beta, \gamma$ implies $n=1$ in $\beta^n\gamma\delta$. Then $\gamma=\pi$ and $\beta\gamma\delta$ imply $\pi = \beta + \delta < ( \frac{2}{3} + \frac{8}{3f} )\pi$ which implies $f<8$, contradicting $f\ge8$. So $\alpha\gamma\delta$ is a vertex. By $\gamma=\pi$, $\alpha>\delta$ and $\alpha\gamma\delta$, we get $\alpha > \frac{1}{2}\pi$. So $\alpha > \frac{1}{2}\pi$ and $\gamma=\pi$ and $\beta = \delta = \frac{4}{f}\pi$ determine all other vertices. Therefore we obtain $\AVC$ \eqref{avc-ga=pi}.
\end{proof}

\begin{prop}\label{RatAlGaDeProp} If one of $\alpha\gamma\delta, \beta\gamma\delta$ is a vertex and $f\ge8$, then tilings of the sphere by congruent almost equilateral quadrilaterals, where every angle is rational, are earth map tilings $E$, their flip modifications $E^{\prime}, E^{\prime\prime}$, and rearrangement $E^{\prime\prime\prime}$. 
\end{prop}

\begin{proof} Up to symmetry, we may assume $\alpha\gamma\delta$ is a vertex. By $\alpha \neq \beta$, this implies that $\beta\gamma\delta$ is not a vertex. By $f\ge8$ and $\alpha\gamma\delta$, we get $\beta = \frac{4}{f}\pi < \pi$. By $f\ge8$, Lemma \ref{AlGaDeLem} implies $\delta<\pi$. Similar to the previous proposition, the proof is divided into three cases: every angle $<\pi$, or exactly one of $\alpha, \gamma \ge\pi$. The angle values are determined via Type I, II, III solutions to \eqref{MyTrigEq}. Only the valid ones are selected. We give a couple of examples of our routine and leave out the details of the others. The routine again can be executed in computer. 

\begin{case*}[$\alpha, \beta, \gamma,\delta < \pi$] 

Type I: By relations $\alpha=2\gamma$ and $\beta=2\delta$ from Lemma \ref{ReducedTrigAng} and $\alpha\gamma\delta$, we get 
\begin{align*}
&f\ge8,&
&\alpha= ( \tfrac{4}{3}-\tfrac{4}{3f} ) \pi, &
&\beta=\tfrac{4}{f}\pi,&
&\gamma= ( \tfrac{2}{3}-\tfrac{2}{3f} ) \pi, &
&\delta=\tfrac{2}{f}\pi.&
\end{align*}
However, $f\ge8$ implies $\alpha>\pi$.

Type II: By matching $( \gamma - \frac{1}{2}\alpha, \frac{1}{2}\beta, \frac{1}{2}\alpha, \delta - \frac{1}{2}\beta )$ and its permutations in \eqref{MyPerm} with each row of the first case of Table \ref{Recali}, and then by $\alpha\gamma\delta$, we get
\begin{align*}
&f\ge8,& &\alpha =\tfrac{1}{3}\pi,& &\beta = \tfrac{4}{f}\pi,& &\gamma = (\tfrac{2}{3}+\tfrac{2}{f})\pi,& &\delta = (1- \tfrac{2}{f})\pi;& \\
&f=8,& &\alpha = \tfrac{1}{3}\pi,& &\beta = \tfrac{1}{2}\pi,& &\gamma=\tfrac{11}{12}\pi,& &\delta = \tfrac{3}{4}\pi;& \\ 
&f=12,& &\alpha=\tfrac{4}{9}\pi,& &\beta = \tfrac{1}{3}\pi,& &\gamma = \tfrac{2}{3}\pi,& &\delta = \tfrac{8}{9}\pi;& \\
&f=18,& & \alpha=\tfrac{4}{9}\pi,& &\beta = \tfrac{2}{9}\pi,& &\gamma = \tfrac{11}{18}\pi,& &\delta = \tfrac{17}{18}\pi;& \\
&f=24,& &\alpha = \tfrac{1}{3}\pi,& &\beta = \tfrac{1}{6}\pi,& &\gamma=\tfrac{3}{4}\pi,& &\delta=\tfrac{11}{12}\pi.&
\end{align*}
In the first two sets, we have $\alpha - \beta = \delta - \gamma$. The angles contradict Lemma \ref{ExchLem}. In the last three sets, we have $\alpha > \beta$ and $\delta > \gamma$, contradicting Lemma \ref{ExchLem}. So they are dismissed. Meanwhile, the set with $f=8$ is the minimal case of the first set. Then there are no valid angle values. 

Type III: By matching each row of Table \ref{MySpSol} and its permutations in \eqref{MyPerm} with each row of the first case of Table \ref{Recali}, and then by $\alpha\gamma\delta$, we get
\begin{align*}
&f=12,&
&\alpha=\tfrac{14}{15}\pi, &
&\beta=\tfrac{1}{3}\pi, &
&\gamma =\tfrac{23}{30}\pi, &
&\delta = \tfrac{3}{10}\pi. &
\end{align*}
Then we obtain the $\AVC$ below,
\begin{align*}
f=12, \quad \AVC=\{ \alpha\gamma\delta, \beta\gamma\delta^3, \beta^6 \}.
\end{align*}
Applying Counting Lemma to $\gamma,\delta$, we know that $\beta\gamma\delta^3$ is not a vertex and
\begin{align*}
f = 12, \quad
\AVC=\{ \alpha\gamma\delta, \beta^6 \}.
\end{align*}
\end{case*}

\begin{case*}[$\beta, \gamma, \delta<\pi$ and $\alpha \ge \pi$] By $\alpha\ge\pi$, we know that $\alpha^2\cdots$ is not a vertex. By Lemma \ref{AlGaDeLem}, the only vertices with strictly more $\delta$ than $\gamma$ are $\alpha\delta^2, \alpha\beta^n\delta^2$. We incorporate this fact in conjunction with Balance Lemma to filter the vertices. This will be explained in Type II and III solutions. 

Type I: By Lemma \ref{ReducedTrigAng} and $\alpha\gamma\delta$, we get
\begin{align*}
&f\ge8,&
&\alpha= ( \tfrac{4}{3}-\tfrac{4}{3f} ) \pi, &
&\beta=\tfrac{4}{f}\pi, &
& \gamma= ( \tfrac{2}{3}-\tfrac{2}{3f} ) \pi, &
&\delta=\tfrac{2}{f}\pi.&
\end{align*}
Then we obtain the $\AVC$ as follows
\begin{align*}
f \ge8, \quad
\AVC = \{ \alpha\gamma\delta,    \gamma^3\delta,  \beta^n,   \alpha\beta^n,   \alpha\beta^n\delta^2,   \beta^n\gamma^2,   \beta^n\gamma\delta,  \beta^n\gamma^2\delta^2  \}.
\end{align*}

Type II: By the same argument we get
\begin{align*}
&f=12,&  
&\alpha = \tfrac{10}{9}\pi, & &
\beta = \tfrac{1}{3}\pi, & &
\gamma = \tfrac{2}{3}\pi, & &
\delta = \tfrac{2}{9}\pi; \\
&f=18,&  
&\alpha = \tfrac{10}{9}\pi,& &
\beta = \tfrac{2}{9}\pi, & &
\gamma = \tfrac{13}{18}\pi,& &
\delta = \tfrac{1}{6}\pi.
\end{align*}
For more efficient checking, we first find that both sets of angle values fail $\alpha\delta^2, \alpha\beta^n\delta^2$. Then $\delta^2\cdots$ is not a vertex. By Balance Lemma, every vertex has either exactly one pair of $\gamma, \delta$ or none of them. Such vertices can only be $\beta^6$ for the first set and $\alpha\beta^4, \beta\gamma^2\delta^2, \beta^5\gamma\delta,  \beta^{9}$ for the second. Hence we have,
\begin{align*}
f=12, \quad &\AVC = \{ \alpha\gamma\delta,  \beta^6 \}; \\
f=18, \quad &\AVC = \{  \alpha\gamma\delta, \alpha\beta^4, \beta\gamma^2\delta^2, \beta^5\gamma\delta,  \beta^{9} \}.
\end{align*}

Type III: By the same argument, we get
\begin{align*}
&f=12,&
&\alpha = \tfrac{7}{5}\pi, & &
\beta = \tfrac{1}{3}\pi, & &
\gamma = \tfrac{8}{15}\pi, & &
\delta = \tfrac{1}{15}\pi; \\
&f=20,&
&\alpha = \tfrac{16}{15}\pi, & &
\beta = \tfrac{1}{5}\pi, & & 
\gamma = \tfrac{23}{30}\pi, & &
\delta = \tfrac{1}{6}\pi; \\
&f=30,&
&\alpha = \tfrac{7}{5}\pi, & &
\beta = \tfrac{2}{15}\pi, & &
\gamma = \tfrac{17}{30}\pi, & &
\delta = \tfrac{1}{30}\pi.
\end{align*}
The angle values in the first set do not satisfy $\alpha\beta^n\delta^2$ for any $n\in \mathbb{N}$. By Balance Lemma and Counting Lemma, every $b$-vertex has exactly one pair of $\gamma, \delta$. Then the only vertex other than $\alpha\gamma\delta$ is $\beta^6$. Hence we get,
\begin{align*}
f=12, \quad &\AVC = \{ \alpha\gamma\delta, \beta^6 \}.
\end{align*} 
For the second set of angle values, we have $3\gamma+\delta>2\pi$ and the remainder of $\gamma^2\cdots$ has value $\frac{7}{15}\pi$. No angles add up to it. Then $\gamma^2\cdots$ is not a vertex. By Balance Lemma and Counting Lemma, every $b$-vertex has exactly one pair of $\gamma, \delta$. Hence we get, 
\begin{align*}
f=20, \quad &\AVC = \{ \alpha\gamma\delta,  \beta^{10} \}.
\end{align*} 
For the third set, $\alpha\beta^4\delta^2$ is the only vertex with strictly more $\delta$ than $\gamma$. In the other $b$-vertices, the number of $\gamma$ is at least that of $\delta$. So they are $\beta^2\gamma^3\delta, \beta^6\gamma^2\delta^2$. The only remaining vertex is $\beta^{15}$. We obtain the third $\AVC$. 
\begin{align*}
f=30, \quad &\AVC = \{ \alpha\gamma\delta, \beta^2\gamma^3\delta, \beta^6\gamma^2\delta^2, \alpha\beta^4\delta^2, \beta^{15} \}.
\end{align*}
\end{case*}

\begin{case*}[$\alpha, \beta, \delta<\pi$ and $\gamma \ge \pi$] By $\gamma\ge\pi$, we know that $\gamma^2\cdots$ is not a vertex. Then Balance Lemma implies no $\delta^2\cdots$ and every $b$-vertex has only one pair of $\gamma, \delta$. 

Type I: By the same argument, we get
\begin{align*}
&f\ge8,&
&\alpha= (1 - \tfrac{4}{f} ) \pi, & &
\beta=\tfrac{4}{f}\pi, & &
\gamma=\pi, & &
\delta=\tfrac{4}{f}\pi; \\
&f\ge8,&
&\alpha= ( \tfrac{2}{3}-\tfrac{4}{3f} ) \pi, & &
\beta=\tfrac{4}{f}\pi, & &
\gamma= ( \tfrac{4}{3}-\tfrac{2}{3f} )\pi, & &
\delta=\tfrac{2}{f}\pi.
\end{align*}
In the first set of angle values, by Lemma \ref{Ga=Pi}, $\gamma=\pi$ implies $\AVC$ \eqref{avc-ga=pi}.

In the second set of angle values, by $\alpha\gamma\delta$ and no $\gamma^2\cdots, \delta^2\cdots$, the other $b$-vertex can only be $\beta^n\gamma\delta$. Meanwhile, the $\hat{b}$-vertices are $\alpha^m, \beta^n, \alpha^m\beta^n$. By $f\ge8$ and $\alpha= ( \tfrac{2}{3}-\tfrac{4}{3f} ) \pi$, we have $\alpha\ge\tfrac{1}{2}\pi$. Then $m\le 3$ in $\alpha^m, \alpha^m\beta^n$. In particular, $\alpha^m=\alpha^3$. So we get
\begin{align} \label{Rat-Ga-I-general-avc}
f\ge8, \quad
\AVC = \{ \alpha\gamma\delta, \alpha^3, \beta^n, \alpha^m\beta^n, \beta^n\gamma\delta \}.
\end{align}

Type II: By the same argument, we get
\begin{align*}
f=12, \quad
\alpha=  \tfrac{2}{3} \pi, \quad 
\beta=\tfrac{1}{3}\pi, \quad 
\gamma=\pi, \quad 
\delta=\tfrac{1}{3}\pi.
\end{align*}
By Lemma \ref{Ga=Pi}, $\gamma=\pi$ and $f=12$ imply
\begin{align*}
f=12, \quad
\AVC = \{ \alpha\gamma\delta, \alpha^3, \alpha^2\beta^2, \alpha\beta^4, \beta^2\gamma\delta, \beta^6 \}.
\end{align*}
That is all the vertices and it is a special case of $\AVC$ \eqref{avc-ga=pi}.  

Type III: By the same argument, we get
\begin{align*}
&f=12, &
&\alpha = \tfrac{8}{15}\pi, & &
\beta = \tfrac{1}{3}\pi, & &
\gamma = \tfrac{41}{30}\pi, & &
\delta= \tfrac{1}{10}\pi; \\
&f=12, &
&\alpha = \tfrac{3}{5}\pi, & &
\beta = \tfrac{1}{3}\pi, & &
\gamma = \tfrac{17}{15}\pi, & &
\delta= \tfrac{4}{15}\pi; \\
&f=20, &
&\alpha = \tfrac{8}{15}\pi, & &
\beta = \tfrac{1}{5}\pi, & &
\gamma = \tfrac{43}{30}\pi, & &
\delta= \tfrac{1}{30}\pi. 
\end{align*}
By no $\gamma^2\cdots, \delta^2\cdots$ and Parity Lemma, the first two sets of angle values give
\begin{align*}
f=12, \quad &\AVC = \{ \alpha\gamma\delta, \beta^6 \}.
\end{align*}
Similarly, the third set of angle values give
\begin{align*}
f=20, \quad &\AVC = \{ \alpha\gamma\delta, \alpha^3\beta^2, \beta^{10}\}. 
\end{align*}
\end{case*}

All of the above $\AVC$s contain at least one subset which constitutes a tiling. 

\newpart{$\AVC$s with tilings}

In $f=12, \AVC=\{ \alpha\gamma\delta, \beta\gamma\delta^3, \beta^6 \}$, if $\beta\gamma\delta^3$ is a vertex, then $\# \gamma < \# \delta$, a contradiction. So $\beta\gamma\delta^3$ is dismissed and we get
\begin{align*}
f = 12, \quad
\AVC=\{ \alpha\gamma\delta, \beta^6 \}.
\end{align*}
It is easy to see that the above $\AVC$ is a special case of the one below
\begin{align}\label{EMT-avc}
f\ge8, \quad
\AVC=\{ \alpha\gamma\delta, \beta^{\frac{f}{2}} \}.
\end{align}

In $f=20, \AVC = \{ \alpha\gamma\delta, \alpha^3\beta^2, \beta^{10}\}$, we get $\gamma\cdots = \delta\cdots = \alpha\gamma\delta$. Then by Counting Lemma, $\alpha^3\beta^2$ is not a vertex. The $\AVC$ is reduced to 
\begin{align*}
f=20,\quad \AVC = \{ \alpha\gamma\delta,  \beta^{10} \}, 
\end{align*}
which is also a special case of $\AVC$ \eqref{EMT-avc}. 

In $f=30, \AVC = \{ \alpha\gamma\delta, \beta^2\gamma^3\delta, \beta^6\gamma^2\delta^2, \alpha\beta^4\delta^2, \beta^{15} \}$, we have $\alpha\beta\cdots =\alpha\beta^4\delta^2$ and no $\alpha^2\cdots$. By no $\alpha^2\cdots$, we know that $\beta^{\alpha}\vert^{\alpha}\beta\cdots, \beta^{\alpha}\vert^{\alpha}\delta\cdots$ are not vertices and $\beta \cdots \beta$ has unique AAD $\beta^{\alpha} \vert \cdots \vert \beta^{\alpha}$. In $\gamma^{\beta} \vert^{\alpha} \beta$ and $\gamma^{\beta}\vert^{\alpha}\delta$, we get $T_1, T_2$ in both pictures of Figure \ref{AAD-gabe-gade}. By $\alpha_1\beta_2\cdots = \alpha\beta^4\delta^2$ and the unique AAD of $\beta^{\alpha} \vert \cdots \vert \beta^{\alpha}$ imply $\beta^{\alpha}\vert^{\alpha}\delta\cdots$, a contradiction. So $\gamma^{\beta} \vert^{\alpha} \beta \cdots, \gamma^{\beta}\vert^{\alpha}\delta \cdots$ are not vertices. Then inside $\bvert \, \gamma^{\beta} \vert \cdots \vert^{\alpha} \delta \, \bvert$, it cannot be empty nor filled by only $\beta$'s. This means that $\beta^2\gamma^3\delta$ is not a vertex. Counting Lemma on $\gamma, \delta$ implies that $\alpha\beta^4\delta^2$ is not a vertex. So $\alpha\cdots=\alpha\gamma\delta$. Counting Lemma on $\alpha,\gamma$ implies that $\beta^6\gamma^2\delta^2$ is not a vertex. The $\AVC$ is reduced to 
\begin{align*}
f=30, \quad
\AVC = \{ \alpha\gamma\delta,  \beta^{15} \},
\end{align*}
which is again a special case of $\AVC$ \eqref{EMT-avc}. 

\begin{figure}[htp]
\centering
\begin{tikzpicture}[>=latex,scale=1]

\begin{scope}

\foreach \a in {1,...,2}
\draw[rotate=90*\a]
	(0,0) -- (0,-1.2) -- (1.2,-1.2) -- (1.2,0);

\draw[]
	(-1.2,1.2) -- (-1.2,0) 
	(1.2,1.2) -- (1.2,0);

\draw[line width=2]
	(1.2,0) -- (1.2,1.2)
	(0,0) -- (-1.2,0);

\draw[]
	(0, 1.2) -- (-0.4, 2.0);

\draw[line width=2]
	(0, 1.2) -- (0.4, 2.0);

\node at (0.2,0.2) {\small $\beta$}; 
\node at (0.2,0.98) {\small $\alpha$};
\node at (0.98,0.95) {\small $\delta$}; 
\node at (0.98,0.2) {\small $\gamma$};

\node at (-0.2,0.2) {\small $\gamma$}; 
\node at (-0.2,0.90) {\small $\beta$};
\node at (-0.98,0.98) {\small $\alpha$};
\node at (-0.98,0.2) {\small $\delta$};

\node at (-0.5,1.75) {\small $\alpha$};
\node at (-0.2,2) {\small $\alpha$};
\node at (0.0, 1.6) {\small $\delta$};
\node at (0.3, 1.4) {\small $\delta$};
\node at (-0.4, 1.4) {\small $\beta...$};
\node at (-1.0, 1.4) {\small $\gamma$};


\node[inner sep=1,draw,shape=circle] at (0.58,0.55) {\small $1$};
\node[inner sep=1,draw,shape=circle] at (-0.60,0.6) {\small $2$};

\end{scope}

\begin{scope}[xshift=4.5cm]

\foreach \a in {1,...,2}
\draw[rotate=90*\a]
	(0,0) -- (0,-1.2) -- (1.2,-1.2) -- (1.2,0);

\draw[]
	(-1.2,1.2) -- (-1.2,0) 
	(1.2,1.2) -- (1.2,0);

\draw[line width=2]
	(0,0) -- (1.2,0)
	(0,0) -- (-1.2,0);

\draw[]
	(0, 1.2) -- (-0.4, 2.0);

\draw[line width=2]
	(0, 1.2) -- (0.4, 2.0);

\node at (0.2,0.2) {\small $\delta$}; 
\node at (0.2,0.98) {\small $\alpha$};
\node at (0.98,0.90) {\small $\beta$}; 
\node at (0.98,0.2) {\small $\gamma$};

\node at (-0.2,0.2) {\small $\gamma$}; 
\node at (-0.2,0.90) {\small $\beta$};
\node at (-0.98,0.98) {\small $\alpha$};
\node at (-0.98,0.2) {\small $\delta$};

\node at (-0.5,1.75) {\small $\alpha$};
\node at (-0.2,2) {\small $\alpha$};
\node at (0.0, 1.6) {\small $\delta$};
\node at (0.3, 1.4) {\small $\delta$};
\node at (-0.4, 1.4) {\small $\beta...$};
\node at (-1.0, 1.4) {\small $\gamma$};


\node[inner sep=1,draw,shape=circle] at (0.58,0.55) {\small $1$};
\node[inner sep=1,draw,shape=circle] at (-0.60,0.6) {\small $2$};

\end{scope}

\end{tikzpicture}
\caption{The AAD of $\gamma^{\beta} \vert^{\alpha} \beta$ and $\gamma^{\beta}\vert^{\alpha}\delta$}
\label{AAD-gabe-gade}
\end{figure}
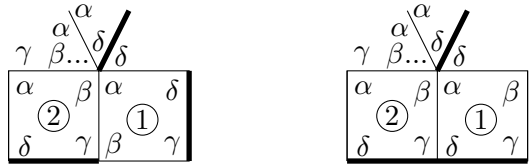 

In $\AVC = \{ \alpha\gamma\delta,    \gamma^3\delta,  \beta^n,   \alpha\beta^n,   \alpha\beta^n\delta^2,   \beta^n\gamma^2,   \beta^n\gamma\delta,  \beta^n\gamma^2\delta^2  \}$, we know $\alpha\gamma\cdots = \alpha\gamma\delta$, and $\alpha\beta\cdots =\alpha\beta^n, \alpha\beta^n\delta^2$, and $\gamma^3\cdots = \gamma^3\delta$, and no $\alpha\gamma^2\cdots$. The vertices with strictly more $\gamma$ than $\delta$ are $\gamma^3\delta, \beta^n\gamma^2$. The vertex with strictly more $\delta$ than $\gamma$ is $\alpha\beta^n\delta^2$.  If $\alpha\beta^n\delta^2$ is a vertex, the AAD determines $T_1, T_2,T_3$ in Figure \ref{AAD-bede2}. Then $\alpha_2\gamma_1\cdots=\alpha\gamma\delta$ and we determine $T_4$. Then $\alpha_4\beta_2\cdots = \alpha\beta^n, \alpha\beta^n\delta^2$. This means $\beta$ or $\delta$ is the angle in $T_5$ just outside $T_2$. By no $\alpha\gamma^2\cdots$, we conclude $\gamma_2\gamma_3\cdots =\gamma^3\cdots = \gamma^3\delta$. This means $\# \alpha\beta^n\delta^2 \le \# \gamma^3\delta$. In each vertex other than $\gamma^3\delta, \alpha\beta^n\delta^2, \beta^n\gamma^2$, the number of $\gamma$ is the same as that of $\delta$. By $\#\gamma=\# \delta$, we have $3\# \gamma^3\delta + 2\# \beta^n \gamma^2 =\# \gamma^3\delta + 2\# \alpha\beta^n\delta^2$. Combining with $\# \alpha\beta^n\delta^2 \le \# \gamma^3\delta$, we get $ \# \gamma^3\delta + \# \beta^n \gamma^2 \le \# \gamma^3\delta$. This implies $\# \beta^n \gamma^2 =0$ and $\# \alpha\beta^n\delta^2 = \# \gamma^3\delta$ and hence $\beta^n\gamma^2$ is not a vertex. The $\AVC$ is reduced to
\begin{align}\label{GenAVC-algade-ga3de}
f\ge 8, \quad
\AVC = \{ \alpha\gamma\delta,   \gamma^3\delta,  \beta^n,   \alpha\beta^n,   \alpha\beta^n\delta^2,  \beta^n\gamma\delta,  \beta^n\gamma^2\delta^2  \}.
\end{align}

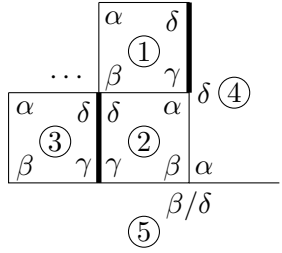
\begin{figure}[htp]
\centering
\begin{tikzpicture}[>=latex,scale=1]

\begin{scope}[xshift=0cm] 
\foreach \a in {0,...,3}
\draw[rotate=90*\a]
	(0,0) -- (0,-1.2); 

\foreach \a in {0,...,2}
\draw[rotate=-90*\a]
	(0,1.2) -- (1.2,1.2) -- (1.2,0); 

\draw[]
	(1.2, -1.2) -- (2.4, -1.2)
;

\draw[line width=2]
	(0,0) -- (0,-1.2)
	(1.2,0) -- (1.2,1.2)
;

\node at (-0.4,0.2) {\small $\cdots$}; 

\node at (0.2,0.2) {\small $\beta$}; 
\node at (0.2,1) {\small $\alpha$};
\node at (0.98,0.95) {\small $\delta$}; 
\node at (0.98,0.2) {\small $\gamma$};

\node at (0.2,-0.25) {\small $\delta$}; 
\node at (0.2,-1) {\small $\gamma$};
\node at (0.98,-0.2) {\small $\alpha$};
\node at (0.98,-1) {\small $\beta$};

\node at (-0.2,-0.25) {\small $\delta$}; 
\node at (-0.2,-1) {\small $\gamma$};
\node at (-0.98,-0.2) {\small $\alpha$};
\node at (-0.98,-1) {\small $\beta$};

\node at (1.4, 0) {\small $\delta$}; 
\node at (1.4, -1) {\small $\alpha$};

\node at (1.2,-1.5) {\small $\beta  / \delta$}; 

\node[inner sep=1,draw,shape=circle] at (0.6,0.55) {\small $1$};
\node[inner sep=1,draw,shape=circle] at (0.6,-0.65) {\small $2$};
\node[inner sep=1,draw,shape=circle] at (-0.58,-0.65) {\small $3$};
\node[inner sep=1,draw,shape=circle] at (1.8, 0) {\small $4$};
\node[inner sep=1,draw,shape=circle] at (0.6,-1.85) {\small $5$};

\end{scope}

\end{tikzpicture}
\caption{The AAD of $\alpha\beta^n\delta^2$}
\label{AAD-bede2}
\end{figure}

We remark that $f=18, \AVC = \{\alpha\gamma\delta, \alpha\beta^4, \beta\gamma^2\delta^2, \beta^5\gamma\delta,  \beta^{9} \}$ as a set may be viewed as a special case of $\AVC$ \eqref{GenAVC-algade-ga3de}. However, the angle values between the two are not compatible. So they are viewed as two different sets.   

If $\gamma=\pi$, then the quadrilateral degenerates into a triangle with $\AVC$ \eqref{avc-ga=pi} which is a special case of $\AVC$ \eqref{Rat-Ga-I-general-avc}.

We summarise the $\AVC$s in their most general forms below
\begin{enumerate}
\item $f\ge8, \AVC = \{ \alpha\gamma\delta, \beta^{\frac{f}{2}} \}$,
\item $f=18, \AVC = \{ \alpha\gamma\delta, \alpha\beta^4, \beta\gamma^2\delta^2, \beta^5\gamma\delta,  \beta^{9} \}$,
\item $f\ge8, \AVC = \{ \alpha\gamma\delta,   \gamma^3\delta,  \beta^n,   \alpha\beta^n,   \alpha\beta^n\delta^2,  \beta^n\gamma\delta,  \beta^n\gamma^2\delta^2  \}$, 
\item $f\ge8, \AVC = \{ \alpha\gamma\delta, \alpha^3, \alpha^m\beta^n, \beta^n, \beta^n\gamma\delta \}$.
\end{enumerate}

By \cite[Proposition 35, 48]{cly}, we get the earth map tilings, their flip modifications and rearrangement, which will be explained below. 

For the first $\AVC$ in the list, for each $f\ge8$ we get the earth map tiling $E$ with $\AVC$ \eqref{EMT-avc}. 

In fact, consecutive $\beta$'s in \cite[Figure 74]{cly} constitute consecutive timezones. Then $\beta^n$ is a vertex in any $\AVC$ in the list means that the tiling is $E$. In the remaining discussion, we may focus on the tilings without $\beta^n$.

The third $\AVC$ without $\beta^n$ is a simplified \cite[$\AVC$ (7.10)]{cly}. Counting Lemma implies that $\gamma^3\delta$ is a vertex if and only if $\alpha\beta^n\delta^2$ is a vertex. If $\gamma^3\delta$ is a vertex, then $\alpha\beta^n\delta^2$ and the angle values imply $f=6q+4$ where $q\in \mathbb{Z}$ and $q\ge1$. For each such $f\ge10$, we get the rearrangement $E^{\prime\prime\prime}$ with 
\begin{align}
E^{\prime\prime\prime}: \AVC \equiv \{ \alpha\gamma\delta, \gamma^3\delta, \alpha\beta^{\frac{f+2}{6}}, \alpha\beta^{\frac{f-4}{6}}\delta^2 \}.
\end{align}
If $\gamma^3\delta$ is not a vertex, then $\alpha\beta^n\delta^2$ is also not a vertex. We get \cite[$\AVC$ (7.9)]{cly} and for each $f\ge8$ we get flip modifications
\begin{align}
\label{Flip-1-avc} &E^{\prime}:  \AVC \equiv \{ \alpha\gamma\delta, \alpha\beta^n, \beta^n\gamma\delta \}, \\ 
\label{Flip-2-avc} &E^{\prime\prime}: \AVC \equiv \{ \alpha\gamma\delta, \alpha\beta^n, \beta^n\gamma^2\delta^2 \}.
\end{align}

The second $\AVC$ without $\beta^n$ is a special case of \cite[$\AVC$ (7.9)]{cly}. If $\alpha\beta^4$ is a vertex, then we get specific flip modifications $E^{\prime}, E^{\prime\prime}$ with $\AVC=\{ \alpha\gamma\delta, \alpha\beta^4, \beta^5\gamma\delta \}$ and $\AVC=\{ \alpha\gamma\delta, \alpha\beta^4, \beta\gamma^2\delta^2 \}$ respectively (which are special cases of $\AVC$s \eqref{Flip-1-avc}, \eqref{Flip-2-avc} respectively). 

The fourth $\AVC$ without $\beta^n$ is \cite[$\AVC$ (7.8)]{cly}. Counting Lemma implies that $\alpha^3$ or $\alpha^m\beta^n$ is a vertex if and only if $\beta^n\gamma\delta$ is a vertex.  If $\alpha^3$ or $\alpha^m\beta^n$ is a vertex, then we get
\begin{align} 
&E^{\prime}: \AVC \equiv \{ \alpha\gamma\delta, \alpha^3, \beta^n\gamma\delta \}, \\ 
&E^{\prime}: \AVC \equiv \{ \alpha\gamma\delta, \alpha^m\beta^n, \beta^n\gamma\delta \}. 
\end{align}

We list the tilings with their $\AVC$s in Table \ref{Rat-AlGaDe-AVCs}. The construction is explained in \cite[Figures 75, 76]{cly}.

\begin{table}[htp]
\begin{center}
\bgroup
\def\arraystretch{1.5}
    \begin{tabular}[]{| c | c | c |}
	\hline
	Tilings & $f$ & $\AVC$  \\ \hhline{|===|}
	$E$ & \multirow{5}{*}{$\ge8$} & $\{ \alpha\gamma\delta, \beta^{\frac{f}{2}} \}$ \\ 
	\cline{1-1} \cline{3-3}
	\multirow{2}{*}{$E^{\prime}$} &  & $\{ \alpha\gamma\delta, \alpha^m\beta^n, \beta^n\gamma\delta \}$ \\
	\cline{3-3}
	 &  & $\{ \alpha\gamma\delta, \alpha^3, \beta^n\gamma\delta \}$ \\
	\cline{1-1} \cline{3-3}
	$E^{\prime\prime}$ &  & $\{ \alpha\gamma\delta, \alpha\beta^n, \beta^n\gamma^2\delta^2 \}$ \\
	\cline{1-1} \cline{3-3}
	$E^{\prime\prime\prime}$ &  & $\{ \alpha\gamma\delta, \gamma^3\delta, \alpha\beta^{\frac{f+2}{6}}, \alpha^{\frac{f-4}{6}}\delta^2 \}$ \\
	\hline
	\end{tabular}
\egroup
\end{center}
\caption{$\AVC$s of rational angles with $\alpha\gamma\delta$}
\label{Rat-AlGaDe-AVCs}
\end{table}

This completes the proof.
\end{proof}

\section{Non-rational Angles} \label{SecNonRat}

In this section, we assume that at least one of the angles $\alpha, \beta,\gamma, \delta$ is non-rational, i.e., its value is not a rational multiple of $\pi$. For integers $m, n, k, l, m_i$, $n_i, k_i, l_i \ge 0$ where $1 \le i \le 2$, the angle sum system of vertices $\alpha^{m_1}\beta^{n_1}\gamma^{k_1}\delta^{l_1}$, $\alpha^{m_2}\beta^{n_2}\gamma^{k_2}\delta^{l_2}, \alpha^m\beta^n\gamma^k\delta^l$ has augmented matrix
\begin{align}\label{AugMat}
[A \vert \vec{b}] =
\left[\begin{array}{cccc|c}  1 & 1 & 1 & 1 & 2+\tfrac{4}{f}\\  
m_1 & n_1 & k_1 & l_1 & 2 \\
m_2 & n_2 & k_2 & l_2 & 2 \\
m & n & k & l & 2
\end{array}\right].
\end{align}
The above system is required to be consistent, namely rank of $[A \vert \vec{b}] =$ rank of $A$. If $A$ is invertible, then the solutions to angle values are rational. Therefore, for some angles being non-rational, we have rank of $A \le 3$. So $\det A = 0$. For $1 \le i \le 2$, when $m_i, n_i, k_i, l_i$ are fixed, we get two equations in terms of $m,n,k,l$ and $f$. We call them the non-rationality condition.

Note that we allow some of $m,n,k,l$ to be $0$ in the discussion involving \eqref{AugMat} and in determining the angle combinations. Only after the angle combinations are determined, we require $m,n,k,l \ge1$ in angle combinations. 

\begin{prop} If $\alpha\gamma\delta, \beta\gamma\delta$ are not vertices and $f\ge8$, then tilings by congruent almost equilateral quadrilaterals with non-rational angles are isolated earth map tilings $S1, S2$, and special tilings $QP_6,S4$.
\end{prop}

\begin{proof} Using each pair of vertices the list of Lemma \ref{PairsLem}, we set up $A$ in \eqref{AugMat} and determine $m,n,k,l$. We demonstrate the argument of solving the related system of linear Diophantine equations and inequalities in two cases. The other cases are determined by the same routine, which can be swiftly implemented in computer.   
 
\begin{case*}[Degree $3$ Pairs] 
Suppose $\alpha\delta^2, \beta\gamma^2$ are vertices. By $f>0$, we do row operations
\begin{align*}
[A \vert \vec{b}]  =
\left[\begin{array}{cccc|c}  
1 & 1 & 1 & 1 & 2+\tfrac{4}{f}\\  
1 & 0 & 0 & 2 & 2 \\
0 & 1 & 2 & 0 & 2 \\
m & n & k & l & 2
\end{array}\right] 
 \rightarrow
\left[\begin{array}{cccc|c}  
1 & 0 & 0 & 2 & 2 \\  
0 & 1 & 2 & 0 & 2 \\
0 & 0 & 1 & 1 & 2(1-\tfrac{2}{f}) \\
0 & 0 & 0 & \lambda & \mu
\end{array}\right],
\end{align*}
where $\lambda = 2m-2n+k-l $ and $\mu = 2(m+k-n-1)+\tfrac{4}{f}(2n-k)$.

The non-rationality condition implies $\lambda=\mu=0$, i.e., $2m - 2n + k  - l  = 0$ and $(n + 1 - m - k)f = 2(2n-k)$. By $f \ge 8$, the latter implies $2n-k\neq0$ if and only if $n + 1 - m - k \neq0$. In this case, we have $f=8 + \frac{2(4m - 2n + 3k-4)}{n + 1 - m - k} \ge8$. So there are three possibilities,
\begin{enumerate}[(1)]
\item $2n-k=0$, $n+1-m-k=0$, $2m-2n+k-l=0$;
\item $2n-k>0$, $n+1-m-k>0$, $4m-2n+3k-4 \ge 0$, $2m-2n + k - l=0$;
\item $2n-k<0$, $n+1-m-k<0$, $4m-2n+3k-4 \le 0$, $2m -2n + k - l=0$.
\end{enumerate}

The non-negative integer solutions to the first possibility are $(m,n,k,l)=(1,0,0,2), (0,1,2,0)$. The vertices are $\alpha\delta^2, \beta\gamma^2$. 

The non-negative integer solution to the second is $(m,n,k,l)=(m,m,0,0)$. The vertex is $\alpha^m\beta^m$. 

There is no non-negative integer solution to the third and hence no vertex. 

Therefore we get the set of admissible vertices
\begin{align*}
\AVC = \{ \alpha\delta^2, \beta\gamma^2, \alpha^m\beta^m \}.
\end{align*}

The arguments for the other pairs are analogous. The $\AVC$s are listed in the first half of Table \ref{NonRatAVCs} where the first two vertices in each row are one of the pairs from the list of Lemma \ref{PairsLem}.
\end{case*}

\begin{case*}[Degree $3,4$ Pairs] In this case, one of $\alpha^3, \alpha\beta^2, \alpha\gamma^2, \alpha\delta^2$ is the unique degree $3$ vertex. If $\alpha^3$ is a vertex, then Lemma \ref{a3Lem} implies $f\ge24$. If one of $\alpha\beta^2, \alpha\gamma^2, \alpha\delta^2$ is a vertex, then Lemma \ref{ab2Lem} implies $f\ge16$. 

Suppose $\alpha\beta^2, \gamma^2\delta^2$ are vertices. We have $f\ge16$. The non-rational condition implies $k-l=0$ and $(2- n - k)f = 4(2m - n)$. The latter and $f\ge16$ imply $2m-n\neq0$ if and only if $2- n - k\neq0$. In this case, we have $f = 16 + \frac{4(2m+3n+4k-8)}{2-n-k} \ge 16$. There are three possibilities,
\begin{enumerate}[(1)]
\item $2m-n=0$, $2-n-k=0$, $k-l=0$;
\item $2m > n$, $2-n-k>0$, $2m+3n+4k-8 \ge 0$, $k-l=0$;
\item $2m < n$, $2-n-k<0$, $2m+3n+4k-8 \le 0$, $k-l=0$.
\end{enumerate}

For each possibility, we obtain the vertices by solving the system of linear Diophantine equations and inequalities for the non-negative integers $m,n,k,l$. Therefore,
\begin{align*}
\AVC = \{ \alpha\beta^2, \gamma^2\delta^2, \alpha^{m}, \alpha^m\beta, \alpha^m\gamma\delta \}.
\end{align*}

The arguments for the other pairs are analogous. The $\AVC$s are listed in the second half of Table \ref{NonRatAVCs}. The first two vertices in each row are assumed to appear since they come from the list of Lemma \ref{PairsLem}. In fact, by Counting Lemma, all vertices in these $\AVC$s must appear, with the exceptions of $f=24, \AVC=\{ \alpha^3, \gamma^2\delta^2, \beta^4, \beta^2\gamma\delta \}$ and $f\ge16, \AVC=\{ \alpha\beta^2, \gamma^2\delta^2, \alpha^m, \alpha^m\beta, \alpha^m\gamma\delta \}$.

We remark that, by exchanging $\alpha \leftrightarrow \beta$ and $\gamma \leftrightarrow \delta$, the $\AVC$s $\{ \alpha^3, \beta\gamma^2, \alpha\beta\delta^2 \}$, $\{ \alpha^3, \beta\delta^2, \alpha\beta\gamma^2 \}$ are equivalent to $\{ \alpha\delta^2, \beta^3, \alpha\beta\gamma^2 \}, \{ \alpha\gamma^2, \beta^3, \alpha\beta\delta^2 \}$, which are special cases of $\{ \alpha\delta^2,  \alpha\beta\gamma^2, \beta^n \}, \{ \alpha\gamma^2,  \alpha\beta\delta^2, \beta^n \}$ respectively.

\begin{table}[h]
\begin{center}
\begin{minipage}[t]{.4\linewidth}
\bgroup
\def\arraystretch{1.5}
     \begin{tabular}[t]{ | c | c | c |}
	\hline
	 $f$ & $\AVC$  \\ \hhline{|==|}
	$12$  & $\{ \alpha^3,  \alpha\gamma^2, \beta^2\delta^2 \}$  \\
          \hline  
	$12$  & $\{ \alpha^3, \alpha\delta^2, \beta^2\gamma^2  \}$  \\
          \hline  
	$24$  & $\{ \alpha^3, \beta\gamma^2, \beta^2\delta^4 \}$  \\
          \hline  
	$24$  & $\{ \alpha^3,  \beta\delta^2, \beta^2\gamma^4 \}$  \\
          \hline   
	$\ge 8$ & $\{ \alpha^2\beta, \beta\delta^2, \gamma^k \}$  \\
          \hline  
	$\ge 8$  & $\{ \alpha\delta^2,  \beta\gamma^2, \alpha^m\beta^m \}$ \\ 
          \hline  
	$24$  & $\{ \alpha^3, \gamma^4, \beta^2\delta^2 \}$  \\
          \hline  
	$24$  & $\{ \alpha^3, \delta^4 , \beta^2\gamma^2 \}$  \\
          \hline  
	\end{tabular}
\egroup
 \end{minipage} 
\begin{minipage}[t]{.48\linewidth}
\bgroup
\def\arraystretch{1.5}
    \begin{tabular}[t]{ | c | c | }
	\hline
	$f$ & $\AVC$  \\ \hhline{|==|}
	$24$  & $\{ \alpha^3, \gamma^2\delta^2, \beta^4, \beta^2\gamma\delta \}$  \\
          \hline  
	$36$  & $\{ \alpha^3, \gamma^2\delta^2, \alpha\beta^3  \}$  \\
          \hline  
	$60$  & $\{ \alpha^3, \gamma^2\delta^2, \beta^5 \}$  \\
          \hline  
	$\ge 16$ & $\{ \alpha\beta^2, \gamma^2\delta^2, \alpha^m, \alpha^m\beta, \alpha^m\gamma\delta \}$  \\
          \hline  
	$\ge16$  & $\{ \alpha\gamma^2, \beta^2\delta^2, \alpha^m \}$  \\
          \hline  
	$\ge16$  & $\{ \alpha\delta^2, \beta^2\gamma^2, \alpha^m \}$  \\
          \hline  
	$\ge16$  & $\{  \alpha\gamma^2,  \alpha\beta\delta^2, \beta^n \}$  \\
          \hline  
	$\ge16$  & $\{  \alpha\delta^2,  \alpha\beta\gamma^2, \beta^n \}$  \\
          \hline  
	\end{tabular}
\egroup
 \end{minipage} 
\end{center}
\caption{Non-rational angles: $\AVC$s without $\alpha\gamma\delta$}
\label{NonRatAVCs}
\end{table}
\end{case*}

\newpart{$\AVC$s without tiling}

We first discuss the $\AVC$s from Table \ref{NonRatAVCs} that do not constitute tilings.

As explained in the AAD discussion in Section \ref{SubsecConcept}, $\alpha^3$ being a vertex implies that $\beta\delta\cdots$ is a vertex. Therefore $\AVC=\{  \alpha^3, \alpha\delta^2, \beta^2\gamma^2  \}, \{ \alpha^3,  \alpha\beta^3, \gamma^2\delta^2  \}$,  $\{ \alpha^3, \beta^2\gamma^2, \delta^4 \}, \{ \alpha^3,  \gamma^2\delta^2, \beta^5 \}$ have no tilings.

In $\AVC=\{ \alpha^3, \alpha\gamma^2, \beta^2\delta^2 \}$, all three vertices appear and the angle sum system implies $\alpha=\gamma$ and $\beta+\delta=\pi$ whereby $\delta \neq \pi$, contradicting Lemma \ref{TriQuadLem}. 

In $\AVC=\{ \alpha^3, \beta\gamma^2, \beta^2\delta^4 \}$, we know that $\beta^2\delta^4$ is a vertex. By no $\alpha\gamma\cdots$, $\gamma\vert\gamma\cdots$, we know that $\beta^{\alpha}\vert^{\gamma}\beta\cdots$, $\beta^{\gamma}\vert^{\gamma}\beta\cdots$ are not vertices. Then $\beta \vert \beta = \beta^{\alpha} \vert^{\alpha} \beta$. By no $\alpha\gamma\cdots$, we know that $\bvert \, \delta^{\alpha} \vert \beta \vert^{\alpha} \delta \, \bvert \cdots$ is not a vertex. By $\beta \vert \beta = \beta^{\alpha} \vert^{\alpha} \beta$, the vertex $\beta^2\delta^4$ has AAD $\bvert \, \delta^{\alpha} \vert^{\gamma} \beta^{\alpha} \vert^{\alpha} \beta^{\gamma}  \vert^{\alpha} \delta \, \bvert$. It implies $\alpha\gamma\cdots$, a contradiction.

In $\AVC=\{ \alpha^3, \beta\delta^2, \beta^2\gamma^4 \}$, the AAD of $\beta\delta^2$ is $\bvert \, \delta^{\alpha} \vert^{\alpha} \beta^{\gamma} \vert^{\alpha} \delta \, \bvert$. It implies $\alpha\gamma\cdots$, a contradiction.

In $\AVC = \{ \alpha^3, \beta^4, \gamma^2\delta^2, \beta^2\gamma\delta \}$, we know $\alpha\cdots= \alpha^3$ is a vertex. Then the AAD of $\alpha^3$ implies that $\beta\delta\cdots=\beta^2\gamma\delta$ is a vertex. By no $\alpha\beta\cdots, \alpha\gamma\cdots$, the AAD of $\beta^2\gamma\delta$ is $\bvert \, \gamma^{\beta} \vert^{\gamma} \beta^{\alpha} \vert \beta \vert^{\alpha} \delta \, \bvert$. It implies $\alpha\gamma\cdots$, a contradiction.

In $\AVC=\{  \alpha^2\beta, \beta\delta^2, \gamma^k \}$, the vertex $\beta\delta^2$ has an AAD $\bvert \, \delta^{\alpha} \vert^{\alpha} \beta^{\gamma} \vert^{\alpha} \delta \, \bvert$. This implies $\alpha\gamma\cdots$, a contradiction.

In $\AVC=\{ \alpha\delta^2, \beta^2\gamma^2, \alpha^m \}$, the vertex $\alpha\delta^2$ has an AAD $\bvert \, \delta^{\alpha} \vert^{\beta} \alpha^{\delta} \vert^{\alpha} \delta \, \bvert$. This implies $\alpha\beta\cdots$, a contradiction. 

In $\AVC=\{ \alpha\delta^2, \beta\gamma^2, \alpha^m\beta^m \}$, we have $\gamma\cdots=\beta\gamma^2$ and $\delta\cdots=\alpha\delta^2$. Meanwhile, $\alpha^{m}\beta^{m}$ has degree $\ge 4$, i.e, $m\ge2$. So $\alpha, \beta<\pi$. By $\alpha\delta^2, \beta\gamma^2$, we have $\gamma,\delta<\pi$. Then the tile is convex. The vertex angle sums $\alpha + 2\delta = 2\pi = \beta + 2\gamma$ imply $\gamma - \frac{1}{2}\alpha=\delta - \frac{1}{2}\beta$. Then \eqref{Coolsaet-Id} implies $\sin( \gamma - \frac{1}{2}\alpha )= 0$ or $\sin \frac{1}{2}\beta = \sin \frac{1}{2}\alpha$. By convexity and $\gamma - \frac{1}{2}\alpha =\delta - \frac{1}{2}\beta$, the former gives $\delta - \tfrac{1}{2}\beta = \gamma - \tfrac{1}{2}\alpha = 0$. Then $\alpha=2\gamma$ and $\beta=2\delta$. By $\alpha\gamma^2$ and $\beta\delta^2$, we get $4\pi = \alpha + 2\gamma + \beta+ 2\delta = 4(\gamma + \delta)$. This implies $\gamma + \delta = \pi$ and $\alpha + \beta = 2\pi$, contradicting $\alpha,\beta < \pi$. So $\sin \frac{1}{2}\beta = \sin \frac{1}{2}\alpha$. For $0<\alpha,\beta<\pi$, we have $\tfrac{1}{2}\beta = \tfrac{1}{2}\alpha$ or $\pi - \tfrac{1}{2}\alpha$. Since $\alpha \neq \beta$, we get $\frac{1}{2}\beta = \pi - \frac{1}{2}\alpha$ which also implies $\alpha + \beta = 2\pi$, the same contradiction.  

\newpart{$\AVC$s with tilings}

In $\AVC =\{ \alpha\gamma^2,  \alpha\beta\delta^2, \beta^n \}$, we shall see in Proposition \ref{Geom-AlGa2-AlBeDe2} that tilings are only geometrically realisable when $f=16$. In that case, the tilings are $S3, S^{\prime}3$ and every angle is rational, a contradiction. 

In $\AVC=\{ \alpha^3, \beta^2\delta^2, \gamma^4 \}$, the tiling is $QP_6$, given by simple quadrilateral subdivision of the cube in Figure \ref{Tiling-QP6}.

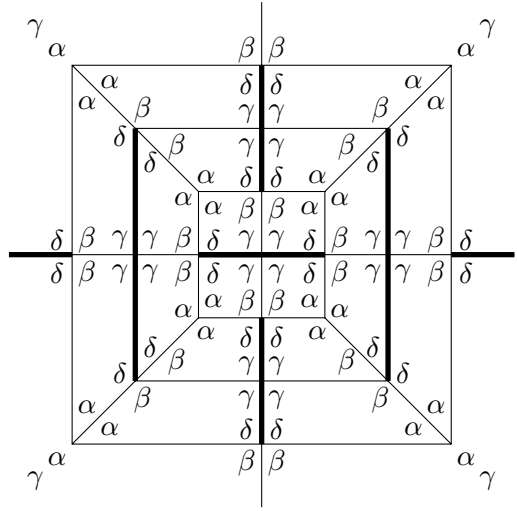
\begin{figure}[htp]
\centering
\begin{tikzpicture}[>=latex, scale=0.7]
\foreach \a in {0,...,3}
	\draw[rotate=90*\a]
	(0,0) -- (4.8,0)
	(3.6,0) -- (3.6,3.6) -- (0,3.6) 
	(1.2,0) -- (1.2,1.2) -- (0,1.2)
	(2.4,0) -- (2.4,2.4) -- (0,2.4)
	(1.2,1.2) -- (3.6,3.6);

\foreach \c in {1,3}
\draw[line width=2, rotate=90*\c]
	(0,1.2) -- (0,-1.2)
	(2.4,2.4) -- (-2.4,2.4)
	(1.2,0) -- (3.6,0)
	(0,3.6) -- (0,4.8);
			
			
\node[shift={(0.2,0.2)}] at (0,0) {\small $\gamma$};
\node[shift={(-0.2,0.2)}] at (0,0) {\small $\gamma$};
\node[shift={(-0.2,-0.25)}] at (0,0) {\small $\gamma$};
\node[shift={(0.2,-0.25)}] at (0,0) {\small $\gamma$};

			
			
\node[shift={(0.2,0.2)}] at (1.2,0) {\small $\beta$};
\node[shift={(-0.2,0.2)}] at (1.2,0) {\small $\delta$};
\node[shift={(-0.2,-0.25)}] at (1.2,0) {\small $\delta$};
\node[shift={(0.2,-0.25)}] at (1.2,0) {\small $\beta$};
			
\node[shift={(0.2,-0.1)}] at (1.2,1.2) {\small $\alpha$};
\node[shift={(-0.1,0.2)}] at (1.2,1.2) {\small $\alpha$};
\node[shift={(-0.2,-0.22)}] at (1.2,1.2) {\small $\alpha$};
			
\node[shift={(0.2,0.2)}] at (0,1.2) {\small $\delta$};
\node[shift={(-0.2,0.2)}] at (0,1.2) {\small $\delta$};
\node[shift={(-0.2,-0.25)}] at (0,1.2) {\small $\beta$};
\node[shift={(0.2,-0.25)}] at (0,1.2) {\small $\beta$};
			
\node[shift={(0.2,-0.22)}] at (-1.2,1.2) {\small $\alpha$};
\node[shift={(0.1,0.2)}] at (-1.2,1.2) {\small $\alpha$};
\node[shift={(-0.2,-0.1)}] at (-1.2,1.2) {\small $\alpha$};
			
\node[shift={(0.2,0.2)}] at (-1.2,0) {\small $\delta$};
\node[shift={(-0.2,0.2)}] at (-1.2,0) {\small $\beta$};
\node[shift={(-0.2,-0.25)}] at (-1.2,0) {\small $\beta$};
\node[shift={(0.2,-0.25)}] at (-1.2,0) {\small $\delta$};
			
\node[shift={(0.2,0.22)}] at (-1.2,-1.2) {\small $\alpha$};
\node[shift={(0.1,-0.2)}] at (-1.2,-1.2) {\small $\alpha$};
\node[shift={(-0.2,0.1)}] at (-1.2,-1.2) {\small $\alpha$};
			
\node[shift={(0.2,0.2)}] at (0,-1.2) {\small $\beta$};
\node[shift={(-0.2,0.2)}] at (0,-1.2) {\small $\beta$};
\node[shift={(-0.2,-0.25)}] at (0,-1.2) {\small $\delta$};
\node[shift={(0.2,-0.25)}] at (0,-1.2) {\small $\delta$};
			
\node[shift={(-0.2,0.22)}] at (1.2,-1.2) {\small $\alpha$};
\node[shift={(-0.1,-0.2)}] at (1.2,-1.2) {\small $\alpha$};
\node[shift={(0.2,0.1)}] at (1.2,-1.2) {\small $\alpha$};

			
\node[shift={(0.2,0.2)}] at (2.4,0) {\small $\gamma$};
\node[shift={(-0.2,0.2)}] at (2.4,0) {\small $\gamma$};
\node[shift={(-0.2,-0.25)}] at (2.4,0) {\small $\gamma$};
\node[shift={(0.2,-0.25)}] at (2.4,0) {\small $\gamma$};

\node[shift={(0.2,-0.1)}] at (2.4,2.4) {\small $\delta$};
\node[shift={(-0.1,0.25)}] at (2.4,2.4) {\small $\beta$};
\node[shift={(-0.55,-0.25)}] at (2.4,2.4) {\small $\beta$};	
\node[shift={(-0.2,-0.45)}] at (2.4,2.4) {\small $\delta$};	
			
\node[shift={(0.2,0.2)}] at (0,2.4) {\small $\gamma$};
\node[shift={(-0.2,0.2)}] at (0,2.4) {\small $\gamma$};
\node[shift={(-0.2,-0.25)}] at (0,2.4) {\small $\gamma$};
\node[shift={(0.2,-0.25)}] at (0,2.4) {\small $\gamma$};
			
\node[shift={(0.55,-0.25)}] at (-2.4,2.4) {\small $\beta$};	
\node[shift={(0.1,0.25)}] at (-2.4,2.4) {\small $\beta$};
\node[shift={(0.2,-0.45)}] at (-2.4,2.4) {\small $\delta$};
\node[shift={(-0.2,-0.1)}] at (-2.4,2.4) {\small $\delta$};
			
\node[shift={(0.2,0.2)}] at (-2.4,0) {\small $\gamma$};
\node[shift={(-0.2,0.2)}] at (-2.4,0) {\small $\gamma$};
\node[shift={(-0.2,-0.25)}] at (-2.4,0) {\small $\gamma$};
\node[shift={(0.2,-0.25)}] at (-2.4,0) {\small $\gamma$};
			
\node[shift={(0.55,0.25)}] at (-2.4,-2.4) {\small $\beta$};
\node[shift={(0.2,0.45)}] at (-2.4,-2.4) {\small $\delta$};
\node[shift={(-0.2,0.1)}] at (-2.4,-2.4) {\small $\delta$};	
\node[shift={(0.1,-0.25)}] at (-2.4,-2.4) {\small $\beta$};
			
\node[shift={(0.2,0.2)}] at (0,-2.4) {\small $\gamma$};
\node[shift={(-0.2,0.2)}] at (0,-2.4) {\small $\gamma$};
\node[shift={(-0.2,-0.25)}] at (0,-2.4) {\small $\gamma$};
\node[shift={(0.2,-0.25)}] at (0,-2.4) {\small $\gamma$};
			
\node[shift={(0.2,0.1)}] at (2.4,-2.4) {\small $\delta$};	
\node[shift={(-0.2,0.45)}] at (2.4,-2.4) {\small $\delta$};
\node[shift={(-0.55,0.25)}] at (2.4,-2.4) {\small $\beta$};
\node[shift={(-0.1,-0.25)}] at (2.4,-2.4) {\small $\beta$};
			
			
\node[shift={(0.2,0.2)}] at (3.6,0) {\small $\delta$};
\node[shift={(-0.2,0.2)}] at (3.6,0) {\small $\beta$};
\node[shift={(-0.2,-0.25)}] at (3.6,0) {\small $\beta$};
\node[shift={(0.2,-0.25)}] at (3.6,0) {\small $\delta$};
			
\node[shift={(0.2,0.2)}] at (3.6,3.6) {\small $\alpha$};
\node[shift={(-0.5,-0.22)}] at (3.6,3.6) {\small $\alpha$};
\node[shift={(-0.2,-0.5)}] at (3.6,3.6) {\small $\alpha$};
			
\node[shift={(0.2,0.2)}] at (0,3.6) {\small $\beta$};
\node[shift={(-0.2,0.2)}] at (0,3.6) {\small $\beta$};
\node[shift={(-0.2,-0.25)}] at (0,3.6) {\small $\delta$};
\node[shift={(0.2,-0.25)}] at (0,3.6) {\small $\delta$};
			
\node[shift={(-0.2,0.2)}] at (-3.6,3.6) {\small $\alpha$};
\node[shift={(0.5,-0.22)}] at (-3.6,3.6) {\small $\alpha$};
\node[shift={(0.2,-0.5)}] at (-3.6,3.6) {\small $\alpha$};
			
\node[shift={(0.2,0.2)}] at (-3.6,0) {\small $\beta$};
\node[shift={(-0.2,0.2)}] at (-3.6,0) {\small $\delta$};
\node[shift={(-0.2,-0.25)}] at (-3.6,0) {\small $\delta$};
\node[shift={(0.2,-0.25)}] at (-3.6,0) {\small $\beta$};
			
\node[shift={(-0.2,-0.2)}] at (-3.6,-3.6) {\small $\alpha$};
\node[shift={(0.5,0.2)}] at (-3.6,-3.6) {\small $\alpha$};
\node[shift={(0.2,0.5)}] at (-3.6,-3.6) {\small $\alpha$};
			
\node[shift={(0.2,0.2)}] at (0,-3.6) {\small $\delta$};
\node[shift={(-0.2,0.2)}] at (0,-3.6) {\small $\delta$};
\node[shift={(-0.2,-0.25)}] at (0,-3.6) {\small $\beta$};
\node[shift={(0.2,-0.25)}] at (0,-3.6) {\small $\beta$};
			
\node[shift={(0.2,-0.2)}] at (3.6,-3.6) {\small $\alpha$};
\node[shift={(-0.5,0.2)}] at (3.6,-3.6) {\small $\alpha$};
\node[shift={(-0.2,0.5)}] at (3.6,-3.6) {\small $\alpha$};
			
			
\node[shift={(45:4.25)}] at (0,0) {\small $\gamma$};
\node[shift={(135:4.25)}] at (0,0) {\small $\gamma$};
\node[shift={(-45:4.25)}] at (0,0) {\small $\gamma$};
\node[shift={(-135:4.25)}] at (0,0) {\small $\gamma$};

\end{tikzpicture}
\caption{Tiling $QP_6$}
\label{Tiling-QP6}
\end{figure}

In $\AVC = \{ \alpha\beta^2, \gamma^2\delta^2, \alpha^m, \alpha^m\beta, \alpha^m\gamma\delta \}$, we have $\beta^2\cdots = \alpha\beta^2$ and no $\beta\gamma\cdots$, $\beta\delta\cdots$. Since the $\AVC$ assumes $\alpha\beta^2$ as the unique degree $3$ vertex, the vertices $\alpha^m, \alpha^m\beta,\alpha^m\gamma\delta$ have degree $\ge4$. The argument in the second picture of Figure \ref{AAD-alal-bebe-gaga-bebe} implies no $\alpha^{\beta}\vert^{\beta}\alpha\cdots$. By no $\beta\delta\cdots, \alpha^{\beta}\vert^{\beta}\alpha\cdots$, the AAD of $\alpha\vert\alpha$ is $\alpha^{\delta}\vert^{\delta}\alpha$. This means no $\alpha\alpha\alpha$. So $m=2$ in $\alpha^m\gamma\delta$ and $\alpha^m, \alpha^m\beta$ are not vertices. The $\AVC$ is reduced to 
 \begin{align}\label{AVC-ab2-a2gd-g2d2}
f=16, \quad \AVC = \{ \alpha\beta^2, \alpha^2\gamma\delta, \gamma^2\delta^2  \}.
\end{align}
By $\AVC$ \eqref{AVC-ab2-a2gd-g2d2}, we construct $S4$ in the second picture of Figure \ref{Tiling-S2S4}. As $2\alpha=\gamma+\delta = \pi$, $S4$ is a subdivision of a non-edge-to-edge parallelogram tiling in Figure \ref{SubdivParallelogram}. The right angles are $\alpha$. The non-indicated parallelogram angles are $\beta=\frac{3}{4}\pi$.

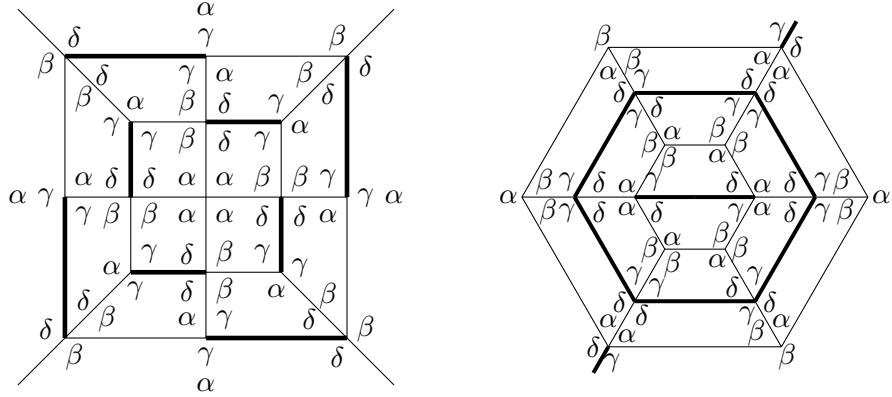
\begin{figure}[htp]
\centering
\begin{tikzpicture}

\begin{scope}[scale=1.25] 

\foreach \a in {0,1,2,3}
{
\begin{scope}[rotate=90*\a]

\draw
	(0,0) -- (1.5,0) -- (1.5,-1.5)
	(0.8,0) -- (0.8,0.8) -- (2,2);

\draw[line width=1.75]
	(0,0.8) -- (0.8,0.8)
	(1.5,0) -- (1.5,1.5);
	
\node at (0.2,0.2) {\small $\alpha$};
\node at (0.6,0.2) {\small $\beta$};
\node at (0.2,0.6) {\small $\delta$};
\node at (0.6,0.6) {\small $\gamma$};

\node at (1,0.75) {\small $\alpha$};
\node at (1,0.2) {\small $\beta$};
\node at (1.3,1.1) {\small $\delta$};
\node at (1.3,0.2) {\small $\gamma$};

\node at (1,-0.75) {\small $\gamma$};
\node at (1,-0.2) {\small $\delta$};
\node at (1.3,-1.05) {\small $\beta$};
\node at (1.3,-0.2) {\small $\alpha$};

\node at (2,0) {\small $\alpha$};
\node at (1.7,-1.4) {\small $\beta$};
\node at (1.7,1.4) {\small $\delta$};
\node at (1.7,0) {\small $\gamma$};

\end{scope}
}

\end{scope}

\begin{scope}[xshift=6.5cm]

\foreach \a in {0,...,5}
{
\begin{scope}[rotate=60*\a]

\draw 
	(0:0.8) -- (60:0.8)
	(0:2.3) -- (60:2.3)
	(0.8,0) -- (2.3,0);

\draw[line width=1.5]
	(0:1.6) -- (60:1.6);

\end{scope}
}

\foreach \a in {1,-1}
{
\begin{scope}[scale=\a]

\draw[line width=1.5]
	(0,0) -- (0.8,0)
	(60:2.3) -- (60:2.7);
	
\node at (0.3,0.55) {\small $\alpha$};
\node at (-0.35,0.5) {\small $\beta$};
\node at (0.5,0.2) {\small $\delta$};
\node at (-0.55,0.2) {\small $\gamma$};

\node at (-0.3,0.85) {\small $\alpha$};
\node at (0.3,0.9) {\small $\beta$};
\node at (-0.5,1.2) {\small $\delta$};
\node at (0.5,1.2) {\small $\gamma$};

\node at (0.9,0.15) {\small $\alpha$};
\node at (0.6,0.7) {\small $\beta$};
\node at (1.27,0.2) {\small $\delta$};
\node at (0.8,1.05) {\small $\gamma$};

\node at (0.9,-0.15) {\small $\alpha$};
\node at (0.6,-0.7) {\small $\beta$};
\node at (1.27,-0.2) {\small $\delta$};
\node at (0.8,-1.05) {\small $\gamma$};

\node at (0.9,1.85) {\small $\alpha$};
\node at (-0.82,1.8) {\small $\beta$};
\node at (0.7,1.57) {\small $\delta$};
\node at (-0.7,1.57) {\small $\gamma$};

\node at (1.15,1.65) {\small $\alpha$};
\node at (1.95,0.2) {\small $\beta$};
\node at (1,1.4) {\small $\delta$};
\node at (1.7,0.2) {\small $\gamma$};

\node at (1.15,-1.65) {\small $\alpha$};
\node at (2,-0.2) {\small $\beta$};
\node at (1,-1.4) {\small $\delta$};
\node at (1.7,-0.2) {\small $\gamma$};

\node at (2.48,0) {\small $\alpha$};
\node at (-60:2.45) {\small $\beta$};
\node at (1.35,2) {\small $\delta$};
\node at (1.1,2.2) {\small $\gamma$};

\end{scope}
}

\end{scope}

\end{tikzpicture}
\caption{Tilings $S2$ and $S4$}
\label{Tiling-S2S4}
\end{figure}

\begin{figure}[htp]
\centering
\begin{tikzpicture}[scale=1]

\draw[gray!50]
(-0.3,1.4) -- (-0.3,-0.5)
(1.1,-0.8) -- (0.5,-0.2) -- (-1.1,-0.8)
(0.5,1.1) -- (0.5,-0.2)
(-0.5,-1.4) -- ++(0.16,0.06) -- ++(-0.12,0.12) -- ++(-0.16,-0.06)
(0.5,-0.2) -- ++(-0.16,-0.06) -- ++(0.12,-0.12) -- ++(0.16,0.06);

\draw[gray!50, line width=1.2]
(0.3,-1.1) -- (-0.3,-0.5)
(1.2,0) -- (0.5,0.45) -- (-0.3,0.45) -- (-1.2,0);

\draw
(-1.3,0.8) -- (-0.3,1.4) -- (1.3,0.8) -- (0.3,0.2) -- cycle
(-1.1,-0.8) -- (-0.5,-1.4) -- (1.1,-0.8)
(0.3,0.2) -- (0.3,-1.1)
(-0.5,0.5) -- (-0.5,-1.4)
(-1.3,0.8) -- (-1.1,-0.8)
(1.3,0.8) -- (1.1,-0.8)
(0.3,0.2) -- ++(0.15,0.09) -- ++(-0.16,0.06) -- ++(-0.15,-0.09)
(-0.3,1.4) -- ++(-0.15,-0.09) -- ++(0.16,-0.06) -- ++(0.15,0.09)
(1.3,0.8) -- ++(-0.15,-0.09) -- ++(-0.03,-0.24) -- ++(0.15,0.09)
(-0.5,0.3) -- ++(0.16,-0.06) -- ++(0,0.2)
(-0.5,0.3) -- ++(-0.16,0.06) -- ++(0,0.2)
(0.3,-0.9) -- ++(0.16,0.06) -- ++(0,-0.2)
(0.3,-0.9) -- ++(-0.16,-0.06) -- ++(0,-0.2)
(-1.1,-0.8) -- ++(0.12,-0.12) -- ++(-0.03,0.24) -- ++(-0.12,0.12);

\draw[line width=1.2]
(0.5,1.1) -- (-0.5,0.5)
(1.2,0) -- (0.3,-0.45) -- (-0.5,-0.45) -- (-1.2,0);

\end{tikzpicture}
\caption{$S4$ as a subdivision of a non-edge-to-edge parallelogram tiling.}
\label{SubdivParallelogram}
\end{figure}
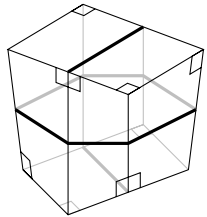

In $\AVC=\{ \alpha\gamma^2, \beta^2\delta^2, \alpha^m \}$, we shall see in Proposition \ref{Geom-AlG2-Be2De2} that tilings are only geometrically realisable when $f=16$ and $\alpha^m=\alpha^4$. We use the $\AVC$ to construct $S2$ in the first picture of Figure \ref{Tiling-S2S4}.  

In $\AVC=\{  \alpha\delta^2,  \alpha\beta\gamma^2, \beta^n \}$, we shall see in Proposition \ref{Geom-AlDe2-AlBeGa2} that tilings are only geometrically realisable when $f=12,16$. We use the $\AVC$ to construct $S_{12}1, S_{16}1$ in Figure \ref{Tilings-S1}.

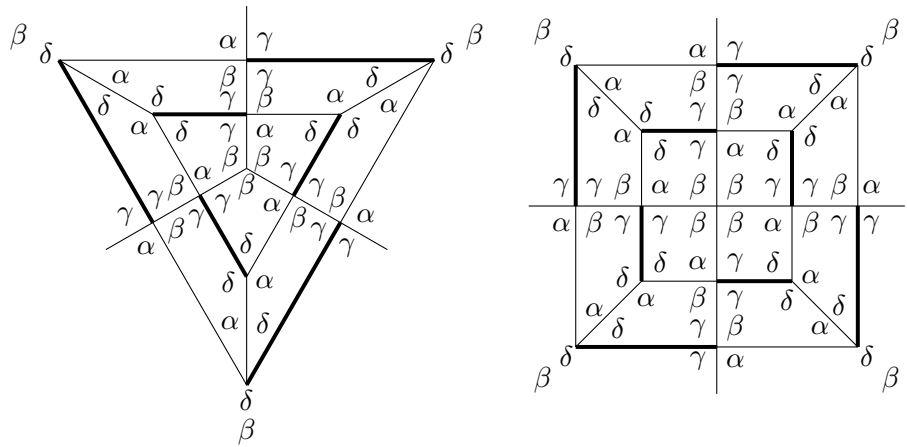
\begin{figure}[htp]
\centering
\begin{tikzpicture}[>=latex,scale=1]

\begin{scope}[yshift=0.5cm, scale=0.6]

\foreach \a in {0,...,2}
\draw[rotate=120*\a]
	(0,0)-- (-30:2.4) -- (30:4.8) -- (90:2.4)
	(-30:1.2) -- (30:2.4) -- (90:1.2)
	(-30:2.4) -- (-30:3.6)
	(30:2.4) -- (30:4.8);
			
\foreach \b in {0,...,2}
	\draw[rotate=120*\b, line width=1.5]
	(-30:1.2) -- (30:2.4)
	(90:2.4) -- (30:4.8);

			
			\node[shift={(30:0.25)}] at (0,0) {\small $\beta$};
			\node[shift={(150:0.25)}] at (0,0) {\small $\beta$};
			\node[shift={(270:0.25)}] at (0,0) {\small $\beta$};
			
			
			\node[shift={(105:0.25)}] at (30:2.4) {\small $\alpha$};
			\node[shift={(-45:0.25)}] at (30:2.4) {\small $\delta$};
			\node[shift={(210:0.45)}] at (30:2.4) {\small $\delta$};
			
			\node[shift={(45:0.35)}] at (90:1.1) {\small $\beta$};
			\node[shift={(135:0.35)}] at (90:1.1) {\small $\gamma$};
			\node[shift={(225:0.35)}] at (90:1.2) {\small $\gamma$};
			\node[shift={(315:0.35)}] at (90:1.25) {\small $\alpha$};

			\node[shift={(225:0.25)}] at (150:2.4) {\small $\alpha$};
			\node[shift={(75:0.25)}] at (150:2.4) {\small $\delta$};
			\node[shift={(330:0.45)}] at (150:2.4) {\small $\delta$};
			
			\node[shift={(165:0.35)}] at (210:1.15) {\small $\beta$};
			\node[shift={(255:0.35)}] at (210:1.1) {\small $\gamma$};
			\node[shift={(345:0.35)}] at (210:1.25) {\small $\gamma$};
			\node[shift={(75:0.35)}] at (210:1.3) {\small $\alpha$};
			
			\node[shift={(345:0.25)}] at (270:2.4) {\small $\alpha$};
			\node[shift={(195:0.25)}] at (270:2.4) {\small $\delta$};
			\node[shift={(90:0.45)}] at (270:2.4) {\small $\delta$};
			
			\node[shift={(285:0.35)}] at (330:1.15) {\small $\beta$};
			\node[shift={(15:0.35)}] at (330:1.15) {\small $\gamma$};
			\node[shift={(105:0.35)}] at (330:1.25) {\small $\gamma$};
			\node[shift={(195:0.35)}] at (330:1.3) {\small $\alpha$};
			
			
			\node[shift={(30:0.2)}] at (30:4.8) {\small $\delta$};
			\node[shift={(195:0.85)}] at (30:4.8) {\small $\delta$};
			\node[shift={(225:0.85)}] at (30:4.8) {\small $\alpha$};
			
			\node[shift={(45:0.35)}] at (90:2.4) {\small $\gamma$};
			\node[shift={(135:0.35)}] at (90:2.4) {\small $\alpha$};
			\node[shift={(225:0.35)}] at (90:2.4) {\small $\beta$};
			\node[shift={(315:0.35)}] at (90:2.4) {\small $\gamma$};
			
			\node[shift={(150:0.2)}] at (150:4.8) {\small $\delta$};
			\node[shift={(315:0.85)}] at (150:4.8) {\small $\delta$};
			\node[shift={(345:0.85)}] at (150:4.8) {\small $\alpha$};
			
			\node[shift={(165:0.35)}] at (210:2.45) {\small $\gamma$};
			\node[shift={(255:0.35)}] at (210:2.4) {\small $\alpha$};
			\node[shift={(345:0.35)}] at (210:2.45) {\small $\beta$};
			\node[shift={(75:0.35)}] at (210:2.5) {\small $\gamma$};
			
			\node[shift={(270:0.2)}] at (270:4.8) {\small $\delta$};
			\node[shift={(75:0.85)}] at (270:4.8) {\small $\delta$};
			\node[shift={(105:0.85)}] at (270:4.8) {\small $\alpha$};

			\node[shift={(285:0.35)}] at (330:2.4) {\small $\gamma$};
			\node[shift={(15:0.35)}] at (330:2.4) {\small $\alpha$};
			\node[shift={(105:0.35)}] at (330:2.5) {\small $\beta$};
			\node[shift={(195:0.35)}] at (330:2.5) {\small $\gamma$};

			
			\node[shift={(30:3.5)}] at (0,0) {\small $\beta$};
			\node[shift={(150:3.5)}] at (0,0) {\small $\beta$};
			\node[shift={(270:3.5)}] at (0,0) {\small $\beta$};

\end{scope}

\begin{scope}[xshift=6.25cm, scale=1.25]

\foreach \a in {0,1,2,3}
{
\begin{scope}[rotate=90*\a]

\draw
	(0,0) -- (1.5,0) -- (1.5,1.5)
	(0,0.8) -- (0.8,0.8) -- (1.5,1.5)
	(0,1.5) -- (0,2)
;

\draw[line width=1.5]
	(0,0.8) -- (-0.8,0.8)
	(-1.5,0) -- (-1.5,1.5);
	
\node at (0.2,0.2) {\small $\beta$};
\node at (0.6,0.2) {\small $\gamma$};
\node at (0.2,0.6) {\small $\alpha$};
\node at (0.6,0.6) {\small $\delta$};

\node at (1,0.75) {\small $\delta$};
\node at (1,0.2) {\small $\gamma$};
\node at (1.3,1.1) {\small $\alpha$};
\node at (1.3,0.2) {\small $\beta$};

\node at (1,-0.75) {\small $\alpha$};
\node at (1,-0.2) {\small $\beta$};
\node at (1.3,-1.05) {\small $\delta$};
\node at (1.3,-0.2) {\small $\gamma$};

\node at (1.6,1.6) {\small $\delta$};
\node at (0.2,1.65) {\small $\gamma$};
\node at (1.65,0.2) {\small $\alpha$};
\node at (1.85,1.85) {\small $\beta$};

\end{scope}
}

\end{scope}

\end{tikzpicture}
\caption{Tilings $S1 = S_{12}1, S_{16}1$}
\label{Tilings-S1}
\end{figure}

This completes the proof.
\end{proof}

\begin{prop} If one of $\alpha\gamma\delta, \beta\gamma\delta$ is a vertex and $f\ge8$, then the tilings by congruent almost equilateral quadrilaterals with non-rational angles are earth map tilings $E$ and their flip modifications $E^{\prime}, E^{\prime\prime}$. 
\end{prop}

\begin{proof} Up to symmetry, we may assume $\alpha\gamma\delta$ is a vertex. By $\alpha \neq \beta$, this implies $\beta\gamma\delta$ is not a vertex. The angle sum system gives
\begin{align*}
\alpha + \gamma + \delta = 2\pi, \quad
\beta = \tfrac{4}{f}\pi.
\end{align*}
By $\alpha\gamma\delta$, at least two of $\alpha, \gamma, \delta$ are non-rational. The key fact we use is the following: suppose $\varphi, \psi$ are non-rational angles and $\varphi + \psi$ is a rational angle, then $u\varphi+v\psi = q\pi$ for some rational numbers $u,v,q$ implies $u=v$.

If $\gamma < \delta$, then $\alpha\gamma\delta$ and Parity Lemma imply $\alpha\delta \cdots = \alpha\gamma\delta$. Similarly, $\gamma>\delta$ implies $\alpha\gamma\cdots = \alpha\gamma\delta$. We have either $\alpha\gamma\cdots = \alpha\gamma\delta$ or $\alpha\delta\cdots = \alpha\gamma\delta$. Meanwhile, Lemma \ref{AlGaDe-Al2} implies that $\alpha^2\cdots$ is a $\hat{b}$-vertex. At least two of $\alpha, \gamma, \delta$ are non-rational. So we divide the discussion into the following cases. 

\begin{case*}[$\gamma, \delta$ are non-rational, $\alpha$ is rational]
By $\alpha\gamma\delta$, we know that $\gamma + \delta$ is a rational angle. As $\gamma,\delta$ are non-rational, at each vertex, the number of $\gamma$ equals to the number of $\delta$. By Parity Lemma, the $b$-vertices are $\alpha\gamma\delta, \beta^n\gamma^k\delta^k, \gamma^k\delta^k$. As $\alpha,\beta$ are rational, the $\hat{b}$-vertices are $\alpha^m, \alpha^m\beta^n, \beta^n$. Therefore
\begin{align}\label{NonRat-algade-fullavc}
\AVC = \{ \alpha\gamma\delta, \alpha^m, \alpha^m\beta^n, \beta^n, \beta^n\gamma^k\delta^k, \gamma^k\delta^k \}.
\end{align}
\end{case*}

\begin{case*}[$\alpha, \gamma$ are non-rational, $\delta$ is rational] By $\alpha\gamma\delta$, we know that $\alpha + \gamma$ is a rational angle. As $\beta, \delta$ are rational and $\alpha, \gamma$ are non-rational, at each vertex, the number of $\alpha$ equals to the number of $\gamma$. Since $\alpha^2\cdots$ can only be $\hat{b}$-vertex, this implies that $\alpha^2\cdots$ is not a vertex. Then $\gamma\cdots=\alpha\gamma\cdots$ with no $\alpha, \gamma$ in the remainder. By $\alpha\gamma\delta$ and Parity Lemma, we have $\gamma\cdots=\alpha\gamma\delta$. Counting Lemma implies that the only other vertex is $\beta^n$ where $n=\frac{f}{2}$. Therefore
\begin{align}\label{AlGaDe-Ben-AVC}
\AVC = \{ \alpha\gamma\delta, \beta^{\frac{f}{2}} \}.
\end{align}
\end{case*}

\begin{case*}[$\alpha, \delta$ are non-rational, $\gamma$ is rational] The previous argument relies only on the parity of $\gamma, \delta$. Exchanging $\gamma \leftrightarrow \delta$ above, we get $\AVC$ \eqref{AlGaDe-Ben-AVC}.
\end{case*}

\begin{case*}[$\alpha, \gamma, \delta$ are non-rational] As $\alpha, \gamma, \delta$ are non-rational and $\beta$ is rational, $\alpha^m, \gamma^k, \delta^l, \alpha^m\beta^n, \beta^n\gamma^k, \beta^n\delta^l$ are not vertices. Since $\alpha^2\cdots$ is a $\hat{b}$-vertex, by no $\alpha^m, \alpha^m\beta^n$, this implies that $\alpha^2\cdots$ is not a vertex. 

Suppose $\gamma>\delta$. We have $\alpha\gamma\cdots = \alpha\gamma\delta$. Then $\alpha\cdots = \alpha\gamma\delta, \alpha\delta^l, \alpha\beta^n\delta^l$. Counting Lemma on $\alpha,\delta$ and Parity Lemma imply that $\alpha\delta^l, \alpha\beta^n\delta^l$ are not vertices. Then $\alpha\cdots = \alpha\gamma\delta$. Counting Lemma further implies that the only other vertex is $\beta^n$. We get $\AVC$ \eqref{AlGaDe-Ben-AVC}. 

Suppose $\gamma<\delta$. We have $\alpha\delta\cdots = \alpha\gamma\delta$, exchanging $\gamma \leftrightarrow \delta$ and $k \leftrightarrow l$ above gives $\AVC$ \eqref{AlGaDe-Ben-AVC}.
\end{case*}

We summarise the AVCs below,
\begin{enumerate}
\item $f\ge8, \AVC = \{ \alpha\gamma\delta, \beta^{\frac{f}{2}} \}$, 
\item $f\ge8, \AVC = \{ \alpha\gamma\delta, \alpha^m, \beta^n, \alpha^m\beta^n, \gamma^k\delta^k, \beta^n\gamma^k\delta^k \}$.
\end{enumerate}

By \cite[Proposition 48]{cly}, the earth map tilings $E$ and their flip modifications $E^{\prime},E^{\prime\prime}$ are obtained from the above $\AVC$s. The same argument of identifying the $\AVC$s in Table \ref{Rat-AlGaDe-AVCs} with tilings in Proposition \ref{RatAlGaDeProp} only uses the vertices, not the rationality nor non-rationality. Then the same discussion applies here. Hence we get 
\begin{align}
&E^{\prime}: \, \AVC \equiv \{ \alpha\gamma\delta, \alpha^m,  \beta^n\gamma\delta \}; \\
&E^{\prime}: \, \AVC \equiv \{ \alpha\gamma\delta, \alpha^m\beta^n, \beta^n\gamma\delta \}; \\
&E^{\prime\prime}: \AVC \equiv \{ \alpha\gamma\delta, \alpha\beta^n,  \gamma^k\delta^k \}; \\ 
&E^{\prime\prime}: \AVC \equiv \{ \alpha\gamma\delta, \alpha\beta^n,  \beta^n\gamma^k\delta^k \}. 
\end{align}

The tilings with their $\AVC$s are given in Table \ref{NonRat-AlGaDe-AVCs}. The construction is explained in \cite[Figures 75, 76]{cly}.  

\begin{table}[h]
\begin{center}
\bgroup
\def\arraystretch{1.5}
    \begin{tabular}[t]{| c | c | c |}
	\hline
	Tilings & $f$ & $\AVC$  \\ \hhline{|===|} 
	$E$ & \multirow{5}{*}{$\ge8$}  & $\{ \alpha\gamma\delta, \beta^{\frac{f}{2}} \}$ \\ 
	\cline{1-1}\cline{3-3}
	\multirow{2}{*}{$E^{\prime}$} & & $\{ \alpha\gamma\delta, \alpha^m,  \beta^n\gamma\delta \}$ \\ 
	\cline{3-3}
	 &  & $\{ \alpha\gamma\delta, \alpha^m\beta^n, \beta^n\gamma\delta \}$ \\ 
	\cline{1-1}\cline{3-3}
	\multirow{2}{*}{$E^{\prime\prime}$} &  & $\{ \alpha\gamma\delta, \alpha\beta^n, \gamma^k\delta^k \}$  \\
          \cline{3-3}
	 &  & $\{ \alpha\gamma\delta, \alpha\beta^n, \beta^n\gamma^k\delta^k \}$  \\
          \hline  
	\end{tabular}
\egroup
\end{center}
\caption{Non-rational angles: $\AVC$s with $\alpha\gamma\delta$}
\label{NonRat-AlGaDe-AVCs}
\end{table}

This completes the proof.
\end{proof}

\section{Geometric Realisation} \label{SecGeom}

Up to this point, we have obtained tilings using the conditions on angles. For the tilings to exist, we still need to show that there is actually a simple almost equilateral quadrilateral with the given angles. We also provide extrinsic formulae to all angles and edges. For the purpose of tilings, we always assume $\alpha, \beta, \gamma, \delta \in (0,2\pi)$ and $a\in(0,\pi)$ and $b\in(0,2\pi)$.

For tilings without $\alpha\gamma\delta$, we use spherical trigonometric tools to justify the existence. In Section \ref{Subsec-Geom-SphTrig}, we establish the tools. The same tools can also be applied to show the existence of tilings by congruent $a^2bc$ quadrilaterals. Sections \ref{Subsec-Geom-al3}, \ref{Subsec-Geom-albe2}, \ref{Subsec-Geom-alga2-alde2} are the applications of these tools. Lastly, we verify the existence of tilings with $\alpha\gamma\delta$ by direct construction in Section \ref{Subsec-Geom-algade}.

\subsection{Spherical Trigonometry} \label{Subsec-Geom-SphTrig}

To verify the existence of $a^2bc$ quadrilateral and the almost equilateral quadrilateral, we state the following facts from \cite{ch,cly}. The $a^2bc$ quadrilateral in Figure \ref{StdQuad} exists if and only if the following matrix equation holds.
\begin{align}\label{MatrixEq}
Y(c)Z(\pi - \beta)Y(a)Z(\pi - \alpha)Y(a)Z(\pi - \delta)Y(b)Z(\pi - \gamma) = I_3 
\end{align}
where, 
\begin{align*}
Y(t) = 
\begin{bmatrix}
\cos t & 0 & \sin t \\
0 & 1 & 0 \\
-\sin t & 0 & \cos t
\end{bmatrix}, 
Z(t) = 
\begin{bmatrix}
\cos t & -\sin t & 0 \\
\sin t & \cos t & 0 \\
0 & 0 & 1
\end{bmatrix}, 
I_3 = 
\begin{bmatrix}
1 & 0 & 0 \\
0& 1 & 0 \\
0 & 0 & 1
\end{bmatrix}.
\end{align*}
It is obvious that $Y(t), Z(t) \in SO(3)$.

\begin{lem} \label{a2bcExistLem} Suppose $\gamma \notin \mathbb{Z}\pi$ and $a \in (0,\pi)$. Then there exists an $a^2bc$ quadrilateral with angles $\alpha, \beta, \gamma, \delta$ and $a$ if and only if the following holds.
\begin{align} \label{TrigEqca}
&(\cos \alpha-1)\sin\beta\sin\delta\cos^2a + \sin \alpha \sin (\beta+\delta) \cos a \\ \notag
&+ \sin\beta\sin\delta - \cos \alpha \cos \beta \cos \delta + \cos\gamma = 0.
\end{align}
\end{lem}

The proof of the lemma also shows that, given the angles $\alpha,\beta,\gamma,\delta$ where $\gamma \notin \mathbb{Z}\pi$ and an edge $a \in (0,\pi)$, we can uniquely determine $b, c \in (0, 2\pi]$ by the following formulae,
\begin{align}
\label{TrigEqcb} 
\sin \gamma \cos b = & \, (1-\cos\alpha)\sin\beta\cos\delta\cos^2a - \sin\alpha \cos(\beta+\delta)\cos a \\ \notag
&\, - \sin\beta\cos\delta - \cos \alpha \cos\beta\sin\delta; \\
\label{AEsb}
\sin \gamma \sin b = & \, ((1-\cos\alpha)\sin\beta\cos a - \sin \alpha \cos \beta) \sin a. \\
\label{TrigEqcc}
\sin \gamma \cos c = & \, (1-\cos\alpha)\cos\beta\sin\delta\cos^2a - \sin\alpha \cos(\beta+\delta)\cos a \\ \notag
&\, - \cos\beta\sin\delta - \cos \alpha \sin\beta\cos\delta; \\
\label{TrigEqsc}
\sin \gamma \sin c =& \, ((1-\cos\alpha)\sin\delta\cos a - \sin \alpha \cos \delta) \sin a.
\end{align}

\begin{proof} Let $s$ denote the left hand side of \eqref{TrigEqca} and $p_{\beta,\delta}, q_{\beta}, p_{\delta, \beta}, q_{\delta}$ denote the right hand sides of \eqref{TrigEqcb}, \eqref{AEsb}, \eqref{TrigEqcc}, \eqref{TrigEqsc} respectively. Let
\begin{align*}
U &:=Z(\pi-\beta)Y(a)Z(\pi-\alpha)Y(a)Z(\pi-\delta), \\
V &:=Y(c)^TZ(\pi-\gamma)^TY(b)^T.
\end{align*}
Then \eqref{MatrixEq} means $U=V$ and calculations show that $U_{22} = V_{22}, U_{21} = V_{21}, U_{23} = V_{23}, U_{12} = V_{12}, U_{32} = V_{32}$ are exactly \eqref{TrigEqca}, \eqref{TrigEqcb}, \eqref{AEsb}, \eqref{TrigEqcc}, \eqref{TrigEqsc} respectively. In particular, we get the necessary direction of the lemma. 

For the sufficient direction, we first calculate to verify the equalities,
\begin{align*}
s^2 + p_{\beta,\delta}^2 + q_{\beta}^2 = s^2 + p_{\delta,\beta}^2 + q_{\delta}^2 = \sin^2\gamma + 2s\cos\gamma.
\end{align*}
Then \eqref{TrigEqca} means $s=0$, or $U_{22}=V_{22}$, and the equalities become 
\begin{align*}
p_{\beta,\delta}^2 + q_{\beta}^2 = p_{\delta,\beta}^2 + q_{\delta}^2 = \sin^2\gamma.
\end{align*}
By $\gamma \notin \mathbb{Z}\pi$, we have $\sin\gamma \neq 0$. Then the above equalities imply unique $b,c \in (0, 2\pi]$ such that $\sin\gamma\cos b = p_{\beta,\delta}$ and $\sin\gamma\sin b= q_{\beta}$ and $\sin\gamma\cos c=p_{\delta,\beta}$ and $\sin\gamma\sin c = q_{\delta}$. This means $U_{21}=V_{21}$ and $U_{23}=V_{23}$ and $U_{12} = V_{12}$ and $U_{32} = V_{32}$. Then $U,V$ have the same second row and the same second column. Let 
\begin{align*}
\vec{w}=
\begin{bmatrix}
U_{12} \\
U_{22} \\
U_{32}
\end{bmatrix}
=\begin{bmatrix}
V_{12} \\
V_{22} \\
V_{32}
\end{bmatrix},
\end{align*}
be the second column of $U$ and $V$.

By $U, V \in SO(3)$ and $U_{21} = V_{21}, U_{22} = V_{22}, U_{23} = V_{23}$, we know that the $(2,2)$-entry of $UV^T$ is $1$. Since $UV^T \in SO(3)$, this implies
\begin{align*}
UV^T
=\begin{bmatrix}
\ast & 0 & \ast \\
0 & 1 & 0 \\
\ast & 0 & \ast \\
\end{bmatrix}
=Y(\theta),
\end{align*}
for some $\theta$. Then $U = Y(\theta)V$ implies $\vec{w}=Y(\theta)\vec{w}$. This implies $Y(\theta)=I_3$ or $\vec{w}=\pm \vec{e}_2$, where $\vec{e}_2$ is the second column of $Y(\theta)$. Suppose it is the latter. Since $V\vec{e}_2 = \vec{w}$, we then get $Y^T(c)Z(\pi-\gamma)^TY(b)^T\vec{e}_2=\pm \vec{e}_2$, which implies 
\begin{align*}
\begin{bmatrix}
\sin(\pi-\gamma) \\
\cos(\pi - \gamma) \\
0
\end{bmatrix}
= Z(\pi-\gamma)^T\vec{e}_2 
= \pm \vec{e}_2
=
\begin{bmatrix}
0 \\
\pm 1 \\
0
\end{bmatrix}.
\end{align*}
This implies $\gamma \in \mathbb{Z}\pi$, a contradiction to the hypothesis. So we have $Y(\theta)=I_3$, i.e., $U=V$. Hence the converse is true.
\end{proof}

The edge reduction $c=a$ implies that the almost equilateral quadrilateral in Figure \ref{StdQuad} exists if and only if
\begin{align}\label{AEMatrixEq}
Y(a)Z(\pi - \beta)Y(a)Z(\pi - \alpha)Y(a)Z(\pi - \delta)Y(b)Z(\pi - \gamma) = I_3.
\end{align}

\begin{lem}\label{AEQuadExists} If $\gamma, \delta \neq \pi$, then there exists an almost equilateral quadrilateral with angles $\alpha,\beta,\gamma,\delta \in (0,2\pi)$ and edge $a \in (0,\pi)$ if and only if \eqref{Coolsaet-Id} and the following identity hold,
\begin{align}
\label{ca3Ang} \cos a &= \tfrac{ \sin\beta \cos \gamma + \sin \delta  }{ (1-\cos \beta) \sin \gamma } = \tfrac{ \sin\alpha \cos \delta + \sin \gamma }{ (1-\cos \alpha) \sin \delta }.
\end{align}
Moreover, \eqref{Coolsaet-Id} implies the second equality above.
\end{lem}

\begin{proof} The criterion for the existence of the quadrilateral is given by \cite[Lemma 18]{cly}. In particular, \eqref{Coolsaet-Id} is \cite[equation (3.8)]{cly} and \eqref{ca3Ang} is \cite[equations (3.9), (3.10)]{cly}. For this lemma, we only need to explain that the second equality in \eqref{ca3Ang} is a consequence of \eqref{Coolsaet-Id}.

By $\alpha, \beta \in (0,2\pi)$, we know $\sin \frac{1}{2}\alpha, \sin\frac{1}{2}\beta \neq 0$. If \eqref{Coolsaet-Id} is true, then dividing both sides by $\sin\frac{1}{2}\alpha\sin\tfrac{1}{2}\beta$, 
\begin{align*}
\sin \delta \cot \tfrac{1}{2}\beta + \cos \gamma = \sin \gamma \cot\tfrac{1}{2}\alpha + \cos \delta.
\end{align*}
Squaring both sides, it implies
\begin{align*}
\tfrac{1}{2} \csc^2 \tfrac{1}{2}\beta \sin \delta ( \sin \delta + \sin \beta \cos \gamma ) = \tfrac{1}{2} \csc^2 \tfrac{1}{2}\alpha \sin \gamma ( \sin \gamma + \sin \alpha \cos \delta ). 
\end{align*}
This gives \eqref{ca3Ang}.
\end{proof}

To justify the existence of an almost equilateral quadrilateral with given $\alpha, \beta, \gamma, \delta$, we need to first verify the angle values satisfy \eqref{Coolsaet-Id}. Then we need one of the expressions for $\cos a$ in \eqref{ca3Ang} to have absolute value $<1$. This determines $a \in (0,\pi)$.

Lastly, it remains to show that the quadrilateral is simple. One sufficient condition is given by Lemma \ref{SimpQuadLem}. To apply this lemma, we need $0<b<\pi$. The two lemmas after Lemma \ref{SimpQuadLem} provide sufficient conditions for $b<\pi$.

\begin{lem}[{\cite[Lemma 20]{cly}}]\label{SimpQuadLem} A spherical quadrilateral having all the edges and at least three angles $<\pi$ is simple. 
\end{lem}

\begin{lem}\label{b-criteria1} In the almost equilateral quadrilateral with $0 < a, \alpha, \beta, \gamma, \delta<\pi$ and $0 < b < 2\pi$, if $\beta > \delta$ or $\alpha > \gamma$, then $b<\pi$.
\end{lem}

\begin{proof} 
By \eqref{ca3Ang} and $0<\alpha, \beta, \gamma, \delta < \pi$ and $\beta>\delta$, we have
\begin{align*}
(1-\cos \alpha) \sin \beta \cos a - \sin \alpha \cos \beta = \tfrac{1}{\sin \delta} ( \sin \beta \sin \gamma + \sin \alpha \sin ( \beta - \delta ) ) > 0.
\end{align*}
By $0<a, \gamma <\pi$, we also have $\sin \gamma, \sin a > 0$. Then by \eqref{AEsb}, we get $\sin b>0$.

The proof for the case of $\alpha > \gamma$ is analogous.
\end{proof}

\begin{lem} \label{b-criteria2} In a spherical almost equilateral quadrilateral with $0 < \gamma, \delta < \pi$ and $0<a<\pi$ and $0 < b < 2\pi$,
\begin{enumerate}
\item $\cos(\beta - \delta) + \cos \gamma >0$ if and only if $b<\pi$;
\item if $\beta \le \pi$ and one of $\beta + \pi > \gamma + \delta$ and $\delta + \pi > \beta + \gamma$ holds, then $b<\pi$.
\end{enumerate}
The same is true by swapping $\alpha \leftrightarrow \beta$ and $\gamma \leftrightarrow \delta$.
\end{lem}

\begin{proof} Multiplying $\sin \gamma$ on both sides of \eqref{AEsb} and applying $\sin \theta = \cot \frac{1}{2}\theta ( 1 - \cos \theta )$, we have
\begin{align*}
\sin^2 \gamma \sin b
& = \sin \gamma ( (1-\cos \alpha)\cos a \sin \beta - \sin \alpha \cos \beta  )  \sin a \\
&= (1-\cos \alpha) ( \sin \beta \sin \gamma \cos a - \cos \beta \cot \tfrac{1}{2}\alpha \sin \gamma  )  \sin a. 
\end{align*}
By \eqref{Coolsaet-Id}, we get $\cos \gamma - \cot \tfrac{1}{2}\alpha \sin \gamma  = \cos \delta - \cot \tfrac{1}{2}\beta \sin \delta$. Then combined with $\sin \theta = \cot \frac{1}{2}\theta ( 1 - \cos \theta )$ and $\sin \theta \cot \frac{1}{2}\theta = ( 1+\cos \theta )$ and \eqref{ca3Ang}, we have
\begin{align*}
&\sin \beta \sin \gamma \cos a - \cos \beta \sin \gamma \cot \tfrac{1}{2}\alpha  \\
&=(1-\cos \beta) \cot \tfrac{1}{2}\beta \sin \gamma \cos a - \cos \beta \sin \gamma \cot \tfrac{1}{2} \alpha \\
&= (\sin\beta\cos\gamma + \sin \delta) \cot \tfrac{1}{2}\beta - \cos \beta \sin \gamma \cot \tfrac{1}{2} \alpha \\
&= \cos \gamma + \sin \delta \cot \tfrac{1}{2}\beta + \cos \beta ( \cos \gamma - \sin \gamma \cot \tfrac{1}{2}\alpha ) \\
&= \cos \gamma + \sin \beta \sin \delta + \cos \beta \cos \delta \\
&= \cos(\beta-\delta) + \cos \gamma.
\end{align*}

If $\beta \in (0, \pi]$ and $\gamma, \delta \in (0, \pi)$, then $\pi + \beta > \gamma + \delta$ implies $\cos( \beta - \delta) + \cos \gamma > 0$. By $\sin \gamma \neq 0$, the formulae above imply $\sin b > 0$. This means $b < \pi$. The argument for the case of $\pi + \delta > \beta + \gamma$ is analogous. 
\end{proof}

For some discussion, it is convenient to know the range of $a$. The following lemma gives a necessary condition on $a$ when the $a^2bc$ quadrilateral is convex. 

\begin{lem} \label{a2bc-aEdge-Lem} In an simple $a^2bc$ quadrilateral $\square XYWZ$ where $XY=XZ=a$ and $\angle YXZ = \rho$, if $\rho$ is bigger than the area of $\square XYWZ$, and $\square XYWZ$ contains the isosceles triangle $\triangle XYZ$ in Figure \ref{a2bcQuad-a-Edge}, then $a < \frac{1}{2}\pi$.
\end{lem}

\begin{proof} A lune on the unit sphere with interior antipodal angles $\rho$ in the picture of Figure \ref{a2bcQuad-a-Edge} has area $2\rho$. So an isosceles triangle given by a half of the lune (with two side edges of length $\frac{1}{2}\pi$ and base edge indicated by dashed line) has area $\rho$.

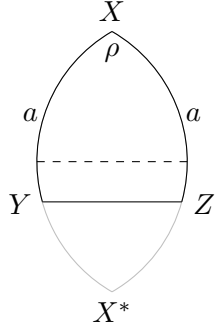
\begin{figure}[htp] 
\centering
\begin{tikzpicture}[scale=1]


\begin{scope}[xshift=0cm] 

\tikzmath{
\r=2; \R=sqrt(3);
\a=120; \aa=\a/2;
\xP=0; \yP=\R;
\aOP=120; 
\aPOne=\aOP;
\mPOne=tan(90+\aPOne);
\aPTwo=\aOP;
\mPTwo=tan(-90-\aPTwo);
\xCQ=-1.5; \yCQ=0.5;
\xCOne = 1; \yCOne=0;
\xCTwo = -1; \yCTwo=0;
}

\coordinate (O) at (0,0);
\coordinate (C1) at (1,0);
\coordinate (C2) at (-1,0);

	
	(C2) circle (\r)
	(C1) circle (\r);

\draw[gray!50]
	([shift={(-60:2)}]-1,0) arc (-60:60:2);

\draw[gray!50]
	([shift={(-60+180:2)}]1,0) arc (-60+180:60+180:2);

\pgfmathsetmacro{\xPOne}{ ( 2*((\mPOne)*(\yCOne)+(\xCOne)) - sqrt( ( 2*( (\mPOne)*(\yCOne)+\xCOne )  )^2 - 4*( (\mPOne)^2 + 1 )*( (\xCOne)^2 + (\yCOne)^2 - \r^2 ) ) )/( 2*( (\mPOne)^2+1 ) ) };
\pgfmathsetmacro{\yPOne}{ \mPOne*\xPOne };

\pgfmathsetmacro{\xPTwo}{ ( 2*((\mPTwo)*(\yCTwo)+(\xCTwo) ) + sqrt( ( 2*( (\mPTwo)*\yCTwo+\xCTwo )  )^2 - 4*( (\mPTwo)^2 + 1 )*( (\xCTwo)^2 + (\yCTwo)^2 - \r^2 ) ) )/( 2*( (\mPTwo)^2+1 ) ) };
\pgfmathsetmacro{\yPTwo}{ \mPTwo*\xPTwo };

\pgfmathsetmacro{\xPPOne}{ ( 2*((\mPOne)*(\yCOne)+(\xCOne)) + sqrt( ( 2*( (\mPOne)*(\yCOne)+\xCOne )  )^2 - 4*( (\mPOne)^2 + 1 )*( (\xCOne)^2 + (\yCOne)^2 - \r^2 ) ) )/( 2*( (\mPOne)^2+1 ) ) };
\pgfmathsetmacro{\yPPOne}{ \mPOne*\xPPOne };

\pgfmathsetmacro{\xPPTwo}{ ( 2*((\mPTwo)*(\yCTwo)+(\xCTwo) ) - sqrt( ( 2*( (\mPTwo)*\yCTwo+\xCTwo )  )^2 - 4*( (\mPTwo)^2 + 1 )*( (\xCTwo)^2 + (\yCTwo)^2 - \r^2 ) ) )/( 2*( (\mPTwo)^2+1 ) ) };
\pgfmathsetmacro{\yPPTwo}{ \mPTwo*\xPPTwo };

\pgfmathsetmacro{\dPOneP}{ sqrt( (\xPOne - \xP)^2 + (\yPOne - \yP)^2 ) };
\pgfmathsetmacro{\aCPOne}{ acos( (2*\r^2 - \dPOneP^2 )/(2*\r^2) ) };
\pgfmathsetmacro{\l}{ \aCPOne/\a };
\pgfmathsetmacro{\rQ}{ sqrt( \R^2 + (\xCQ)^2  )  };
\pgfmathsetmacro{\aPCQ}{ acos( (\xCQ)/(\rQ) ) }
\pgfmathsetmacro{\aQ}{  -( 360 - 2*\aPCQ )*(\l) };

\coordinate (CQ) at (\xCQ, \yCQ);

\coordinate (P) at (0,{sqrt(3)});

\coordinate (Q) at (-75:0.85);

\coordinate (PP) at (0,{-sqrt(3)});

\coordinate (P1) at (\xPOne, \yPOne);
\coordinate (P2) at (\xPTwo, \yPTwo);

\coordinate (PP1) at (\xPPOne,\yPPOne);

\coordinate (PP2) at (\xPPTwo,\yPPTwo);

\arcThroughThreePoints[]{P}{PP}{P1};
\arcThroughThreePoints[]{P2}{PP}{P};


\draw[dashed]
	(180:1) -- (0:1)
;

\draw[]
	(P1) -- (P2)
;

\node at (90: 2) {\small $X$};
\node at (270: 2) {\small $X^{\ast}$};

\node at (205: 1.35) {\small $Y$};
\node at (335: 1.35) {\small $Z$};



\node at (150:1.25) {\small $a$};
\node at (30:1.25) {\small $a$};
\node at (90:1.45) {\small $\rho$};

\end{scope}

\end{tikzpicture}
\caption{$a^2bc$ quadrilateral containing isosceles triangle $\triangle XYZ$}
\label{a2bcQuad-a-Edge}
\end{figure}

Assume $a\ge\frac{1}{2}\pi$. Since the $a^2bc$ quadrilateral contains $\triangle XYZ$, the isosceles triangle given by the upper half of the lune is contained in the $a^2bc$ quadrilateral. Then the area of the isosceles triangle is $\rho$, which is less than the area of the quadrilateral, a contradiction. 
\end{proof}

\subsection{Tiles \boldmath{$a^2bc, a^3b$} and Vertices \boldmath{$\alpha^3, \beta^2\delta^2, \gamma^4$}} \label{Subsec-Geom-al3}

The discussion includes $QP_6, QP^{\prime}_6$. The tiling $QP_6$ is given by the quadrilateral subdivision of the cube $P_6$. 

In the most general setting, the tiles of $QP_6$ are the $a^2bc$ quadrilaterals. The tiling is displayed in the left column of Figure \ref{Tiling-a2bc-f24-a3}. The edge reduction $c=a$ gives the tiling $QP_6$ by congruent almost equilateral quadrilaterals. On the other hand, for a particular combination of angle values, $QP_6$ has a flip modification $QP^{\prime}_6$ depicted in the right column of of Figure \ref{Tiling-a2bc-f24-a3}. As shown in the picture, $QP^{\prime}_6$ can also be viewed combinatorially as a \quotes{special quadrilateral subdivision} of the triangular prism.

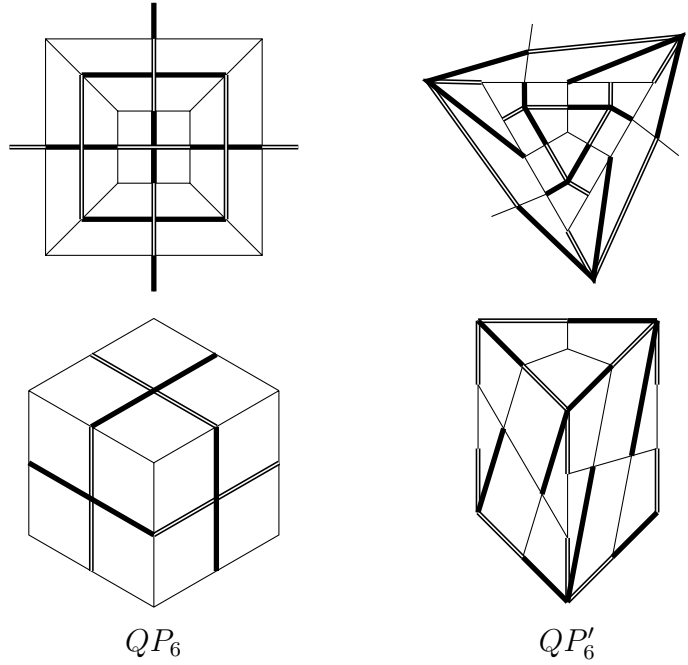
\begin{figure}[htp]
\centering
\begin{tikzpicture}[>=latex]

\begin{scope}

\begin{scope}[] 

\begin{scope}[scale=0.4]

	\foreach \a in {0,...,3}
	\draw[rotate=90*\a]
	(0,0) -- (4.8,0)
	(3.6,0) -- (3.6,3.6) -- (0,3.6) 
	(1.2,0) -- (1.2,1.2) -- (0,1.2)
	(2.4,0) -- (2.4,2.4) -- (0,2.4)
	(1.2,1.2) -- (3.6,3.6);
	
	\foreach \b in {0,1}
	\draw[line width=2,rotate=180*\b]
	(0,1.2) -- (0,-1.2)
	(2.4,2.4) -- (-2.4,2.4)
	(1.2,0) -- (3.6,0)
	(0,3.6) -- (0,4.8);
	
	\foreach \c in {1,3}
	\draw[double, line width=0.6,rotate=90*\c]
	(0,1.2) -- (0,-1.2)
	(2.4,2.4) -- (-2.4,2.4)
	(1.2,0) -- (3.6,0)
	(0,3.6) -- (0,4.8);
		
\end{scope}

\end{scope}

\begin{scope}[xshift=5.5cm, yshift=0.2cm] 

\begin{scope}[scale=0.275]

	\foreach \a in {0,1,2}
	\draw[rotate=120*\a]
	(0,0)-- (-30:2.4) -- (30:4.8) -- (90:2.4)
	(-30:1.2) -- (30:2.4) -- (90:1.2)
	(11:3.18) -- (30:2.4) -- (49:3.18)
	(131:3.18) -- (116:4.32) -- (40:7.2)
	(150:4.8) -- (160:7.2) -- (116:4.32)
	(90:2.4) -- (40:7.2)
	(116:4.32) -- (108:5.5);
	
	\foreach \c in {0,1,2}
	\draw[double, line width=0.6, rotate=120*\c]
	(-30:1.2) -- (30:2.4) -- (49:3.18)	
	(30:4.8) -- (40:7.2) -- (116:4.32);
	
	\foreach \b in {0,1,2}
	\draw[line width=2, rotate=120*\b]
	(90:1.2) -- (30:2.4) -- (11:3.18)
	(90:2.4) -- (40:7.2) -- (-4:4.32);

\end{scope}

\end{scope}

\end{scope}

\begin{scope}[yshift=-4.2 cm] 

\begin{scope}[scale=0.8] 

	
	\coordinate (O) at (0,0);
	\coordinate (A1) at (2.078,1.2);
	\coordinate (A2) at (-2.078,1.2);
	\coordinate (A3) at (0,-2.4);	
	
	\coordinate (A4) at (2.078,-1.2);
	\coordinate (A5) at (-2.078,-1.2);
	
	\coordinate (A6) at (0,2.4);
	
	\coordinate (OA1) at (1.039,0.6);
	\coordinate (OA2) at (-1.039,0.6);
	\coordinate (A1A6) at (1.039,1.8);
	\coordinate (A2A6) at (-1.039,1.8);
	
	\coordinate (OA3) at (0,-1.2);
	\coordinate (A3A4) at (1.039,-1.8);
	\coordinate (A1A4) at (2.078,0);
	
	\coordinate (A3A5) at (-1.039,-1.8);
	\coordinate (A2A5) at (-2.078,0);

	\draw 
	(O) -- (A1)
	(O) -- (A2)
	(O) -- (A3)
	(A1) -- (A6) -- (A2)
	(A3) -- (A4) -- (A1)
	(A3) -- (A5) -- (A2);
	
	\draw[double, line width=0.6] 
	(OA1) -- (A2A6)
	(OA3) -- (A1A4)
	(OA2) -- (A3A5);
	
	\draw[line width=2]
	(OA2) -- (A1A6)
	(OA1) -- (A3A4)	
	(OA3) -- (A2A5);
	
\end{scope}

\node at (0,-2.4) {$QP_6$};

\begin{scope}[xshift = 5.5 cm] 

\begin{scope}[yshift=0.7cm, scale=0.85]
		
	\draw 
	(0,0) -- (0,-3)
	(1.4,1.4) -- (1.4,-1.6)
	(-1.4,1.4) -- (-1.4,-1.6)
	(-1.4,1.4) -- (0,0) -- (1.4,1.4)--cycle
	(-1.4,-1.6) -- (0,-3) -- (1.4,-1.6);
	
	\draw
	(0,0.96) -- (0,1.4)
	(0,0.96) -- (-0.7,0.7)
	(0,0.96) -- (0.7,0.7);
	
	\draw
	(-0.7,0.7) -- (-1.4,-1.6)
	(0,0) -- (-0.7,-2.3)
	(0.7,0.7) -- (0,-3)
	(1.4,1.4) -- (0.7,-2.3);
	
	\draw 
	(-1.4,0.4) -- (0,-2)
	(0,-1) -- (1.4,-0.6);

	\draw[double, line width=0.6]
	(0,1.4) -- (-1.4,1.4)
	(-0.7,0.7) -- (0,0)
	(0.7,0.7) -- (1.4,1.4)
	(-1.4,1.4) -- (-1.4,0.4)
	(-1.4,-0.6) -- (-1.4,-1.6)
	(0,0) -- (0,-1)
	(0,-2) -- (0,-3)
	(1.4,1.4) -- (1.4,0.4)
	(1.4,-0.6) -- (1.4,-1.6)
	(-1.4,-1.6) -- (-0.7,-2.3)
	(0,-3) -- (0.7,-2.3);

	\draw[line width=2]
	(-1.4,1.4) -- (-0.7,0.7)
	(0,0) -- (0.7,0.7)
	(0,1.4) -- (1.4,1.4)
	(-0.7,-2.3) -- (0,-3)
	(0.7,-2.3) -- (1.4,-1.6)
	(-1.4,-1.6) -- (-1,-0.2857)
	(-0.4,-1.314) -- (0,0)
	(0.4,-0.8857) -- (0,-3)
	(1,-0.7143) -- (1.4,1.4);

\end{scope}

\node at (0,-2.4) {$QP^{\prime}_6$};

\end{scope}

\end{scope}

	\end{tikzpicture}
	\caption{Tilings by congruent $a^2bc$ quadrilaterals: $QP_6, QP^{\prime}_6$}
	\label{Tiling-a2bc-f24-a3}
\end{figure}

\begin{prop} Tilings exist for the following $\AVC$s

\begin{enumerate}
\item $QP_6, a^2bc, \AVC \equiv \{ \alpha^3, \beta^2\delta^2, \gamma^4 \}$,
\item $QP^{\prime}_6, a^2bc,  \AVC \equiv \{  \alpha^3, \alpha\beta^2, \alpha^2\delta^2, \beta^2\delta^2, \gamma^4 \}$,
\item $QP_6, a^3b, \AVC \equiv \{ \alpha^3, \beta^2\delta^2, \gamma^4 \}$.
\end{enumerate}

Their edges and angles are listed in Tables \ref{a2bc-a3b-f24-Al3-Edge-Angle}, \ref{a2bc-f24-Al3-AlBe2-Edge-Angle} respectively. There are three possible edge reductions as illustrated in Figure \ref{a2bcEdgeRed}: $c=a$ if and only if $\delta = \sin^{-1} (\frac{1}{6} (4+\sqrt{3}))^{\frac{1}{2}}$; $c=b$ if and only if $\delta=\frac{1}{2}\pi$; and $b=a$ if and only if $\delta=\pi - \sin^{-1} (\frac{1}{6} (4+\sqrt{3}))^{\frac{1}{2}}$.

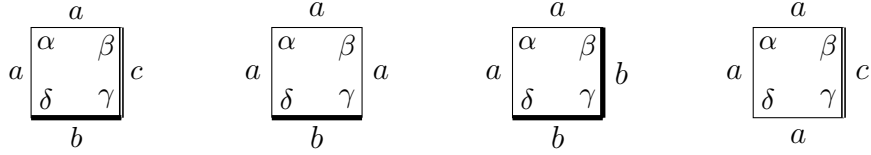
\begin{figure}[htp] 
\centering
\begin{tikzpicture}

\begin{scope}[] 

\draw
	(0,0) -- (1.2,0) -- (1.2,1.2) -- (0,1.2) -- cycle;

\draw[double, line width=0.6]
	(1.2,0) -- (1.2,1.2);

\draw[line width=2]
	(0,0) -- (1.2,0);

\node at (0.2,1) {\small $\alpha$};
\node at (1,0.925) {\small $\beta$};
\node at (1,0.25) {\small $\gamma$};
\node at (0.2,0.25) {\small $\delta$};

\node at (0.6,1.4) {\small $a$};
\node at (0.6,-0.25) {\small $b$};
\node at (-0.2,0.6) {\small $a$};
\node at (1.4,0.6) {\small $c$};

\end{scope} 

\begin{scope}[xshift = 3.2cm] 

\draw
	(0,0) -- (0,1.2) -- (1.2,1.2) -- (1.2, 0.0);

\draw[ line width=2]
	(0,0) -- (1.2,0);

\node at (0.2,1) {\small $\alpha$};
\node at (1,0.925) {\small $\beta$};
\node at (1,0.25) {\small $\gamma$};
\node at (0.2,0.25) {\small $\delta$};

\node at (0.6, -0.25) {$b$};
\node at (-0.25, 0.6) {$a$};
\node at (1.45, 0.6) {$a$};
\node at (0.6, 1.45) {$a$};

\end{scope}

\begin{scope}[xshift = 6.4cm] 

\draw
	(0,0) -- (0,1.2) -- (1.2,1.2) -- (1.2, 0.0) -- cycle;

\draw[ line width=2]
	(0,0) -- (1.2,0)
	(1.2,0) -- (1.2,1.2)
;

\node at (0.2,1) {\small $\alpha$};
\node at (1,0.925) {\small $\beta$};
\node at (1,0.25) {\small $\gamma$};
\node at (0.2,0.25) {\small $\delta$};

\node at (0.6, -0.25) {$b$};
\node at (-0.25, 0.6) {$a$};
\node at (1.45, 0.6) {$b$};
\node at (0.6, 1.45) {$a$};

\end{scope} 

\begin{scope}[xshift = 9.6cm] 

\draw
	(0,0) -- (0,1.2) -- (1.2,1.2) -- (1.2, 0.0) -- cycle;

\draw[double, line width=0.6]
	(1.2,0) -- (1.2,1.2);

\node at (0.2,1) {\small $\alpha$};
\node at (1,0.925) {\small $\beta$};
\node at (1,0.25) {\small $\gamma$};
\node at (0.2,0.25) {\small $\delta$};

\node at (0.6, -0.25) {$a$};
\node at (-0.25, 0.6) {$a$};
\node at (1.45, 0.6) {$c$};
\node at (0.6, 1.45) {$a$};

\end{scope} 

\end{tikzpicture}
\caption{Edge reductions of $a^2bc$: $c=a$, $c=b$, $b=a$}
\label{a2bcEdgeRed}
\end{figure}

\begin{table}[h]
\begin{center}
\scalebox{1}{
\bgroup
\def\arraystretch{1.75}
    \begin{tabular}[t]{ | c | c | c | c |}
	\hline
	\multicolumn{4}{|c|}{ $a^2bc, \AVC \equiv \{  \alpha^3, \beta^2\delta^2, \gamma^4 \}$ } \\ \hhline{|====|}
	 $f$ &  $24$ & $\alpha$ & $\frac{2}{3}\pi$ \\
	\hline 
	$a$ & $\sin^{-1}  \tfrac{1}{\sqrt{3}\sin \delta}$ & $\beta$ & $\pi - \delta$ \\
	\hline 
	$b$ & $\cos^{-1}  \frac{ \sqrt{3\sin^2 \delta - 1} - \cos \delta  }{2\sin \delta} $  & $\gamma$  & $\frac{1}{2}\pi$ \\
	\hline 
	$c$ & $\cos^{-1}  \frac{ \sqrt{3\sin^2 \delta-1} + \cos \delta }{2\sin \delta} $ & $\delta$  & $\frac{1}{4}\pi < \delta < \frac{3}{4}\pi$  \\
	\hline 
	\end{tabular}
\egroup
}
\end{center}
\caption{Angles and edges for $QP_6$, including edge reductions}
\label{a2bc-a3b-f24-Al3-Edge-Angle}
\end{table}

\begin{table}[h]
\begin{center}
\scalebox{1}{
\bgroup
\def\arraystretch{1.6}
    \begin{tabular}[t]{ | c | c | c | c | }
	\hline
	\multicolumn{4}{|c|}{ $a^2bc, \AVC \equiv \{\alpha^3, \alpha\beta^2, \alpha^2\delta^2, \beta^2\delta^2, \gamma^4 \}$ } \\ \hhline{|====|}
	 $f$ &  $24$ & $\alpha$ & $\frac{2}{3}\pi$ \\
	\hline 
	$a$ & $ \cos^{-1} \frac{\sqrt{5}}{3} $ & $\beta$ & $\frac{2}{3}\pi$ \\
	\hline 
	$b$ & $\cos^{-1}\frac{\sqrt{5}-1}{2\sqrt{3}} $  & $\gamma$ & $\frac{1}{2}\pi$ \\
	\hline 
	$c$ & $ \cos^{-1}  \frac{\sqrt{5}+1}{2\sqrt{3}} $ & $\delta$ & $\frac{1}{3}\pi$ \\
	\hline 
	\end{tabular}
\egroup
}
\end{center}
\caption{Angles and edges for $QP_6^{\prime}$}
\label{a2bc-f24-Al3-AlBe2-Edge-Angle}
\end{table}

\end{prop}

\begin{proof} Suppose the tilings exist and we calculate the angles and edge lengths. The angle sum system of $\{ \alpha^3, \beta^2\delta^2, \gamma^4 \}$ implies 
\begin{align*}
f=24, \quad \alpha=\tfrac{2}{3}\pi, 
\quad \beta + \delta=\pi, \quad  
\gamma=\tfrac{1}{2}\pi. 
\end{align*}
So every angle of the quadrilateral $\square ABCD$ is $<\pi$ and the tile is convex. Then $0<b,c<\pi$. The area of $\square ABCD$ is $\frac{1}{6}\pi < \alpha$. Then Lemma \ref{a2bc-aEdge-Lem} implies $0 < a < \frac{1}{2}\pi$ and $0 < \cos a, \sin a < 1$. By \eqref{TrigEqca}, \eqref{TrigEqcb}, \eqref{AEsb}, \eqref{TrigEqcc}, \eqref{TrigEqsc}, we get
\begin{align*} 
&\cos a = \tfrac{\sqrt{3\sin^2 \delta - 1}}{\sqrt{3}\sin \delta }, &
 &\cos b = \sin c= \tfrac{\sqrt{3\sin^2 \delta - 1}-\cos \delta}{2\sin \delta}, &\\
&\sin a = \tfrac{1}{\sqrt{3}\sin \delta},&
&\sin b= \cos c= \tfrac{\sqrt{3\sin^2 \delta-1}+\cos \delta }{2\sin \delta}.&
\end{align*}
These identities determine $a, b, c$ and imply $b+c=\frac{1}{2}\pi$, which means $0< b, c < \frac{1}{2}\pi$. 

Conversely, Lemma \ref{a2bcExistLem} assures the existence of an $a^2bc$ quadrilateral with the given angles and edges satisfying the corresponding values determined above. Furthermore, as $a,b,c \in (0, \frac{1}{2}\pi)$ and all angles $<\pi$, Lemma \ref{SimpQuadLem} implies that the quadrilateral is simple. Therefore the range of $\delta$ is determined by all the sine and cosine values above being positive. This is equivalent to $\frac{1}{4}\pi < \delta < \frac{3}{4}\pi$.

We now discuss the three edge reductions, $c=a$, $b=a$ and $c=b$. 

For $c = a$ and $\delta \in (\frac{1}{4}\pi, \frac{3}{4}\pi)$, the equalities $\cos a =\cos c$ and $\sin a = \sin c$ is equivalent to $\delta = \sin^{-1}  ( \tfrac{1}{6} (4+\sqrt{3}) )^{\frac{1}{2}}$.

For $b = a$, by similar argument, $\delta = \pi - \sin^{-1} (  \tfrac{1}{6} (4+\sqrt{3}) )^{\frac{1}{2}}$. 

For $c = b$, by $b+c=\frac{1}{2}\pi$, we get $b=c=\frac{1}{4}\pi$ and the quadrilateral is a kite such that $\beta=\delta=\frac{1}{2}\pi$. 

The three reductions correspond to three different values of $\delta$. So further reduction $a=b=c$ is impossible.

For $\AVC \equiv \{ \alpha^3, \alpha\beta^2, \alpha^2\delta^2, \beta^2\delta^2, \gamma^4 \}$, $\alpha = \beta = \frac{2}{3}\pi$, $\gamma = \frac{1}{2}\pi$, $\delta = \frac{1}{3}\pi$, we get
\begin{align*}
\cos a = \tfrac{1}{3}\sqrt{5}, \quad \cos b = \tfrac{1}{2\sqrt{3}}(\sqrt{5}-1), \quad \cos c =  \tfrac{1}{2\sqrt{3}}(\sqrt{5}+1).
\end{align*}
This completes the proof.
\end{proof}

We give an alternative justification for the existence of $QP_6$ by congruent almost equilateral quadrilaterals ($a^3b$) with $\AVC \equiv \{ \alpha^3, \beta^2\delta^2, \gamma^4 \}$. As every angle $<\pi$ and $\alpha>\gamma$, Lemma \ref{ATriLem} implies $\beta > \delta$. By $\beta + \delta =\pi$, we have $\beta > \frac{1}{2}\pi > \delta$. Simplifying \eqref{Coolsaet-Id}, we get $( 4\sqrt{3} \cos^2 \tfrac{1}{2}\delta - 3\sqrt{3} + 1 ) \cos \tfrac{1}{2}\delta = 0$. By $0<\delta < \frac{1}{2}\pi$, we have $\cos \delta = \frac{1}{2}(1-\frac{\sqrt{3}}{3})$. Then by \eqref{ca3Ang}, 
\begin{align*}
\cos a = \tfrac{\sin\alpha\cos\delta+\sin\gamma}{(1-\cos\alpha)\sin\delta} =( \tfrac{1}{13} (5 + 2\sqrt{3}) )^{\frac{1}{2}} \approx 0.8069,
\end{align*}
which uniquely determines $0 < a < \pi$. Then Lemma \ref{AEQuadExists} implies that the quadrilateral exists. By \eqref{TrigEqcb}, \eqref{AEsb}, we get
\begin{align*}
\cos b = ( \tfrac{2}{13} (4 - \sqrt{3} ) )^{\frac{1}{2}} = \sin a, \quad
\sin b = ( \tfrac{1}{13}(5+2\sqrt{3}) )^{\frac{1}{2}} = \cos a.
\end{align*} 
This implies $a + b = \frac{1}{2}\pi$. Lemma \ref{SimpQuadLem} implies that the quadrilateral is simple. So the tile exists and is convex. The angles and edges are listed in Table \ref{SpecialTilingData2}.

\subsection{Tiles \boldmath{$a^2bc, a^3b$} and Vertex \boldmath{$\alpha\beta^2$}} \label{Subsec-Geom-albe2}

The discussion includes tilings by congruent $a^2bc$ quadrilaterals with $\AVC = \{ \alpha\beta^2, \alpha^2\delta^2, \gamma^4, \alpha\delta^{\frac{f+8}{8}}, \beta^2\delta^{\frac{f-8}{8}}, \delta^{\frac{f}{4}} \}$ and tilings by congruent almost equilateral quadrilaterals with $\AVC = \{ \alpha\beta^2, \alpha^2\delta^2, \gamma\delta^3, \alpha\gamma^3\delta, \gamma^6 \}$. From \cite{cl, cly}, we know that the former constitutes a family of tilings. Figure \ref{E5EMT}, equivalent to \cite[$E_{\square}5$ of Figure 5]{cly}, illustrates three timezones of this earth map tiling.

\begin{figure}[htp]
	\centering


\begin{tikzpicture}[>=latex,scale=0.8]

\tikzmath{ 
\l=0.6;
}

\fill[gray!25]
	(0.0,0.0) -- (0.0,-2*\l) -- (\l,-2*\l) -- (\l,-4*\l) -- (2*\l,-4*\l) -- (2*\l,-6*\l) -- (6*\l,-6*\l) -- (6*\l,-4*\l) -- (5*\l,-4*\l) -- (5*\l,-2*\l) -- (4*\l,-2*\l) -- (4*\l,0) -- (2*\l,0) 
;

\foreach \b in {0,1,2,3,4,5,6}{

	\begin{scope}[xshift=1.2*\b cm] 
	\draw[]
	(0,0) -- (0,-2*\l)	
	(\l, -2*\l) -- (\l, -4*\l)
	(2*\l, -4*\l) -- (2*\l, -6*\l);

	\end{scope}
}

\foreach \b in {0,2,4,6}{

	\begin{scope}[xshift=1.2*\b cm] 
	\draw[]
	(0,-2*\l) -- (\l, -2*\l) 
	(\l, -4*\l) -- (2*\l,-4*\l);
	\end{scope}

}

\foreach \b in {0,2,4}{

	\begin{scope}[xshift=1.2*\b cm] 
	\draw[]
	(3*\l, -2*\l) -- (4*\l, -2*\l)
	(2*\l, -4*\l) -- (3*\l, -4*\l);
	\end{scope}

}


\foreach \b in {0, 2, 4}{

	\begin{scope}[xshift=1.2*\b cm] 
	
	\draw[double, line width=0.6]
	(\l, -2*\l) -- (2*\l, -2*\l) -- (3*\l, -2*\l)
	(3*\l, -4*\l) -- (4*\l, -4*\l) -- (5*\l, -4*\l);

	\end{scope}
}

\foreach \b in {0, 2, 4}{

	\begin{scope}[xshift=1.2*\b cm] 
	
	\draw[line width=2]
	(2*\l, 0) -- (2*\l, -2*\l)
	(2*\l, -2*\l) -- (2*\l, -4*\l)
	(4*\l, -2*\l) -- (4*\l, -4*\l)
	(4*\l, -4*\l) -- (4*\l, -6*\l);
	\end{scope}
}

\node at (0.0,-3*\l) {\small $\cdots$};

\node at (14*\l,-3*\l) {\small $\cdots$};

\end{tikzpicture}
\caption{Earth map tiling $E_{\square}5$ with $\AVC \equiv \{ \alpha\beta^2, \alpha^2\delta^2, \gamma^4, \delta^{\frac{f}{4}} \}$}
\label{E5EMT}
\end{figure}
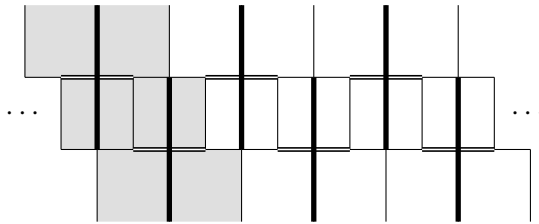

\begin{prop} Tilings for $a^2bc, \AVC = \{ \alpha\beta^2, \alpha^2\delta^2, \gamma^4, \alpha\delta^{\frac{f+8}{8}}, \beta^2\delta^{\frac{f-8}{8}},$ $\delta^{\frac{f}{4}} \}$ exist. The edges and angles are listed in Table \ref{AlBe2-Al2De2-AngEdge}.

\begin{table}[h]
\begin{center}
\scalebox{1}{
\bgroup
\def\arraystretch{1.75}
    \begin{tabular}[t]{ | c | c | c | c | }
	\hline
	\multicolumn{4}{|c|}{ $\AVC \equiv \{ \alpha\beta^2, \alpha^2\delta^2, \gamma^4, \alpha\delta^{\frac{f+8}{8}}, \beta^2\delta^{\frac{f-8}{8}}, \delta^{\frac{f}{4}} \}$ } \\ \hhline{|====|}
	 $f$ &  $\ge16$ & $\alpha$ & $( 1 - \frac{8}{f} ) \pi$ \\
	\hline 
	$a$ & $\cos^{-1} ( 1 - \tfrac{1}{4}(3-\sqrt{5})\sec^2 \tfrac{4}{f}\pi )$ & $\beta$ & $(\frac{1}{2}+\frac{4}{f})\pi$ \\
	\hline 
	$b$ & $ \cos^{-1}  \tfrac{\sqrt{5} -1}{4} \sec \frac{4}{f}\pi   $  & $\gamma$ & $ \frac{1}{2}\pi $ \\
	\hline 
	$c$ & $\cos^{-1} ( (3 - \sqrt{5} )\cos \frac{4}{f}\pi + ( \sqrt{5}-2 ) \sec \frac{4}{f}\pi )$ & $\delta$ & $\frac{8}{f}\pi$ \\
	\hline 
	\end{tabular}
\egroup
}
\end{center}
\caption{Angles and edges for earth map tilings $E_{\square}5$ and their flip modifications}
\label{AlBe2-Al2De2-AngEdge}
\end{table}
\end{prop}

\begin{proof} The angle sum system of $\alpha\beta^2, \alpha^2\delta^2, \gamma^4$ gives
\begin{align*}
\alpha= ( 1 - \tfrac{8}{f} ) \pi, \quad
\beta= (\tfrac{1}{2}+\tfrac{4}{f} )\pi, \quad 
\gamma=\tfrac{1}{2}\pi, \quad
\delta=\tfrac{8}{f}\pi.
\end{align*}
So the tile is convex.

By \cite[Proposition 33]{cly}, the tilings have $f\ge16$. Then $\alpha > \frac{4}{f}\pi$. Lemma \ref{a2bc-aEdge-Lem} implies $0 < a < \frac{1}{2}\pi$. So $0 < \cos a < 1$. Substituting the angle formulae into \eqref{TrigEqca}, we obtain
\begin{align*} 
\sin \tfrac{1}{2}\delta(
&- ( \cos^2 \delta + 2\cos \delta  + 1  )\cos^2 a \\
&+ ( 2\cos^2\delta + \cos \delta - 1 ) \cos a \\ 
&+ ( -\cos^2\delta + \cos\delta + 1 )
)= 0.
\end{align*}
By $\cos a> 0$ and $\sin \tfrac{1}{2}\delta \neq 0$, the solution to the above quadratic equation is $\cos a=\tfrac{2\cos \delta + \sqrt{5} - 1 }{2\cos \delta+2}$, which uniquely determines $a \in (0, \frac{1}{2}\pi)$. Then by \eqref{TrigEqcb}, \eqref{TrigEqcc} and $\delta=\frac{8}{f}\pi$, we get $\cos a, \cos b$ and $\cos c$ in terms of $f$ below,
\begin{align*} 
&\cos a = 1 - \tfrac{1}{4}(3-\sqrt{5})\sec^2 \tfrac{4}{f}\pi, \\
&\cos b = \tfrac{1}{4}(\sqrt{5}-1)\sec \tfrac{4}{f}\pi, \\
&\cos c = (3-\sqrt{5})\cos \tfrac{4}{f}\pi + (\sqrt{5}-2) \sec \tfrac{4}{f}\pi.
\end{align*}
On the other hand, by \eqref{AEsb}, \eqref{TrigEqsc}, we get $\sin b = \frac{1}{2}(\sqrt{5}+1)\cos \frac{1}{2}\delta \sin a$ and $\sin c =  \frac{1}{2}(\sqrt{5}-1)\sin \delta \sin a$ respectively. The ranges of $\delta$ and $a$ imply $\sin b, \sin c>0$. Then we have $0<b, c<\pi$ and $\cos b, \cos c$ uniquely determine $b,c$ respectively. Lemma \ref{a2bcExistLem} implies that the $a^2bc$ quadrilateral with given angles exists. As every angle $<\pi$, Lemma \ref{SimpQuadLem} implies that the quadrilateral is simple. So the tile exists. 

Now we discuss the edge reductions, $c=a$ and $b=a$ and $b=c$.

By $0 < \cos \tfrac{4}{f}\pi < 1$, the equalities $\cos a= \cos b$ and $\cos a= \cos c$ have no integer solution for $f\ge16$. If $b=c$, then the quadrilateral is a kite such that $\beta=\delta$, which implies $f=8$. Then $b \neq c$ for $f>8$. So there is no edge reduction. 

This completes the proof.
\end{proof}

The same approach applies to $a^3b, \AVC = \{ \alpha\beta^2, \alpha^2\delta^2, \gamma\delta^3, \alpha\gamma^3\delta, \gamma^6 \}$. Nevertheless, a more efficient approach is presented in the next proposition. 

\begin{prop}\label{albe2-al2de2-ExistProp} The tiling for $\AVC = \{ \alpha\beta^2, \alpha^2\delta^2, \gamma\delta^3, \alpha\gamma^3\delta, \gamma^6 \}$ exists and it is $S5$. The values for edges and angles are listed in Table \ref{SpecialTilingData2}.
\end{prop}

\begin{proof} The angle sum system gives
\begin{align*}
f =36, \quad
\alpha=\tfrac{4}{9}\pi, \quad 
\beta=\tfrac{7}{9}\pi, \quad 
\gamma=\tfrac{1}{3}\pi, \quad
\delta=\tfrac{5}{9}\pi.
\end{align*}
By $a, \alpha, \beta, \gamma, \delta < \pi$ and $\beta>\delta$, Lemma \ref{b-criteria1} implies $b < \pi$. The angles are (Type II) solution to \eqref{Coolsaet-Id}. By \eqref{ca3Ang} and \eqref{TrigEqcb}, 
\begin{align*}
&\cos a =\tfrac{ \sin \frac{2}{9}\pi + 2 \cos \frac{1}{18}\pi } { \sqrt{3}  (1 + \cos \frac{2}{9}\pi ) },\\
&\cos b=\tfrac{1}{3} (4 \sin^2\tfrac{1}{9}\pi - \sqrt{3} \tan \tfrac{1}{18}\pi + 2\sqrt{3} \cos \tfrac{2}{9}\pi \tan \tfrac{1}{18}\pi + 4 \cos \tfrac{1}{18}\pi \tan\tfrac{1}{9}\pi ).
\end{align*}
Lemma \ref{AEQuadExists} implies that the quadrilateral exists. As $ b < \pi$, the equation for $\cos b > 0$ uniquely determines $b$. In fact, $\cos b$ is the biggest root of $9x^3 + 9x^2 - 9x - 1=0$. By $a,b<\pi$ and every angle $<\pi$, Lemma \ref{SimpQuadLem} implies that the quadrilateral is simple. 
\end{proof}

\begin{prop} The tiling for $\AVC = \{ \alpha\beta^2, \alpha^2\gamma\delta, \gamma^2\delta^2 \}$ exists and it is $S4$. The edges and angles are listed in Table \ref{SpecialTilingData2}.
\end{prop}

\begin{proof}The angle sum system gives
\begin{align*}
f=16, \quad
\alpha = \tfrac{1}{2}\pi, \quad
\beta = \tfrac{3}{4}\pi, \quad
\gamma + \delta = \pi.
\end{align*}
Then the tile is convex. So Lemma \ref{ExchLem} implies $\alpha < \beta$ and $\gamma < \delta$. By $\gamma+\delta =\pi$, we get $\delta > \frac{1}{2}\pi > \gamma$ and $\cos \delta < 0$ and $\alpha > \gamma$. Then Lemma \ref{b-criteria1} implies $b < \pi$. By \eqref{Coolsaet-Id},
\begin{align*}
\tan \delta = - \tfrac{ 2\sin \frac{3}{8} \pi }{ \sin \frac{3}{8} \pi - \cos \frac{ 3}{8}\pi } =  -(2+\sqrt{2}).
\end{align*} 
By \eqref{ca3Ang} and \eqref{TrigEqcb}, we get
\begin{align*}
\cos a &=  \tfrac{1}{2}\sqrt{2}, \\
\cos b &= \tfrac{1}{4} (2\sqrt{2} - 1).
\end{align*}
Lemma \ref{AEQuadExists} implies that the quadrilateral exists. By $b < \pi$, the equation for $\cos b$ uniquely determines $b$. By $a,b<\pi$ and every angle $<\pi$, Lemma \ref{SimpQuadLem} implies that the quadrilateral is simple.
\end{proof}

\subsection{Tile \boldmath{$a^3b$} and Vertex \boldmath{$\alpha\gamma^2$} or \boldmath{$\alpha\delta^2$}} \label{Subsec-Geom-alga2-alde2}

\begin{prop}\label{Geom-AlG2-Be2De2} The tiling for $\AVC \equiv \{ \alpha\gamma^2, \beta^2\delta^2, \alpha^{\frac{f}{4}} \}$ exists only for $f=16$ and it is $S2$. The values of angles and edges are given in Table \ref{SpecialTilingData1}.
\end{prop}

\begin{proof} The angle sum system gives
\begin{align*}
\alpha = \tfrac{8}{f}\pi, \quad
\beta + \delta = \pi, \quad
\gamma = ( 1 - \tfrac{4}{f} )\pi.
\end{align*}
By $\beta + \delta = \pi$, we get $\delta < \pi$. Then by Lemma \ref{TriQuadLem}, we get $\alpha \neq \gamma$. By $\alpha\gamma^2$, this implies $\frac{f}{4}\ge4$ in $\alpha^{\frac{f}{4}}$. So $f\ge16$ and every angle $<\pi$. 

By convexity and $\alpha < \gamma$, Lemma \ref{ATriLem} implies $\beta<\delta$. Then $\beta + \delta = \pi$ implies $\beta < \frac{1}{2}\pi < \delta$.

Assume $f \ge 20$. We get $\alpha \le \frac{2}{5}\pi$ and $\gamma \ge \frac{4}{5}\pi$. By convexity, Lemma \ref{LunEstLem} implies $\gamma + \delta < \pi + \beta$. Combined with $\beta + \delta = \pi$ and $\gamma \ge \frac{4}{5}\pi$, we get $\frac{4}{5}\pi + \pi - \beta < \pi + \beta$ which gives $\beta > \frac{2}{5}\pi \ge \alpha$ and $\delta < \frac{3}{5}\pi$. By $\alpha < \beta$, Lemma \ref{ExchLem} implies $\gamma < \delta$, which means $\frac{4}{5}\pi < \frac{3}{5}\pi$, a contradiction. Hence $f=16$ and
\begin{align*}
\alpha = \tfrac{1}{2}\pi, \quad
\beta + \delta = \pi, \quad
\gamma =\tfrac{3}{4}\pi.
\end{align*}
The solution to \eqref{Coolsaet-Id} for $\frac{1}{2}\pi < \delta < \pi$ is
\begin{align*} 
\cos \delta =  \tfrac{1}{2}( 1 - \sqrt{2} ), \quad \sin \delta = \tfrac{ 1 }{2} ( 1+2\sqrt{2} )^{\frac{1}{2}}.
\end{align*}
By \eqref{ca3Ang} and \eqref{TrigEqcb},
\begin{align*}
\cos a &= (  \tfrac{1}{7} (2\sqrt{2} - 1 ) )^{\frac{1}{2}}, \\
\cos b &= ( \tfrac{1}{7}( 22\sqrt{2} - 25) )^{\frac{1}{2}}.
\end{align*}
Lemma \ref{AEQuadExists} implies that the quadrilateral exists. We also have $\cos(\beta - \delta) + \cos \gamma > 0$. Then Lemma \ref{b-criteria2} implies $b<\pi$ and hence $\cos b$ uniquely determines $b$. By $a,b<\pi$ and every angle $<\pi$, Lemma \ref{SimpQuadLem} implies that the quadrilateral is simple. 
\end{proof}

\begin{prop}\label{Geom-AlGa2-AlBeDe2} The tilings for $\AVC \equiv \{ \alpha\gamma^2, \alpha\beta\delta^2, \beta^{\frac{f}{4}} \}$ exist only for $f=16$ and they are $S3, S'3$. The values of angles and edges are given in Table \ref{SpecialTilingData1}.
\end{prop}

\begin{proof} The degree $\ge3$ vertex $\beta^{\frac{f}{4}}$ implies $f \equiv 0 \mod 4$ and $f\ge12$. By $\alpha\gamma^2$, we have $\gamma < \pi$. The angle sum system gives 
\begin{align*}
\alpha = 2\pi - 2\gamma, \quad
\beta = \tfrac{8}{f}\pi, \quad
\gamma = \delta + \tfrac{4}{f}\pi.
\end{align*}
We have $\gamma > \delta$. Then Lemma \ref{ExchLem} implies $\alpha > \beta$. The angles for $f=16$ are Type I solution to \eqref{Coolsaet-Id}.

We first consider $f\neq16$. By \eqref{Coolsaet-Id} and $0 < \gamma < \pi$, we have $\sin \gamma > 0$ and 
\begin{align}\label{tanGa}
\tan \gamma = \tfrac{ 2 ( \cos \frac{1}{2}\beta - 1 )\sin \frac{1}{2}\beta }{ \cos \beta }.
\end{align}
For $f>16$, we get $\beta < \frac{1}{2}\pi$. Then \eqref{tanGa} implies $\tan \gamma < 0$. This means $\tfrac{1}{2}\pi < \gamma < \pi$. By $\gamma>\frac{1}{2}\pi$ and $\alpha\gamma^2$, we get $\alpha < \pi$. Hence the tile is convex. Then Lemma \ref{LunEstLem} implies $\beta + \pi > \gamma + \delta$. This implies $\gamma < ( \frac{1}{2} + \frac{6}{f} )\pi$. 

Assume $f\ge 36$. Then $\gamma < ( \frac{1}{2} + \frac{6}{f} )\pi \le \frac{2}{3}\pi$ and $\alpha\gamma^2$ imply $\alpha > \frac{2}{3}\pi > \gamma $. By $\pi>\gamma>\delta$, Lemma \ref{ATriLem} implies $\beta > \delta$ which means $ \gamma < \frac{12}{f}\pi$, which implies $f<24$, a contradiction. 

If $20 \le f \le 32$, we have $0 < \beta \le \frac{2}{5}\pi$. By $\alpha\gamma^2$, we have $\tan\gamma=-\tan \frac{1}{2}\alpha$. Then by \eqref{tanGa} and $0 < \beta \le \frac{2}{5}\pi$, we get $\tan \frac{1}{2}\alpha \le \tan \frac{1}{2}\beta$, which implies $\alpha \le \beta$, contradicting $\alpha > \beta$. 

If $f=12$, then $\beta = \frac{2}{3}\pi$ and $\delta = \gamma - \frac{1}{3}\pi$. Then \eqref{tanGa} gives $\gamma = \frac{1}{3}\pi$. This implies $\delta = 0$, a contradiction.

We conclude $f=16$ and $\beta^4$ is a vertex. The angle sum system gives
\begin{align*}
\alpha = \pi, \quad
\beta = \tfrac{1}{2}\pi, \quad
\gamma = \tfrac{1}{2}\pi, \quad
\delta = \tfrac{1}{4}\pi.
\end{align*}
As $\alpha = \pi$, the quadrilateral is in fact an isosceles triangle $\triangle BCD$ with the top angle $\delta=\frac{1}{4}\pi$ at the north pole and base angles $\beta = \gamma = \frac{1}{2} \pi$ at the equator. Hence
\begin{align*}
a = \tfrac{1}{4}\pi, \quad 
b = \tfrac{1}{2}\pi. 
\end{align*}
This means that $\triangle BCD$ is half of a face of the regular octahedron. The tilings are effectively non-edge-to-edge tilings by congruent isosceles triangles, which serve as examples to the isosceles triangle with angles $(\tfrac{4}{f}\pi, \tfrac{1}{2}\pi, \tfrac{1}{2}\pi)$ in \cite[Table 1]{da}.
\end{proof}

\begin{prop}\label{Geom-AlDe2-AlBeGa2} The tilings for $\AVC \equiv  \{ \alpha\delta^2, \alpha\beta\gamma^2, \beta^{\frac{f}{4}} \}$ exist only for $f=12, 16$ and they are $S_{12}1, S_{16}1$. The values of angles and edges are listed in Table \ref{SpecialTilingData1}.
\end{prop}

\begin{proof} The degree $\ge3$ vertex $\beta^{\frac{f}{4}}$ implies $f \equiv 0 \mod 4$ and $f\ge12$. By $\alpha\delta^2, \alpha\beta\gamma^2$, we have $\gamma, \delta < \pi$. The angle sum system gives 
\begin{align*}
\alpha = 2\pi - 2\delta, \quad
\beta = \tfrac{8}{f}\pi, \quad
\gamma = \delta - \tfrac{4}{f}\pi.
\end{align*}

By $\gamma < \delta$ and Lemma \ref{ExchLem}, we get $\alpha < \beta$ and $\delta > ( 1 - \frac{4}{f} )\pi$. Then we have $( 1 - \frac{4}{f} )\pi < \delta < \pi$. By $f\ge12$, we have every angle $<\pi$. By convexity, $C$ lies outside the isosceles triangle $\triangle ABD$ in Figure \ref{ConvexAEQuad}. Then we have $\alpha+2\beta>\alpha+2\theta>\pi$. By $\alpha < \beta$, this implies $\beta > \frac{1}{3}\pi$. So we get $f<24$. This means $f=12, 16, 20$.

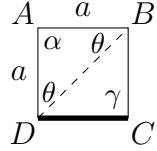
\begin{figure}[htp] 
\centering
\begin{tikzpicture}

\begin{scope}[] 

\draw
	(0,0) -- (0,1.2) -- (1.2,1.2) -- (1.2, 0.0);

\draw[line width=2]
	(0,0) -- (1.2,0);

\draw[dashed]
	(0,0) -- (1.2,1.2) 
;

\node at (-0.2, -0.2) {$D$};
\node at (-0.2, 1.4) {$A$};
\node at (1.4, 1.4) {$B$};
\node at (1.4, -0.2) {$C$};

\node at (0.2,1) {\small $\alpha$};
\node at (0.8,1) {\small $\theta$};
\node at (1,0.25) {\small $\gamma$};
\node at (0.15,0.35) {\small $\theta$};

\node at (-0.25, 0.6) {$a$};
\node at (0.6, 1.45) {$a$};

\end{scope}

\end{tikzpicture}
\caption{Convex almost equilateral quadrilateral and $\triangle ABD$}
\label{ConvexAEQuad}
\end{figure}

By \eqref{Coolsaet-Id}, we get
\begin{align} \label{AlDe2-Coolsaet}
\sin \delta \sin ( \delta - \tfrac{4}{f}\pi ) = - \sin \tfrac{4}{f}\pi \sin ( 2\delta - \tfrac{4}{f}\pi ).
\end{align}
Meanwhile, we have $\pi + \beta - \gamma - \delta = (1+\frac{12}{f})\pi - 2\delta$. 

For $f=20$ and $( 1 - \frac{4}{f} )\pi < \delta < \pi$, there is no solution to \eqref{AlDe2-Coolsaet}. 

For $f=12, 16$ and $( 1 - \frac{4}{f} )\pi < \delta < \pi$, the solutions to \eqref{AlDe2-Coolsaet} are 
\begin{align*}
&f=12,&
&\cos \delta =  -\tfrac{1}{4}\sqrt{10},&
&\delta \approx 0.7902\pi;& \\
&f=16,&
&\tan \delta = 2 - \sqrt{5} - (7 - 3 \sqrt{5})^{\frac{1}{2}},&
&\delta \approx 0.7898\pi.&
\end{align*}
By \eqref{ca3Ang} and \eqref{TrigEqcb},
\begin{align*}
&f=12,&
&\cos a = \tfrac{2}{3}\sqrt{5} - 1,& \\
&& &\cos b = 3 \sqrt{5} - 6;& \\
&f=16,& 
&\cos a = \tfrac{1}{2}(-3 - \sqrt{2} + \sqrt{5} + \sqrt{10} ),& \\
&&&\cos b = -9 - 6\sqrt{2} + 4\sqrt{5} + 3\sqrt{10}.&
\end{align*}
In both cases, Lemma \ref{AEQuadExists} implies that the quadrilateral exists. 

Meanwhile, for $f=12,16$, we have $\pi + \beta - \gamma - \delta = (1+\frac{12}{f})\pi - 2\delta > 0$. Lemma \ref{b-criteria2} implies $b<\pi$. So $\cos b$ uniquely determines $b$ and Lemma \ref{SimpQuadLem} implies that the quadrilateral is simple.
\end{proof}

\begin{prop} The tiling for $\AVC \equiv \{ \alpha\delta^2, \alpha\beta^3, \gamma^3\delta, \alpha^2\beta\gamma^2 \}$ exists and it is $S6$. The values of edges and angles are listed in Table \ref{SpecialTilingData2}. 
\end{prop}

\begin{proof} The angle sum system gives
\begin{align*}
f=36, \quad
\alpha=\tfrac{1}{3}\pi, \quad 
\beta=\tfrac{5}{9}\pi, \quad 
\gamma=\tfrac{7}{18}\pi, \quad 
\delta=\tfrac{5}{6}\pi.
\end{align*}
The angles are (Type II) solutions to \eqref{Coolsaet-Id}. By \eqref{ca3Ang} and \eqref{TrigEqcb}, 
\begin{align*}
&\cos a = 4 \cos \tfrac{1}{9}\pi - 3, \\
&\cos b = 6 \cos \tfrac{1}{9}\pi + 2 \sqrt{3} \sin \tfrac{1}{9}\pi - 3 \sqrt{3} \tan \tfrac{1}{9}\pi  - 4.
\end{align*}
Then Lemma \ref{AEQuadExists} implies that the quadrilateral exists. Since every angle $<\pi$ and $\pi + \beta > \gamma + \delta$, Lemma \ref{b-criteria2} implies $b<\pi$. So $\cos b$ uniquely determines $b$. 
In fact, $\cos b$ is the biggest root to $x^3 + 39 x^2 + 39 x - 71=0$. Lemma \ref{SimpQuadLem} implies that the quadrilateral is simple. 
\end{proof}

\subsection{Tiles \boldmath{$a^2bc, a^3b$} and Vertex $\alpha\gamma\delta$} \label{Subsec-Geom-algade}

It remains to discuss the tilings with $\alpha\gamma\delta$ as a vertex, i.e., the earth map tilings $E$ and their flip modifications $E^{\prime}, E^{\prime\prime}$ and rearrangement $E^{\prime\prime\prime}$. We also include the discussion of the earth map tilings by congruent $a^2bc$ and show the relation between these two types of earth map tilings via edge reduction.

As seen in \cite{cl, cly}, there is a family of earth map tilings by the $a^2bc$ quadrilateral in Figure \ref{StdQuad} and $\AVC \equiv \{ \beta\gamma\delta, \alpha^{\frac{f}{2}} \}$. If $c=a$, then we obtain $E$ in the first picture of Figure \ref{EMT}.

The existence of $a^2bc$ quadrilateral in Figure \ref{StdQuad} with area $\alpha = \frac{4}{f}\pi$ and distinct edges $a,b,c$ is guaranteed by the next proposition. Note that, by the quadrilateral angle sum, $\alpha = \frac{4}{f}\pi$ is equivalent to $\beta+ \gamma + \delta = 2\pi$. Similarly, $\beta = \frac{4}{f}\pi$ is equivalent to $\alpha + \gamma + \delta = 2\pi$ in $a^3b$.

The following proposition is a result of the proof of \cite[Proposition 29]{cly}. 

\begin{prop}[] \label{a2bcQuad-Area} If $\alpha = \frac{4}{f}\pi$, for $f\ge6$, there exists $a \in (0, \pi)$ and one-parameter family of simple $a^2bc$ quadrilaterals with distinct edges $a,b,c$ and area $\alpha$. Hence tilings exist for $a^2bc, \AVC \equiv \{ \beta\gamma\delta, \alpha^{\frac{f}{2}} \}$ for each $f\ge6$.  
\end{prop}

The quadrilateral in the proposition is exactly a simple $\square ABCD$, satisfying $AB=AD=a$ and $\angle A = \alpha$ and area $=\alpha$. This quadrilateral can take various shapes, illustrated in Figure \ref{ALunes}. We remark that the area $=\alpha$ is equivalent to $AC=\pi - a$, which is also equivalent to $A^{\ast}C=a$. For quadrilateral in the first, third and fourth picture, by Lemma \ref{a2bc-aEdge-Lem} the range of $a$ is $(0, \frac{1}{2}\pi)$ whereas for the quadrilateral in the second picture the range of $a$ is $(\frac{1}{2}\pi, \pi)$. For the first two pictures, $\angle C = \pi$ if and only if $a = \frac{1}{2}\pi$. We remark that the discussion of the shapes is independent of tilings and $\alpha$ can take any value in $(0,\pi)$.

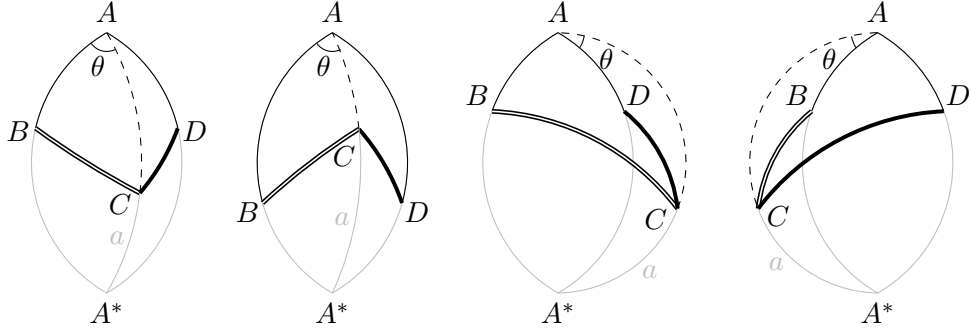
\begin{figure}[htp] 
\centering
\begin{tikzpicture}[scale=1]


\begin{scope}[] 

\tikzmath{
\r=2; \R=sqrt(3);
\a=120; \aa=\a/2;
\xP=0; \yP=\R;
\aOP=64; 
\aPOne=\aOP;
\mPOne=tan(90+\aPOne);
\aPTwo=\aOP;
\mPTwo=tan(-90-\aPTwo);
\xCQ=-3; \yCQ=0;
\xCOne = sqrt(\r^2-\R^2); \yCOne=0;
\xCTwo =-sqrt(\r^2-\R^2); \yCTwo=0;
}

\coordinate (O) at (0,0);
\coordinate (C1) at (\xCOne,0);
\coordinate (C2) at (\xCTwo,0);


	(O) circle (\R);
	
	(C2) circle (\r)
	(C1) circle (\r);

\draw[gray!50]
	([shift={(-\aa:\r)}]\xCTwo,0) arc (-\aa:\aa:\r);

\draw[gray!50]
	([shift={(-\aa+180:\r)}]\xCOne,0) arc (-\aa+180:\aa+180:\r);

\draw[] ([shift=(90:\R)]215:0.25) arc (215:296:0.25)
;

\pgfmathsetmacro{\xPOne}{ ( 2*((\mPOne)*(\yCOne)+(\xCOne)) - sqrt( ( 2*( (\mPOne)*(\yCOne)+\xCOne )  )^2 - 4*( (\mPOne)^2 + 1 )*( (\xCOne)^2 + (\yCOne)^2 - \r^2 ) ) )/( 2*( (\mPOne)^2+1 ) ) };
\pgfmathsetmacro{\yPOne}{ \mPOne*\xPOne };

\pgfmathsetmacro{\xPTwo}{ ( 2*((\mPTwo)*(\yCTwo)+(\xCTwo) ) + sqrt( ( 2*( (\mPTwo)*\yCTwo+\xCTwo )  )^2 - 4*( (\mPTwo)^2 + 1 )*( (\xCTwo)^2 + (\yCTwo)^2 - \r^2 ) ) )/( 2*( (\mPTwo)^2+1 ) ) };
\pgfmathsetmacro{\yPTwo}{ \mPTwo*\xPTwo };

\pgfmathsetmacro{\xPPOne}{ ( 2*((\mPOne)*(\yCOne)+(\xCOne)) + sqrt( ( 2*( (\mPOne)*(\yCOne)+\xCOne )  )^2 - 4*( (\mPOne)^2 + 1 )*( (\xCOne)^2 + (\yCOne)^2 - \r^2 ) ) )/( 2*( (\mPOne)^2+1 ) ) };
\pgfmathsetmacro{\yPPOne}{ \mPOne*\xPPOne };

\pgfmathsetmacro{\xPPTwo}{ ( 2*((\mPTwo)*(\yCTwo)+(\xCTwo) ) - sqrt( ( 2*( (\mPTwo)*\yCTwo+\xCTwo )  )^2 - 4*( (\mPTwo)^2 + 1 )*( (\xCTwo)^2 + (\yCTwo)^2 - \r^2 ) ) )/( 2*( (\mPTwo)^2+1 ) ) };
\pgfmathsetmacro{\yPPTwo}{ \mPTwo*\xPPTwo };

\pgfmathsetmacro{\dPOneP}{ sqrt( (\xPOne - \xP)^2 + (\yPOne - \yP)^2 ) };
\pgfmathsetmacro{\aCPOne}{ acos( (2*\r^2 - \dPOneP^2 )/(2*\r^2) ) };
\pgfmathsetmacro{\l}{ \aCPOne/\a };
\pgfmathsetmacro{\rQ}{ sqrt( \R^2 + (\xCQ)^2  )  };
\pgfmathsetmacro{\aPCQ}{ acos( (\xCQ)/(\rQ) ) }
\pgfmathsetmacro{\aQ}{  -( 360 - 2*\aPCQ )*(1-\l) };

\coordinate (CQ) at (\xCQ, \yCQ);


\coordinate (P) at (0,{sqrt(3)});

\coordinate[rotate around={\aQ:(CQ)}] (Q) at (P);


\coordinate (PP) at (0,{-sqrt(3)});

\coordinate (P1) at (\xPOne, \yPOne);
\coordinate (P2) at (\xPTwo, \yPTwo);

\coordinate (PP1) at (\xPPOne,\yPPOne);

\coordinate (PP2) at (\xPPTwo,\yPPTwo);



	(P1) -- (PP1)
	(P2) -- (PP2)
;

\arcThroughThreePoints[]{P}{PP}{P1};
\arcThroughThreePoints[]{P2}{PP}{P};

\arcThroughThreePoints[double, line width=0.6]{P1}{PP1}{Q};
\arcThroughThreePoints[line width=1.5]{Q}{PP2}{P2};

\arcThroughThreePoints[gray!50]{PP}{P}{Q}; 
\arcThroughThreePoints[dashed]{Q}{PP}{P};

\node at (90: 2.0) {\small $A$};
\node at (270: 2.0) {\small $A^{\ast}$};
\node at (160: 1.25) {\small $B$};
\node at (20: 1.25) {\small $D$};
\node at (288: 0.58) {\small $C$};

\node at (95:1.3) {\small $\theta$};
\node at (278:1.0) {\textcolor{gray!50}{\small $a$}};

\end{scope}


\begin{scope}[xshift=3.0cm] 

\tikzmath{
\r=2; \R=sqrt(3);
\a=120; \aa=\a/2;
\xP=0; \yP=\R;
\aOP=120; 
\aPOne=\aOP;
\mPOne=tan(90+\aPOne);
\aPTwo=\aOP;
\mPTwo=tan(-90-\aPTwo);
\xCQ=-4.5; \yCQ=-0.2;
\xCOne = sqrt(\r^2-\R^2); \yCOne=0;
\xCTwo =-sqrt(\r^2-\R^2); \yCTwo=0;
}

\coordinate (O) at (0,0);
\coordinate (C1) at (\xCOne,0);
\coordinate (C2) at (\xCTwo,0);


	(O) circle (\R);
	
	(C2) circle (\r)
	(C1) circle (\r);

\draw[gray!50]
	([shift={(-\aa:\r)}]\xCTwo,0) arc (-\aa:\aa:\r);

\draw[gray!50]
	([shift={(-\aa+180:\r)}]\xCOne,0) arc (-\aa+180:\aa+180:\r);

\draw[] ([shift=(90:\R)]215:0.25) arc (215:296:0.25)
;

\pgfmathsetmacro{\xPOne}{ ( 2*((\mPOne)*(\yCOne)+(\xCOne)) - sqrt( ( 2*( (\mPOne)*(\yCOne)+\xCOne )  )^2 - 4*( (\mPOne)^2 + 1 )*( (\xCOne)^2 + (\yCOne)^2 - \r^2 ) ) )/( 2*( (\mPOne)^2+1 ) ) };
\pgfmathsetmacro{\yPOne}{ \mPOne*\xPOne };

\pgfmathsetmacro{\xPTwo}{ ( 2*((\mPTwo)*(\yCTwo)+(\xCTwo) ) + sqrt( ( 2*( (\mPTwo)*\yCTwo+\xCTwo )  )^2 - 4*( (\mPTwo)^2 + 1 )*( (\xCTwo)^2 + (\yCTwo)^2 - \r^2 ) ) )/( 2*( (\mPTwo)^2+1 ) ) };
\pgfmathsetmacro{\yPTwo}{ \mPTwo*\xPTwo };

\pgfmathsetmacro{\xPPOne}{ ( 2*((\mPOne)*(\yCOne)+(\xCOne)) + sqrt( ( 2*( (\mPOne)*(\yCOne)+\xCOne )  )^2 - 4*( (\mPOne)^2 + 1 )*( (\xCOne)^2 + (\yCOne)^2 - \r^2 ) ) )/( 2*( (\mPOne)^2+1 ) ) };
\pgfmathsetmacro{\yPPOne}{ \mPOne*\xPPOne };

\pgfmathsetmacro{\xPPTwo}{ ( 2*((\mPTwo)*(\yCTwo)+(\xCTwo) ) - sqrt( ( 2*( (\mPTwo)*\yCTwo+\xCTwo )  )^2 - 4*( (\mPTwo)^2 + 1 )*( (\xCTwo)^2 + (\yCTwo)^2 - \r^2 ) ) )/( 2*( (\mPTwo)^2+1 ) ) };
\pgfmathsetmacro{\yPPTwo}{ \mPTwo*\xPPTwo };

\pgfmathsetmacro{\xPQRef}{\xPTwo};
\pgfmathsetmacro{\yPQRef}{\yPTwo};

\pgfmathsetmacro{\dPRefP}{ sqrt( (\xPQRef - \xP)^2 + (\yPQRef - \yP)^2 ) };
\pgfmathsetmacro{\aCPRef}{ acos( (2*\r^2 - (\dPRefP)^2 )/(2*\r^2) ) };
\pgfmathsetmacro{\l}{ \aCPRef/\a };
\pgfmathsetmacro{\rQ}{ sqrt( \R^2 + (\xCQ)^2  )  };
\pgfmathsetmacro{\aPCQ}{ acos( (\xCQ)/(\rQ) ) }
\pgfmathsetmacro{\aQ}{ -( 360 -  2*\aPCQ )*(1-\l) };

\coordinate (CQ) at (\xCQ, \yCQ);


\coordinate (P) at (0,\R);

\coordinate[rotate around={\aQ:(CQ)}] (Q) at (P);


\coordinate (PP) at (0,-\R);

\coordinate (P1) at (\xPOne, \yPOne);
\coordinate (P2) at (\xPTwo, \yPTwo);

\coordinate (PP1) at (\xPPOne,\yPPOne);

\coordinate (PP2) at (\xPPTwo,\yPPTwo);



	(P1) -- (PP1)
	(P2) -- (PP2)
;

\arcThroughThreePoints[]{P}{PP}{P1};
\arcThroughThreePoints[]{P2}{PP}{P};

\arcThroughThreePoints[double, line width=0.6]{Q}{PP1}{P1};
\arcThroughThreePoints[line width=1.5]{P2}{PP2}{Q};

\arcThroughThreePoints[gray!50]{PP}{P}{Q}; 
\arcThroughThreePoints[dashed]{Q}{PP}{P}; 

\node at (90: 2.0) {\small $A$};
\node at (270: 2.0) {\small $A^{\ast}$};
\node at (210: 1.3) {\small $B$};
\node at (-30: 1.3) {\small $D$};
\node at (30: 0.18) {\small $C$};

\node at (95:1.3) {\small $\theta$};
\node at (280:0.75) {\textcolor{gray!50}{\small $a$}};

\end{scope}


\begin{scope}[xshift=6cm]

\tikzmath{
\r=2; \R=sqrt(3);
\a=120; \aa=\a/2;
\xP=0; \yP=\R;
\aOP=52; 
\aPOne=\aOP;
\mPOne=tan(90+\aPOne);
\aPTwo=\aOP;
\mPTwo=tan(-90-\aPTwo);
\xCQ=1; \yCQ=0.7;
\xCOne = sqrt(\r^2-\R^2); \yCOne=0;
\xCTwo =-sqrt(\r^2-\R^2); \yCTwo=0;
}

\coordinate (O) at (0,0);
\coordinate (C1) at (\xCOne ,0);
\coordinate (C2) at (\xCTwo,0);


	(O) circle (\R);
	
	(C2) circle (\r)
	(C1) circle (\r);

\draw[gray!50]
	([shift={(-\aa:2)}]\xCTwo,0) arc (-\aa:\aa:2);

\draw[gray!50]
	([shift={(-\aa+180:\r)}]\xCOne,0) arc (-\aa+180:\aa+180:\r);

\draw[] ([shift=(90:\R)]325:0.35) arc (325:353:0.35)
;

\pgfmathsetmacro{\xPOne}{ ( 2*((\mPOne)*(\yCOne)+(\xCOne)) - sqrt( ( 2*( (\mPOne)*(\yCOne)+\xCOne )  )^2 - 4*( (\mPOne)^2 + 1 )*( (\xCOne)^2 + (\yCOne)^2 - \r^2 ) ) )/( 2*( (\mPOne)^2+1 ) ) };
\pgfmathsetmacro{\yPOne}{ \mPOne*\xPOne };

\pgfmathsetmacro{\xPTwo}{ ( 2*((\mPTwo)*(\yCTwo)+(\xCTwo) ) + sqrt( ( 2*( (\mPTwo)*\yCTwo+\xCTwo )  )^2 - 4*( (\mPTwo)^2 + 1 )*( (\xCTwo)^2 + (\yCTwo)^2 - \r^2 ) ) )/( 2*( (\mPTwo)^2+1 ) ) };
\pgfmathsetmacro{\yPTwo}{ \mPTwo*\xPTwo };

\pgfmathsetmacro{\xPPOne}{ ( 2*((\mPOne)*(\yCOne)+(\xCOne)) + sqrt( ( 2*( (\mPOne)*(\yCOne)+\xCOne )  )^2 - 4*( (\mPOne)^2 + 1 )*( (\xCOne)^2 + (\yCOne)^2 - \r^2 ) ) )/( 2*( (\mPOne)^2+1 ) ) };
\pgfmathsetmacro{\yPPOne}{ \mPOne*\xPPOne };

\pgfmathsetmacro{\xPPTwo}{ ( 2*((\mPTwo)*(\yCTwo)+(\xCTwo) ) - sqrt( ( 2*( (\mPTwo)*\yCTwo+\xCTwo )  )^2 - 4*( (\mPTwo)^2 + 1 )*( (\xCTwo)^2 + (\yCTwo)^2 - \r^2 ) ) )/( 2*( (\mPTwo)^2+1 ) ) };
\pgfmathsetmacro{\yPPTwo}{ \mPTwo*\xPPTwo };

\pgfmathsetmacro{\xPQRef}{\xPTwo};
\pgfmathsetmacro{\yPQRef}{\yPTwo};

\pgfmathsetmacro{\dPRefP}{ sqrt( (\xPQRef - \xP)^2 + (\yPQRef - \yP)^2 ) };
\pgfmathsetmacro{\aCPRef}{ acos( (2*\r^2 - (\dPRefP)^2 )/(2*\r^2) ) };
\pgfmathsetmacro{\l}{ \aCPRef/\a };
\pgfmathsetmacro{\rQ}{ sqrt( \R^2 + (\xCQ)^2  )  };
\pgfmathsetmacro{\aPCQ}{ acos( (\xCQ)/(\rQ) ) }
\pgfmathsetmacro{\aQ}{ ( 360 -  2*\aPCQ )*(1-\l) };

\coordinate (CQ) at (\xCQ, \yCQ);


\coordinate (P) at (0,\R);

\coordinate[rotate around={\aQ:(CQ)}] (Q) at (P);


\coordinate (PP) at (0,-\R);

\coordinate (P1) at (\xPOne, \yPOne);
\coordinate (P2) at (\xPTwo, \yPTwo);

\coordinate (PP1) at (\xPPOne,\yPPOne);

\coordinate (PP2) at (\xPPTwo,\yPPTwo);



	(P1) -- (PP1)
	(P2) -- (PP2)
;

\arcThroughThreePoints[]{P}{PP}{P1};
\arcThroughThreePoints[]{P2}{PP}{P};

\arcThroughThreePoints[double, line width=0.6]{Q}{PP1}{P1};
\arcThroughThreePoints[line width=1.5]{Q}{PP2}{P2};

\arcThroughThreePoints[dashed]{Q}{PP}{P}; 
\arcThroughThreePoints[gray!50]{PP}{P}{Q}; 

\node at (90: 2.0) {\small $A$};
\node at (270: 2.0) {\small $A^{\ast}$};
\node at (140: 1.4) {\small $B$};
\node at (40: 1.4) {\small $D$};
\node at (-30: 1.5) {\small $C$};



\node at (65:1.525) {\small $\theta$};
\node at (310:1.9) {\textcolor{gray!50}{\small $a$}};

\end{scope}


\begin{scope}[xshift=10.25 cm]

\tikzmath{
\r=2; \R=sqrt(3);
\a=120; \aa=\a/2;
\xP=0; \yP=\R;
\aOP=52; 
\aPOne=\aOP;
\mPOne=tan(90+\aPOne);
\aPTwo=\aOP;
\mPTwo=tan(-90-\aPTwo);
\xCQ=-1; \yCQ=0.7;
\xCOne = sqrt(\r^2-\R^2); \yCOne=0;
\xCTwo =-sqrt(\r^2-\R^2); \yCTwo=0;
}

\coordinate (O) at (0,0);
\coordinate (C1) at (\xCOne,0);
\coordinate (C2) at (\xCTwo,0);


	(O) circle (\R);
	
	(C2) circle (\r)
	(C1) circle (\r);

\draw[gray!50]
	([shift={(-\aa:\r)}]\xCTwo,0) arc (-\aa:\aa:\r);

\draw[gray!50]
	([shift={(-\aa+180:\r)}]\xCOne,0) arc (-\aa+180:\aa+180:\r);

\draw[] ([shift=(90:\R)]215:0.35) arc (215:188:0.35)
;

\pgfmathsetmacro{\xPOne}{ ( 2*((\mPOne)*(\yCOne)+(\xCOne)) - sqrt( ( 2*( (\mPOne)*(\yCOne)+\xCOne )  )^2 - 4*( (\mPOne)^2 + 1 )*( (\xCOne)^2 + (\yCOne)^2 - \r^2 ) ) )/( 2*( (\mPOne)^2+1 ) ) };
\pgfmathsetmacro{\yPOne}{ \mPOne*\xPOne };

\pgfmathsetmacro{\xPTwo}{ ( 2*((\mPTwo)*(\yCTwo)+(\xCTwo) ) + sqrt( ( 2*( (\mPTwo)*\yCTwo+\xCTwo )  )^2 - 4*( (\mPTwo)^2 + 1 )*( (\xCTwo)^2 + (\yCTwo)^2 - \r^2 ) ) )/( 2*( (\mPTwo)^2+1 ) ) };
\pgfmathsetmacro{\yPTwo}{ \mPTwo*\xPTwo };

\pgfmathsetmacro{\xPPOne}{ ( 2*((\mPOne)*(\yCOne)+(\xCOne)) + sqrt( ( 2*( (\mPOne)*(\yCOne)+\xCOne )  )^2 - 4*( (\mPOne)^2 + 1 )*( (\xCOne)^2 + (\yCOne)^2 - \r^2 ) ) )/( 2*( (\mPOne)^2+1 ) ) };
\pgfmathsetmacro{\yPPOne}{ \mPOne*\xPPOne };

\pgfmathsetmacro{\xPPTwo}{ ( 2*((\mPTwo)*(\yCTwo)+(\xCTwo) ) - sqrt( ( 2*( (\mPTwo)*\yCTwo+\xCTwo )  )^2 - 4*( (\mPTwo)^2 + 1 )*( (\xCTwo)^2 + (\yCTwo)^2 - \r^2 ) ) )/( 2*( (\mPTwo)^2+1 ) ) };
\pgfmathsetmacro{\yPPTwo}{ \mPTwo*\xPPTwo };

\pgfmathsetmacro{\xPQRef}{\xPTwo};
\pgfmathsetmacro{\yPQRef}{\yPTwo};

\pgfmathsetmacro{\dPRefP}{ sqrt( (\xPQRef - \xP)^2 + (\yPQRef - \yP)^2 ) };
\pgfmathsetmacro{\aCPRef}{ acos( (2*\r^2 - (\dPRefP)^2 )/(2*\r^2) ) };
\pgfmathsetmacro{\l}{ \aCPRef/\a };
\pgfmathsetmacro{\rQ}{ sqrt( \R^2 + (\xCQ)^2  )  };
\pgfmathsetmacro{\aPCQ}{ acos( (\xCQ)/(\rQ) ) }
\pgfmathsetmacro{\aQ}{ -(  2*\aPCQ )*(1-\l) };

\coordinate (CQ) at (\xCQ, \yCQ);


\coordinate (P) at (0,\R);

\coordinate[rotate around={\aQ:(CQ)}] (Q) at (P);


\coordinate (PP) at (0,-\R);

\coordinate (P1) at (\xPOne, \yPOne);
\coordinate (P2) at (\xPTwo, \yPTwo);

\coordinate (PP1) at (\xPPOne,\yPPOne);

\coordinate (PP2) at (\xPPTwo,\yPPTwo);



	(P1) -- (PP1)
	(P2) -- (PP2)
;

\arcThroughThreePoints[]{P}{PP}{P1};
\arcThroughThreePoints[]{P2}{PP}{P};

\arcThroughThreePoints[double, line width=0.6]{P1}{PP1}{Q};
\arcThroughThreePoints[line width=1.5]{P2}{PP2}{Q};

\arcThroughThreePoints[dashed]{P}{PP}{Q}; 
\arcThroughThreePoints[gray!50]{Q}{P}{PP}; 

\node at (90: 2.0) {\small $A$};
\node at (270: 2.0) {\small $A^{\ast}$};
\node at (140: 1.4) {\small $B$};
\node at (40: 1.4) {\small $D$};
\node at (208: 1.5) {\small $C$};



\node at (115:1.525) {\small $\theta$};
\node at (225:1.9) {\textcolor{gray!50}{\small $a$}};

\end{scope}

\end{tikzpicture}
\caption{Shapes of the $a^2bc$ quadrilateral}
\label{ALunes}
\end{figure}

For the quadrilaterals in the first two pictures of Figure \ref{ALunes}, edge reduction $b=c$ happens when $\theta = \frac{1}{2}\alpha$, and $c=a$ when $\cos \theta = \frac{\cos a}{1-\cos a}$ and $b=a$ when $\cos ( \alpha - \theta ) = \frac{\cos a}{1-\cos a}$. For the third tile, $b=c$ when $\theta =\pi - \frac{1}{2}\alpha$, and $c=a$ when $\cos(\alpha+\theta) = \frac{\cos a}{1-\cos a}$ and $b=a$ when $\cos \theta = \frac{\cos a}{1-\cos a}$. Discussion for the fourth picture is analogous.

\begin{prop}If $\beta = \frac{4}{f}\pi$, for $f\ge6$, there exists $a\in (0, \pi)$ and a simple almost equilateral quadrilateral with the distinct edges $a,b$ and area $\beta$. Moreover, tilings exist for the following from Tables \ref{Rat-AlGaDe-AVCs}, \ref{NonRat-AlGaDe-AVCs},
\begin{enumerate}
\item $E, \AVC \equiv \{ \alpha\gamma\delta, \beta^{\frac{f}{2}} \}$,
\item $E^{\prime}, \AVC \equiv \{ \alpha\gamma\delta,\alpha^m\beta^{n_1},\beta^{n}\gamma\delta \}, \{ \alpha\gamma\delta,\alpha^m,\beta^{n}\gamma\delta \}$,
\item $E^{\prime\prime}, \AVC \equiv \{ \alpha\gamma\delta,\alpha\beta^{n_1}, \gamma^k\delta^k \},  \{ \alpha\gamma\delta,\alpha\beta^{n_1}, \beta^{n_2}\gamma^k\delta^k \}$,
\item $E^{\prime\prime\prime}, \AVC \equiv \{ \alpha\gamma\delta, \gamma^3\delta, \alpha\beta^{\frac{f+2}{6}}, \alpha\beta^{\frac{f-4}{6}}\delta^2 \}$.
\end{enumerate}
Their corresponding edge and angle values are summarised in Tables \ref{EMTData1}, \ref{EMTData2}.
\end{prop}

Tilings with $\AVC \equiv \{ \alpha\gamma\delta, \beta^{\frac{f}{2}} \}$ in the proposition can be obtained by edge reduction $c=a$ of the tilings with $\AVC \equiv \{ \beta\gamma\delta, \alpha^{\frac{f}{2}} \}$ in Proposition \ref{a2bcQuad-Area} via the relabelling $\alpha \leftrightarrow \beta$ and $\gamma \leftrightarrow \delta$ and $(A,B,C,D) \rightarrow (B,A,D,C)$. We remark $a \in (0,\pi)$ for $f=6$ and $a \in [\frac{1}{3}\pi, \frac{1}{2}\pi)$ for $f\ge8$. In each case, choosing $a$ is equivalent to choosing $\alpha$, depending on the shape. 

\begin{proof} If $f=6$, then the existence of the tile in Proposition \ref{MinProp} is induced by projection of the cube onto the sphere. For $f\ge8$, recall that $\delta < \pi$ (Lemma \ref{AlGaDeLem}). Then after relabelling as mentioned before, we can see that the second picture in Figure \ref{ALunes} cannot happen and there are three shapes of quadrilateral $\square ABCD$ in Figure \ref{BLunes}. This means the tiles exist and $0 < a < \frac{1}{2}\pi$.

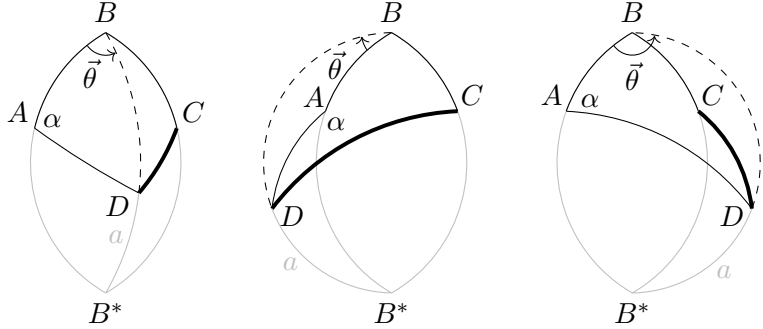
\begin{figure}[htp] 
\centering
\begin{tikzpicture}[scale=1]

\begin{scope}[] 

\tikzmath{
\r=2; \R=sqrt(3);
\a=120; \aa=\a/2;
\xP=0; \yP=\R;
\aOP=64; 
\aPOne=\aOP;
\mPOne=tan(90+\aPOne);
\aPTwo=\aOP;
\mPTwo=tan(-90-\aPTwo);
\xCQ=-3; \yCQ=0;
\xCOne = sqrt(\r^2-\R^2); \yCOne=0;
\xCTwo =-sqrt(\r^2-\R^2); \yCTwo=0;
\xCThree = sqrt(\r^2-\R^2); \yCThree=0;
}

\coordinate (O) at (0,0);
\coordinate (C1) at (\xCOne ,0);
\coordinate (C2) at (\xCTwo,0);


	(O) circle (\R);
	
	(C2) circle (\r)
	(C1) circle (\r);

\draw[gray!50]
	([shift={(-\aa:2)}]\xCTwo,0) arc (-\aa:\aa:\r);

\draw[gray!50]
	([shift={(-\aa+180:\r)}]\xCOne,0) arc (-\aa+180:\aa+180:\r);

\pgfmathsetmacro{\xPOne}{ ( 2*((\mPOne)*(\yCOne)+(\xCOne)) - sqrt( ( 2*( (\mPOne)*(\yCOne)+\xCOne )  )^2 - 4*( (\mPOne)^2 + 1 )*( (\xCOne)^2 + (\yCOne)^2 - \r^2 ) ) )/( 2*( (\mPOne)^2+1 ) ) };
\pgfmathsetmacro{\yPOne}{ \mPOne*\xPOne };

\pgfmathsetmacro{\xPTwo}{ ( 2*((\mPTwo)*(\yCTwo)+(\xCTwo) ) + sqrt( ( 2*( (\mPTwo)*\yCTwo+\xCTwo )  )^2 - 4*( (\mPTwo)^2 + 1 )*( (\xCTwo)^2 + (\yCTwo)^2 - \r^2 ) ) )/( 2*( (\mPTwo)^2+1 ) ) };
\pgfmathsetmacro{\yPTwo}{ \mPTwo*\xPTwo };

\pgfmathsetmacro{\xPPOne}{ ( 2*((\mPOne)*(\yCOne)+(\xCOne)) + sqrt( ( 2*( (\mPOne)*(\yCOne)+\xCOne )  )^2 - 4*( (\mPOne)^2 + 1 )*( (\xCOne)^2 + (\yCOne)^2 - \r^2 ) ) )/( 2*( (\mPOne)^2+1 ) ) };
\pgfmathsetmacro{\yPPOne}{ \mPOne*\xPPOne };

\pgfmathsetmacro{\xPPTwo}{ ( 2*((\mPTwo)*(\yCTwo)+(\xCTwo) ) - sqrt( ( 2*( (\mPTwo)*\yCTwo+\xCTwo )  )^2 - 4*( (\mPTwo)^2 + 1 )*( (\xCTwo)^2 + (\yCTwo)^2 - \r^2 ) ) )/( 2*( (\mPTwo)^2+1 ) ) };
\pgfmathsetmacro{\yPPTwo}{ \mPTwo*\xPPTwo };

\pgfmathsetmacro{\dPOneP}{ sqrt( (\xPOne - \xP)^2 + (\yPOne - \yP)^2 ) };
\pgfmathsetmacro{\aCPOne}{ acos( (2*\r^2 - \dPOneP^2 )/(2*\r^2) ) };
\pgfmathsetmacro{\l}{ \aCPOne/\a };
\pgfmathsetmacro{\rQ}{ sqrt( \R^2 + (\xCQ)^2  )  };
\pgfmathsetmacro{\aPCQ}{ acos( (\xCQ)/(\rQ) ) }
\pgfmathsetmacro{\aQ}{  -( 360 - 2*\aPCQ )*(1-\l) };

\coordinate (CQ) at (\xCQ, \yCQ);


\coordinate (P) at (0,\R);

\coordinate[rotate around={\aQ:(CQ)}] (Q) at (P);


\coordinate (PP) at (0,-\R);

\coordinate (P1) at (\xPOne, \yPOne);
\coordinate (P2) at (\xPTwo, \yPTwo);

\coordinate (PP1) at (\xPPOne,\yPPOne);

\coordinate (PP2) at (\xPPTwo,\yPPTwo);



	(P1) -- (PP1)
	(P2) -- (PP2)
;

\draw[->] ([shift=(90:\R)]215:0.3) arc (215:296:0.3)
;

\arcThroughThreePoints[]{P}{PP}{P1};
\arcThroughThreePoints[]{P2}{PP}{P};

\arcThroughThreePoints[]{P1}{PP1}{Q};
\arcThroughThreePoints[line width=1.5]{Q}{PP2}{P2};

\arcThroughThreePoints[gray!50]{PP}{P}{Q}; 
\arcThroughThreePoints[dashed]{Q}{PP}{P}; 

\node at (90: 2.0) {\small $B$};
\node at (270: 2.0) {\small $B^{\ast}$};
\node at (150: 1.35) {\small $A$};
\node at (30: 1.35) {\small $C$};
\node at (285: 0.6) {\small $D$};



\node at (278:1.0) {\textcolor{gray!50}{\small $a$}};

\node at (142:0.9) {\small $\alpha$};

\node at (100:1.2) {\small $\vec{\theta}$};

\end{scope} 


\begin{scope}[xshift=3.8cm] 

\tikzmath{
\r=2; \R=sqrt(3);
\a=120; \aa=\a/2;
\xP=0; \yP=\R;
\aOP=52; 
\aPOne=\aOP;
\mPOne=tan(90+\aPOne);
\aPTwo=\aOP;
\mPTwo=tan(-90-\aPTwo);
\xCQ=-1; \yCQ=0.7;
\xCOne = sqrt(\r^2-\R^2); \yCOne=0;
\xCTwo =-sqrt(\r^2-\R^2); \yCTwo=0;
\xCThree = sqrt(\r^2-\R^2); \yCThree=0;
}

\coordinate (O) at (0,0);
\coordinate (C1) at (\xCOne ,0);
\coordinate (C2) at (\xCTwo,0);


	(O) circle (\R);
	
	(C2) circle (\r)
	(C1) circle (\r);

\draw[gray!50]
	([shift={(-\aa:2)}]\xCTwo,0) arc (-\aa:\aa:\r);

\draw[gray!50]
	([shift={(-\aa+180:\r)}]\xCOne,0) arc (-\aa+180:\aa+180:\r);

\pgfmathsetmacro{\xPOne}{ ( 2*((\mPOne)*(\yCOne)+(\xCOne)) - sqrt( ( 2*( (\mPOne)*(\yCOne)+\xCOne )  )^2 - 4*( (\mPOne)^2 + 1 )*( (\xCOne)^2 + (\yCOne)^2 - \r^2 ) ) )/( 2*( (\mPOne)^2+1 ) ) };
\pgfmathsetmacro{\yPOne}{ \mPOne*\xPOne };

\pgfmathsetmacro{\xPTwo}{ ( 2*((\mPTwo)*(\yCTwo)+(\xCTwo) ) + sqrt( ( 2*( (\mPTwo)*\yCTwo+\xCTwo )  )^2 - 4*( (\mPTwo)^2 + 1 )*( (\xCTwo)^2 + (\yCTwo)^2 - \r^2 ) ) )/( 2*( (\mPTwo)^2+1 ) ) };
\pgfmathsetmacro{\yPTwo}{ \mPTwo*\xPTwo };

\pgfmathsetmacro{\xPPOne}{ ( 2*((\mPOne)*(\yCOne)+(\xCOne)) + sqrt( ( 2*( (\mPOne)*(\yCOne)+\xCOne )  )^2 - 4*( (\mPOne)^2 + 1 )*( (\xCOne)^2 + (\yCOne)^2 - \r^2 ) ) )/( 2*( (\mPOne)^2+1 ) ) };
\pgfmathsetmacro{\yPPOne}{ \mPOne*\xPPOne };

\pgfmathsetmacro{\xPPTwo}{ ( 2*((\mPTwo)*(\yCTwo)+(\xCTwo) ) - sqrt( ( 2*( (\mPTwo)*\yCTwo+\xCTwo )  )^2 - 4*( (\mPTwo)^2 + 1 )*( (\xCTwo)^2 + (\yCTwo)^2 - \r^2 ) ) )/( 2*( (\mPTwo)^2+1 ) ) };
\pgfmathsetmacro{\yPPTwo}{ \mPTwo*\xPPTwo };

\pgfmathsetmacro{\xPQRef}{\xPOne};
\pgfmathsetmacro{\yPQRef}{\yPOne};

\pgfmathsetmacro{\dPRefP}{ sqrt( (\xPQRef - \xP)^2 + (\yPQRef - \yP)^2 ) };
\pgfmathsetmacro{\aCPRef}{ acos( (2*\r^2 - (\dPRefP)^2 )/(2*\r^2) ) };
\pgfmathsetmacro{\l}{ \aCPRef/\a };
\pgfmathsetmacro{\rQ}{ sqrt( \R^2 + (\xCQ)^2  )  };
\pgfmathsetmacro{\aPCQ}{ acos( (\xCQ)/(\rQ) ) }
\pgfmathsetmacro{\aQ}{ -(  2*\aPCQ )*(1-\l) };

\coordinate (CQ) at (\xCQ, \yCQ);


\coordinate (P) at (0,\R);

\coordinate[rotate around={\aQ:(CQ)}] (Q) at (P);


\coordinate (PP) at (0,-\R);

\coordinate (P1) at (\xPOne, \yPOne);
\coordinate (P2) at (\xPTwo, \yPTwo);

\coordinate (PP1) at (\xPPOne,\yPPOne);

\coordinate (PP2) at (\xPPTwo,\yPPTwo);



	(P1) -- (PP1)
	(P2) -- (PP2)
;

\draw[->] ([shift=(90:\R)]215:0.4) arc (215:188:0.4)
;

\arcThroughThreePoints[]{P}{PP}{P1};
\arcThroughThreePoints[]{P2}{PP}{P};

\arcThroughThreePoints[]{P1}{PP1}{Q};
\arcThroughThreePoints[line width=1.5]{P2}{PP2}{Q};

\arcThroughThreePoints[dashed]{P}{PP}{Q}; 
\arcThroughThreePoints[gray!50]{Q}{P}{PP}; 

\node at (90: 2.0) {\small $B$};
\node at (270: 2.0) {\small $B^{\ast}$};
\node at (140: 1.35) {\small $A$};
\node at (40: 1.4) {\small $C$};
\node at (208: 1.5) {\small $D$};



\node at (225:1.9) {\textcolor{gray!50}{\small $a$}};

\node at (145:0.9) {\small $\alpha$};

\node at (120:1.5) {\small $\vec{\theta}$};

\end{scope} 


\begin{scope}[xshift=7cm]

\tikzmath{
\r=2; \R=sqrt(3);
\a=120; \aa=\a/2;
\xP=0; \yP=\R;
\aOP=52; 
\aPOne=\aOP;
\mPOne=tan(90+\aPOne);
\aPTwo=\aOP;
\mPTwo=tan(-90-\aPTwo);
\xCQ=1; \yCQ=0.7;
\xCOne = sqrt(\r^2-\R^2); \yCOne=0;
\xCTwo = -sqrt(\r^2-\R^2); \yCTwo=0;
}

\coordinate (O) at (0,0);
\coordinate (C1) at (\xCOne,0);
\coordinate (C2) at (\xCTwo,0);


	(O) circle (\R);
	
	(C2) circle (\r)
	(C1) circle (\r);

\draw[gray!50]
	([shift={(-\aa:\r)}]\xCTwo,0) arc (-\aa:\aa:\r);

\draw[gray!50]
	([shift={(-\aa+180:\r)}]\xCOne,0) arc (-\aa+180:\aa+180:\r);

\draw[->] ([shift=(90:\R)]215:0.3) arc (215:352:0.3)
;

\pgfmathsetmacro{\xPOne}{ ( 2*((\mPOne)*(\yCOne)+(\xCOne)) - sqrt( ( 2*( (\mPOne)*(\yCOne)+\xCOne )  )^2 - 4*( (\mPOne)^2 + 1 )*( (\xCOne)^2 + (\yCOne)^2 - \r^2 ) ) )/( 2*( (\mPOne)^2+1 ) ) };
\pgfmathsetmacro{\yPOne}{ \mPOne*\xPOne };

\pgfmathsetmacro{\xPTwo}{ ( 2*((\mPTwo)*(\yCTwo)+(\xCTwo) ) + sqrt( ( 2*( (\mPTwo)*\yCTwo+\xCTwo )  )^2 - 4*( (\mPTwo)^2 + 1 )*( (\xCTwo)^2 + (\yCTwo)^2 - \r^2 ) ) )/( 2*( (\mPTwo)^2+1 ) ) };
\pgfmathsetmacro{\yPTwo}{ \mPTwo*\xPTwo };

\pgfmathsetmacro{\xPPOne}{ ( 2*((\mPOne)*(\yCOne)+(\xCOne)) + sqrt( ( 2*( (\mPOne)*(\yCOne)+\xCOne )  )^2 - 4*( (\mPOne)^2 + 1 )*( (\xCOne)^2 + (\yCOne)^2 - \r^2 ) ) )/( 2*( (\mPOne)^2+1 ) ) };
\pgfmathsetmacro{\yPPOne}{ \mPOne*\xPPOne };

\pgfmathsetmacro{\xPPTwo}{ ( 2*((\mPTwo)*(\yCTwo)+(\xCTwo) ) - sqrt( ( 2*( (\mPTwo)*\yCTwo+\xCTwo )  )^2 - 4*( (\mPTwo)^2 + 1 )*( (\xCTwo)^2 + (\yCTwo)^2 - \r^2 ) ) )/( 2*( (\mPTwo)^2+1 ) ) };
\pgfmathsetmacro{\yPPTwo}{ \mPTwo*\xPPTwo };

\pgfmathsetmacro{\xPQRef}{\xPTwo};
\pgfmathsetmacro{\yPQRef}{\yPTwo};

\pgfmathsetmacro{\dPRefP}{ sqrt( (\xPQRef - \xP)^2 + (\yPQRef - \yP)^2 ) };
\pgfmathsetmacro{\aCPRef}{ acos( (2*\r^2 - (\dPRefP)^2 )/(2*\r^2) ) };
\pgfmathsetmacro{\l}{ \aCPRef/\a };
\pgfmathsetmacro{\rQ}{ sqrt( \R^2 + (\xCQ)^2  )  };
\pgfmathsetmacro{\aPCQ}{ acos( (\xCQ)/(\rQ) ) }
\pgfmathsetmacro{\aQ}{ ( 360 -  2*\aPCQ )*(1-\l) };

\coordinate (CQ) at (\xCQ, \yCQ);


\coordinate (P) at (0,\R);

\coordinate[rotate around={\aQ:(CQ)}] (Q) at (P);


\coordinate (PP) at (0,-\R);

\coordinate (P1) at (\xPOne, \yPOne);
\coordinate (P2) at (\xPTwo, \yPTwo);

\coordinate (PP1) at (\xPPOne,\yPPOne);

\coordinate (PP2) at (\xPPTwo,\yPPTwo);



	(P1) -- (PP1)
	(P2) -- (PP2)
;

\arcThroughThreePoints[]{P}{PP}{P1};
\arcThroughThreePoints[]{P2}{PP}{P};

\arcThroughThreePoints[]{Q}{PP1}{P1};
\arcThroughThreePoints[line width=1.5]{Q}{PP2}{P2};

\arcThroughThreePoints[dashed]{Q}{PP}{P}; 
\arcThroughThreePoints[gray!50]{PP}{P}{Q}; 

\node at (90: 2.0) {\small $B$};
\node at (270: 2.0) {\small $B^{\ast}$};
\node at (140: 1.4) {\small $A$};
\node at (40: 1.4) {\small $C$};
\node at (-28: 1.5) {\small $D$};



\node at (124:1) {\small $\alpha$};
\node at (90:1.15) {\small $\vec{\theta}$};
\node at (310:1.9) {\textcolor{gray!50}{\small $a$}};

\end{scope}

\end{tikzpicture}
\caption{Shapes of the $a^3b$ quadrilateral}
\label{BLunes}
\end{figure}

Let $\vec{\theta}$ denote oriented $\angle ABD$ (which is positive if measured counter-clockwise from $AB$) in all three shapes and $\theta$ denote the angle size (i.e., absolute value) of $\vec{\theta}$. This means $\vec{\theta} \ge 0$ in the first and the third picture of Figure \ref{BLunes} and $\vec{\theta} \le 0$ in the second. The value of $\vec{\theta}$ and $B^{\ast}D=a$ determine the location of $D$.

The isosceles triangle $\triangle ABD$ in all three pictures is half of the lune defined by $A, B, D, B^{\ast},  \vec{\theta}$ which has area $2\theta$. Then the area of $\triangle ABD$ is $\theta$. By triangle angle sum of $\triangle ABD$, the first and third picture imply $\alpha+2\theta - \pi = \theta$ and the second picture implies $(2\pi-\alpha)+2\theta - \pi = \theta$. Combining both, we get  
\begin{align}\label{algade-alth-eq}
\alpha = \pi - \vec{\theta}.
\end{align}

In each shape, $\triangle ABC$ is a standard triangle contained in $\square ABCD$. By $AB=a$ and $BD=\pi - a$, spherical cosine law on $\triangle ABD$ implies $\cos \vec{\theta} = \frac{\cos a}{1-\cos a}$. By $a \in (0, \frac{1}{2}\pi)$, we have $\cos \vec{\theta} > 0$, which implies $\vec{\theta} \in (-\frac{1}{2}\pi, \frac{1}{2}\pi)$. Moreover, by \eqref{algade-alth-eq}, we get $\alpha \in (\frac{1}{2}\pi, \frac{3}{2}\pi)$ and
\begin{align} \label{algade-cosa-eq} 
\cos a = \tfrac{ \cos \vec{\theta} }{ 1 + \cos \vec{\theta} }  = \tfrac{ \cos \alpha }{ \cos \alpha - 1 }.
\end{align}
Then the first equality above implies $\cos a \le \frac{1}{2}$. So $a \ge \frac{1}{3}\pi$ and hence $a \in [\frac{1}{3}\pi, \frac{1}{2}\pi)$.

In the first shape, $\angle CBD$ of $\triangle BCD$ is given by $\beta-\theta=\beta-\vec{\theta}=\alpha+\beta-\pi$. In the second shape, $\angle CBD = \beta+\theta = \beta - \vec{\theta}=\alpha+\beta-\pi$. In the third shape, $\angle CBD = \theta - \beta = - (\beta - \vec{\theta}) = -(\alpha+\beta-\pi)$. Spherical cosine law on $\triangle BCD$ gives $\cos b= \cos BD \cos BC + \sin BD \cos BC \cos \angle CBD$. Then by $BC=a$ and $BD=\pi-a$ and \eqref{algade-alth-eq} and \eqref{algade-cosa-eq}, we have 
\begin{align} \label{algade-cosb-eq} 
\cos b = \tfrac{ (2\cos \alpha - 1)\cos (\alpha + \beta) - \cos^2 \alpha  }{ ( 1 - \cos \alpha )^2 }.
\end{align}

For $\alpha \in (\frac{1}{2}\pi, \frac{3}{2}\pi)$, the quadrilateral constructed in Figure \ref{BLunes} is simple. The edges $a,b$ are determined by \eqref{algade-cosa-eq}, \eqref{algade-cosb-eq} with a choice of $\alpha$.

We summarise the relations between the shapes in Figure \ref{BLunes} and $\vec{\theta}, \alpha, \beta$ as follows,
\begin{enumerate}[\textit{Shape} 1.]
\item $0 \le \vec{\theta} \le \beta \iff \pi - \beta \le \alpha \le \pi$;
\item $\vec{\theta}<0 \iff \alpha > \pi$;
\item $\vec{\theta} > \beta  \iff \gamma > \pi \iff \alpha + \beta < \pi$.
\end{enumerate}
Then $\alpha+\beta=\pi$ if and only if $\gamma=\pi$.

There are three degenerate situations, $\alpha=\pi$ ($\vec{\theta}=0$) and $\gamma=\pi$ ($\vec{\theta}=\beta$) and $b=a$ ($\vec{\theta} = \frac{1}{2}\beta$). In the first two situations, the quadrilateral degenerates into triangles. In the last situation, the almost equilateral quadrilateral becomes a rhombus.  

By $\alpha > \frac{1}{2}\pi$, we have $m \le 3$ at $\alpha^m \cdots$ in each $\AVC$, notably $\alpha^m=\alpha^3$. By $\alpha < \frac{3}{2}\pi$, we have $\gamma+\delta > \frac{1}{2}\pi$, which implies $k\le3$ in $\gamma^k\delta^k\cdots$. 

\begin{figure}[htp] 
\centering
\begin{tikzpicture}[scale=1]

\begin{scope}[] 

\tikzmath{
\r=2; \R=sqrt(3);
\a=120; \aa=\a/2;
\xP=0; \yP=\R;
\aOP=54; 
\aPOne=\aOP;
\mPOne=tan(90+\aPOne);
\aPTwo=\aOP;
\mPTwo=tan(-90-\aPTwo);
\aPThree=2.5*\aOP;
\mPThree=tan(90+\aPThree);
\xCQ=1; \yCQ=0;
\xCOne = sqrt(\r^2-\R^2); \yCOne=0;
\xCTwo =-sqrt(\r^2-\R^2); \yCTwo=0;
\xCThree = sqrt(\r^2-\R^2); \yCThree=0;
\A=acos(\xCOne/\r);
\AA=2*\A;
}

\coordinate (O) at (0,0);
\coordinate (C1) at (\xCOne,0);
\coordinate (C2) at (\xCTwo,0);


	(O) circle (\R);
	
	(C2) circle (\r)
	(C1) circle (\r);

\draw[gray!50]
	([shift={(-\aa:\r)}]\xCTwo,0) arc (-\aa:\aa:\r);

\draw[gray!50]
	([shift={(-\aa+180:\r)}]\xCOne,0) arc (-\aa+180:\aa+180:\r);

\pgfmathsetmacro{\xPOne}{ ( 2*((\mPOne)*(\yCOne)+(\xCOne)) - sqrt( ( 2*( (\mPOne)*(\yCOne)+\xCOne )  )^2 - 4*( (\mPOne)^2 + 1 )*( (\xCOne)^2 + (\yCOne)^2 - \r^2 ) ) )/( 2*( (\mPOne)^2+1 ) ) };
\pgfmathsetmacro{\yPOne}{ \mPOne*\xPOne };

\pgfmathsetmacro{\xPTwo}{ ( 2*((\mPTwo)*(\yCTwo)+(\xCTwo) ) + sqrt( ( 2*( (\mPTwo)*\yCTwo+\xCTwo )  )^2 - 4*( (\mPTwo)^2 + 1 )*( (\xCTwo)^2 + (\yCTwo)^2 - \r^2 ) ) )/( 2*( (\mPTwo)^2+1 ) ) };
\pgfmathsetmacro{\yPTwo}{ \mPTwo*\xPTwo };

\pgfmathsetmacro{\xPThree}{ ( 2*((\mPThree)*(\yCThree)+(\xCThree)) - sqrt( ( 2*( (\mPThree)*(\yCThree)+\xCThree )  )^2 - 4*( (\mPThree)^2 + 1 )*( (\xCThree)^2 + (\yCThree)^2 - \r^2 ) ) )/( 2*( (\mPThree)^2+1 ) ) };
\pgfmathsetmacro{\yPThree}{ \mPThree*\xPThree };

\pgfmathsetmacro{\xPPOne}{ ( 2*((\mPOne)*(\yCOne)+(\xCOne)) + sqrt( ( 2*( (\mPOne)*(\yCOne)+\xCOne )  )^2 - 4*( (\mPOne)^2 + 1 )*( (\xCOne)^2 + (\yCOne)^2 - \r^2 ) ) )/( 2*( (\mPOne)^2+1 ) ) };
\pgfmathsetmacro{\yPPOne}{ \mPOne*\xPPOne };

\pgfmathsetmacro{\xPPTwo}{ ( 2*((\mPTwo)*(\yCTwo)+(\xCTwo) ) - sqrt( ( 2*( (\mPTwo)*\yCTwo+\xCTwo )  )^2 - 4*( (\mPTwo)^2 + 1 )*( (\xCTwo)^2 + (\yCTwo)^2 - \r^2 ) ) )/( 2*( (\mPTwo)^2+1 ) ) };
\pgfmathsetmacro{\yPPTwo}{ \mPTwo*\xPPTwo };

\pgfmathsetmacro{\dPOneP}{ sqrt( (\xPOne - \xP)^2 + (\yPOne - \yP)^2 ) };
\pgfmathsetmacro{\aCPOne}{ acos( (2*\r^2 - \dPOneP^2 )/(2*\r^2) ) };
\pgfmathsetmacro{\l}{ \aCPOne/\a };
\pgfmathsetmacro{\rQ}{ sqrt( \R^2 + (\xCQ)^2  )  };
\pgfmathsetmacro{\aPCQ}{ acos( (\xCQ)/(\rQ) ) }
\pgfmathsetmacro{\aQ}{  ( 180 - 1*\aPCQ )*(1-\l) };

\pgfmathsetmacro{\rPOne}{ sqrt( (\xPOne)^2 + (\yPOne)^2 ) }; 
\pgfmathsetmacro{\aPOneCOne}{ acos( ( (\xCOne)^2 + (\r)^2 - (\rPOne)^2 )/( 2*(\xCOne)*(\r)  ) ) }; 
\pgfmathsetmacro{\aPOneCOneP}{ \A - acos( ( (\xCOne)^2 + (\r)^2 - (\rPOne)^2 )/( 2*(\xCOne)*(\r)  ) ) }; 

\pgfmathsetmacro{\angRatioPOne}{\aPOneCOneP/\AA}; 
\pgfmathsetmacro{\aPP}{\R-2*\R*\angRatioPOne}; 

\coordinate (CQ) at (\xCQ, \yCQ);




\coordinate (P) at (0,\R);

\coordinate[rotate around={\aQ:(CQ)}] (Q) at (P);


\node at (Q) {\textcolor{red}{$\cdot$}};

\coordinate (PP) at (0,-\R);

\coordinate (P1) at (\xPOne, \yPOne);
\coordinate (P2) at (\xPTwo, \yPTwo);

\coordinate (PP1) at (\xPPOne,\yPPOne);

\coordinate (PP2) at (\xPPTwo,\yPPTwo);




	(P1) -- (PP1)
	(P2) -- (PP2)
;

\draw[->] ([shift=(90:\R)]215:0.3) arc (215:325:0.3)
;

\draw[line width=1.5]
	(Q) -- (P2)
;

\arcThroughThreePoints[]{P2}{PP}{P}; 
\arcThroughThreePoints[]{P}{PP}{Q}; 


\node at (90: 2.0) {\small $B$};
\node at (270: 2.0) {\small $B^{\ast}$};
\node at (150: 1.35) {\small $A$};
\node at (35: 1.35) {\small $C$};
\node at (215: 1.35) {\small $D$};



\node at (245:1.54) {\textcolor{gray!50}{\small $a$}};

\node at (140:0.95) {\small $\alpha$};

\node at (90:1.2) {\small $\vec{\theta}$};

\end{scope}

\begin{scope}[xshift=3cm] 

\tikzmath{
\r=2.8; \R=sqrt(3);
\a=120; \aa=\a/2;
\xP=0; \yP=\R;
\aOP=56; 
\aPOne=\aOP;
\mPOne=tan(90+\aPOne);
\aPTwo=\aOP;
\mPTwo=tan(-90-\aPTwo);
\aPThree=100;
\mPThree=tan(90+\aPThree);
\xCQ=-sqrt(\r^2-\R^2); \yCQ=0;
\xCOne = sqrt(\r^2-\R^2); \yCOne=0;
\xCTwo =-sqrt(\r^2-\R^2); \yCTwo=0;
\xCThree =-sqrt(\r^2-\R^2); \yCThree=0;
}

\coordinate (O) at (0,0);
\coordinate (C1) at (\xCOne ,0);
\coordinate (C2) at (\xCTwo,0);


	(O) circle ({\R});
	
	(C2) circle (\r)
	(C1) circle (\r);

\draw[gray!50]
	([shift={(-38.25:\r)}]-\xCOne,0) arc (-38.25:38.25:\r);

\draw[gray!50]
	([shift={(-38.25+180:\r)}]\xCOne,0) arc (-38.25+180:38.25+180:\r);

\pgfmathsetmacro{\xPOne}{ ( 2*((\mPOne)*(\yCOne)+(\xCOne)) - sqrt( ( 2*( (\mPOne)*(\yCOne)+\xCOne )  )^2 - 4*( (\mPOne)^2 + 1 )*( (\xCOne)^2 + (\yCOne)^2 - \r^2 ) ) )/( 2*( (\mPOne)^2+1 ) ) };
\pgfmathsetmacro{\yPOne}{ \mPOne*\xPOne };

\pgfmathsetmacro{\xPTwo}{ ( 2*((\mPTwo)*(\yCTwo)+(\xCTwo) ) + sqrt( ( 2*( (\mPTwo)*(\yCTwo)+(\xCTwo) )  )^2 - 4*( (\mPTwo)^2 + 1 )*( (\xCTwo)^2 + (\yCTwo)^2 - (\r)^2 ) ) )/( 2*( (\mPTwo)^2+1 ) ) };
\pgfmathsetmacro{\yPTwo}{ \mPTwo*\xPTwo };

\pgfmathsetmacro{\xPThree}{ ( 2*((\mPThree)*(\yCThree)+(\xCThree)) + sqrt( ( 2*( (\mPThree)*(\yCThree)+\xCThree )  )^2 - 4*( (\mPThree)^2 + 1 )*( (\xCThree)^2 + (\yCThree)^2 - \r^2 ) ) )/( 2*( (\mPThree)^2+1 ) ) };
\pgfmathsetmacro{\yPThree}{ \mPThree*\xPThree };

\pgfmathsetmacro{\xPPOne}{ ( 2*((\mPOne)*(\yCOne)+(\xCOne)) + sqrt( ( 2*( (\mPOne)*(\yCOne)+\xCOne )  )^2 - 4*( (\mPOne)^2 + 1 )*( (\xCOne)^2 + (\yCOne)^2 - \r^2 ) ) )/( 2*( (\mPOne)^2+1 ) ) };
\pgfmathsetmacro{\yPPOne}{ \mPOne*\xPPOne };

\pgfmathsetmacro{\xPPTwo}{ ( 2*((\mPTwo)*(\yCTwo)+(\xCTwo) ) - sqrt( ( 2*( (\mPTwo)*\yCTwo+\xCTwo )  )^2 - 4*( (\mPTwo)^2 + 1 )*( (\xCTwo)^2 + (\yCTwo)^2 - \r^2 ) ) )/( 2*( (\mPTwo)^2+1 ) ) };
\pgfmathsetmacro{\yPPTwo}{ \mPTwo*\xPPTwo };

\pgfmathsetmacro{\xPPThree}{ ( 2*((\mPThree)*(\yCThree)+(\xCThree)) - sqrt( ( 2*( (\mPThree)*(\yCThree)+\xCThree )  )^2 - 4*( (\mPThree)^2 + 1 )*( (\xCThree)^2 + (\yCThree)^2 - \r^2 ) ) )/( 2*( (\mPThree)^2+1 ) ) };
\pgfmathsetmacro{\yPPThree}{ \mPThree*\xPPThree };

\pgfmathsetmacro{\dPOneP}{ sqrt( (\xPOne - \xP)^2 + (\yPOne - \yP)^2 ) };
\pgfmathsetmacro{\aCPOne}{ acos( (2*\r^2 - \dPOneP^2 )/(2*\r^2) ) };
\pgfmathsetmacro{\l}{ \aCPOne/\a };
\pgfmathsetmacro{\rQ}{ sqrt( \R^2 + (\xCQ)^2  )  };
\pgfmathsetmacro{\aPCQ}{ acos( (\xCQ)/(\rQ) ) }
\pgfmathsetmacro{\aQ}{-( 360 - 2*\aPCQ )*(1-\l) };

\coordinate (CQ) at (\xCQ, \yCQ);


\coordinate (P) at (0,{\R});

\coordinate (PP) at (0,{-\R});

\coordinate (P1) at (\xPOne, \yPOne);
\coordinate (P2) at (\xPTwo, \yPTwo);
\coordinate (P3) at (\xPThree, \yPThree);

\coordinate (PP1) at (\xPPOne,\yPPOne);
\coordinate (PP2) at (\xPPTwo,\yPPTwo);
\coordinate (PP3) at (\xPPThree,\yPPThree);

\coordinate[rotate around={180:(O)}] (Q) at (P1);





	(P1) -- (PP1)
	(P2) -- (PP2)
;

\draw[->] ([shift=(90:\R)]234:0.3) arc (234:306:0.3)
;

\draw[]
	(Q) -- (0,0) -- (P1)
;

\arcThroughThreePoints[]{P}{PP}{P1};
\arcThroughThreePoints[]{P2}{PP}{P};

\arcThroughThreePoints[line width=1.5]{Q}{P}{P2};


\node at (90: 2.0) {\small $B$};
\node at (270: 2.0) {\small $B^{\ast}$};
\node at (140: 1.0) {\small $A$};
\node at (35: 1.0) {\small $C$};
\node at (325: 0.95) {\small $D$};



\node at (292.5:1.32) {\textcolor{gray!50}{\small $a$}};

\node at (122:0.6) {\small $\alpha$};

\node at (90:1.15) {\small $\vec{\theta}$};

\end{scope}

\begin{scope}[xshift=6cm] 


\tikzmath{
\r=2; \R=sqrt(3);
\a=120; \aa=\a/2;
\xP=0; \yP=\R;
\aOP=54; 
\aPOne=\aOP;
\mPOne=tan(90+\aPOne);
\aPTwo=\aOP;
\mPTwo=tan(-90-\aPTwo);
\aPThree=72;
\mPThree=tan(90+\aPThree);
\xCQ=-1; \yCQ=0;
\xCOne = sqrt(\r^2-\R^2); \yCOne=0;
\xCTwo =-sqrt(\r^2-\R^2); \yCTwo=0;
\xCThree = sqrt(\r^2-\R^2); \yCThree=0;
\A=acos(\xCOne/\r);
\AA=2*\A;
}

\coordinate (O) at (0,0);
\coordinate (C1) at (\xCOne ,0);
\coordinate (C2) at (\xCTwo,0);




	(O) circle (\R);
	
	(C2) circle (\r)
	(C1) circle (\r);

\draw[gray!50]
	([shift={(-\aa:\r)}]\xCTwo,0) arc (-\aa:\aa:\r);

\draw[gray!50]
	([shift={(-\aa+180:\r)}]\xCOne,0) arc (-\aa+180:\aa+180:\r);

\pgfmathsetmacro{\xPOne}{ ( 2*((\mPOne)*(\yCOne)+(\xCOne)) - sqrt( ( 2*( (\mPOne)*(\yCOne)+\xCOne )  )^2 - 4*( (\mPOne)^2 + 1 )*( (\xCOne)^2 + (\yCOne)^2 - \r^2 ) ) )/( 2*( (\mPOne)^2+1 ) ) };
\pgfmathsetmacro{\yPOne}{ \mPOne*\xPOne };

\pgfmathsetmacro{\xPTwo}{ ( 2*((\mPTwo)*(\yCTwo)+(\xCTwo) ) + sqrt( ( 2*( (\mPTwo)*\yCTwo+\xCTwo )  )^2 - 4*( (\mPTwo)^2 + 1 )*( (\xCTwo)^2 + (\yCTwo)^2 - \r^2 ) ) )/( 2*( (\mPTwo)^2+1 ) ) };
\pgfmathsetmacro{\yPTwo}{ \mPTwo*\xPTwo };

\pgfmathsetmacro{\xPPOne}{ ( 2*((\mPOne)*(\yCOne)+(\xCOne)) + sqrt( ( 2*( (\mPOne)*(\yCOne)+\xCOne )  )^2 - 4*( (\mPOne)^2 + 1 )*( (\xCOne)^2 + (\yCOne)^2 - \r^2 ) ) )/( 2*( (\mPOne)^2+1 ) ) };
\pgfmathsetmacro{\yPPOne}{ \mPOne*\xPPOne };

\pgfmathsetmacro{\xPPTwo}{ ( 2*((\mPTwo)*(\yCTwo)+(\xCTwo) ) - sqrt( ( 2*( (\mPTwo)*\yCTwo+\xCTwo )  )^2 - 4*( (\mPTwo)^2 + 1 )*( (\xCTwo)^2 + (\yCTwo)^2 - \r^2 ) ) )/( 2*( (\mPTwo)^2+1 ) ) };
\pgfmathsetmacro{\yPPTwo}{ \mPTwo*\xPPTwo };

\pgfmathsetmacro{\dPOneP}{ sqrt( (\xPOne - \xP)^2 + (\yPOne - \yP)^2 ) };
\pgfmathsetmacro{\aCPOne}{ acos( (2*\r^2 - \dPOneP^2 )/(2*\r^2) ) };
\pgfmathsetmacro{\l}{ \aCPOne/\a };
\pgfmathsetmacro{\rQ}{ sqrt( \R^2 + (\xCQ)^2  )  };
\pgfmathsetmacro{\aPCQ}{ acos( (\xCQ)/(\rQ) ) }
\pgfmathsetmacro{\aQ}{  -( 360 - 2*\aPCQ )*(1-\l) };

\pgfmathsetmacro{\rPOne}{ sqrt( (\xPOne)^2 + (\yPOne)^2 ) }; 
\pgfmathsetmacro{\aPOneCOne}{ acos( ( (\xCOne)^2 + (\r)^2 - (\rPOne)^2 )/( 2*(\xCOne)*(\r)  ) ) }; 
\pgfmathsetmacro{\aPOneCOneP}{ \A - acos( ( (\xCOne)^2 + (\r)^2 - (\rPOne)^2 )/( 2*(\xCOne)*(\r)  ) ) }; 

\pgfmathsetmacro{\angRatioPOne}{\aPOneCOneP/\AA}; 
\pgfmathsetmacro{\aPP}{\R-2*\R*\angRatioPOne}; 



\coordinate (V) at (270:\aPP);

\coordinate (CQ) at (\xCQ, \yCQ);


\coordinate (P) at (0,{sqrt(3)});

\coordinate[rotate around={\aQ:(CQ)}] (Q) at (P);


\coordinate (PP) at (0,-\R);

\coordinate (P1) at (\xPOne, \yPOne);
\coordinate (P2) at (\xPTwo, \yPTwo);

\coordinate (PP1) at (\xPPOne,\yPPOne);

\coordinate (PP2) at (\xPPTwo,\yPPTwo);



	(P1) -- (PP1)
	(P2) -- (PP2)
;

\draw[->] ([shift=(90:\R)]215:0.3) arc (215:270:0.3)
;

\draw[dashed]
	(P) -- (V)
;

\draw[gray!50]
	(PP) -- (V)
;

\arcThroughThreePoints[]{P}{PP}{P1};
\arcThroughThreePoints[]{P2}{PP}{P};

\arcThroughThreePoints[]{P1}{PP1}{V};
\arcThroughThreePoints[line width=1.5]{V}{PP2}{P2};


\node at (90: 2.0) {\small $B$};
\node at (270: 2.0) {\small $B^{\ast}$};
\node at (150: 1.35) {\small $A$};
\node at (30: 1.35) {\small $C$};
\node at (285: 0.75) {\small $D$};



\node at (262.5:1.1) {\textcolor{gray!50}{\small $a$}};

\node at (135:0.95) {\small $\alpha$};

\node at (105:1.2) {\small $\vec{\theta}$};

\end{scope}

\end{tikzpicture}
\caption{Degenerated $a^3b$ quadrilaterals: triangles and rhombus}
\label{a3bDegenShapes}
\end{figure}
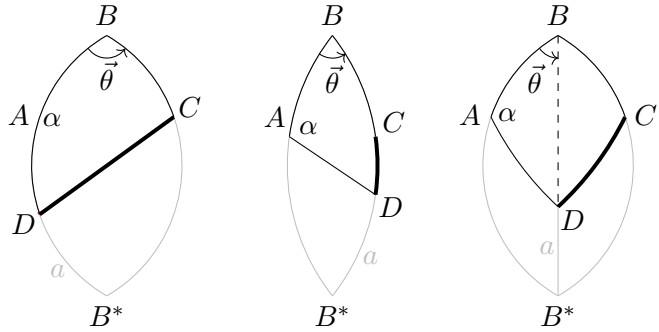

\subsubsection*{\bm{$E^{\prime}: \AVC \equiv \{ \alpha\gamma\delta,\alpha^m\beta^{n_1},\beta^{n}\gamma\delta \}, \{ \alpha\gamma\delta,\alpha^3,\beta^{n}\gamma\delta \}$}}

By $\alpha\gamma\delta$ and $\beta^{n}\gamma\delta$, we have $\alpha = n\beta$ in both $\AVC$s. We further have $n_1 = \tfrac{f}{2} - mn$ where $m = 1,2,3$.

\begin{case*}[Convex]

If the quadrilateral is convex, then it is equivalent to $n\beta=\alpha \in [\pi-\beta, \pi]$. 

If $m=3$, then by $\alpha+\beta\ge\pi$ and $\alpha>\beta$, we know that $\alpha^3\beta^{n_1}$ is not a vertex and $\alpha^3$ is a vertex. So $\pi - \beta \le n\beta = \alpha = \frac{2}{3}\pi$ implies $\beta=\frac{1}{3}\pi$ and $f=12$ and $n=2$. We get
\begin{align*}
f=12, \quad \AVC \equiv \{ \alpha\gamma\delta,\alpha^3,\beta^{2}\gamma\delta \}.
\end{align*}
We have $\alpha+\beta = \pi$ and $\gamma=\pi$. The quadrilateral is in fact the triangle $\triangle ABD$  in the second picture of Figure \ref{a3bDegenShapes} and the tiling is effectively a non-edge-to-edge tiling by congruent triangles which is an example of tilings with angles $\frac{2}{3}\pi, \frac{1}{3}\pi, \frac{1}{3}\pi$ in \cite[Table 1]{da}. We illustrate the non-edge-to-edge triangular tiling (or the earth map tiling $E$) given by $\AVC\equiv\{ \alpha\gamma\delta, \beta^6 \}$ in the first picture of Figure \ref{Tiling-NE2E-Isos-algade} and the non-edge-to-edge triangular tiling (or $E^{\prime}$) given by $\AVC \equiv \{ \alpha\gamma\delta,\alpha^3,\beta^{2}\gamma\delta \}$ in the second. One may fill the angles by first locating $\gamma=\pi$. The spherical ones are illustrated in Figure \ref{Tilings-NE2E-E-EF1}.

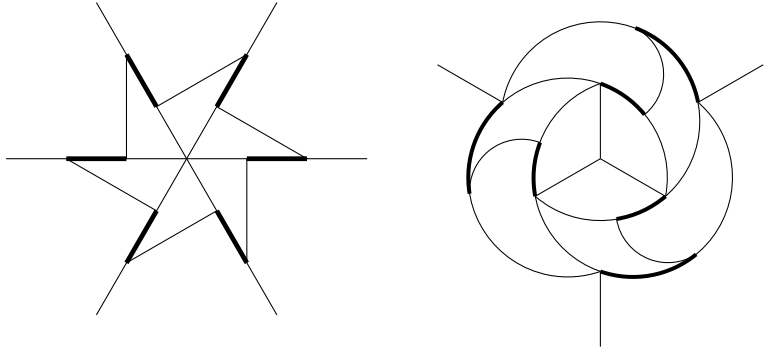
\begin{figure}[htp] 
\centering
\begin{tikzpicture}[scale=1]

\begin{scope} 

\tikzmath{
\r=0.8; 
}

\foreach \a in {0,...,5}{

\draw[rotate=60*\a]
	(0:0) -- (0:3*\r)
	(60:\r) -- (0:2*\r)
;

\draw[rotate=60*\a, line width=1.8]
	(0:\r) -- (0:2*\r)
;

}

\end{scope} 

\begin{scope}[xshift=5.5cm] 

\tikzmath{
\r=1; 
}

\foreach \a in {0,...,2}{

\draw[rotate=120*\a]
	(0:0) -- (90:\r)
;

\coordinate (A) at (90+120*\a:\r);
\coordinate (B) at (210+120*\a:\r);
\coordinate (C) at (30+120*\a:1.5*\r);
\arcThroughThreePoints[]{A}{B}{C};

\coordinate (D) at (45+120*\a:0.83*\r);
\coordinate (E) at (75+120*\a:1.8*\r);
\coordinate (F) at (90+120*\a:1.8*\r);
\arcThroughThreePoints[]{D}{F}{E};

\draw[rotate=120*\a]
	(30+120*\a:1.5*\r) -- (30+120*\a:2.5*\r)
;

\coordinate (D2) at (45+120+120*\a:0.83*\r);

\arcThroughThreePoints[line width=1.5]{D2}{C}{B};

\coordinate (C2) at (30+120+120*\a:1.5*\r);
\coordinate (C3) at (30+2*120+120*\a:1.5*\r);

\coordinate (E2) at (75+120+120*\a:1.8*\r);

\arcThroughThreePoints[line width=1.5]{C2}{C3}{E2};

}

\end{scope} 

\end{tikzpicture}
\caption{Non-edge-to-edge tilings by congruent isosceles triangles given by $\AVC \equiv \{ \alpha\gamma\delta, \beta^6\}$ and $\AVC \equiv \{ \alpha\gamma\delta, \alpha^3, \beta^2\gamma\delta \}$ where $\gamma=\pi$}
\label{Tiling-NE2E-Isos-algade}
\end{figure}

\begin{figure}
	\centering
		\adjustbox{trim=\dimexpr.5\Width-15mm\relax{} \dimexpr.5\Height-15mm\relax{}  \dimexpr.5\Width-15mm\relax{} \dimexpr.5\Height-15mm\relax{} ,clip}{\includegraphics[height=6cm]{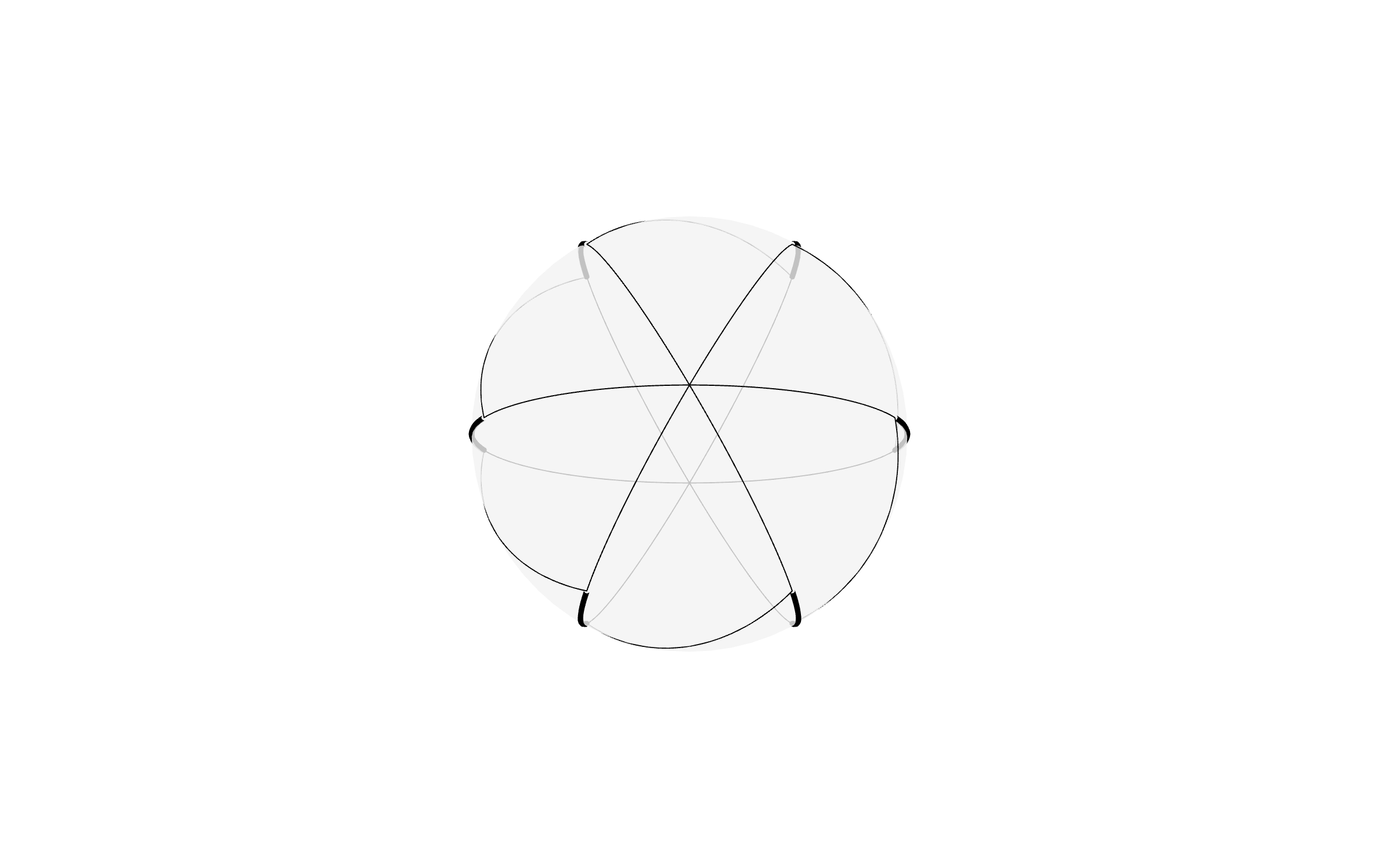}} \quad \quad \quad
		\adjustbox{trim=\dimexpr.5\Width-15mm\relax{} \dimexpr.5\Height-15mm\relax{}  \dimexpr.5\Width-15mm\relax{} \dimexpr.5\Height-15mm\relax{},clip}{\includegraphics[height=6cm]{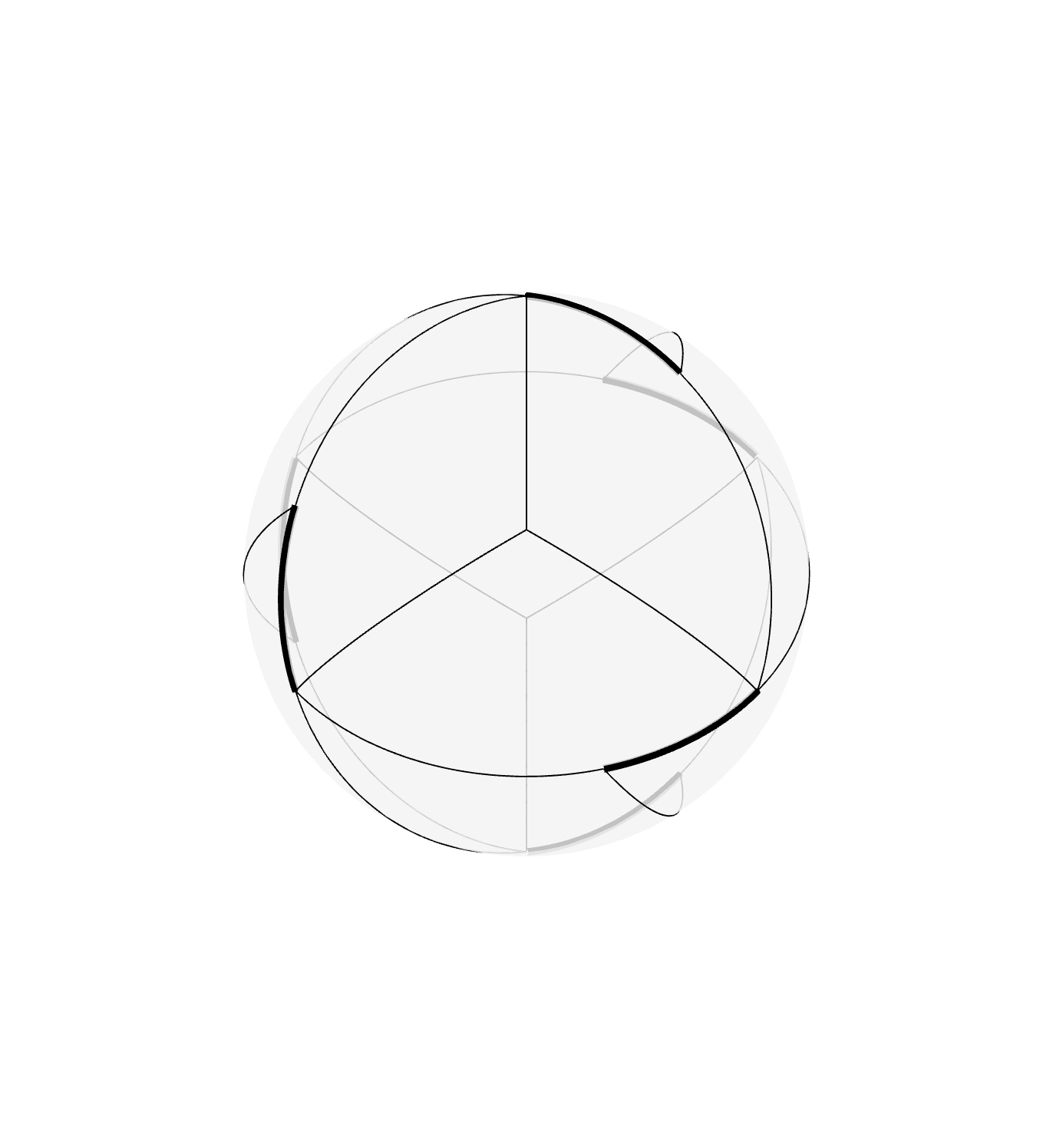}}  \quad
	\caption{Non-edge-to-edge tilings (degenerated $E,E^{\prime}$) by congruent isosceles triangles}
	\label{Tilings-NE2E-E-EF1}
\end{figure}

If $m=1,2$, then one of $\alpha\beta^{n_1}, \alpha^2\beta^{n_1}$ is a vertex. The length of $[\pi-\beta, \pi]$ is $\beta$. Then there are two possibilities of $n\beta \in [\pi-\beta, \pi]$. 

The first possibility is that both $\pi-\beta, \pi$ are multiples of $\beta$, which means $(p-1)\beta=\pi-\beta$ and $p\beta=\pi$ where $p\in \mathbb{N}$. This implies $f=4p$ and $p\ge2$. Correspondingly we further get $n=p-1, p$. If $n=p-1$, then 
\begin{align*}
\AVC &\equiv \{ \alpha\gamma\delta,\alpha\beta^{p+1},\beta^{p-1}\gamma\delta \}, \quad p=\tfrac{f}{4}; \\
\AVC &\equiv \{ \alpha\gamma\delta,\alpha^2\beta^{2},\beta^{p-1}\gamma\delta \}, \quad p=\tfrac{f}{4}.
\end{align*}
As $\alpha=(p-1)\beta$, we have $\alpha+\beta=\pi$ and $\gamma=\pi$. The quadrilateral is in fact the triangle in the second picture of Figure \ref{a3bDegenShapes}. These families of tilings are another examples of non-edge-to-edge triangular tilings with angles $(1-\frac{4}{f})\pi, \frac{4}{f}\pi, \frac{4}{f}\pi$ in \cite[Table 1]{da}. These examples are illustrated in Figures \ref{Tiling-NE2E-Isos-algade-albe4-be2gade}, \ref{Tiling-NE2E-Isos-algade-al2be2-be2gade} for $f=12$.

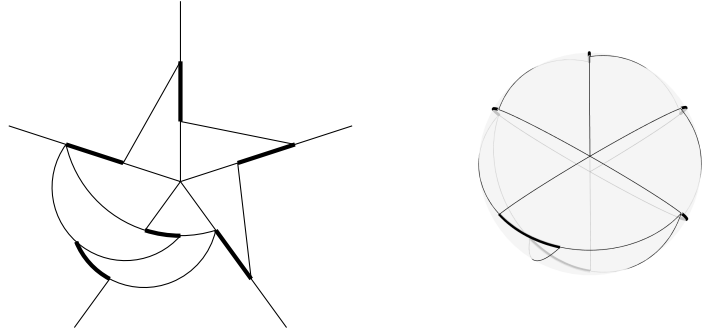
\begin{figure}[htp] 

\minipage{0.5\textwidth}
\raggedleft
\begin{tikzpicture}[scale=1]

\begin{scope}

\tikzmath{
\r=0.8; 
}

\foreach \a in {0,3,4}{

\draw[rotate=72*\a]
	(0:0) -- (90:3*\r)
	(72+90:\r) -- (90:2*\r)
;

\draw[rotate=72*\a, line width=1.6]
	(90:\r) -- (90:2*\r)
;

}

\draw[]
	(0:0) -- (90+72:3*\r)
	(0:0) -- (90+2*72:1*\r)
;

\coordinate (A) at (90+1*72:2*\r);
\coordinate (B) at (90+2*72:\r);
\coordinate (C) at (90+3*72:\r);

\arcThroughThreePoints[]{A}{B}{C};

\coordinate (D) at (210:2*\r);
\coordinate (E) at (270:0.9*\r);

\arcThroughThreePoints[]{A}{D}{E};

\coordinate (F) at (90+2*72:2*\r);

\arcThroughThreePoints[]{D}{F}{C};

\draw[]
	(90+2*72:2*\r) -- (90+2*72:3*\r)
;

\draw[line width=1.6]
	(90+72:\r) -- (90+72:2*\r)
;

\arcThroughThreePoints[line width=1.6]{B}{C}{E};

\arcThroughThreePoints[line width=1.6]{D}{C}{F};

\end{scope} 

\end{tikzpicture}
\endminipage \quad \quad \quad \quad
\minipage{0.5\textwidth}
\raggedright
		\adjustbox{trim=\dimexpr.5\Width-15mm\relax{} \dimexpr.5\Height-15mm\relax{}  \dimexpr.5\Width-15mm\relax{} \dimexpr.5\Height-15mm\relax{} ,clip}{\includegraphics[height=6cm]{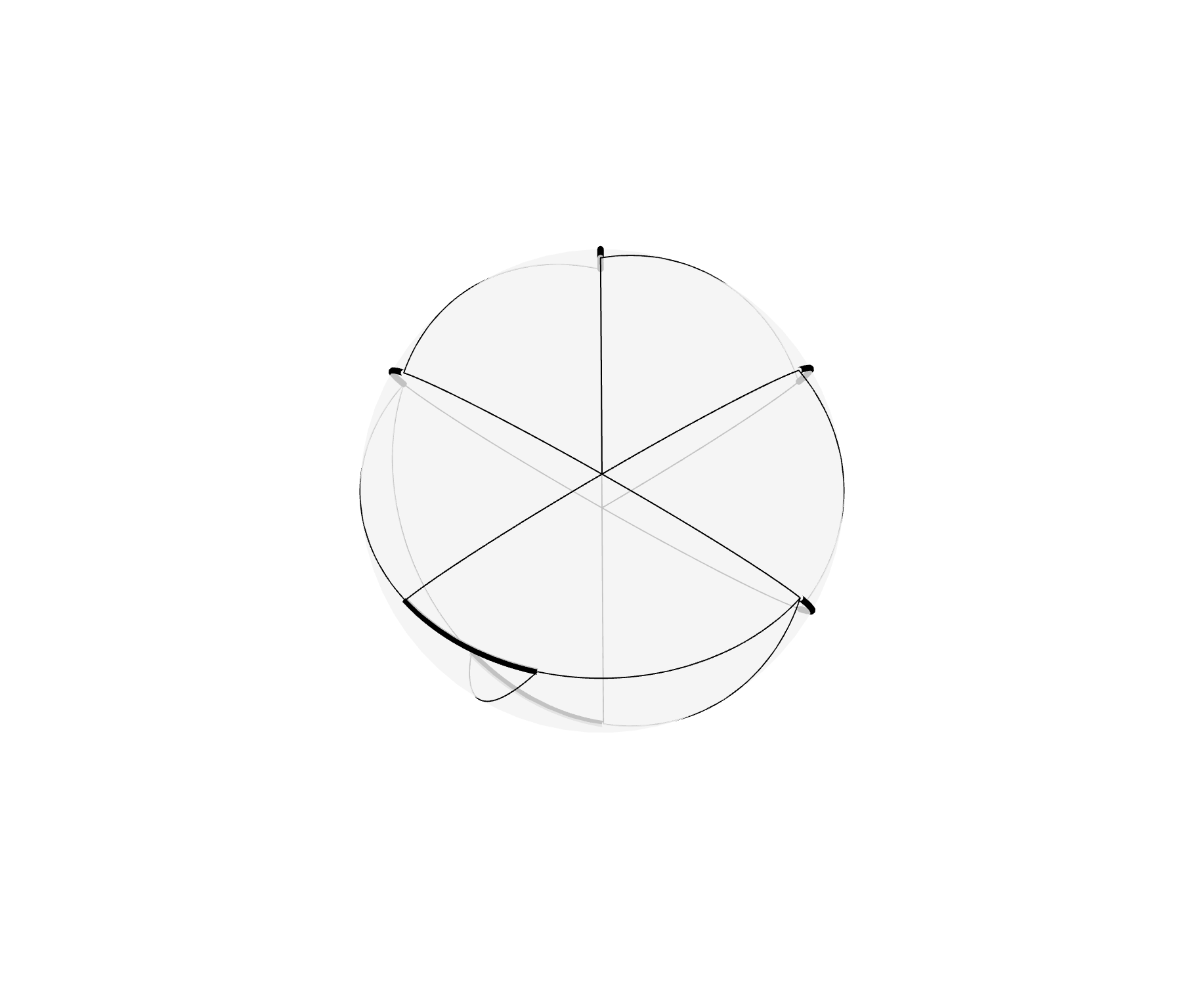}} 
\endminipage

\caption{Non-edge-to-edge tiling by congruent isosceles triangles given by $\AVC  \equiv \{ \alpha\gamma\delta, \alpha\beta^4, \beta^2\gamma\delta \}$ where $\gamma=\pi$}
\label{Tiling-NE2E-Isos-algade-albe4-be2gade}

\end{figure}

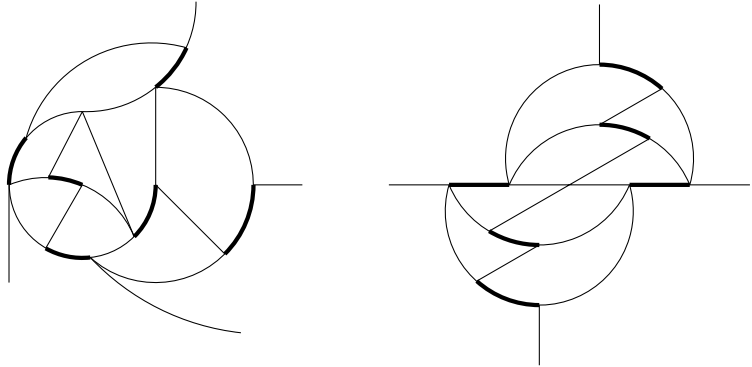
\begin{figure}[htp] 
\centering
\begin{tikzpicture}[scale=1]

\begin{scope}

\tikzmath{
\r=0.65; 
\rr=\r/2; 
}

\coordinate (O) at (0:0);
\coordinate (O1) at (180:1.5*\r);
\coordinate (O2) at ([shift={(180:1.5*\r)}]90:3.75*\r);

\coordinate (A) at (228:2*\r);
\coordinate (B) at (0:2*\r);
\coordinate (C) at (90:2*\r);

\coordinate (A1) at (180:3*\r);

\coordinate (A2) at ([shift={(180:2*\r)}]135:1*\r);

\coordinate (A3) at ([shift={(O1)}]140:1.5*\r);

\coordinate (A4) at ([shift={(O1)}]-45:1.5*\r);

\coordinate (A5) at ([shift={(180:3*\r/2)}]240:1.5*\r);

\coordinate (B1) at (-45:2*\r);

\coordinate (C1) at (-60:3.5*\r);

\coordinate (C2) at (90:4*\r);

\coordinate (D) at (135:\r);

\coordinate (D1) at (180:3*\r/2);

\coordinate (D2) at ([shift={(180:3*\r/2)}]90:1.5*\r);

\coordinate (D3) at ([shift={(180:1.5*\r)}]90:2.5*\r);

\coordinate (D4) at (176:2.2*\r);

\coordinate (B2) at ([shift={(D2)}]31.5:2.5*\r);

\coordinate (B3) at ([shift={(O2)}]0:2.325*\r);

\coordinate (B4) at (270:2.5*\r);

	(O1) circle (1.5*\r);

	(D2) circle (2.5*\r);

\arcThroughThreePoints[]{A}{B}{C};

\arcThroughThreePoints[]{D2}{A}{O};

\arcThroughThreePoints[]{D2}{C}{B3};

\arcThroughThreePoints[]{B2}{D3}{A3};

\arcThroughThreePoints[]{A}{B4}{C1};

\arcThroughThreePoints[]{A4}{D1}{A1};

\draw[]
	(90:0) -- (90:2*\r) 
	(0:2*\r) -- (0:3*\r)
	(180:3*\r) -- ([shift={(180:3*\r)}]270:2*\r)
	(0:0) -- (-45:2*\r)
	([shift={(180:3*\r/2)}]90:1.5*\r) -- (176:2.2*\r)
	([shift={(180:3*\r/2)}]240:1.5*\r) -- (180:1.5*\r)
	([shift={(180:3*\r/2)}]90:1.5*\r) -- ([shift={(180:1.5*\r)}]-45:1.5*\r)
;

\arcThroughThreePoints[line width=1.6]{A5}{A4}{A};

\arcThroughThreePoints[line width=1.6]{O1}{A1}{D4};

\arcThroughThreePoints[line width=1.6]{A3}{A5}{A1};

\arcThroughThreePoints[line width=1.6]{C}{B3}{B2};

\arcThroughThreePoints[line width=1.6]{B1}{C}{B};

\arcThroughThreePoints[line width=1.6]{A4}{D2}{O};

\end{scope}

\begin{scope}[xshift=5.5cm] 

\tikzmath{
\r=0.8; 
\rr=\r/2;
\A = atan(\r/\rr);
\h = sqrt(\r^2+\rr^2);
\AA = atan(2*\r/\rr);
\hh = sqrt((2*\r)^2+\rr^2);
}

\coordinate (O1) at (0:\r/2);
\coordinate (O2) at (0:-\r/2);

\foreach \a in {0,1}{

\draw[rotate=180*\a]
	(0:0) -- (0:3*\r)
	(0:0) -- (30:1.54*\r)
;

\coordinate (A) at (0+180*\a:2*\r); 
\coordinate (B) at (\A+180*\a:\h); 
\coordinate (C) at (180+180*\a:1*\r); 

\arcThroughThreePoints[]{A}{B}{C};

\coordinate (B2) at (\AA+180*\a:\hh); 

\arcThroughThreePoints[]{A}{B2}{C};

\draw[]
	(B) -- ([shift={(\A+180*\a:\h)}]30+180*\a:1.2*\r)
;

\coordinate (B3) at (\AA+180*\a:\hh); 

\draw[]
	(B3) -- ([shift={(\AA+180*\a:\hh)}]90+180*\a:\r)
;
	
\draw[line width=1.6, rotate=180*\a]
	(0:1*\r) -- (0:2*\r)
;

\coordinate (D) at (30+180*\a:1.54*\r);
	\arcThroughThreePoints[line width=1.6]{D}{C}{B};

\coordinate (F) at ([shift={(\A+180*\a:\h)}]30+180*\a:1.2*\r);
	\arcThroughThreePoints[line width=1.6]{F}{C}{B3};

}

\end{scope}

\end{tikzpicture}
\caption{Non-edge-to-edge tilings by congruent isosceles triangles given by $\AVC  \equiv \{ \alpha\gamma\delta, \alpha^2\beta^2, \beta^2\gamma\delta \}$ where $\gamma=\pi$}
\label{Tiling-NE2E-Isos-algade-al2be2-be2gade}
\end{figure}

\begin{figure}
	\centering
		\adjustbox{trim=\dimexpr.5\Width-15mm\relax{} \dimexpr.5\Height-15mm\relax{}  \dimexpr.5\Width-15mm\relax{} \dimexpr.5\Height-15mm\relax{} ,clip}{\includegraphics[height=6cm]{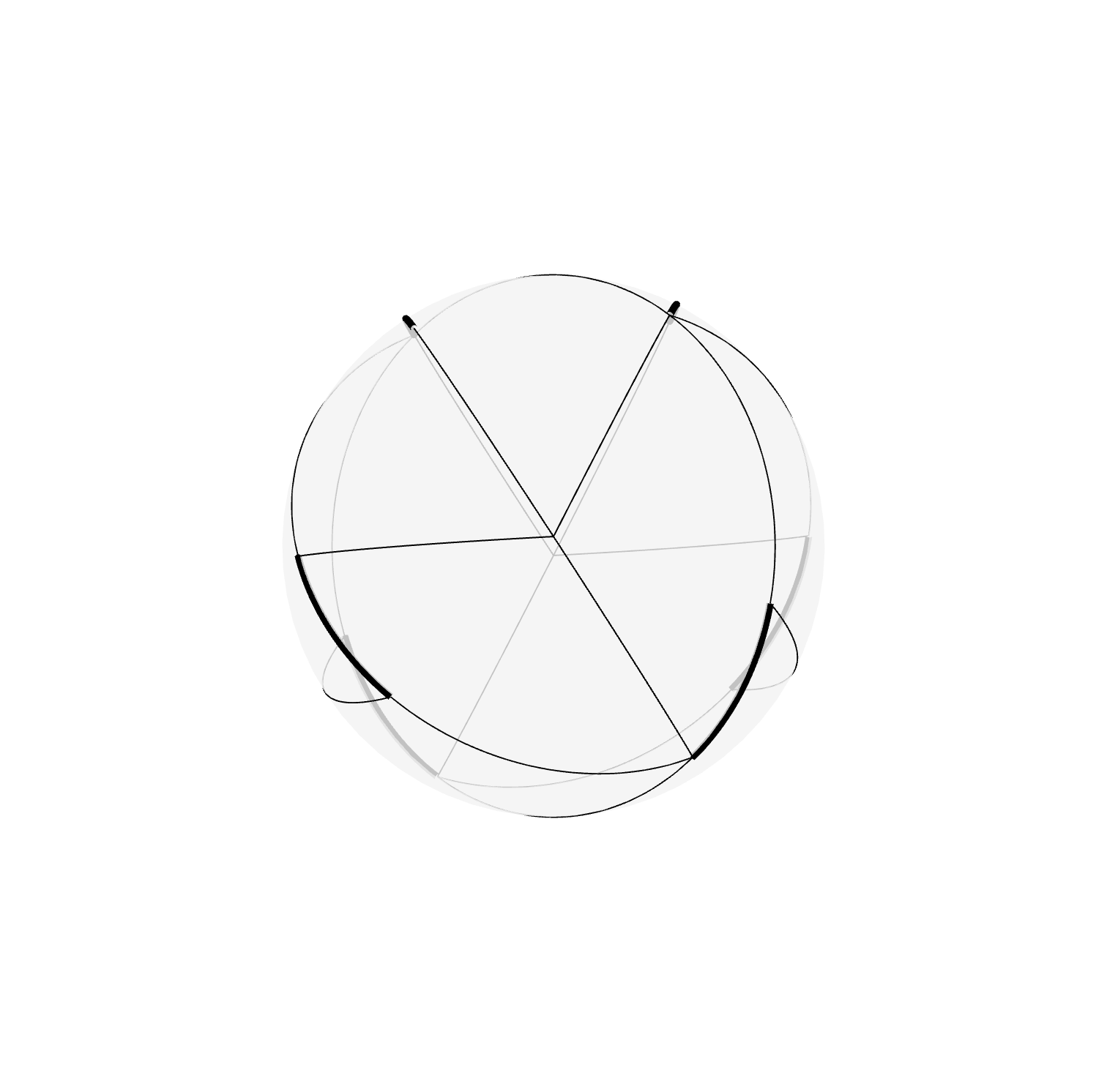}} \quad \quad \quad \quad
		\adjustbox{trim=\dimexpr.5\Width-15mm\relax{} \dimexpr.5\Height-15mm\relax{}  \dimexpr.5\Width-15mm\relax{} \dimexpr.5\Height-15mm\relax{},clip}{\includegraphics[height=6cm]{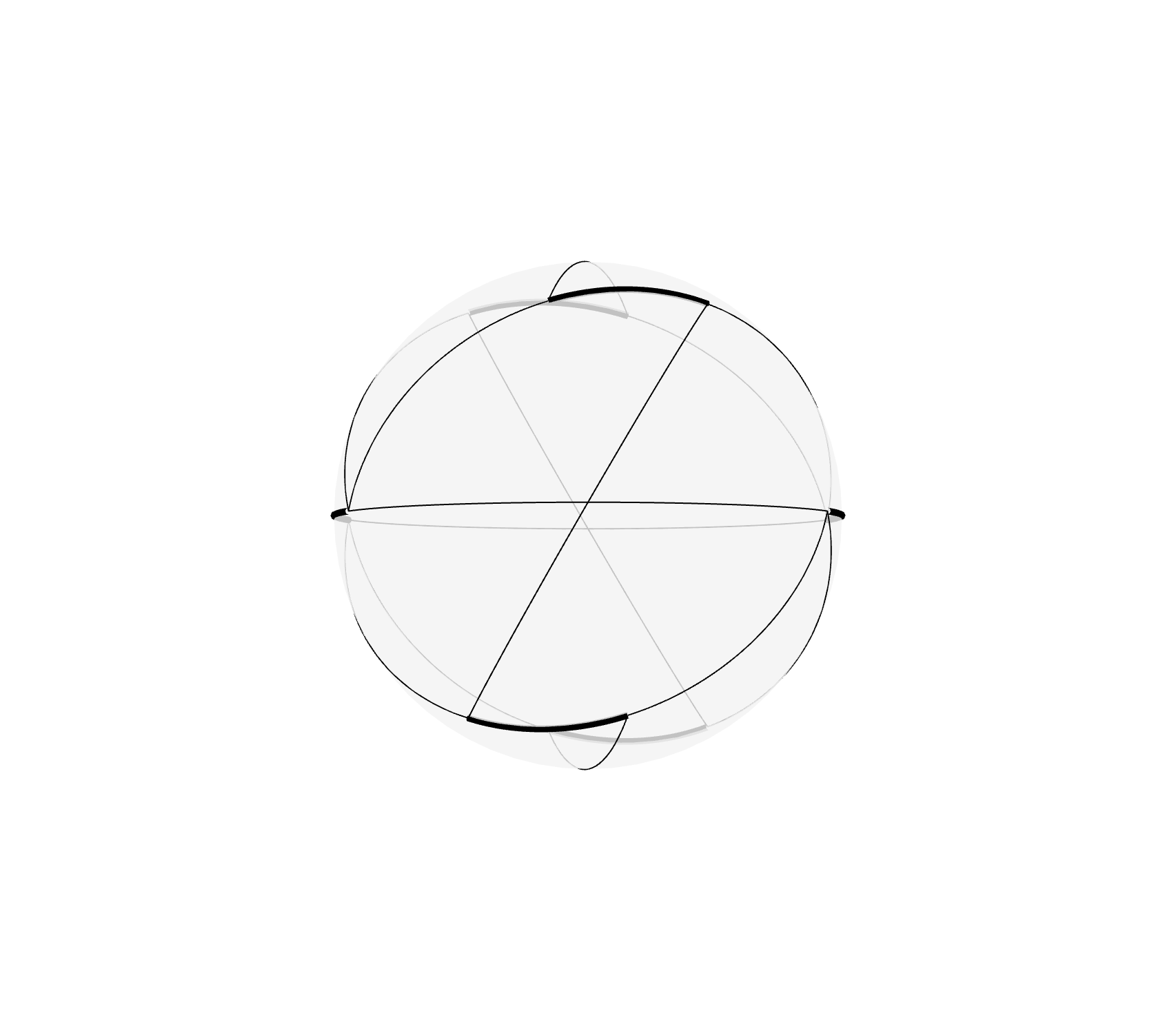}}  \quad
	\caption{Non-edge-to-edge tilings (two tilings of degenerated $E^{\prime}$) by congruent isosceles triangles}
	\label{Tilings-NE2E-EF2}
\end{figure}

On the other hand, if $n=p$, then 
\begin{align*}
\AVC &\equiv \{ \alpha\gamma\delta,\alpha\beta^{p},\beta^{p}\gamma\delta \}, \quad p=\tfrac{f}{4}.
\end{align*}
As $\alpha=p\beta$, we have $\alpha=\pi$.  The quadrilateral is in fact the triangle $\triangle BCD$ in the first picture of Figure \ref{a3bDegenShapes}. This family of tilings are examples of non-edge-to-edge tilings by congruent triangles. An example is illustrated in Figure \ref{Tiling-NE2E-algade-albe3-be3gade} for $f=12$.

\begin{figure}[htp]

\minipage{0.5\textwidth}
\raggedleft
\begin{tikzpicture}[scale=1]

\tikzmath{
\r=0.8; 
\rr=\r/2; 
}

\foreach \a in {0,1,2}{

\draw[rotate=60*\a]
	(0:0) -- (0:3*\r)
	(60:\r) -- (0:2*\r)
;

\draw[rotate=60*\a, line width=1.6]
	(60:\r) -- (0:2*\r)
;

}

\draw[]
	(0:0) -- (180:1*\r)
;

\coordinate (A) at (180:1*\r);
\coordinate (B) at (\rr,-\r);
\coordinate (C) at (0:2*\r);

\arcThroughThreePoints[]{A}{B}{C};

\coordinate (B2) at (\rr,-2*\r);

\coordinate (B3) at (210:1.5*\r);

\arcThroughThreePoints[]{A}{B2}{C};

\coordinate (C2) at (180+60:0.825*\r);

\coordinate (C3) at (180+150:2.21*\r);

\coordinate (C4) at (180+150:3*\r);

\arcThroughThreePoints[]{A}{B3}{C4};

\draw[line width=1.6]
	(180+60:0.825*\r) -- (0:\r)
	(\rr,-\r)-- (\rr,-2*\r)
	(180+150:2.21*\r) -- (180+150:3*\r)
;

\draw[]
	(180+150:2.21*\r) -- (180+150:4*\r)
;
\end{tikzpicture}
\endminipage \quad  \quad \quad
\minipage{0.5\textwidth}
\raggedright
		\adjustbox{trim=\dimexpr.5\Width-15mm\relax{} \dimexpr.5\Height-15mm\relax{}  \dimexpr.5\Width-15mm\relax{} \dimexpr.5\Height-15mm\relax{} ,clip}{\includegraphics[height=6cm]{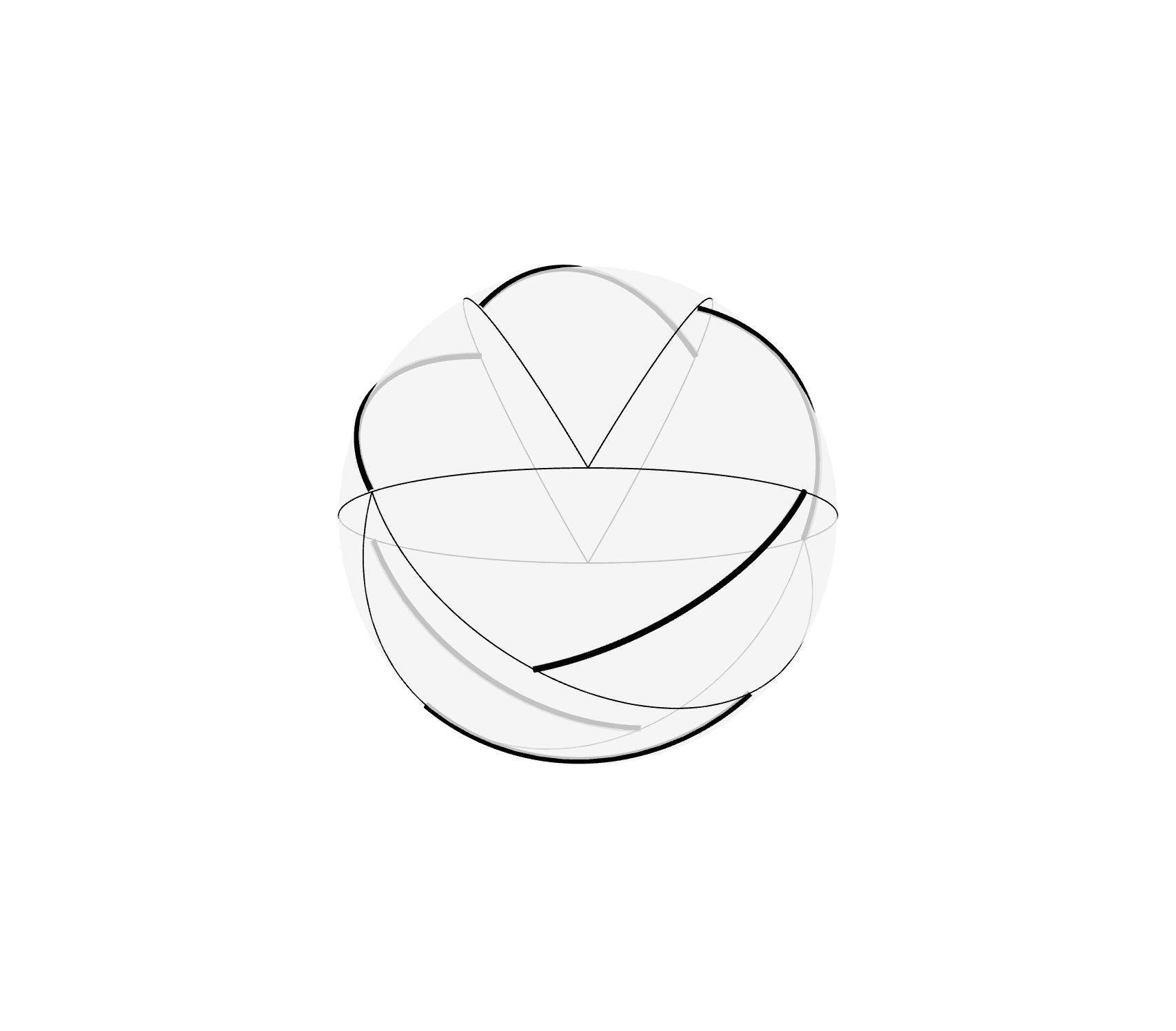}} 
\endminipage
\caption{Non-edge-to-edge tiling by congruent triangles given by $\AVC  \equiv \{ \alpha\gamma\delta, \alpha\beta^3, \beta^3\gamma\delta \}$ where $\alpha=\pi$}
\label{Tiling-NE2E-algade-albe3-be3gade}
\end{figure}

The second possibility is that $n\beta \neq  \pi - \beta, \pi$. Then even $f$ implies $f=4p+2$. By $(p-\frac{1}{2})\beta = \pi - \beta < \alpha = n\beta < (p+\frac{1}{2})\beta$, we get $p-\frac{1}{2} < n < p +\frac{1}{2}$, which implies $n=p$. So we have $\alpha=\frac{f-2}{4}\beta=\pi-\frac{1}{2}\beta$. By $m=1,2$ and $n_1 = \tfrac{f}{2} - mn$, the $\AVC$ becomes
\begin{align*}
\AVC &\equiv \{ \alpha\gamma\delta,\alpha\beta^{p+1},\beta^{p}\gamma\delta \}, \quad p=\tfrac{f-2}{4}; \\
\AVC &\equiv \{ \alpha\gamma\delta,\alpha^2\beta,\beta^{p}\gamma\delta \}, \quad p=\tfrac{f-2}{4}.
\end{align*}
By \eqref{algade-alth-eq}, we get $\vec{\theta}=\frac{1}{2}\beta$. The quadrilateral is in fact the rhombus in the third picture of Figure \ref{a3bDegenShapes}. These families of tilings are the ones given by \cite[$\AVC$ (4.1)]{cly}.
\end{case*}

\begin{case*}[$\alpha>\pi$] If $\alpha>\pi$, then $\alpha^m$ is not a vertex and $\alpha^m\beta^n=\alpha\beta^n$. Then by $n\beta=\alpha<\frac{3}{2}\pi$, we have $n\beta\in (\pi, \frac{3}{2}\pi)$, which implies $n \in (\frac{f}{4}, \frac{3f}{8})$. By $m=1$ and $n_1 = \tfrac{f}{2} - mn$, we get
\begin{align*}
\AVC \equiv \{ \alpha\gamma\delta,\alpha\beta^{\frac{f}{2} - n}, \beta^{n}\gamma\delta \}, \quad n \in ( \tfrac{f}{4}, \tfrac{3f}{8} ).
\end{align*}
\end{case*}

\begin{case*}[$\gamma>\pi$] If $\gamma>\pi$, then $\alpha+\beta<\pi$. 

If $\alpha^3$ is a vertex, then $\alpha=\frac{2}{3}\pi$ and $\gamma+\delta = \frac{4}{3}\pi$. By $\alpha+\beta<\pi$ and $\beta^n\gamma\delta$ and $\beta=\frac{4}{f}\pi$, we then get $n=\frac{f}{6}$ and $f>12$. We have
\begin{align*}
\AVC \equiv \{ \alpha\gamma\delta, \alpha^3, \beta^{\frac{f}{6}}\gamma\delta \}.
\end{align*}

If $\alpha^m\beta^{n_1}$ is a vertex, then by $\frac{1}{2}\pi<\alpha=n\beta$ and $\alpha+\beta<\pi$ and $0<n_1 = \frac{f}{2} - mn$ and $m=1,2,3$, we get
\begin{align*}
\AVC &\equiv \{ \alpha\gamma\delta,\alpha\beta^{\frac{f}{2} - n}, \beta^{n}\gamma\delta \}, \quad n \in ( \tfrac{f}{8},  \tfrac{f}{4} - 1 ); \\
\AVC &\equiv \{ \alpha\gamma\delta,\alpha^2\beta^{\frac{f}{2} - 2n}, \beta^{n}\gamma\delta \}, \quad n \in ( \tfrac{f}{8}, \tfrac{f}{4} - 1 );\\
\AVC &\equiv \{ \alpha\gamma\delta,\alpha^3\beta^{\frac{f}{2} - 3n}, \beta^{n}\gamma\delta \}, \quad n \in ( \tfrac{f}{8}, \tfrac{f}{6} - \tfrac{1}{3} ].
\end{align*} 
\end{case*}

\subsubsection*{\bm{$E^{\prime\prime}: \AVC \equiv \{ \alpha\gamma\delta,\alpha\beta^{n}, \gamma^k\delta^k \},  \{ \alpha\gamma\delta,\alpha\beta^{n}, \beta^{n_1}\gamma^k\delta^k \}$}}

Since $\beta^{n_1}\gamma\delta$ has been previously discussed, we may assume $k \ge 2$. This implies $\gamma + \delta \le \pi$. Recall $k \le 3$. Then we have $k=2,3$.

By $\alpha\gamma\delta,\alpha\beta^{n}$, the vertex angle sums imply $n\beta=\gamma+\delta$. In the $\AVC$s, we have $kn \le \frac{f}{2}$. If $\gamma^k\delta^k$ is a vertex, then $kn=\frac{f}{2}$ and if $\beta^{n_1}\gamma^k\delta^k$ is a vertex, then $kn < \frac{f}{2}$ and $n_1 = \frac{f}{2} - kn$.

\begin{case*}[Convex] Convexity means $\alpha \in [\pi-\beta, \pi]$. By $n\beta=\gamma+\delta$ and $\alpha\gamma\delta$, it is equivalent to $n\beta = \gamma+\delta \in [\pi, \pi+\beta]$. Combined with $\gamma+\delta \le \pi$, we get $\gamma+\delta=\pi$ and hence $k=2$ and $n  = \frac{f}{4}$ and $\alpha=\pi$. Hence
\begin{align*}
\AVC \equiv \{ \alpha\gamma\delta,\alpha\beta^{\frac{f}{4}}, \gamma^2\delta^2 \}.
\end{align*}
As $\alpha=\pi$, the quadrilateral is in fact the triangle $\triangle BCD$ and the tiling is non-edge-to-edge tiling by congruent triangles. An example is illustrated in Figure \ref{Tiling-NE2E-algade-albe3-ga2de2} for $f=12$.

\begin{figure}[htp] 

\minipage{0.5\textwidth}
\raggedleft
\begin{tikzpicture}[scale=1]

\tikzmath{
\r=0.8; 
\rr=\r/2;
\A = atan(\r/\rr);
\h = sqrt(\r^2+\rr^2);
\AA = atan(2*\r/\rr);
\hh = sqrt((2*\r)^2+\rr^2);
}

\coordinate (O1) at (0:\r/2);
\coordinate (O2) at (0:-\r/2);

\foreach \a in {0,1}{

\draw[rotate=180*\a]
	(0:0) -- (0:3*\r)
;

\draw[line width=1.6, rotate=180*\a]
	(0:0) -- (30:1.54*\r)
;

\coordinate (A) at (0+180*\a:2*\r); 
\coordinate (B) at (\A+180*\a:\h); 
\coordinate (C) at (180+180*\a:1*\r); 

\arcThroughThreePoints[]{A}{B}{C};

\coordinate (B2) at (\AA+180*\a:\hh); 

\arcThroughThreePoints[]{A}{B2}{C};

\draw[line width=1.6]
	(B) -- ([shift={(B)}]30+180*\a:1.2*\r)
;

\coordinate (B3) at (\AA+180*\a:\hh); 

\draw[line width=1.6]
	(B3) -- ([shift={(B3)}]90+180*\a:\r)
;
	
\draw[]
	(0:1*\r) -- (0:2*\r)
;

\coordinate (D) at (30+180*\a:1.54*\r);

\arcThroughThreePoints[]{D}{C}{B};

\coordinate (F) at ([shift={(\A+180*\a:\h)}]30+180*\a:1.2*\r);
	
\arcThroughThreePoints[]{F}{C}{B3};

}
\end{tikzpicture}
\endminipage \quad  \quad \quad
\minipage{0.5\textwidth}
\raggedright
		\adjustbox{trim=\dimexpr.5\Width-15mm\relax{} \dimexpr.5\Height-15mm\relax{}  \dimexpr.5\Width-15mm\relax{} \dimexpr.5\Height-15mm\relax{} ,clip}{\includegraphics[height=6cm]{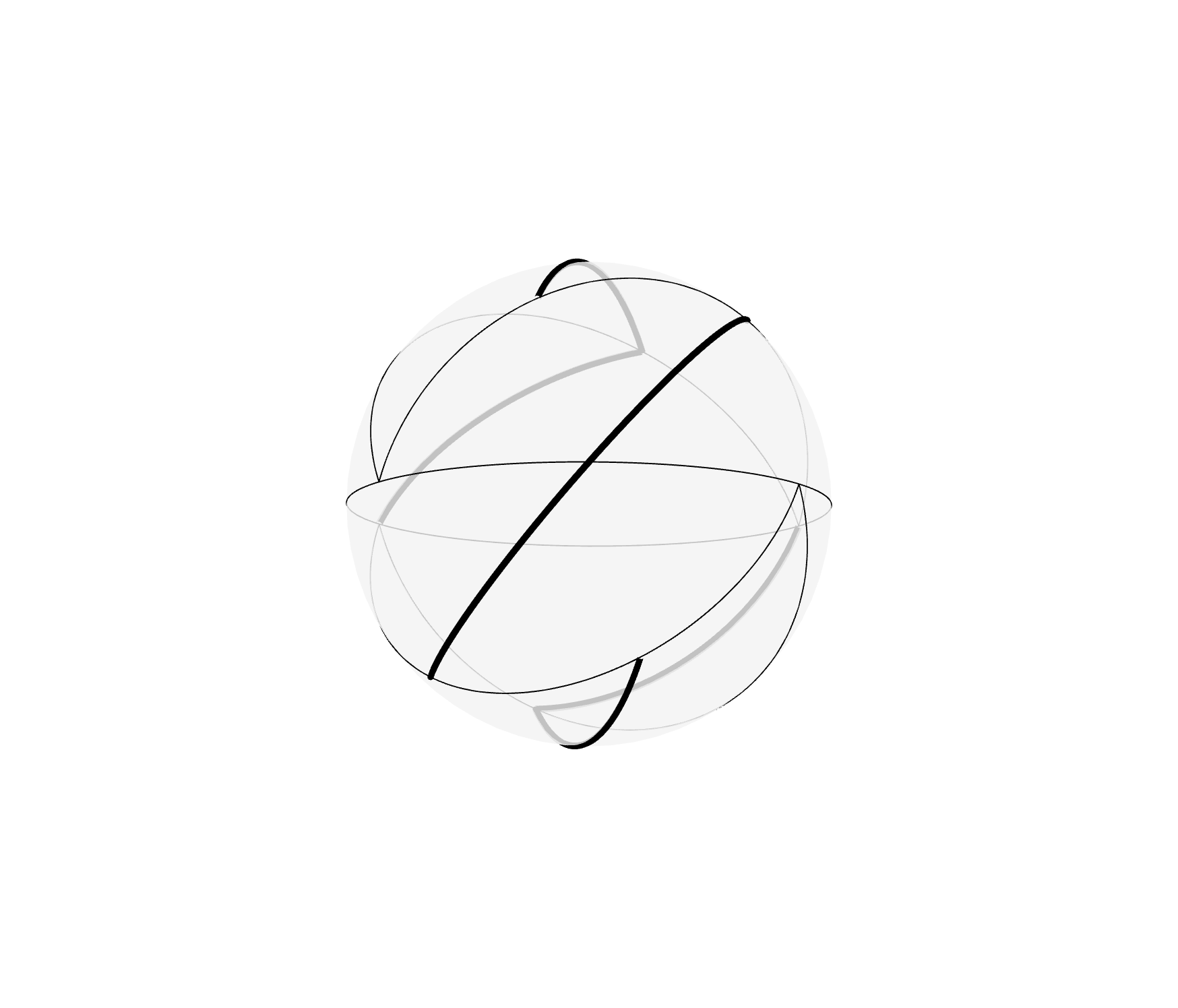}} 
\endminipage
\caption{Non-edge-to-edge tiling by congruent triangles given by $\AVC  \equiv \{ \alpha\gamma\delta, \alpha\beta^3, \gamma^2\delta^2 \}$ where $\alpha=\pi$}
\label{Tiling-NE2E-algade-albe3-ga2de2}
\end{figure}

\end{case*}

\begin{case*}[$\alpha>\pi$] For $\alpha>\pi$, we already know $\alpha \in (\pi, \frac{3}{2}\pi)$. By $\alpha\gamma\delta$ and $n\beta=\gamma+\delta$, we have $\alpha \in (\pi, \frac{3}{2}\pi)$ if and only if $n \beta = \gamma+\delta \in (\frac{1}{2}\pi, \pi)$. This is equivalent to $\frac{f}{8} < n  <\frac{f}{4}$. By $k=2,3$ and $\frac{f}{2}\ge kn$, we further have $n \le \frac{f}{4}, \frac{f}{6}$ respectively. Combined with $n \in (\frac{f}{8}, \frac{f}{4})$, we get
\begin{align*}
\AVC &\equiv \{ \alpha\gamma\delta,\alpha\beta^{n}, \beta^{\frac{f}{2} - 2n}\gamma^2\delta^2 \}, \quad n \in (\tfrac{f}{8}, \tfrac{f}{4}); \\
\AVC &\equiv \{ \alpha\gamma\delta,\alpha\beta^{n}, \beta^{\frac{f}{2} - 3n}\gamma^3\delta^3 \}, \quad n \in (\tfrac{f}{8}, \tfrac{f}{6}); \\
\AVC &\equiv \{ \alpha\gamma\delta,\alpha\beta^{\frac{f}{6}}, \gamma^3\delta^3 \}.
\end{align*}
\end{case*}

\subsubsection*{\bm{$E^{\prime\prime\prime}: \AVC \equiv \{ \alpha\gamma\delta, \gamma^3\delta, \alpha\beta^{\frac{f+2}{6}}, \alpha\beta^{\frac{f-4}{6}}\delta^2 \}$}}

From Proposition \ref{RatAlGaDeProp}, we already know $f\ge8$ and
\begin{align*}
\alpha =( \tfrac{4}{3} - \tfrac{4}{3f} )\pi, \quad
\beta = \tfrac{4}{f}\pi, \quad
\gamma = ( \tfrac{2}{3} - \tfrac{2}{3f} )\pi. \quad
\delta = \tfrac{2}{f}\pi.
\end{align*}
For each fixed $f\ge 8$, choosing this $\alpha$ guarantees a tile with required $\gamma, \delta$. The existence can also be shown by Lemma \ref{AEQuadExists} and Lemma \ref{b-criteria2}.

This completes the proof.
\end{proof}

We hereby conclude our study with the following two theorems.

\begin{theorem} Tilings of the sphere by congruent almost equilateral quadrilaterals, where every angle is rational, are earth map tilings $E$ and their flip modifications, $E^{\prime}, E^{\prime\prime}$, and rearrangement $E^{\prime\prime\prime}$, and isolated earth map tilings, $S3, S^{\prime}3, S5$, and special tiling $S6$. 
\end{theorem}

\begin{theorem} Tilings of the sphere by congruent almost equilateral quadrilaterals with some non-rational angles are earth map tilings $E$ and their flip modifications, $E^{\prime}, E^{\prime\prime}$, and isolated earth map tilings, $S1, S2$, and special tilings $QP_6,S4$.
\end{theorem}

\end{document}